\UseRawInputEncoding
\documentclass[12pt]{article}
\usepackage[T1]{fontenc}
\usepackage{lmodern}
\usepackage[margin=1in]{geometry}
\usepackage{amsmath,amssymb,amsthm,mathtools}
\usepackage{microtype,indentfirst}
\usepackage{tikz}
\usepackage{pgfplots}
\pgfplotsset{compat=newest}
\definecolor{revision}{RGB}{0,0,0}
\definecolor{editorial}{RGB}{0,0,0}
\usepackage[section]{placeins}
\definecolor{annBlue}{RGB}{0,114,178}
\definecolor{annOrange}{RGB}{230,159,0}
\DeclareSymbolFont{uprightintegrals}{U}{euex}{m}{n}
\DeclareMathSymbol{\intop}{\mathop}{uprightintegrals}{"52}
\DeclareMathSymbol{\ointop}{\mathop}{uprightintegrals}{"48}
\usepackage[colorlinks=false,pdfborder={0 0 0},pdfborderstyle={}]{hyperref}
\hypersetup{pdftitle={Quantitative Polya inequalities for annular sectors},
pdfauthor={Guowei Dai and Yingxin Sun},
pdfkeywords={Polya conjecture, annular sectors, Dirichlet and Neumann eigenvalues, quantitative spectral inequalities, continuous angular order}}
\DeclareRobustCommand{\revblue}[1]{#1}
\numberwithin{equation}{section}
\newtheorem{theorem}{Theorem}[section]
\newtheorem{lemma}[theorem]{Lemma}
\newtheorem{proposition}[theorem]{Proposition}
\newtheorem{corollary}[theorem]{Corollary}
\theoremstyle{remark}

\newcommand{\R}{\mathbb R}
\newcommand{\dd}{\,\mathrm{d}}
\newcommand{\paren}[1]{\left(#1\right)}
\newcommand{\bracket}[1]{\left[#1\right]}
\newcommand{\set}[1]{\left\{#1\right\}}
\newcommand{\abs}[1]{\left|#1\right|}
\newcommand{\floor}[1]{\left\lfloor#1\right\rfloor}
\newcommand{\pos}[1]{\left(#1\right)_{+}}
\title{\revblue{Quantitative P\'olya inequalities for annular sectors}\thanks{Research
supported by the National Natural Science Foundation of China (Grant No.~12371110).}}
\author{Guowei Dai\thanks{School of Mathematical Sciences, Dalian
University of Technology, Dalian 116024, P.R. China.
Email: \href{mailto:daiguowei@dlut.edu.cn}{daiguowei@dlut.edu.cn}.}
\and Yingxin Sun\thanks{School of Mathematical Sciences, Dalian
University of Technology, Dalian 116024, P.R. China.
Email: \href{mailto:sunyingxin2023@mail.dlut.edu.cn}{sunyingxin2023@mail.dlut.edu.cn}.}}
\date{}
\begin{document}
\maketitle
{\begin{abstract}
{\color{editorial}We prove explicit positive corrections to the Dirichlet and Neumann
P\'olya inequalities at every positive eigenvalue index of every
annular sector. The coefficients depend only on the radii and are
uniform in the aperture. Real-order angular integration preserves
the exact area term, while directional rounding estimates and radial
trials control all frequencies. Optimized bounds attain the sharp
circular-arc coefficient for the continuous count. We identify the
full sector perimeter asymptotic and a low-index obstruction to
using its coefficient throughout the spectrum. The strict Neumann
inequality also passes to products with bounded Lipschitz domains.}

\end{abstract}}


\section{Introduction and main results}\label{sec:introduction}
{\color{editorial}Dirichlet and Neumann Laplacian eigenvalues model membrane and
acoustic frequencies~\cite{CourantHilbert}. Annular sectors provide
a separable geometry in which radial width, angular aperture, and
boundary conditions enter the estimates explicitly.}

{For a bounded connected Lipschitz domain $\Omega\subset\R^m$, write
its eigenvalues, with multiplicities, as}
{\[
 0<\lambda_1\paren\Omega\le\lambda_2\paren\Omega\le\cdots,
 \qquad
 0=\mu_0\paren\Omega<\mu_1\paren\Omega\le\mu_2\paren\Omega\le\cdots.
\]}
{If $\omega_m$ is the volume of the unit ball in $\R^m$, Weyl's
law~\cite{Weyl,Davies} identifies the leading scale}
{\[
 W_k\paren\Omega=4\pi^2
 \paren{\frac{k}{\omega_m\abs\Omega}}^{2/m},\qquad
 \frac{\lambda_k\paren\Omega}{W_k\paren\Omega},\,
 \frac{\mu_k\paren\Omega}{W_k\paren\Omega}\longrightarrow1
 \quad\paren{k\to+\infty}.
\]}
P\'olya's conjecture asks for the individual inequalities
{\[
\mu_k\paren\Omega\le W_k\paren\Omega\le\lambda_k\paren\Omega,
 \qquad k\ge1.
\]}
{\color{editorial}The conjecture requires the exact Weyl coefficient at every index.
In the plane, $W_k=4\pi k/\abs\Omega$. We seek explicit positive
corrections to both comparisons that remain valid throughout the
spectrum.}

{P\'olya proved the inequalities for planar tiling domains by
geometric replication and spectral bracketing~\cite{Polya1961};
Kellner removed the additional restriction in the Neumann tiling
argument~\cite{Kellner}. The sharp averaged bounds of Berezin,
Li--Yau, and Kr\"oger control spectral sums on general
domains~\cite{Berezin,LiYau,Kroger}. Laptev's product
principle~\cite{Laptev}, the non-tiling families of Freitas and
Salavessa~\cite{FS}, and the thin-product results of He and
Wang~\cite{HW} identify further settings where the individual
Weyl coefficient can be recovered. See~\cite{Henrot} for a broader
account.}

Separation of variables is especially effective for circular
geometries. Filonov, Levitin, Polterovich, and Sher proved both
inequalities for disks and circular sectors of arbitrary aperture,
and the Dirichlet inequality for balls in all
dimensions~\cite{Balls}. They also proved the Dirichlet inequality
for every circular annulus~\cite{DirichletAnnuli}. The remaining
Neumann ball cases are treated in~\cite{NeumannBalls,LiNeumann}.
Song and Zou obtained Neumann estimates for annuli and annular
sectors~\cite{SZ}. In our earlier work, we proved the strict Neumann
inequality for every complete circular annulus~\cite[Theorem~1.1]{DaiSunAnnuli}.
We also obtained positive corrections of square-root order, with an
explicit bound for $k\ge1604$ and a universal, non-explicit coefficient
for all indices~\cite[Corollary~6.2 and Theorem~6.9]{DaiSunAnnuli}.
Congruent subdivision gives strict Neumann inequalities for
sector apertures $2\pi/q$ with integer
$q\ge2$~\cite[Corollary~1.3]{DaiSunAnnuli}; sufficiently fine
angular subdivisions give further families~\cite[Section~3.1]{FS}.
In a sector, however, the angular orders $j\pi/\beta$ have
multiplicity one and arbitrary spacing. The integer orders and
double multiplicities of a complete annulus do not control all
these lattices by themselves.

{\color{editorial}Tang and Wang proved the strict Dirichlet P\'olya inequality for
every bounded simply connected planar
domain~\cite[Theorem~1.2]{TangWangDirichlet}, including all sectors
considered here and the retained slit at aperture $2\pi$. Their
Remark~1.4 also gives a positive square-root correction for smooth
simply connected domains, with a domain-dependent constant not
specified by a geometric formula. Our Dirichlet contribution is
quantitative: positive coefficients given explicitly by the two
radii, uniform in aperture and valid at every index, also in the
presence of corners and the slit. The Neumann estimates provide
the corresponding explicit all-index subtraction for arbitrary
aperture. Optimized versions additionally give explicit errors at
the circular-arc scale. The radial phase formulas and Bessel-phase
bounds are taken from~\cite{DaiSunAnnuli,DirichletAnnuli,Phase};
their use is identified in the proofs.}

{\color{editorial}A rough polar-metric comparison loses the area coefficient, while
order-one errors in the radial count can consume the boundary gain.
We integrate over real angular order before sampling the sector
lattice, so the action integral gives the area term exactly.
Pairing opposite half-periods of the rounding error against a
decreasing weight yields the sharp one-sided constants $1/3$ and
$\sqrt2/6$. Radial comparisons and optimized trials cover the
turning region and low frequencies, which a two-term asymptotic
law~\cite{Ivrii} alone does not control.}

{The proof is entirely analytic. The endpoint constants used
to join the frequency ranges are proved in
Appendix~\ref{sec:elementary-estimates}. The figures and decimal
approximations illustrate the formulas. The validity of the
refined bounds, including those defined by an extremum or a
supremum, is established by the analytic arguments; the closed
all-index conclusions require no numerical optimization.}

{For $0<r<R$ and $0<\beta<2\pi$, define}
\[
 S_{r,R,\beta}=\set{\paren{s\cos\theta,s\sin\theta}:
 r<s<R,\quad0<\theta<\beta},\qquad
 \abs{S_{r,R,\beta}}=\frac\beta2\paren{R^2-r^2}.
\]
{At $\beta=2\pi$, the two radial faces are retained as a slit.
The Dirichlet and Neumann forms are defined on the polar rectangle
$\paren{r,R}\times\paren{0,2\pi}$ with their respective boundary
conditions and with no identification of angular traces. This
endpoint has the sector angular spectrum. Write}
{\[
 N_S^{\mathrm D}\paren{E}=\#\set{k\ge1:\lambda_k\paren{S}\le E},
 \qquad N_S^{<}\paren{E}=\#\set{k\ge0:\mu_k\paren{S}<E}.
\]}
{The strict Neumann cutoff is essential when passing from a
counting estimate to a strict individual bound.}

{To explain the scale of the corrections, recall the planar
\emph{two-term Weyl law}.  For a fixed domain of area $A$ and
boundary length $P$, its form is
\begin{equation}\label{eq:intro-two-term-weyl}
 \begin{aligned}
 N^{\mathrm D}\paren E
   &=\underbrace{\frac A{4\pi}E}_{\text{area term}}
     -\underbrace{\frac P{4\pi}\sqrt E}_{\text{boundary term}}
     +o\paren{\sqrt E},\\
 N^{\mathrm N,<}\paren E
   &=\underbrace{\frac A{4\pi}E}_{\text{area term}}
     +\underbrace{\frac P{4\pi}\sqrt E}_{\text{boundary term}}
     +o\paren{\sqrt E},\qquad E\longrightarrow+\infty.
 \end{aligned}
\end{equation}
The first term measures the interior; the second measures the
boundary and has opposite signs for the two boundary conditions.
This is the classical Weyl boundary correction, rather than a new
formula of the present paper.  Ivrii proved the two-term law for
smooth boundaries under the hypothesis that periodic billiard
trajectories have measure zero; see~\cite{Ivrii} and
\cite[Section~1.1, formulas~(1.1.2) and~(1.1.4)]{IvriiSurvey}.
We do not assume that this general theorem applies to every
Lipschitz domain or directly to the slit endpoint.  For the annular
sectors here, Theorem~\ref{thm:Bq-full-perimeter} proves
\eqref{eq:intro-two-term-weyl} directly, including the strict and
inclusive counting conventions.}

{For $S=S_{r,R,\beta}$ the boundary has two circular arcs and two
radial sides, with total lengths
\[
 P_{\rm arc}=\beta\paren{R+r},\qquad
 P_{\rm side}=2\paren{R-r},\qquad P=P_{\rm arc}+P_{\rm side}.
\]
At $\beta=2\pi$ the radial sides are the two retained faces of the
slit and are both counted in this boundary length.  Since
$A=\beta\paren{R-r}\paren{R+r}/2$, we obtain
\begin{equation}\label{eq:intro-arc-coefficient}
 \frac{P_{\rm arc}}A=\frac2{R-r},\qquad
 \frac{P_{\rm side}}A=\frac4{\beta\paren{R+r}},\qquad
 \frac PA=\frac2{R-r}+\frac4{\beta\paren{R+r}}.
\end{equation}
The phrase \emph{full circular-arc coefficient} will mean the
first of these ratios: it includes both arcs.  It differs from
the full perimeter coefficient $P/A$, which includes the radial
sides as well.  In the counting function the corresponding
coefficients are $P_{\rm arc}/\paren{4\pi}$ and $P/\paren{4\pi}$;
after inversion to individual eigenvalues they become
$P_{\rm arc}/A$ and $P/A$.  Figure~\ref{fig:intro-boundary-parts}
shows this geometric decomposition.}

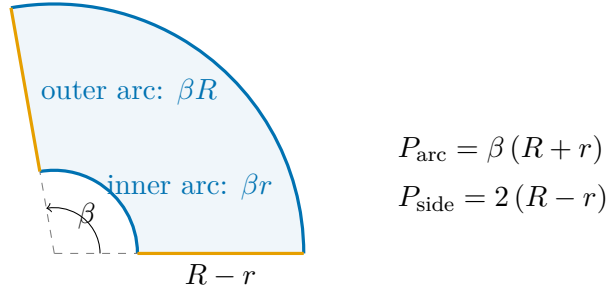
\begin{figure}[htbp]
\centering
\begin{tikzpicture}[scale=1.1,font=\small]
 \fill[annBlue,opacity=.06] (1,0)--(3,0)
   arc (0:100:3)
   --({cos(100)},{sin(100)})
   arc (100:0:1)--cycle;
 \draw[annBlue,very thick] (3,0)
   arc (0:100:3);
 \draw[annBlue,very thick] (1,0)
   arc (0:100:1);
 \draw[annOrange,very thick] (1,0)--(3,0);
 \draw[annOrange,very thick] ({cos(100)},{sin(100)})
   --({3*cos(100)},{3*sin(100)});
 \draw[gray,dashed] (0,0)--(1,0);
 \draw[gray,dashed] (0,0)--({cos(100)},{sin(100)});
 \draw[->] (.55,0) arc (0:100:.55);
 \node at (.39,.45) {$\beta$};
 \node[below] at (2,0) {$R-r$};
 \node[annBlue] at (.9,1.95) {outer arc: $\beta R$};
 \node[annBlue] at (1.63,.8) {inner arc: $\beta r$};
 \node[anchor=west] at (4,.95)
   {\begin{tabular}{@{}l@{}}$P_{\rm arc}=\beta\paren{R+r}$\\[5pt]
    $P_{\rm side}=2\paren{R-r}$\end{tabular}};
\end{tikzpicture}
\caption{Boundary lengths in an annular sector (schematic).
Blue marks the two circular arcs; orange marks the two radial
sides.  The optimized continuous angular counts recover the full
blue contribution.  The additional orange contribution is
identified by the discrete angular calculation in
Section~\ref{sec:boundary-optimality}.}
\label{fig:intro-boundary-parts}
\end{figure}

{The motivation for our estimates is now precise.  The boundary
term suggests a correction proportional to $\sqrt E$, but the
unspecified $o\paren{\sqrt E}$ in the two-term law does not give
its sign or size at a prescribed energy.  In particular it does
not prove a P\'olya inequality at every index.  Our main theorems
replace this asymptotic information by explicit positive
corrections valid at every $E>0$.  The stronger estimates below
then recover the circular-arc scale with explicit losses.}

{Put $\delta=\paren{R-r}/r$ and define}
{\[
a_{\mathrm D}=\frac{4r c_\delta}{R^2-r^2},\qquad
 a_{\mathrm N}=\frac{4r b_\delta^{\mathrm c}}{R^2-r^2},
\]}
{where the explicit positive numbers $c_\delta$ and
$b_\delta^{\mathrm c}$ are given in \eqref{eq:Dq-c} and
\eqref{eq:Nq-uniform-b}. They require no spectral data.}

{\begin{theorem}[Dirichlet correction]\label{thm:dirichlet}
Let $0<r<R$ and $0<\beta\le2\pi$. For every $E>0$,
\[
 N_{S_{r,R,\beta}}^{\mathrm D}\paren{E}
 \le\frac{\abs{S_{r,R,\beta}}}{4\pi}
       \pos{E-a_{\mathrm D}\sqrt E}.
\]
For every $k\ge1$, with $W_k=4\pi k/\abs{S_{r,R,\beta}}$,
\[
 \lambda_k\paren{S_{r,R,\beta}}
 \ge\paren{\sqrt{W_k+a_{\mathrm D}^2/4}+a_{\mathrm D}/2}^2
 =W_k+a_{\mathrm D}\sqrt{W_k+a_{\mathrm D}^2/4}
       +a_{\mathrm D}^2/2>W_k.
\]
\end{theorem}}

{\begin{theorem}[Neumann correction]\label{thm:main}
Let $0<r<R$ and $0<\beta\le2\pi$. For every $E>0$,
\[
 N_{S_{r,R,\beta}}^{<}\paren{E}
 >\frac{\abs{S_{r,R,\beta}}}{4\pi}
       \paren{E+a_{\mathrm N}\sqrt E}.
\]
For every $k\ge1$, with $W_k=4\pi k/\abs{S_{r,R,\beta}}$,
\[
 \mu_k\paren{S_{r,R,\beta}}
 <\paren{\sqrt{W_k+a_{\mathrm N}^2/4}-a_{\mathrm N}/2}^2
 =W_k-a_{\mathrm N}
       \paren{\sqrt{W_k+a_{\mathrm N}^2/4}-a_{\mathrm N}/2}<W_k.
\]
\end{theorem}}

{The two main theorems give simple closed coefficients
$a_{\mathrm D},a_{\mathrm N}>0$ that are uniform in $\beta$.
Section~\ref{sec:dirichlet-quantitative} constructs $c_\delta$
and proves Theorem~\ref{thm:dirichlet}; its full-arc estimate and
discrete radial-row inverse further strengthen that theorem.
For Neumann conditions, Section~\ref{sec:neumann} first proves
the strict continuous area comparison.
Section~\ref{sec:neumann-quantitative} retains its positive
margins, optimizes the trial functions and the turning cutoff,
and extracts $b_\delta^{\mathrm c}$ to prove
Theorem~\ref{thm:main}.  Its envelope inverse
\eqref{eq:Nq-envelope-inverse} and the Dirichlet inverse
\eqref{eq:Dq-sigma-eigen} give stronger, frequency-dependent
corrections, with leading term
$2\sqrt{W_k}/\paren{R-r}$.}

{These improvements are possible because the proof controls more
than the sign required by P\'olya's conjecture.  Integration of
the radial action already gives the exact area term, so the
remaining task is to estimate a surplus or deficit without
changing that coefficient. At the eigenvalue frequencies, the
main theorems use a single positive lower bound for this margin
divided by frequency.
The later formulas retain the margin's dependence on frequency,
use the individual radial rows instead of their integral when
advantageous, and reduce the explicit phase losses by optimizing
the trial and cutoff parameters.  Each refinement retains
information discarded in the simpler bound.  The additional
boundary-scale gain is justified by estimates for the remaining
errors, rather than by an assumption that an asymptotic expansion
has the correct sign at every energy.}

{Section~\ref{sec:boundary-optimality} has a different logical
role: it supplies an asymptotic benchmark for these estimates.
It proves that the circular-arc coefficient is sharp for the
continuous angular counts, and explains how sampling the angular
orders contributes the radial sides in the full sector law.
Its low-index example shows why this larger perimeter
coefficient cannot simply replace the explicit all-index
constants.  Thus the two-term law explains the target scale,
while the finite-frequency estimates prove the main theorems;
the latter do not follow by discarding an asymptotic remainder.
Optimality of a cutoff or trial below is understood within its
stated family, and is not a claim that the closed coefficients
$a_{\mathrm D}$ and $a_{\mathrm N}$ are globally best possible.
Finally, Corollary~\ref{cor:product} transfers the strict
Neumann inequality to products with bounded connected Lipschitz
factors.}

{Section~\ref{sec:reduction} develops the real-order reduction and
rounding estimates. Section~\ref{sec:dirichlet-quantitative} proves
the Dirichlet bounds. Sections~\ref{sec:neumann} and
\ref{sec:neumann-quantitative} establish and optimize the Neumann
count. Section~\ref{sec:boundary-optimality} identifies the boundary
coefficients and their scope, and Section~\ref{sec:products}
proves the product consequence.
\textcolor{editorial}{Lemma~\ref{lem:endpoint-estimates} in
Appendix~\ref{sec:elementary-estimates} collects the endpoint bounds
used to make the Neumann parameter ranges overlap.}}

\section{Preliminaries and continuous angular order}\label{sec:reduction}
\subsection{The radial family and angular transfer}
{Scaling all lengths by $r^{-1}$ multiplies the eigenvalues by $r^2$
and the area by $r^{-2}$. Their product, and hence the desired
counting inequality, is invariant. Normalize}
{\[
 r=1,\qquad R=1+\delta,\qquad\delta>0,
 \qquad D=R^2-1=\delta\paren{2+\delta}.
\]}
{The frequency $t$ is the square root of the energy, and $x=\delta t$
measures radial oscillation on the width scale.}

{For each real angular order $v\ge0$, consider the form}
\[
 a_v\bracket{u}
 =\int_1^R\paren{s\abs{u'}^2+\frac{v^2}{s}\abs{u}^2}\dd s
\]
{in $L^2\paren{\paren{1,R},s\dd s}$, with domain
$H^1\paren{1,R}$. Its eigenvalues $\lambda_{v,n}$, $n\ge0$,
solve}
{\[
 -u''-\frac1s u'+\frac{v^2}{s^2}u=\lambda u,
 \qquad u'\paren{1}=u'\paren{R}=0.
\]}
{This is the radial family of~\cite[Section~2.1]{DaiSunAnnuli},
with the order allowed to vary continuously. Define}
{\[
 n_v\paren{t^2}=\#\set{n\ge0:\lambda_{v,n}\le t^2},
 \qquad
 M\paren{t}=2\int_0^{+\infty}n_v\paren{t^2}\dd v.
\]}
{For Dirichlet conditions, use the same form $a_v$ on
$H_0^1(1,R)$ and write its eigenvalues as
$\lambda^{\mathrm D}_{v,n}$, $n\ge1$. Define
\[
 n_v^{\mathrm D}(t^2)
   =\#\{n\ge1:\lambda^{\mathrm D}_{v,n}\le t^2\},\qquad
 M_{\mathrm D}(t)=2\int_0^{+\infty}n_v^{\mathrm D}(t^2)\dd v.
\]
The notation without a superscript continues to denote the
Neumann family. Both families decrease in count as the angular
order increases; their different angular sampling conventions
will give the two comparisons below.}

{The factor two is convenient for the action identity. Our principal
radial estimate is}
\begin{proposition}\label{prop:continuous}
For every $\delta>0$ and $t>0$,
\[
 M\paren{t}>\frac D4t^2.
\]
\end{proposition}
{It is proved in Section~\ref{sec:neumann}. The next lemma supplies
the monotonicity and continuity needed for its strict angular transfer.}

{\begin{lemma}\label{lem:order-continuity}
For $v\ge w\ge0$ and every $n\ge0$,
\[
 \frac{v^2-w^2}{R^2}
 \le\lambda_{v,n}-\lambda_{w,n}\le v^2-w^2.
\]
In particular, $n_v\paren{t^2}$ decreases with $v$, vanishes for
$v>Rt$, and $M$ is continuous on $\paren{0,+\infty}$. The same order bounds,
monotonicity, support bound, and continuity hold for
$\lambda^{\mathrm D}_{v,n}$, $n\ge1$, $n_v^{\mathrm D}$,
and $M_{\mathrm D}$.
\end{lemma}}
\begin{proof}
{For every nonzero $u\in H^1\paren{1,R}$, the difference between
the two Rayleigh quotients is}
{\[
 \frac{a_v\bracket{u}-a_w\bracket{u}}
      {\int_1^R s\abs{u}^2\dd s}
 =\paren{v^2-w^2}
   \frac{\int_1^R s^{-1}\abs{u}^2\dd s}
        {\int_1^R s\abs{u}^2\dd s}.
\]}
{Since $R^{-2}\le s^{-2}\le1$, the last ratio lies in
$\bracket{R^{-2},1}$. Applying the min--max principle
to every finite-dimensional trial space gives both eigenvalue
bounds~\cite[Chapter~4]{Davies}. Taking $w=0$ and using
$\lambda_{0,n}\ge0$ gives $\lambda_{v,n}\ge v^2/R^2$.
Thus $n_v\paren{t^2}=0$ when $v>Rt$.}

{We also need a uniform bound on the number of radial eigenvalues.
The inequalities $1\le s\le R$ give}
{\[
 \frac{a_0\bracket{u}}{\int_1^R s\abs{u}^2\dd s}
 \ge\frac1R\frac{\int_1^R\abs{u'}^2\dd s}
                         {\int_1^R\abs{u}^2\dd s},
 \qquad
 \lambda_{0,n}\ge\frac1R\paren{\frac{\pi n}{\delta}}^2.
\]}
{Consequently, for $0<t\le B$ only
$n\le\delta B\sqrt R/\pi$ can contribute and the order support
is contained in $\bracket{0,RB}$.}

{Fix $t_*>0$. Each function $v\mapsto\lambda_{v,n}$ is continuous
by the two-sided bound and strictly increasing on
$\bracket{0,+\infty}$. Therefore the equation
$\lambda_{v,n}=t_*^2$ has at most one solution for each relevant
$n$, and these exceptional orders form a finite set. If
$t_j\to t_*$, then}
{\[
 n_v\paren{t_j^2}\longrightarrow n_v\paren{t_*^2}
 \quad\text{for almost every }v.
\]}
{The common support and radial-index bound just obtained allow
dominated convergence in the definition of $M$. This proves continuity.
It also shows why including a radial eigenvalue at its energy
threshold does not introduce a discontinuity after order integration.}
{For Dirichlet conditions, restrict the form domain to
$H_0^1(1,R)$ in the same min--max argument. The zero-order
radial comparison is now
\[
 \lambda^{\mathrm D}_{0,n}\ge\frac1R
       \left(\frac{\pi n}{\delta}\right)^2,\qquad n\ge1,
\]
so there is again a uniform finite number of contributing rows
on bounded frequency intervals. The same dominated-convergence
argument proves continuity of $M_{\mathrm D}$.}
\end{proof}

{\color{revision}
\begin{proposition}[Angular transfer and the strict energy cutoff]
\label{prop:angular-transfer}
For normalized radii $1<R$ and $0<\beta\le2\pi$, let
$S=S_{1,R,\beta}$ and $p=\pi/\beta$.
For every $E>0$,
\[
 N_S^{<}\paren{E}\ge\frac{M\paren{\sqrt E}}{2p}.
\]
For the inclusive Dirichlet count,
\[
 N_S^{\mathrm D}\paren E\le\frac{M_{\mathrm D}(\sqrt E)}{2p}.
\]
The two comparisons are illustrated in Figure~\ref{fig:sampling}.
Consequently, the Neumann estimate in
Proposition~\ref{prop:continuous} implies
\[
 N_S^{<}\paren{E}>\frac{\abs S}{4\pi}E,
 \qquad \mu_k\paren S<\frac{4\pi k}{\abs S}\quad(k\ge1).
\]
More generally, any lower bound $M\paren t>Dt^2/4+tB\paren t$
transfers to
\[
 N_S^{<}\paren E>\frac{\abs S}{4\pi}E
              +\frac{\sqrt E}{2p}B\paren{\sqrt E}.
\]
\end{proposition}}
\begin{proof}
{\color{revision}Write $N_S^{\le}\paren E$ for the Neumann count including
eigenvalues equal to $E$. The angular Neumann eigenfunctions are
$\cos\paren{jp\theta}$, $j\ge0$, each with multiplicity one.
Separation of variables gives}
{\[
 N_S^{\le}\paren{t^2}=\sum_{j=0}^{+\infty}n_{jp}\paren{t^2}.
\]}
{For $jp\le v\le\paren{j+1}p$, monotonicity implies
$n_v\paren{t^2}\le n_{jp}\paren{t^2}$. Integrating on each
interval and summing gives the left-endpoint comparison}
{\[
 pN_S^{\le}\paren{t^2}\ge\int_0^{+\infty}n_v\paren{t^2}\dd v
 =\frac12M\paren{t}.
\]}
{For Dirichlet conditions the angular eigenfunctions are
$\sin(jp\theta)$, $j\ge1$. On $[(j-1)p,jp]$,
monotonicity instead gives
$n_v^{\mathrm D}(t^2)\ge n_{jp}^{\mathrm D}(t^2)$, and hence
\[
 pN_S^{\mathrm D}(t^2)
 =p\sum_{j\ge1}n_{jp}^{\mathrm D}(t^2)
 \le\int_0^{+\infty}n_v^{\mathrm D}(t^2)\dd v
 =\frac12M_{\mathrm D}(t).
\]
This proves the Dirichlet transfer with its inclusive cutoff.
It will yield quantitative bounds once the continuous deficit
has been estimated in Section~\ref{sec:dirichlet-quantitative}.}
{The Neumann rectangle comparison above uses an inclusive energy cutoff. To obtain
strict eigenvalue inequalities, fix $E>0$ and let $E_j\uparrow E$
with $0<E_j<E$. The sector spectrum is discrete, so
$N_S^{\le}\paren{E_j}\to N_S^{<}\paren{E}$.
{\color{revision}Passing to the limit using the continuity of $M$ from
Lemma~\ref{lem:order-continuity} gives the first inequality below;
Proposition~\ref{prop:continuous}, proved in Section~\ref{sec:neumann},
gives the strict second inequality:}}
{\[
 pN_S^{<}\paren{E}
 \ge\frac12M\paren{\sqrt E}
 >\frac D8 E.
\]}
{Since $D/\paren{8p}=\abs{S}/\paren{4\pi}$, this gives the qualitative
strict counting inequality. Taking $E=4\pi k/\abs{S}$ gives
$N_S^{<}\paren{E}>k$: at least $k+1$ eigenvalues, starting with
$\mu_0$, lie strictly below $E$, and hence $\mu_k\paren{S}<E$.
Conversely, the strict eigenvalue inequalities imply the strict
counting bound by choosing $k=\floor{\abs{S} E/\paren{4\pi}}$
when $k\ge1$, and by the constant mode when $k=0$.
{\color{revision}For physical radii $r<rR$, the general remainder transfers as
\[
 N_{S_{r,rR,\beta}}^{<}\paren E
 >\frac{\abs{S_{r,rR,\beta}}}{4\pi}E
   +\frac{r\sqrt E}{2p}B\paren{r\sqrt E}.
\]
This is the form used in Section~\ref{sec:neumann-quantitative}.}}
\end{proof}

{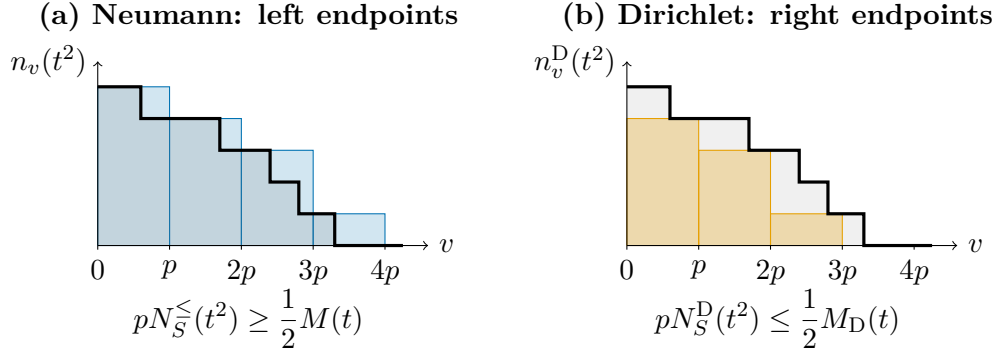
\begin{figure}[htbp]
\centering
\begin{tikzpicture}[x=.95cm,y=.42cm,font=\small]
 \begin{scope}
  \node[font=\small\bfseries] at (2.1,7.2)
   {(a) Neumann: left endpoints};
  \foreach \j/\h in {0/5,1/4,2/3,3/1}{
   \fill[annBlue,opacity=.18] (\j,0) rectangle ({\j+1},\h);
   \draw[annBlue] (\j,0) rectangle ({\j+1},\h);
  }
  \fill[gray,opacity=.18] (0,0)--(0,5)--(.6,5)--(.6,4)
   --(1.7,4)--(1.7,3)--(2.4,3)--(2.4,2)--(2.8,2)
   --(2.8,1)--(3.3,1)--(3.3,0)--cycle;
  \draw[black,very thick] (0,5)--(.6,5)--(.6,4)--(1.7,4)
   --(1.7,3)--(2.4,3)--(2.4,2)--(2.8,2)--(2.8,1)
   --(3.3,1)--(3.3,0)--(4.25,0);
  \draw[->] (0,0)--(4.6,0) node[right] {$v$};
  \draw[->] (0,0)--(0,5.8) node[left] {$n_v(t^2)$};
  \foreach \j/\lab in {0/0,1/p,2/{2p},3/{3p},4/{4p}}{
   \draw (\j,0)--(\j,-.13) node[below] {$\lab$};
  }
  \node[anchor=north] at (2.1,-1.05)
   {$\displaystyle pN_S^{\le}(t^2)\ge\frac12M(t)$};
 \end{scope}
 \begin{scope}[xshift=7cm]
  \node[font=\small\bfseries] at (2.1,7.2)
   {(b) Dirichlet: right endpoints};
  \fill[gray,opacity=.12] (0,0)--(0,5)--(.6,5)--(.6,4)
   --(1.7,4)--(1.7,3)--(2.4,3)--(2.4,2)--(2.8,2)
   --(2.8,1)--(3.3,1)--(3.3,0)--cycle;
  \foreach \j/\h in {1/4,2/3,3/1}{
   \fill[annOrange,opacity=.28] ({\j-1},0) rectangle (\j,\h);
   \draw[annOrange] ({\j-1},0) rectangle (\j,\h);
  }
  \draw[black,very thick] (0,5)--(.6,5)--(.6,4)--(1.7,4)
   --(1.7,3)--(2.4,3)--(2.4,2)--(2.8,2)--(2.8,1)
   --(3.3,1)--(3.3,0)--(4.25,0);
  \draw[->] (0,0)--(4.6,0) node[right] {$v$};
  \draw[->] (0,0)--(0,5.8) node[left] {$n_v^{\mathrm D}(t^2)$};
  \foreach \j/\lab in {0/0,1/p,2/{2p},3/{3p},4/{4p}}{
   \draw (\j,0)--(\j,-.13) node[below] {$\lab$};
  }
  \node[anchor=north] at (2.1,-1.05)
   {$\displaystyle pN_S^{\mathrm D}(t^2)\le\frac12M_{\mathrm D}(t)$};
 \end{scope}
\end{tikzpicture}
\caption{Angular transfer for both boundary conditions.
(a) The Neumann left-endpoint rectangles cover the region under
the radial count; angular rows start at $j=0$.
(b) The Dirichlet right-endpoint rectangles lie inside that
region; angular rows start at $j=1$.
The staircase areas are $M(t)/2=\int_0^{+\infty}n_v(t^2)\dd v$
and $M_{\mathrm D}(t)/2=\int_0^{+\infty}n_v^{\mathrm D}(t^2)\dd v$.
For Neumann conditions the strict energy cutoff follows by
continuity in this integral. Theorems~\ref{thm:dirichlet}
and~\ref{thm:main} use these comparisons together with the
continuous deficit and surplus proved later. The staircases
are schematic, rather than an assertion that the two actual
radial counts coincide.}
\label{fig:sampling}
\label{fig:Dq-angular-transfer}
\end{figure}}

\subsection{The action and directional rounding}\label{sec:action}
{For $x>v\ge0$, let}
\[
 A_v\paren{x}=\sqrt{x^2-v^2}-v\arccos\paren{v/x},
\]
{with $A_v\paren{v}=0$ by continuity. The radial action is}
\[
 f_t\paren{v}=
 \begin{cases}
  \paren{A_v\paren{Rt}-A_v\paren{t}}/\pi,&0\le v\le t,\\
  A_v\paren{Rt}/\pi,&t\le v\le Rt,\\
  0,&v\ge Rt.
 \end{cases}
\]
{Equivalently, writing $\pos{z}=\max\set{z,0}$,}
{\[
 f_t\paren{v}=\frac1\pi\int_1^R\sqrt{\pos{t^2-v^2/s^2}}\dd s.
\]}
{This representation makes the exact area term transparent.
The calculation is the real-order action identity
of~\cite[Lemma~2.5]{DaiSunAnnuli}; explicitly, Tonelli's theorem
and $v=st\eta$ give}
{\[
\begin{aligned}
 2\int_0^{Rt}f_t\paren{v}\dd v
 &=\frac2\pi\int_1^R\int_0^{st}
                  \sqrt{t^2-v^2/s^2}\dd v\dd s\\
 &=\frac{2t^2}\pi\int_1^R s\dd s
                    \int_0^1\sqrt{1-\eta^2}\dd\eta
 =\frac{2t^2}\pi\frac{R^2-1}{2}\frac\pi4.
\end{aligned}
\]}
{Thus}
\begin{equation}\label{eq:action}
 2\int_0^{Rt}f_t\paren{v}\dd v=\frac D4t^2.
\end{equation}
{No metric comparison is made, so no loss in the Weyl coefficient occurs.}

{Differentiating $A_v$ with respect to its order gives}
{\[
 \partial_v A_v\paren{z}=-\arccos\paren{v/z},
 \qquad
 \partial_v^2 A_v\paren{z}=\frac1{\sqrt{z^2-v^2}}.
\]}
{For $0\le v<t$, the inverse-cosine terms cancel at $v=0$, and}
{\[
 f_t'\paren{0}=0,\qquad
 -f_t''\paren{v}
 =\frac1\pi\paren{\frac1{\sqrt{t^2-v^2}}
              -\frac1{\sqrt{R^2t^2-v^2}}}
 \ge\frac{\delta}{\pi Rt}.
\]}
{Indeed, the difference of reciprocal square roots increases with
$v$, since its derivative for $v>0$ is
$v\bracket{\paren{t^2-v^2}^{-3/2}-\paren{R^2t^2-v^2}^{-3/2}}>0$;
its value at zero is $\delta/\paren{Rt}$. On $t<v<Rt$, however,}
{\[
 f_t''\paren{v}=\frac1{\pi\sqrt{R^2t^2-v^2}}>0.
\]}
{The turning point $s=v/t$ then lies inside the radial interval.
This change of curvature is why the outer order region is reflected
before applying a concave-phase rounding estimate.}

Write $\set{y}=y-\floor{y}$ and $\psi\paren{y}=\set{y}-1/2$.
At integers, $\psi\paren{n}=-1/2$ preserves the counting convention
without affecting the integrals. The absolute-value bound
in~\cite[Lemma~2.8]{DaiSunAnnuli} is $1/\paren{2\sqrt\kappa}$.
The upper bound $1/\paren{3\sqrt\kappa}$ for $C^2$ phases
is proved in~\cite[Lemma~4.4]{DaiSunAnnuli}. The lemma below also gives
the sharp lower bound and the initial-phase refinements under the stated
absolute-continuity hypothesis.

{\begin{lemma}[Sharp one-sided rounding bounds]\label{lem:saw}
Let $q\in C^1\paren{\bracket{0,T}}$ have absolutely continuous
derivative. Suppose that $q'\paren{0}=0$ and
$q''\le-\kappa<0$ almost everywhere. Then
\begin{equation}\label{eq:one-sided-rounding}
 -\frac{\sqrt2}{6\sqrt\kappa}
 \le\int_0^T\psi\paren{q\paren{v}}\dd v
 \le\frac1{3\sqrt\kappa}.
\end{equation}
Both constants are sharp. More precisely, with
$a=\set{q\paren{0}}$, the following phase-dependent bounds hold:
\begin{align}
 \int_0^T\psi\paren{q\paren{v}}\dd v
 &\ge-\sqrt{\frac{2a}{\kappa}}
          \paren{\frac12-\frac{2a}{3}},
 &&0\le a\le\frac12,\label{eq:rounding-phase-lower}\\
 \int_0^T\psi\paren{q\paren{v}}\dd v
 &\le\frac{2\sqrt2}{3\sqrt\kappa}
          \paren{a-\frac12}^{3/2},
 &&\frac12\le a<1.\label{eq:rounding-phase-upper}
\end{align}
Consequently, the bounds relevant to the initial phases $1/4$ and
$3/4$ are
\[
a=\frac14:\quad
 \int_0^T\psi\paren{q}\dd v\ge-\frac{\sqrt2}{6\sqrt\kappa},
 \qquad
 a=\frac34:\quad
 \int_0^T\psi\paren{q}\dd v\le\frac{\sqrt2}{12\sqrt\kappa}.
\]
\end{lemma}}

\begin{proof}
{The case $T=0$ is immediate, so assume $T>0$.
Put $s\left(v\right)=-q'\left(v\right)$ and $y\left(v\right)=q\left(0\right)-q\left(v\right)$. Absolute continuity gives
$s\left(v\right)\ge\kappa v$, since $s\left(0\right)=0$ and $s'\ge\kappa$ almost
everywhere. Hence $y'=s$ and $y$ is strictly increasing away from
zero. Applying the chain rule for absolutely continuous functions,}
{\[
 s\left(v\right)^2=2\int_0^v s\left(z\right)s'\left(z\right)\dd z
 \ge2\kappa\int_0^v s\left(z\right)\dd z=2\kappa y\left(v\right).
\]}
{Set $Y=y\left(T\right)$, and let $v=v\left(y\right)$ be the inverse map for $0<y\le Y$.
The phase integral becomes}
{\begin{equation}\label{eq:rounding-phase-change}
 \begin{aligned}
 I:=\int_0^T\psi\paren{q\left(v\right)}\dd v
   &=\int_0^Y g_a\left(y\right)w\left(y\right)\dd y,\\
 g_a\left(y\right)&=\psi\left(a-y\right),\qquad w\left(y\right)=\frac1{s\left(v\left(y\right)\right)}.
 \end{aligned}
\end{equation}}
{The weight is positive and nonincreasing, and}
{\begin{equation}\label{eq:rounding-weight}
 0<w\left(y\right)\le\frac1{\sqrt{2\kappa y}},\qquad y>0.
\end{equation}}
{One may first change variables on $\bracket{\varepsilon,T}$ and
then let $\varepsilon\downarrow0$: the majorant in
\eqref{eq:rounding-weight} is integrable at zero. Since $q$ is
strictly decreasing on $\paren{0,T}$, each integer level is attained
at most once. Thus the values of $\psi$ at the jumps have no effect
on the integral, although they determine the filled points in
Figure~\ref{fig:quadratic-sawtooth}.}

{We record both pairing signs, including every possible truncation.
If a negative half-period begins at $y=b$, then, up to its isolated
jump value,}
{\[
 g_a\left(b+t\right)=
 \begin{cases}-t,&0<t<1/2,\\1-t,&1/2<t<1.\end{cases}
\]}
{Its negative value $-z$ at $b+z$ is followed by the positive value
$z$ at $b+1-z$. The complete pair therefore has integral}
{\begin{equation}\label{eq:rounding-negative-positive}
 \int_0^{1/2}z\bracket{w\left(b+1-z\right)-w\left(b+z\right)}\dd z\le0.
\end{equation}}
{For a truncation of length $0\le h\le1/2$ the integrand is
nonpositive. If $1/2<h\le1$, the integral over the prefix of
length $h$ is exactly}
{\[
 \int_{1-h}^{1/2}z\bracket{w\left(b+1-z\right)-w\left(b+z\right)}\dd z
 -\int_0^{1-h}z\,w\left(b+z\right)\dd z\le0.
\]}
{The first term matches every included positive value to an earlier
negative value; the second consists of the unmatched negative values.}

{For a positive half-period followed by a negative one, the values
$z$ and $-z$ occur at $b+1/2-z$ and $b+1/2+z$, respectively.
The complete pair has integral}
{\begin{equation}\label{eq:rounding-positive-negative}
 \int_0^{1/2}z
 \bracket{w\left(b+1/2-z\right)-w\left(b+1/2+z\right)}\dd z\ge0.
\end{equation}}
{A prefix of length at most $1/2$ is nonnegative. For a longer prefix,
put $d=h-1/2\in\bracket{0,1/2}$. Its integral equals}
{\[
 \int_0^d z\bracket{w\left(b+1/2-z\right)-w\left(b+1/2+z\right)}\dd z
 +\int_d^{1/2}z\,w\left(b+1/2-z\right)\dd z\ge0.
\]}
{Here every included negative value has its earlier positive match;
only positive values can remain unmatched. These identities also
hold when $b=0$, by the integrability established above. Any tail
made of successive pairs is a finite collection of complete pairs
and at most one truncated pair, so its sign is the same. The two
matching arrangements are illustrated in
Figure~\ref{fig:rounding-pairing}.}

{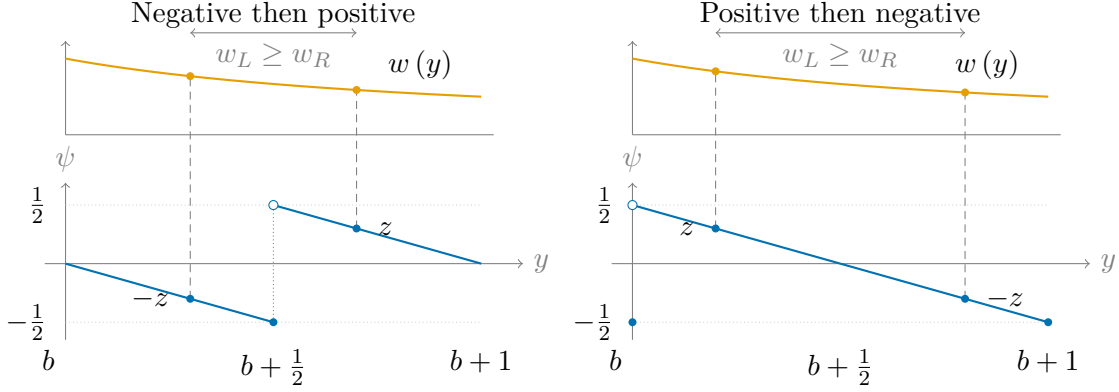
\begin{figure}[htbp]
\centering
\begin{tikzpicture}[x=5.5cm,y=1.55cm,font=\small]
\begin{scope}[xshift=0.0cm]
\node at (0.5,2.12) {Negative then positive};
\draw[->,gray] (-0.05,0) -- (1.10,0) node[right] {$y$};
\draw[->,gray] (0,-0.65) -- (0,0.69) node[above] {$\psi$};
\draw[gray!35,densely dotted] (0,0.5) -- (1,0.5);
\draw[gray!35,densely dotted] (0,-0.5) -- (1,-0.5);
\node[anchor=east] at (-0.02,0.5) {$\frac12$};
\node[anchor=east] at (-0.02,-0.5) {$-\frac12$};
\node[anchor=north east] at (0,-0.64) {$b$};
\node[anchor=north] at (0.5,-0.64) {$b+\frac12$};
\node[anchor=north] at (1,-0.64) {$b+1$};
\draw[->,gray] (0,1.10) -- (0,1.90);
\draw[gray] (0,1.10) -- (1.03,1.10);
\draw[annOrange,thick] plot coordinates {(0.0000000,1.7500000) (0.0250000,1.7341463) (0.0500000,1.7190476) (0.0750000,1.7046512) (0.1000000,1.6909091) (0.1250000,1.6777778) (0.1500000,1.6652174) (0.1750000,1.6531915) (0.2000000,1.6416667) (0.2250000,1.6306122) (0.2500000,1.6200000) (0.2750000,1.6098039) (0.3000000,1.6000000) (0.3250000,1.5905660) (0.3500000,1.5814815) (0.3750000,1.5727273) (0.4000000,1.5642857) (0.4250000,1.5561404) (0.4500000,1.5482759) (0.4750000,1.5406780) (0.5000000,1.5333333) (0.5250000,1.5262295) (0.5500000,1.5193548) (0.5750000,1.5126984) (0.6000000,1.5062500) (0.6250000,1.5000000) (0.6500000,1.4939394) (0.6750000,1.4880597) (0.7000000,1.4823529) (0.7250000,1.4768116) (0.7500000,1.4714286) (0.7750000,1.4661972) (0.8000000,1.4611111) (0.8250000,1.4561644) (0.8500000,1.4513514) (0.8750000,1.4466667) (0.9000000,1.4421053) (0.9250000,1.4376623) (0.9500000,1.4333333) (0.9750000,1.4291139) (1.0000000,1.4250000)};
\node[anchor=west] at (0.75,1.70) {$w\left(y\right)$};
\draw[annBlue,thick] (0,0) -- (0.5,-0.5);
\draw[annBlue,thick] (0.5,0.5) -- (1,0);
\draw[gray,densely dotted] (0.5,-0.5) -- (0.5,0.5);
\fill[annBlue] (0.5,-0.5) circle (1.5pt);
\draw[annBlue,fill=white] (0.5,0.5) circle (1.7pt);
\draw[gray,densely dashed] (0.3,-0.3) -- (0.3,1.6000000);
\fill[annOrange] (0.3,1.6000000) circle (1.5pt);
\fill[annBlue] (0.3,-0.3) circle (1.5pt);
\draw[gray,densely dashed] (0.7,0.3) -- (0.7,1.4823529);
\fill[annOrange] (0.7,1.4823529) circle (1.5pt);
\fill[annBlue] (0.7,0.3) circle (1.5pt);
\node[anchor=east] at (0.27499999999999997,-0.3) {$-z$};
\node[anchor=west] at (0.725,0.3) {$z$};
\draw[<->,gray] (0.3,1.97) -- (0.7,1.97) node[midway,below] {$w_L\ge w_R$};
\end{scope}
\begin{scope}[xshift=7.5cm]
\node at (0.5,2.12) {Positive then negative};
\draw[->,gray] (-0.05,0) -- (1.10,0) node[right] {$y$};
\draw[->,gray] (0,-0.65) -- (0,0.69) node[above] {$\psi$};
\draw[gray!35,densely dotted] (0,0.5) -- (1,0.5);
\draw[gray!35,densely dotted] (0,-0.5) -- (1,-0.5);
\node[anchor=east] at (-0.02,0.5) {$\frac12$};
\node[anchor=east] at (-0.02,-0.5) {$-\frac12$};
\node[anchor=north east] at (0,-0.64) {$b$};
\node[anchor=north] at (0.5,-0.64) {$b+\frac12$};
\node[anchor=north] at (1,-0.64) {$b+1$};
\draw[->,gray] (0,1.10) -- (0,1.90);
\draw[gray] (0,1.10) -- (1.03,1.10);
\draw[annOrange,thick] plot coordinates {(0.0000000,1.7500000) (0.0250000,1.7341463) (0.0500000,1.7190476) (0.0750000,1.7046512) (0.1000000,1.6909091) (0.1250000,1.6777778) (0.1500000,1.6652174) (0.1750000,1.6531915) (0.2000000,1.6416667) (0.2250000,1.6306122) (0.2500000,1.6200000) (0.2750000,1.6098039) (0.3000000,1.6000000) (0.3250000,1.5905660) (0.3500000,1.5814815) (0.3750000,1.5727273) (0.4000000,1.5642857) (0.4250000,1.5561404) (0.4500000,1.5482759) (0.4750000,1.5406780) (0.5000000,1.5333333) (0.5250000,1.5262295) (0.5500000,1.5193548) (0.5750000,1.5126984) (0.6000000,1.5062500) (0.6250000,1.5000000) (0.6500000,1.4939394) (0.6750000,1.4880597) (0.7000000,1.4823529) (0.7250000,1.4768116) (0.7500000,1.4714286) (0.7750000,1.4661972) (0.8000000,1.4611111) (0.8250000,1.4561644) (0.8500000,1.4513514) (0.8750000,1.4466667) (0.9000000,1.4421053) (0.9250000,1.4376623) (0.9500000,1.4333333) (0.9750000,1.4291139) (1.0000000,1.4250000)};
\node[anchor=west] at (0.75,1.70) {$w\left(y\right)$};
\draw[annBlue,thick] (0,0.5) -- (1,-0.5);
\fill[annBlue] (0,-0.5) circle (1.5pt);
\draw[annBlue,fill=white] (0,0.5) circle (1.7pt);
\fill[annBlue] (1,-0.5) circle (1.5pt);
\draw[gray,densely dashed] (0.2,0.3) -- (0.2,1.6416667);
\fill[annOrange] (0.2,1.6416667) circle (1.5pt);
\fill[annBlue] (0.2,0.3) circle (1.5pt);
\draw[gray,densely dashed] (0.8,-0.3) -- (0.8,1.4611111);
\fill[annOrange] (0.8,1.4611111) circle (1.5pt);
\fill[annBlue] (0.8,-0.3) circle (1.5pt);
\node[anchor=east] at (0.17500000000000002,0.3) {$z$};
\node[anchor=west] at (0.8250000000000001,-0.3) {$-z$};
\draw[<->,gray] (0.2,1.97) -- (0.8,1.97) node[midway,below] {$w_L\ge w_R$};
\end{scope}
\end{tikzpicture}
\caption{Equal magnitudes are paired at an earlier and a later phase.
The upper curves depict a nonincreasing weight. In the left panel the
earlier value is $-z$, so the pair contributes
$z\left(w_R-w_L\right)\le0$; in the right panel it is $z$, so the pair contributes
$z\left(w_L-w_R\right)\ge0$. The explicit truncated-pair formulas show why
stopping inside either half-period preserves these signs.}
\label{fig:rounding-pairing}
\end{figure}}

{We first prove the upper bound. If $0\le a\le1/2$, the initial
negative interval is $0<y<a$. If $Y\le a$, then $I\le0$.
Otherwise discard that interval and retain the first positive
half-period, which lies in $a<y<a+1/2$. Its contribution, including
any truncation, is at most}
{\[
 \int_0^{1/2}\frac{1/2-z}{\sqrt{2\kappa\left(a+z\right)}}\dd z
 \le\int_0^{1/2}\frac{1/2-z}{\sqrt{2\kappa z}}\dd z
 =\frac1{3\sqrt\kappa}.
\]}
{If $Y>a+1/2$, the remaining tail starts with a negative half-period
and contributes at most zero by
\eqref{eq:rounding-negative-positive} and its truncated form.
The last integral is evaluated by the antiderivative
$2A\sqrt z-\left(2/3\right)z^{3/2}$ of $\left(A-z\right)/\sqrt z$, with $A=1/2$.}

{If $1/2\le a<1$, put $l=a-1/2$. The initial positive values are
$g_a\left(y\right)=l-y$ for $0<y<l$. A stop within this interval retains only
a subset of its positive contribution; a later stop adds a
nonpositive tail of negative--positive pairs. Consequently}
{\[
 I\le\frac1{\sqrt{2\kappa}}
 \bracket{2l\sqrt y-\frac23 y^{3/2}}_0^l
 =\frac{2\sqrt2}{3\sqrt\kappa}l^{3/2}.
\]}
{This proves~\eqref{eq:rounding-phase-upper}. Since $l<1/2$, its
right-hand side is at most $1/\paren{3\sqrt\kappa}$, completing
the universal upper bound. The argument includes $a=1/2$, when
$l=0$ and the whole integral is a nonpositive paired tail.}

{For the lower bound with $0\le a\le1/2$, the initial negative
values are $g_a\left(y\right)=-\left(1/2-a+y\right)$ on $0<y<a$. A stop in this interval
has loss at most that of the whole interval. After $y=a$, the tail
starts with a positive half-period and is nonnegative by
\eqref{eq:rounding-positive-negative}, including its final truncation.
Using the antiderivative $2A\sqrt y+\left(2/3\right)y^{3/2}$ gives}
{\begin{align*}
 -I&\le\frac1{\sqrt{2\kappa}}
       \int_0^a\frac{1/2-a+y}{\sqrt y}\dd y\\
 &=\frac1{\sqrt{2\kappa}}
       \bracket{2\paren{1/2-a}\sqrt y+\frac23y^{3/2}}_0^a
 =\sqrt{\frac{2a}{\kappa}}\paren{\frac12-\frac{2a}{3}}.
\end{align*}}
{This proves~\eqref{eq:rounding-phase-lower}, also at $a=0$ by an
empty initial interval. For $a>0$, differentiation gives}
{\[
 \frac{\mathrm d}{\mathrm d a}
 \bracket{\sqrt{2a}\paren{\frac12-\frac{2a}{3}}}
 =\frac{\sqrt2}{\sqrt a}\paren{\frac14-a}.
\]}
{The derivative changes sign at $a=1/4$, giving the
maximum $\sqrt2/6$ on $\bracket{0,1/2}$.}

{It remains to bound the loss when $1/2<a<1$. Put $l=a-1/2$ again.
The initial positive value $z$ occurs at $y=l-z$ for $0<z<l$;
its negative match $-z$ occurs at $y=l+z$. If $Y\le l$, then
$I\ge0$. For $l<Y\le a$, write $d=Y-l$. When $d\le l$ the
integral is}
{\[
 \int_0^d z\bracket{w\left(l-z\right)-w\left(l+z\right)}\dd z
 +\int_d^l z\,w\left(l-z\right)\dd z\ge0.
\]}
{When $l<d\le1/2$, it is}
{\begin{equation}\label{eq:rounding-initial-positive-pair}
 \int_0^l z\bracket{w\left(l-z\right)-w\left(l+z\right)}\dd z
 -\int_l^d z\,w\left(l+z\right)\dd z.
\end{equation}}
{Thus only the negative magnitudes $l<z<d$ can cause a loss. If
$Y>a$, the initial expression is the same with $d=1/2$, and the
tail after the first integer crossing at $y=a$ is a nonnegative
collection of positive--negative pairs. In every case,}
{\[
 I\ge-\int_l^{1/2}\frac{z}{\sqrt{2\kappa\left(l+z\right)}}\dd z
 \ge-\int_0^{1/2}\sqrt{\frac z{2\kappa}}\dd z
 =-\frac1{6\sqrt\kappa}
 >-\frac{\sqrt2}{6\sqrt\kappa}.
\]}
{The comparison uses $l+z\ge z$ and enlarges the integration
interval; the antiderivative of $\sqrt z$ is $\left(2/3\right)z^{3/2}$.
This completes both sides of~\eqref{eq:one-sided-rounding}.}

{For sharpness, take an integer $m$ and the quadratic
$q\left(v\right)=m+a-\kappa v^2/2$. Then $w\left(y\right)=1/\sqrt{2\kappa y}$,
so the curvature majorant is an equality. With $a=0$ and
$T=1/\sqrt\kappa$, the integrand equals
$1/2-\kappa v^2/2$ for $0<v<T$, and}
{\[
 \int_0^T\psi\paren{q\left(v\right)}\dd v
 =\bracket{\frac v2-\frac{\kappa v^3}{6}}_0^{1/\sqrt\kappa}
 =\frac1{3\sqrt\kappa}.
\]}
{With $a=1/4$ and $T=1/\sqrt{2\kappa}$, it equals
$-1/4-\kappa v^2/2$ except at the endpoint, and}
{\[
 \int_0^T\psi\paren{q\left(v\right)}\dd v
 =\bracket{-\frac v4-\frac{\kappa v^3}{6}}_0^{1/\sqrt{2\kappa}}
 =-\frac{\sqrt2}{6\sqrt\kappa}.
\]}
{For every $0<a\le1/2$, stopping the same quadratic at
$T=\sqrt{2a/\kappa}$ attains
\eqref{eq:rounding-phase-lower}. For every $1/2<a<1$, stopping
at $T=\sqrt{2\left(a-1/2\right)/\kappa}$ attains
\eqref{eq:rounding-phase-upper}. The zero endpoint constants are
limits as $T\downarrow0$. The exact examples at $a=0$, $1/4$,
and $3/4$ with $\kappa=1$ are shown in
Figure~\ref{fig:quadratic-sawtooth}.}
\end{proof}

{\begin{figure}[htbp]
\centering
\begin{tikzpicture}[x=1.65cm,y=2.25cm,font=\small]
\begin{scope}[xshift=0.00cm]
\fill[annBlue!12] (0,0) -- (0.0000000,0.5000000) -- (0.0222222,0.4997531) -- (0.0444444,0.4990123) -- (0.0666667,0.4977778) -- (0.0888889,0.4960494) -- (0.1111111,0.4938272) -- (0.1333333,0.4911111) -- (0.1555556,0.4879012) -- (0.1777778,0.4841975) -- (0.2000000,0.4800000) -- (0.2222222,0.4753086) -- (0.2444444,0.4701235) -- (0.2666667,0.4644444) -- (0.2888889,0.4582716) -- (0.3111111,0.4516049) -- (0.3333333,0.4444444) -- (0.3555556,0.4367901) -- (0.3777778,0.4286420) -- (0.4000000,0.4200000) -- (0.4222222,0.4108642) -- (0.4444444,0.4012346) -- (0.4666667,0.3911111) -- (0.4888889,0.3804938) -- (0.5111111,0.3693827) -- (0.5333333,0.3577778) -- (0.5555556,0.3456790) -- (0.5777778,0.3330864) -- (0.6000000,0.3200000) -- (0.6222222,0.3064198) -- (0.6444444,0.2923457) -- (0.6666667,0.2777778) -- (0.6888889,0.2627160) -- (0.7111111,0.2471605) -- (0.7333333,0.2311111) -- (0.7555556,0.2145679) -- (0.7777778,0.1975309) -- (0.8000000,0.1800000) -- (0.8222222,0.1619753) -- (0.8444444,0.1434568) -- (0.8666667,0.1244444) -- (0.8888889,0.1049383) -- (0.9111111,0.0849383) -- (0.9333333,0.0644444) -- (0.9555556,0.0434568) -- (0.9777778,0.0219753) -- (1.0000000,0.0000000) -- (1.0000000,0) -- cycle;
\draw[->,gray] (-0.08,0) -- (2.38,0) node[right] {$v$};
\draw[->,gray] (0,-0.63) -- (0,0.68) node[above] {$\psi$};
\draw[gray!35,densely dotted] (0,0.5) -- (2.25,0.5);
\draw[gray!35,densely dotted] (0,-0.5) -- (2.25,-0.5);
\node[anchor=east] at (-0.04,0.5) {$\frac12$};
\node[anchor=east] at (-0.04,-0.5) {$-\frac12$};
\node[anchor=north east] at (-0.02,-0.03) {$0$};
\draw (1,0.025) -- (1,-0.025) node[below] {$1$};
\draw (2,0.025) -- (2,-0.025) node[below] {$2$};
\node at (1.10,0.89) {$a=0$};
\draw[gray,dashed] (1.0000000,-0.57) -- (1.0000000,0.58);
\node[anchor=south] at (1.0000000,0.59) {$T$};
\draw[annBlue,thick] plot coordinates {(0.0000000,0.5000000) (0.0372161,0.4993075) (0.0744323,0.4972299) (0.1116484,0.4937673) (0.1488646,0.4889197) (0.1860807,0.4826870) (0.2232969,0.4750693) (0.2605130,0.4660665) (0.2977292,0.4556787) (0.3349453,0.4439058) (0.3721615,0.4307479) (0.4093776,0.4162050) (0.4465938,0.4002770) (0.4838099,0.3829640) (0.5210260,0.3642659) (0.5582422,0.3441828) (0.5954583,0.3227147) (0.6326745,0.2998615) (0.6698906,0.2756233) (0.7071068,0.2500000) (0.7443229,0.2229917) (0.7815391,0.1945983) (0.8187552,0.1648199) (0.8559714,0.1336565) (0.8931875,0.1011080) (0.9304037,0.0671745) (0.9676198,0.0318560) (1.0048360,-0.0048476) (1.0420521,-0.0429363) (1.0792682,-0.0824100) (1.1164844,-0.1232687) (1.1537005,-0.1655125) (1.1909167,-0.2091413) (1.2281328,-0.2541551) (1.2653490,-0.3005540) (1.3025651,-0.3483380) (1.3397813,-0.3975069) (1.3769974,-0.4480609) (1.4142136,-0.5000000)};
\draw[annBlue,thick] plot coordinates {(1.4142136,0.5000000) (1.4296290,0.4780805) (1.4450444,0.4559233) (1.4604599,0.4335285) (1.4758753,0.4108961) (1.4912907,0.3880260) (1.5067062,0.3649183) (1.5221216,0.3415729) (1.5375370,0.3179900) (1.5529525,0.2941693) (1.5683679,0.2701111) (1.5837833,0.2458152) (1.5991988,0.2212817) (1.6146142,0.1965105) (1.6300296,0.1715017) (1.6454451,0.1462553) (1.6608605,0.1207712) (1.6762759,0.0950495) (1.6916913,0.0690902) (1.7071068,0.0428932) (1.7225222,0.0164586) (1.7379376,-0.0102136) (1.7533531,-0.0371235) (1.7687685,-0.0642710) (1.7841839,-0.0916562) (1.7995994,-0.1192790) (1.8150148,-0.1471394) (1.8304302,-0.1752374) (1.8458457,-0.2035731) (1.8612611,-0.2321465) (1.8766765,-0.2609574) (1.8920920,-0.2900060) (1.9075074,-0.3192922) (1.9229228,-0.3488161) (1.9383383,-0.3785776) (1.9537537,-0.4085768) (1.9691691,-0.4388135) (1.9845846,-0.4692880) (2.0000000,-0.5000000)};
\draw[annBlue,thick] plot coordinates {(2.0000000,0.5000000) (2.0065789,0.4868205) (2.0131579,0.4735976) (2.0197368,0.4603315) (2.0263158,0.4470222) (2.0328947,0.4336695) (2.0394737,0.4202735) (2.0460526,0.4068343) (2.0526316,0.3933518) (2.0592105,0.3798260) (2.0657895,0.3662569) (2.0723684,0.3526446) (2.0789474,0.3389889) (2.0855263,0.3252900) (2.0921053,0.3115478) (2.0986842,0.2977623) (2.1052632,0.2839335) (2.1118421,0.2700615) (2.1184211,0.2561461) (2.1250000,0.2421875) (2.1315789,0.2281856) (2.1381579,0.2141404) (2.1447368,0.2000519) (2.1513158,0.1859202) (2.1578947,0.1717452) (2.1644737,0.1575268) (2.1710526,0.1432652) (2.1776316,0.1289604) (2.1842105,0.1146122) (2.1907895,0.1002207) (2.1973684,0.0857860) (2.2039474,0.0713080) (2.2105263,0.0567867) (2.2171053,0.0422221) (2.2236842,0.0276143) (2.2302632,0.0129631) (2.2368421,-0.0017313) (2.2434211,-0.0164690) (2.2500000,-0.0312500)};
\fill[annBlue] (0,-0.5000000) circle (1.5pt);
\draw[annBlue,fill=white] (0,0.5) circle (1.7pt);
\draw[gray,densely dotted] (1.4142136,-0.5) -- (1.4142136,0.5);
\fill[annBlue] (1.4142136,-0.5) circle (1.5pt);
\draw[annBlue,fill=white] (1.4142136,0.5) circle (1.7pt);
\draw[gray,densely dotted] (2.0000000,-0.5) -- (2.0000000,0.5);
\fill[annBlue] (2.0000000,-0.5) circle (1.5pt);
\draw[annBlue,fill=white] (2.0000000,0.5) circle (1.7pt);
\end{scope}
\begin{scope}[xshift=5.15cm]
\fill[annOrange!12] (0,0) -- (0.0000000,-0.2500000) -- (0.0157135,-0.2501235) -- (0.0314270,-0.2504938) -- (0.0471405,-0.2511111) -- (0.0628539,-0.2519753) -- (0.0785674,-0.2530864) -- (0.0942809,-0.2544444) -- (0.1099944,-0.2560494) -- (0.1257079,-0.2579012) -- (0.1414214,-0.2600000) -- (0.1571348,-0.2623457) -- (0.1728483,-0.2649383) -- (0.1885618,-0.2677778) -- (0.2042753,-0.2708642) -- (0.2199888,-0.2741975) -- (0.2357023,-0.2777778) -- (0.2514157,-0.2816049) -- (0.2671292,-0.2856790) -- (0.2828427,-0.2900000) -- (0.2985562,-0.2945679) -- (0.3142697,-0.2993827) -- (0.3299832,-0.3044444) -- (0.3456966,-0.3097531) -- (0.3614101,-0.3153086) -- (0.3771236,-0.3211111) -- (0.3928371,-0.3271605) -- (0.4085506,-0.3334568) -- (0.4242641,-0.3400000) -- (0.4399776,-0.3467901) -- (0.4556910,-0.3538272) -- (0.4714045,-0.3611111) -- (0.4871180,-0.3686420) -- (0.5028315,-0.3764198) -- (0.5185450,-0.3844444) -- (0.5342585,-0.3927160) -- (0.5499719,-0.4012346) -- (0.5656854,-0.4100000) -- (0.5813989,-0.4190123) -- (0.5971124,-0.4282716) -- (0.6128259,-0.4377778) -- (0.6285394,-0.4475309) -- (0.6442528,-0.4575309) -- (0.6599663,-0.4677778) -- (0.6756798,-0.4782716) -- (0.6913933,-0.4890123) -- (0.7071068,-0.5000000) -- (0.7071068,0) -- cycle;
\draw[->,gray] (-0.08,0) -- (2.38,0) node[right] {$v$};
\draw[->,gray] (0,-0.63) -- (0,0.68) node[above] {$\psi$};
\draw[gray!35,densely dotted] (0,0.5) -- (2.25,0.5);
\draw[gray!35,densely dotted] (0,-0.5) -- (2.25,-0.5);
\node[anchor=east] at (-0.04,0.5) {$\frac12$};
\node[anchor=east] at (-0.04,-0.5) {$-\frac12$};
\node[anchor=north east] at (-0.02,-0.03) {$0$};
\draw (1,0.025) -- (1,-0.025) node[below] {$1$};
\draw (2,0.025) -- (2,-0.025) node[below] {$2$};
\node at (1.10,0.89) {$a=\frac14$};
\draw[gray,dashed] (0.7071068,-0.57) -- (0.7071068,0.58);
\node[anchor=south] at (0.7071068,0.59) {$T$};
\draw[annBlue,thick] plot coordinates {(0.0000000,-0.2500000) (0.0186081,-0.2501731) (0.0372161,-0.2506925) (0.0558242,-0.2515582) (0.0744323,-0.2527701) (0.0930404,-0.2543283) (0.1116484,-0.2562327) (0.1302565,-0.2584834) (0.1488646,-0.2610803) (0.1674727,-0.2640235) (0.1860807,-0.2673130) (0.2046888,-0.2709488) (0.2232969,-0.2749307) (0.2419050,-0.2792590) (0.2605130,-0.2839335) (0.2791211,-0.2889543) (0.2977292,-0.2943213) (0.3163372,-0.3000346) (0.3349453,-0.3060942) (0.3535534,-0.3125000) (0.3721615,-0.3192521) (0.3907695,-0.3263504) (0.4093776,-0.3337950) (0.4279857,-0.3415859) (0.4465938,-0.3497230) (0.4652018,-0.3582064) (0.4838099,-0.3670360) (0.5024180,-0.3762119) (0.5210260,-0.3857341) (0.5396341,-0.3956025) (0.5582422,-0.4058172) (0.5768503,-0.4163781) (0.5954583,-0.4272853) (0.6140664,-0.4385388) (0.6326745,-0.4501385) (0.6512826,-0.4620845) (0.6698906,-0.4743767) (0.6884987,-0.4870152) (0.7071068,-0.5000000)};
\draw[annBlue,thick] plot coordinates {(0.7071068,0.5000000) (0.7301076,0.4834714) (0.7531085,0.4664138) (0.7761093,0.4488272) (0.7991102,0.4307115) (0.8221110,0.4120668) (0.8451118,0.3928930) (0.8681127,0.3731902) (0.8911135,0.3529583) (0.9141144,0.3321975) (0.9371152,0.3109075) (0.9601161,0.2890886) (0.9831169,0.2667406) (1.0061177,0.2438635) (1.0291186,0.2204575) (1.0521194,0.1965224) (1.0751203,0.1720582) (1.0981211,0.1470650) (1.1211220,0.1215428) (1.1441228,0.0954915) (1.1671236,0.0689112) (1.1901245,0.0418018) (1.2131253,0.0141635) (1.2361262,-0.0140040) (1.2591270,-0.0427004) (1.2821279,-0.0719259) (1.3051287,-0.1016805) (1.3281296,-0.1319641) (1.3511304,-0.1627767) (1.3741312,-0.1941183) (1.3971321,-0.2259890) (1.4201329,-0.2583888) (1.4431338,-0.2913175) (1.4661346,-0.3247754) (1.4891355,-0.3587622) (1.5121363,-0.3932781) (1.5351371,-0.4283230) (1.5581380,-0.4638970) (1.5811388,-0.5000000)};
\draw[annBlue,thick] plot coordinates {(1.5811388,0.5000000) (1.5953541,0.4774226) (1.6095694,0.4546431) (1.6237847,0.4316616) (1.6380000,0.4084779) (1.6522153,0.3850922) (1.6664306,0.3615044) (1.6806460,0.3377146) (1.6948613,0.3137227) (1.7090766,0.2895287) (1.7232919,0.2651326) (1.7375072,0.2405344) (1.7517225,0.2157342) (1.7659378,0.1907319) (1.7801531,0.1655275) (1.7943684,0.1401211) (1.8085837,0.1145125) (1.8227990,0.0887019) (1.8370143,0.0626893) (1.8512296,0.0364745) (1.8654449,0.0100577) (1.8796602,-0.0165612) (1.8938755,-0.0433822) (1.9080908,-0.0704052) (1.9223061,-0.0976304) (1.9365214,-0.1250576) (1.9507367,-0.1526869) (1.9649520,-0.1805182) (1.9791673,-0.2085516) (1.9933826,-0.2367871) (2.0075979,-0.2652247) (2.0218132,-0.2938644) (2.0360285,-0.3227061) (2.0502438,-0.3517499) (2.0644591,-0.3809958) (2.0786744,-0.4104437) (2.0928897,-0.4400937) (2.1071050,-0.4699458) (2.1213203,-0.5000000)};
\draw[annBlue,thick] plot coordinates {(2.1213203,0.5000000) (2.1247067,0.4928108) (2.1280930,0.4856102) (2.1314793,0.4783981) (2.1348656,0.4711745) (2.1382519,0.4639395) (2.1416382,0.4566929) (2.1450245,0.4494350) (2.1484108,0.4421655) (2.1517971,0.4348846) (2.1551834,0.4275922) (2.1585697,0.4202884) (2.1619560,0.4129731) (2.1653423,0.4056463) (2.1687286,0.3983080) (2.1721149,0.3909583) (2.1755013,0.3835972) (2.1788876,0.3762245) (2.1822739,0.3688404) (2.1856602,0.3614448) (2.1890465,0.3540378) (2.1924328,0.3466192) (2.1958191,0.3391893) (2.1992054,0.3317478) (2.2025917,0.3242949) (2.2059780,0.3168305) (2.2093643,0.3093547) (2.2127506,0.3018673) (2.2161369,0.2943685) (2.2195232,0.2868583) (2.2229095,0.2793366) (2.2262959,0.2718034) (2.2296822,0.2642587) (2.2330685,0.2567026) (2.2364548,0.2491350) (2.2398411,0.2415560) (2.2432274,0.2339654) (2.2466137,0.2263635) (2.2500000,0.2187500)};
\fill[annBlue] (0,-0.2500000) circle (1.5pt);
\draw[gray,densely dotted] (0.7071068,-0.5) -- (0.7071068,0.5);
\fill[annBlue] (0.7071068,-0.5) circle (1.5pt);
\draw[annBlue,fill=white] (0.7071068,0.5) circle (1.7pt);
\draw[gray,densely dotted] (1.5811388,-0.5) -- (1.5811388,0.5);
\fill[annBlue] (1.5811388,-0.5) circle (1.5pt);
\draw[annBlue,fill=white] (1.5811388,0.5) circle (1.7pt);
\draw[gray,densely dotted] (2.1213203,-0.5) -- (2.1213203,0.5);
\fill[annBlue] (2.1213203,-0.5) circle (1.5pt);
\draw[annBlue,fill=white] (2.1213203,0.5) circle (1.7pt);
\end{scope}
\begin{scope}[xshift=10.30cm]
\fill[annBlue!12] (0,0) -- (0.0000000,0.2500000) -- (0.0157135,0.2498765) -- (0.0314270,0.2495062) -- (0.0471405,0.2488889) -- (0.0628539,0.2480247) -- (0.0785674,0.2469136) -- (0.0942809,0.2455556) -- (0.1099944,0.2439506) -- (0.1257079,0.2420988) -- (0.1414214,0.2400000) -- (0.1571348,0.2376543) -- (0.1728483,0.2350617) -- (0.1885618,0.2322222) -- (0.2042753,0.2291358) -- (0.2199888,0.2258025) -- (0.2357023,0.2222222) -- (0.2514157,0.2183951) -- (0.2671292,0.2143210) -- (0.2828427,0.2100000) -- (0.2985562,0.2054321) -- (0.3142697,0.2006173) -- (0.3299832,0.1955556) -- (0.3456966,0.1902469) -- (0.3614101,0.1846914) -- (0.3771236,0.1788889) -- (0.3928371,0.1728395) -- (0.4085506,0.1665432) -- (0.4242641,0.1600000) -- (0.4399776,0.1532099) -- (0.4556910,0.1461728) -- (0.4714045,0.1388889) -- (0.4871180,0.1313580) -- (0.5028315,0.1235802) -- (0.5185450,0.1155556) -- (0.5342585,0.1072840) -- (0.5499719,0.0987654) -- (0.5656854,0.0900000) -- (0.5813989,0.0809877) -- (0.5971124,0.0717284) -- (0.6128259,0.0622222) -- (0.6285394,0.0524691) -- (0.6442528,0.0424691) -- (0.6599663,0.0322222) -- (0.6756798,0.0217284) -- (0.6913933,0.0109877) -- (0.7071068,-0.0000000) -- (0.7071068,0) -- cycle;
\draw[->,gray] (-0.08,0) -- (2.38,0) node[right] {$v$};
\draw[->,gray] (0,-0.63) -- (0,0.68) node[above] {$\psi$};
\draw[gray!35,densely dotted] (0,0.5) -- (2.25,0.5);
\draw[gray!35,densely dotted] (0,-0.5) -- (2.25,-0.5);
\node[anchor=east] at (-0.04,0.5) {$\frac12$};
\node[anchor=east] at (-0.04,-0.5) {$-\frac12$};
\node[anchor=north east] at (-0.02,-0.03) {$0$};
\draw (1,0.025) -- (1,-0.025) node[below] {$1$};
\draw (2,0.025) -- (2,-0.025) node[below] {$2$};
\node at (1.10,0.89) {$a=\frac34$};
\draw[gray,dashed] (0.7071068,-0.57) -- (0.7071068,0.58);
\node[anchor=south] at (0.7071068,0.59) {$T$};
\draw[annBlue,thick] plot coordinates {(0.0000000,0.2500000) (0.0322301,0.2494806) (0.0644603,0.2479224) (0.0966904,0.2453255) (0.1289205,0.2416898) (0.1611506,0.2370152) (0.1933808,0.2313019) (0.2256109,0.2245499) (0.2578410,0.2167590) (0.2900712,0.2079294) (0.3223013,0.1980609) (0.3545314,0.1871537) (0.3867615,0.1752078) (0.4189917,0.1622230) (0.4512218,0.1481994) (0.4834519,0.1331371) (0.5156821,0.1170360) (0.5479122,0.0998961) (0.5801423,0.0817175) (0.6123724,0.0625000) (0.6446026,0.0422438) (0.6768327,0.0209488) (0.7090628,-0.0013850) (0.7412929,-0.0247576) (0.7735231,-0.0491690) (0.8057532,-0.0746191) (0.8379833,-0.1011080) (0.8702135,-0.1286357) (0.9024436,-0.1572022) (0.9346737,-0.1868075) (0.9669038,-0.2174515) (0.9991340,-0.2491343) (1.0313641,-0.2818560) (1.0635942,-0.3156163) (1.0958244,-0.3504155) (1.1280545,-0.3862535) (1.1602846,-0.4231302) (1.1925147,-0.4610457) (1.2247449,-0.5000000)};
\draw[annBlue,thick] plot coordinates {(1.2247449,0.5000000) (1.2417471,0.4790321) (1.2587493,0.4577751) (1.2757515,0.4362291) (1.2927537,0.4143939) (1.3097559,0.3922697) (1.3267581,0.3698565) (1.3437603,0.3471541) (1.3607625,0.3241627) (1.3777647,0.3008822) (1.3947669,0.2773126) (1.4117691,0.2534540) (1.4287713,0.2293062) (1.4457735,0.2048694) (1.4627758,0.1801435) (1.4797780,0.1551286) (1.4967802,0.1298246) (1.5137824,0.1042315) (1.5307846,0.0783493) (1.5477868,0.0521780) (1.5647890,0.0257177) (1.5817912,-0.0010317) (1.5987934,-0.0280702) (1.6157956,-0.0553977) (1.6327978,-0.0830143) (1.6498000,-0.1109200) (1.6668022,-0.1391148) (1.6838044,-0.1675987) (1.7008066,-0.1963716) (1.7178088,-0.2254336) (1.7348110,-0.2547847) (1.7518133,-0.2844248) (1.7688155,-0.3143541) (1.7858177,-0.3445724) (1.8028199,-0.3750797) (1.8198221,-0.4058762) (1.8368243,-0.4369617) (1.8538265,-0.4683363) (1.8708287,-0.5000000)};
\draw[annBlue,thick] plot coordinates {(1.8708287,0.5000000) (1.8808069,0.4812827) (1.8907851,0.4624659) (1.9007633,0.4435495) (1.9107415,0.4245335) (1.9207197,0.4054180) (1.9306978,0.3862029) (1.9406760,0.3668883) (1.9506542,0.3474740) (1.9606324,0.3279602) (1.9706106,0.3083469) (1.9805888,0.2886340) (1.9905670,0.2688215) (2.0005452,0.2489095) (2.0105234,0.2288979) (2.0205016,0.2087867) (2.0304798,0.1885760) (2.0404580,0.1682657) (2.0504362,0.1478558) (2.0604143,0.1273464) (2.0703925,0.1067374) (2.0803707,0.0860288) (2.0903489,0.0652207) (2.1003271,0.0443130) (2.1103053,0.0233058) (2.1202835,0.0021989) (2.1302617,-0.0190074) (2.1402399,-0.0403134) (2.1502181,-0.0617189) (2.1601963,-0.0832240) (2.1701745,-0.1048286) (2.1801527,-0.1265328) (2.1901308,-0.1483366) (2.2001090,-0.1702399) (2.2100872,-0.1922428) (2.2200654,-0.2143452) (2.2300436,-0.2365473) (2.2400218,-0.2588488) (2.2500000,-0.2812500)};
\fill[annBlue] (0,0.2500000) circle (1.5pt);
\draw[gray,densely dotted] (1.2247449,-0.5) -- (1.2247449,0.5);
\fill[annBlue] (1.2247449,-0.5) circle (1.5pt);
\draw[annBlue,fill=white] (1.2247449,0.5) circle (1.7pt);
\draw[gray,densely dotted] (1.8708287,-0.5) -- (1.8708287,0.5);
\fill[annBlue] (1.8708287,-0.5) circle (1.5pt);
\draw[annBlue,fill=white] (1.8708287,0.5) circle (1.7pt);
\end{scope}
\end{tikzpicture}
\caption{The sawtooth along $q\left(v\right)=a-v^2/2$ for $a=0$, $1/4$, and
$3/4$. Filled points give $\psi\left(n\right)=-1/2$ at integer crossings; open
points give the right-hand limits $1/2$. The shaded initial intervals
end at $T=1$, $1/\sqrt2$, and $1/\sqrt2$, respectively, and give the
exact integrals $1/3$, $-\sqrt2/6$, and $\sqrt2/12$. The curves display
the jump convention and the equality examples used in the proof.}
\label{fig:quadratic-sawtooth}
\end{figure}
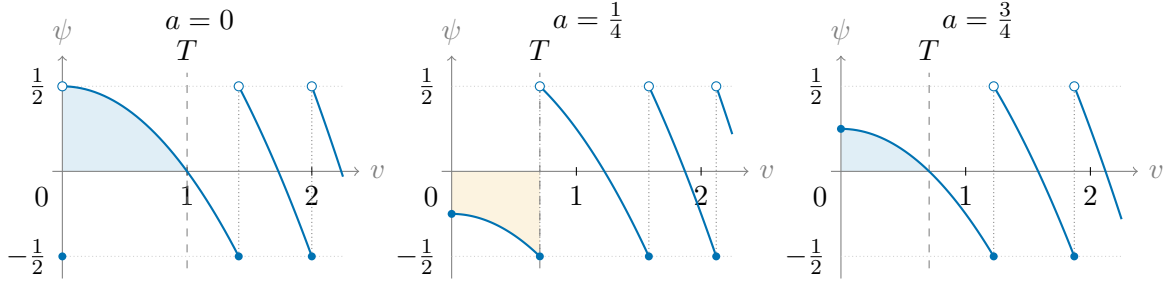}

{The lemma also improves the absolute-value estimate to
$\abs{\int_0^T\psi\paren{q}\dd v}\le1/\paren{3\sqrt\kappa}$.
The lower bound and the dependence on the initial phase
give the refinements needed for the two spectral counts. The next lemma
also retains the initial phase in the two cases not covered by
\eqref{eq:rounding-phase-lower} and
\eqref{eq:rounding-phase-upper}.}

% Keep the lemma opening and its first display on the same page.
\par
\ifdim\dimexpr\pagegoal-\pagetotal\relax<16\baselineskip
  \newpage
\fi
{\begin{lemma}[Bounds retaining every initial phase]\label{lem:all-phase-rounding}
Under the hypotheses of Lemma~\ref{lem:saw}, put
$a=\set{q\paren{0}}$. Define on $0\le a\le1$
\[
\begin{aligned}
 C_+\paren{a}&=
 \begin{cases}
 \displaystyle\int_0^{1/2-a}
       \frac{z}{\sqrt{2\paren{1/2+a-z}}}\dd z,
       &0\le a\le1/2,\\[4pt]
 \displaystyle\frac{2\sqrt2}{3}\paren{a-1/2}^{3/2},
       &1/2\le a\le1,
 \end{cases}\\
 C_-\paren{a}&=
 \begin{cases}
 \displaystyle\sqrt{2a}\paren{\frac12-\frac{2a}{3}},
       &0\le a\le1/2,\\[4pt]
 \displaystyle\int_{a-1/2}^{1/2}
       \frac{z}{\sqrt{2\paren{a-1/2+z}}}\dd z,
       &1/2\le a\le1.
 \end{cases}
 \end{aligned}
\]
Then
\[
-\frac{C_-\paren{a}}{\sqrt\kappa}
 \le\int_0^T\psi\paren{q\paren{v}}\dd v
 \le\frac{C_+\paren{a}}{\sqrt\kappa}.
\]
Both functions are nonnegative and continuous, with matching
endpoint values, and therefore extend continuously with period one.
Their maxima are $1/3$ and $\sqrt2/6$, respectively.
The integral branches have the equivalent expressions
\begin{align}
 C_+\paren{a}
 &=\frac{\sqrt2}{3}\,
   \frac{\paren{1/2-a}^2
     \paren{2\sqrt{1/2+a}+\sqrt{2a}}}
   {\paren{\sqrt{1/2+a}+\sqrt{2a}}^2},
 &&0\le a\le1/2,\label{eq:rounding-Cplus-rationalized}\\
 C_-\paren{a}
 &=\frac{\paren{1-a}\paren{1+2\sqrt{a\paren{2a-1}}}}
   {3\sqrt2\paren{\sqrt a+\sqrt{2a-1}}},
 &&1/2\le a\le1.\label{eq:rounding-Cminus-rationalized}
\end{align}
\end{lemma}}

\begin{proof}
{Keep $g_a$, $Y$, and the nonincreasing weight $w$ from
\eqref{eq:rounding-phase-change}. The lower bound when $a\le1/2$
and the upper bound when $a\ge1/2$ are exactly the initial
half-period estimates in Lemma~\ref{lem:saw}. We verify the two
remaining branches for every terminal phase $Y$.}

{For the upper bound with $0\le a\le1/2$, put $l=1/2-a$.
The initial negative interval $0<y<a$ consists of magnitudes
$l<z<1/2$, with $-z$ occurring at $y=z-l$. The first positive
interval is $a<y<a+1/2=1-l$, with $z$ occurring at
$y=1-l-z$. For $l\le z\le1/2$,}
{\[
 z-l\le1-l-z,\qquad
 z\bracket{w\left(1-l-z\right)-w\left(z-l\right)}\le0.
\]}
{These are precisely the negative and positive values that can be
matched. The residual positive magnitudes are $0<z<l$.
If $Y\le a$, the integral is nonpositive. If
$a<Y\le1-l$, set $r=1-l-Y\in\bracket{0,1/2}$; the included
positive magnitudes are $r<z<1/2$. When $r\ge l$, every included
positive value has a match, and the unmatched negative magnitudes
$l<z<r$ give}
{\[
 I=\int_r^{1/2}z\bracket{w\left(1-l-z\right)-w\left(z-l\right)}\dd z
   -\int_l^r z\,w\left(z-l\right)\dd z\le0.
\]}
{When $0\le r<l$, all negative values are matched, and}
{\[
 I=\int_l^{1/2}z\bracket{w\left(1-l-z\right)-w\left(z-l\right)}\dd z
   +\int_r^l z\,w\left(1-l-z\right)\dd z
 \le\int_0^l z\,w\left(1-l-z\right)\dd z.
\]}
{If $Y>1-l$, the same initial decomposition holds with $r=0$,
while the tail starts with a negative half-period and is
nonpositive by~\eqref{eq:rounding-negative-positive} and its
truncated version. Thus in every case,}
{\[
 I\le\int_0^l\frac{z}{\sqrt{2\kappa\left(1-l-z\right)}}\dd z
 =\frac{C_+\left(a\right)}{\sqrt\kappa}.
\]}
{The endpoint $a=0$ has no initial negative interval and $l=1/2$;
the endpoint $a=1/2$ has no residual positive interval and $l=0$.
The same identities include both endpoints.}

{For the lower bound with $1/2\le a<1$, put $l=a-1/2$.
Match $z$ at $y=l-z$ to $-z$ at $y=l+z$ for $0<z<l$.
The matched integral is nonnegative. A stop with $Y\le l$ is
nonnegative; a stop with $l<Y\le2l$ retains only negative values
that already have a positive match and leaves unmatched positive
values, so it is also nonnegative. For $2l<Y\le a$, formula
\eqref{eq:rounding-initial-positive-pair} with $d=Y-l$ leaves
only the residual negative magnitudes $l<z<d\le1/2$.
For $Y>a$, the same residual interval is complete, and the tail
of positive--negative pairs is nonnegative. It follows that}
{\[
 I\ge-\int_l^{1/2}\frac{z}{\sqrt{2\kappa\left(l+z\right)}}\dd z
 =-\frac{C_-\left(a\right)}{\sqrt\kappa}.
\]}
{When $a=1/2$ this is the whole initial negative half-period,
since $l=0$. The formulas at $a=1$ agree with those at $a=0$,
as checked below; the initial fractional part itself belongs to
$\bracket{0,1}$ with $1$ excluded.}

{We next evaluate the integral branches. For $0\le a\le1/2$,
write $l=1/2-a$, $B=1/2+a$, $p=\sqrt B$, and $r=\sqrt{2a}$.
The substitution $u=B-z$ gives}
{\begin{align*}
 C_+\left(a\right)
 &=\frac1{\sqrt2}\int_{B-l}^{B}\left(B-u\right)u^{-1/2}\dd u\\
 &=\frac1{\sqrt2}
   \bracket{2B\sqrt u-\frac23u^{3/2}}_{B-l}^{B}\\
 &=\frac{\sqrt2}{3}\paren{2p^3-3p^2r+r^3}
 =\frac{\sqrt2}{3}\left(p-r\right)^2\left(2p+r\right).
\end{align*}}
{Since $p^2-r^2=l$, replacing $p-r$ by $l/\left(p+r\right)$ proves
\eqref{eq:rounding-Cplus-rationalized}. This replacement also
avoids subtracting nearly equal square roots near $a=1/2$.}

{For $1/2\le a\le1$, write $l=a-1/2$, $p=\sqrt a$, and
$r=\sqrt{2a-1}$. Substitution $u=l+z$ yields}
{\begin{align*}
 C_-\left(a\right)
 &=\frac1{\sqrt2}\int_{2l}^{a}\left(u-l\right)u^{-1/2}\dd u\\
 &=\frac1{\sqrt2}
   \bracket{\frac23u^{3/2}-2l\sqrt u}_{2l}^{a}
 =\frac{\left(3-4a\right)p+r^3}{3\sqrt2}.
\end{align*}}
{The factorization follows by multiplying the numerator by $p+r$:}
{\[
 \bracket{\left(3-4a\right)p+r^3}\left(p+r\right)
 =\left(1-a\right)\paren{1+2pr},
\]}
{using $p^2=a$ and $r^2=2a-1$. Division by $p+r>0$ proves
\eqref{eq:rounding-Cminus-rationalized}, with no subtractive
cancellation near $a=1$.}

{All branches are nonnegative by their defining integrals or their
explicit positive factors. At the junction and endpoints they give}
{\[
 \begin{aligned}
 &C_+\left(1/2\right)=0,\qquad C_-\left(1/2\right)=1/6,\\
 &C_+\left(0\right)=C_+\left(1\right)=1/3,\qquad C_-\left(0\right)=C_-\left(1\right)=0.
 \end{aligned}
\]}
{The rationalized formulas have positive denominators throughout
their closed intervals, and hence are continuous there; in
particular they also establish continuity at the integrable
singularity of the original $C_+$ integral when $a=0$.
The other branches are continuous directly. Matching values at
$a=1/2$ and at $a=0,1$ therefore give the claimed continuous
periodic extensions.}

{For $C_+$ on $\bracket{0,1/2}$, increasing $a$ shrinks the
positive integration interval and increases the denominator at
every retained $z$, so this branch decreases from $1/3$ to zero.
On $\bracket{1/2,1}$ its power expression increases back to
$1/3$. For $C_-$ on $\bracket{0,1/2}$, the derivative computed
in Lemma~\ref{lem:saw} is positive before $1/4$, vanishes there,
and is negative afterwards; its maximum is $\sqrt2/6$.
On $\bracket{1/2,1}$, increasing $l=a-1/2$ both shrinks the
integration interval $\bracket{l,1/2}$ and enlarges its denominator.
This branch decreases from $1/6$ to zero. Thus}
{\[
 \max_{0\le a\le1} C_+\left(a\right)=\frac13,\qquad
 \max_{0\le a\le1} C_-\left(a\right)=\frac{\sqrt2}{6},
\]}
{as depicted in Figure~\ref{fig:rounding-phase-constants}.}
\end{proof}

{\begin{figure}[htbp]
\centering
\begin{tikzpicture}[x=11.7cm,y=11cm,font=\small]
\draw[->,gray] (-0.015,0) -- (1.055,0) node[right] {$a$};
\draw[->,gray] (0,-0.012) -- (0,0.365) node[above] {$C_\pm\left(a\right)$};
\draw[gray!35,densely dotted] (0,0.3333333) -- (1,0.3333333);
\draw[gray!35,densely dotted] (0,0.2357023) -- (0.25,0.2357023);
\draw[gray!35,densely dotted] (0.25,0) -- (0.25,0.2357023);
\draw[gray!35,densely dotted] (0.5,0) -- (0.5,0.1666667);
\node[anchor=east] at (-0.01,0.3333333) {$\frac13$};
\node[anchor=east] at (-0.01,0.2357023) {$\frac{\sqrt2}{6}$};
\node[anchor=east] at (-0.01,0.1666667) {$\frac16$};
\node[anchor=north east] at (0,-0.003) {$0$};
\node[anchor=north] at (0.25,-0.005) {$\frac14$};
\node[anchor=north] at (0.5,-0.005) {$\frac12$};
\node[anchor=north] at (1,-0.005) {$1$};
\draw[annBlue,thick] plot coordinates {(0.0000000,0.3333333) (0.0000454,0.3266441) (0.0001814,0.3200444) (0.0004082,0.3135330) (0.0007256,0.3071088) (0.0011338,0.3007706) (0.0016327,0.2945172) (0.0022222,0.2883477) (0.0029025,0.2822610) (0.0036735,0.2762560) (0.0045351,0.2703316) (0.0054875,0.2644870) (0.0065306,0.2587212) (0.0076644,0.2530331) (0.0088889,0.2474220) (0.0102041,0.2418869) (0.0116100,0.2364269) (0.0131066,0.2310411) (0.0146939,0.2257289) (0.0163719,0.2204893) (0.0181406,0.2153216) (0.0200000,0.2102250) (0.0219501,0.2051989) (0.0239909,0.2002424) (0.0261224,0.1953550) (0.0283447,0.1905359) (0.0306576,0.1857844) (0.0330612,0.1811001) (0.0355556,0.1764823) (0.0381406,0.1719303) (0.0408163,0.1674437) (0.0435828,0.1630219) (0.0464399,0.1586644) (0.0493878,0.1543707) (0.0524263,0.1501402) (0.0555556,0.1459727) (0.0587755,0.1418675) (0.0620862,0.1378244) (0.0654875,0.1338428) (0.0689796,0.1299225) (0.0725624,0.1260629) (0.0762358,0.1222639) (0.0800000,0.1185251) (0.0838549,0.1148461) (0.0878005,0.1112267) (0.0918367,0.1076665) (0.0959637,0.1041654) (0.1001814,0.1007231) (0.1044898,0.0973393) (0.1088889,0.0940138) (0.1133787,0.0907465) (0.1179592,0.0875371) (0.1226304,0.0843855) (0.1273923,0.0812915) (0.1322449,0.0782550) (0.1371882,0.0752759) (0.1422222,0.0723540) (0.1473469,0.0694893) (0.1525624,0.0666817) (0.1578685,0.0639310) (0.1632653,0.0612373) (0.1687528,0.0586004) (0.1743311,0.0560204) (0.1800000,0.0534972) (0.1857596,0.0510308) (0.1916100,0.0486213) (0.1975510,0.0462686) (0.2035828,0.0439727) (0.2097052,0.0417337) (0.2159184,0.0395516) (0.2222222,0.0374265) (0.2286168,0.0353584) (0.2351020,0.0333475) (0.2416780,0.0313937) (0.2483447,0.0294972) (0.2551020,0.0276582) (0.2619501,0.0258766) (0.2688889,0.0241526) (0.2759184,0.0224864) (0.2830385,0.0208780) (0.2902494,0.0193277) (0.2975510,0.0178356) (0.3049433,0.0164018) (0.3124263,0.0150265) (0.3200000,0.0137099) (0.3276644,0.0124522) (0.3354195,0.0112535) (0.3432653,0.0101141) (0.3512018,0.0090342) (0.3592290,0.0080139) (0.3673469,0.0070535) (0.3755556,0.0061532) (0.3838549,0.0053132) (0.3922449,0.0045339) (0.4007256,0.0038153) (0.4092971,0.0031578) (0.4179592,0.0025617) (0.4267120,0.0020271) (0.4355556,0.0015543) (0.4444898,0.0011437) (0.4535147,0.0007955) (0.4626304,0.0005099) (0.4718367,0.0002873) (0.4811338,0.0001279) (0.4905215,0.0000320) (0.5000000,0.0000000)};
\draw[annBlue,thick] plot coordinates {(0.5000000,0.0000000) (0.5055556,0.0003904) (0.5111111,0.0011042) (0.5166667,0.0020286) (0.5222222,0.0031232) (0.5277778,0.0043649) (0.5333333,0.0057378) (0.5388889,0.0072304) (0.5444444,0.0088338) (0.5500000,0.0105409) (0.5555556,0.0123457) (0.5611111,0.0142431) (0.5666667,0.0162288) (0.5722222,0.0182991) (0.5777778,0.0204506) (0.5833333,0.0226805) (0.5888889,0.0249859) (0.5944444,0.0273646) (0.6000000,0.0298142) (0.6055556,0.0323330) (0.6111111,0.0349189) (0.6166667,0.0375702) (0.6222222,0.0402855) (0.6277778,0.0430632) (0.6333333,0.0459020) (0.6388889,0.0488006) (0.6444444,0.0517577) (0.6500000,0.0547723) (0.6555556,0.0578432) (0.6611111,0.0609694) (0.6666667,0.0641500) (0.6722222,0.0673841) (0.6777778,0.0706708) (0.6833333,0.0740092) (0.6888889,0.0773986) (0.6944444,0.0808383) (0.7000000,0.0843274) (0.7055556,0.0878653) (0.7111111,0.0914514) (0.7166667,0.0950850) (0.7222222,0.0987654) (0.7277778,0.1024922) (0.7333333,0.1062647) (0.7388889,0.1100823) (0.7444444,0.1139447) (0.7500000,0.1178511) (0.7555556,0.1218012) (0.7611111,0.1257945) (0.7666667,0.1298305) (0.7722222,0.1339088) (0.7777778,0.1380289) (0.7833333,0.1421904) (0.7888889,0.1463929) (0.7944444,0.1506360) (0.8000000,0.1549193) (0.8055556,0.1592425) (0.8111111,0.1636052) (0.8166667,0.1680070) (0.8222222,0.1724475) (0.8277778,0.1769265) (0.8333333,0.1814437) (0.8388889,0.1859986) (0.8444444,0.1905911) (0.8500000,0.1952207) (0.8555556,0.1998872) (0.8611111,0.2045903) (0.8666667,0.2093297) (0.8722222,0.2141052) (0.8777778,0.2189164) (0.8833333,0.2237632) (0.8888889,0.2286452) (0.8944444,0.2335622) (0.9000000,0.2385139) (0.9055556,0.2435002) (0.9111111,0.2485207) (0.9166667,0.2535753) (0.9222222,0.2586636) (0.9277778,0.2637856) (0.9333333,0.2689409) (0.9388889,0.2741294) (0.9444444,0.2793508) (0.9500000,0.2846050) (0.9555556,0.2898917) (0.9611111,0.2952107) (0.9666667,0.3005619) (0.9722222,0.3059450) (0.9777778,0.3113599) (0.9833333,0.3168063) (0.9888889,0.3222842) (0.9944444,0.3277932) (1.0000000,0.3333333)};
\draw[annOrange,thick] plot coordinates {(0.0000000,0.0000000) (0.0000454,0.0047616) (0.0001814,0.0095215) (0.0004082,0.0142779) (0.0007256,0.0190292) (0.0011338,0.0237735) (0.0016327,0.0285092) (0.0022222,0.0332346) (0.0029025,0.0379478) (0.0036735,0.0426472) (0.0045351,0.0473311) (0.0054875,0.0519977) (0.0065306,0.0566453) (0.0076644,0.0612721) (0.0088889,0.0658765) (0.0102041,0.0704568) (0.0116100,0.0750110) (0.0131066,0.0795377) (0.0146939,0.0840350) (0.0163719,0.0885012) (0.0181406,0.0929345) (0.0200000,0.0973333) (0.0219501,0.1016959) (0.0239909,0.1060204) (0.0261224,0.1103052) (0.0283447,0.1145485) (0.0306576,0.1187486) (0.0330612,0.1229038) (0.0355556,0.1270123) (0.0381406,0.1310725) (0.0408163,0.1350826) (0.0435828,0.1390409) (0.0464399,0.1429455) (0.0493878,0.1467949) (0.0524263,0.1505873) (0.0555556,0.1543210) (0.0587755,0.1579942) (0.0620862,0.1616052) (0.0654875,0.1651522) (0.0689796,0.1686336) (0.0725624,0.1720477) (0.0762358,0.1753926) (0.0800000,0.1786667) (0.0838549,0.1818682) (0.0878005,0.1849954) (0.0918367,0.1880466) (0.0959637,0.1910201) (0.1001814,0.1939141) (0.1044898,0.1967269) (0.1088889,0.1994568) (0.1133787,0.2021020) (0.1179592,0.2046608) (0.1226304,0.2071316) (0.1273923,0.2095124) (0.1322449,0.2118017) (0.1371882,0.2139978) (0.1422222,0.2160988) (0.1473469,0.2181030) (0.1525624,0.2200088) (0.1578685,0.2218143) (0.1632653,0.2235180) (0.1687528,0.2251179) (0.1743311,0.2266125) (0.1800000,0.2280000) (0.1857596,0.2292786) (0.1916100,0.2304467) (0.1975510,0.2315024) (0.2035828,0.2324442) (0.2097052,0.2332701) (0.2159184,0.2339786) (0.2222222,0.2345679) (0.2286168,0.2350362) (0.2351020,0.2353819) (0.2416780,0.2356032) (0.2483447,0.2356984) (0.2551020,0.2356657) (0.2619501,0.2355034) (0.2688889,0.2352099) (0.2759184,0.2347833) (0.2830385,0.2342219) (0.2902494,0.2335241) (0.2975510,0.2326880) (0.3049433,0.2317121) (0.3124263,0.2305944) (0.3200000,0.2293333) (0.3276644,0.2279271) (0.3354195,0.2263741) (0.3432653,0.2246725) (0.3512018,0.2228206) (0.3592290,0.2208166) (0.3673469,0.2186589) (0.3755556,0.2163457) (0.3838549,0.2138752) (0.3922449,0.2112459) (0.4007256,0.2084558) (0.4092971,0.2055034) (0.4179592,0.2023868) (0.4267120,0.1991043) (0.4355556,0.1956543) (0.4444898,0.1920350) (0.4535147,0.1882446) (0.4626304,0.1842815) (0.4718367,0.1801438) (0.4811338,0.1758300) (0.4905215,0.1713382) (0.5000000,0.1666667)};
\draw[annOrange,thick] plot coordinates {(0.5000000,0.1666667) (0.5055556,0.1641419) (0.5111111,0.1617999) (0.5166667,0.1595613) (0.5222222,0.1573981) (0.5277778,0.1552942) (0.5333333,0.1532388) (0.5388889,0.1512242) (0.5444444,0.1492444) (0.5500000,0.1472947) (0.5555556,0.1453713) (0.5611111,0.1434710) (0.5666667,0.1415910) (0.5722222,0.1397290) (0.5777778,0.1378830) (0.5833333,0.1360512) (0.5888889,0.1342321) (0.5944444,0.1324242) (0.6000000,0.1306264) (0.6055556,0.1288374) (0.6111111,0.1270563) (0.6166667,0.1252822) (0.6222222,0.1235142) (0.6277778,0.1217516) (0.6333333,0.1199936) (0.6388889,0.1182397) (0.6444444,0.1164892) (0.6500000,0.1147415) (0.6555556,0.1129963) (0.6611111,0.1112529) (0.6666667,0.1095110) (0.6722222,0.1077701) (0.6777778,0.1060299) (0.6833333,0.1042900) (0.6888889,0.1025501) (0.6944444,0.1008099) (0.7000000,0.0990690) (0.7055556,0.0973273) (0.7111111,0.0955844) (0.7166667,0.0938401) (0.7222222,0.0920942) (0.7277778,0.0903465) (0.7333333,0.0885967) (0.7388889,0.0868447) (0.7444444,0.0850903) (0.7500000,0.0833333) (0.7555556,0.0815736) (0.7611111,0.0798110) (0.7666667,0.0780454) (0.7722222,0.0762766) (0.7777778,0.0745045) (0.7833333,0.0727290) (0.7888889,0.0709499) (0.7944444,0.0691673) (0.8000000,0.0673808) (0.8055556,0.0655905) (0.8111111,0.0637963) (0.8166667,0.0619980) (0.8222222,0.0601956) (0.8277778,0.0583890) (0.8333333,0.0565781) (0.8388889,0.0547629) (0.8444444,0.0529433) (0.8500000,0.0511192) (0.8555556,0.0492905) (0.8611111,0.0474572) (0.8666667,0.0456193) (0.8722222,0.0437766) (0.8777778,0.0419292) (0.8833333,0.0400769) (0.8888889,0.0382198) (0.8944444,0.0363577) (0.9000000,0.0344907) (0.9055556,0.0326187) (0.9111111,0.0307417) (0.9166667,0.0288596) (0.9222222,0.0269724) (0.9277778,0.0250801) (0.9333333,0.0231826) (0.9388889,0.0212799) (0.9444444,0.0193720) (0.9500000,0.0174588) (0.9555556,0.0155403) (0.9611111,0.0136166) (0.9666667,0.0116875) (0.9722222,0.0097531) (0.9777778,0.0078133) (0.9833333,0.0058681) (0.9888889,0.0039175) (0.9944444,0.0019615) (1.0000000,0.0000000)};
\fill[annBlue] (0,0.3333333) circle (1.8pt);
\fill[annBlue] (1,0.3333333) circle (1.8pt);
\fill[annBlue] (0.5,0) circle (1.8pt);
\fill[annOrange] (0,0) circle (1.8pt);
\fill[annOrange] (1,0) circle (1.8pt);
\fill[annOrange] (0.25,0.2357023) circle (1.8pt);
\fill[annOrange] (0.5,0.1666667) circle (1.8pt);
\node[annBlue,anchor=west] at (0.83,0.283) {$C_+\left(a\right)$};
\node[annOrange,anchor=west] at (0.64,0.125) {$C_-\left(a\right)$};
\end{tikzpicture}
\caption{The continuous phase bounds over one period. The upper
constant $C_+$ decreases to zero at $a=1/2$ and then increases to its
matching endpoint value $1/3$. The lower constant $C_-$ reaches
$\sqrt2/6$ at $a=1/4$, equals $1/6$ at $a=1/2$, and vanishes at both
endpoints. The extrema and endpoint identifications follow from
the formulas and monotonicity in the proof of
Lemma~\ref{lem:all-phase-rounding}.}
\label{fig:rounding-phase-constants}
\end{figure}
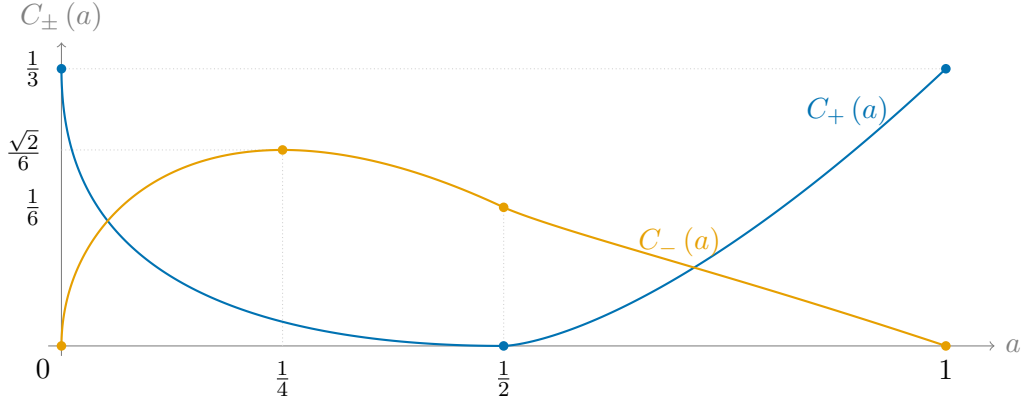}

{The universal constants and the two initial adverse half-period
branches are sharp by Lemma~\ref{lem:saw}. The matched residual
formulas supply bounds at every remaining initial phase.
In particular, $C_+$ decreases up to $1/2$ and increases thereafter.
Thus for $0\le a_0\le a_1\le1$,}
{\[
 \max_{a_0\le a\le a_1}C_+\paren{a}
 =\max\set{C_+\paren{a_0},C_+\paren{a_1}}.
\]}
{This endpoint control bounds an entire initial-phase interval even
when it crosses $1/2$, and is used in the frequency estimates below.}

\section{\revblue{Explicit Dirichlet gaps}}\label{sec:dirichlet-quantitative}
{\color{editorial}This section has two objectives. First, we prove
Theorem~\ref{thm:dirichlet} by finding a positive lower bound for
the continuous deficit at every frequency at which a sector
eigenvalue can occur. A finite-row estimate controls low
frequencies, and a phase estimate controls the tail; the switching
calculation in Subsection~3.2 makes the two ranges overlap.
Second, Subsections~3.3--3.4 retain more information from those
estimates. They strengthen the explicit lower bounds and recover
the full circular-arc coefficient at high frequency. These
improvements strengthen the main theorem, while the proof of its
all-index assertion does not depend on them.}

{Normalize the inner radius to one and put}
{\[
 R>1,\qquad \delta=R-1,\qquad D=R^2-1,\qquad
 p=\frac\pi\beta,\qquad x=\delta t,\qquad A=\frac{R+1}4.
\]}
{Use the Dirichlet radial family $\lambda^{\mathrm D}_{v,n}$,
$n\ge1$, and its counts $n_v^{\mathrm D}$ and $M_{\mathrm D}$
defined in Section~\ref{sec:reduction}. Its continuous deficit is}
{\[
 B_{\mathrm D}\paren{t}=\frac D4t^2-M_{\mathrm D}\paren{t}.
\]}
{The qualitative P\'olya comparison follows from
$B_{\mathrm D}(t)\ge0$ at the relevant frequencies. The main
theorem proves the stronger uniform bound
\[
 c_\delta\le\inf_{\delta t\ge\pi}
                   \frac{B_{\mathrm D}\paren t}{t}.
\]
Replacing a frequency-dependent deficit by $c_\delta t$ makes
the result simple and uniform, but discards the excess above that
lower envelope. The later improvements retain that excess and
the angular floor information that continuous integration
discards. In every case the leading area coefficient stays
unchanged.}
{The term $Dt^2/4$ is the area term in this continuous angular
count. The two circular arcs contribute the expected linear
deficit $(R+1)t/2$; right-endpoint sampling will transfer any
lower bound for $B_{\mathrm D}$ to the sector count. The radial
sides enter through that sampling, and are not part of the
continuous count. Section~\ref{sec:boundary-optimality} proves
these coefficient statements. We use them here to explain the
target of the refinements, rather than as assumptions in the
proof of Theorem~\ref{thm:dirichlet}.}

\subsection{\revblue{Finite radial rows and directional rounding}}
{The unitary substitution $z=\sqrt{s}\,u$ transforms the radial
operator, with Dirichlet endpoint conditions, into}
{\[ -z''+\paren{v^2-1/4}s^{-2}z. \]}  {Taking the minimum of this signed potential
and applying the min--max principle gives}
{\begin{equation}\label{eq:Dq-signed}
 \lambda^{\mathrm D}_{v,n}\ge
 \paren{\frac{\pi n}\delta}^2+
 \begin{cases}
  v^2-1/4,&0\le v\le1/2,\\
  \paren{v^2-1/4}/R^2,&v\ge1/2.
 \end{cases}
\end{equation}}
{The potential is strictly above its minimum in the interior when
$v\ne1/2$, so the comparison is strict at those orders.  At
$v=1/2$ the potential is zero and equality holds.}

{\color{revision}To obtain the row lengths, fix $n\ge1$ and put
$y=t^2-\paren{\pi n/\delta}^2$.  If $y<0$, the branch $v\ge1/2$
in \eqref{eq:Dq-signed} cannot contribute, while the other branch
requires $0\le v\le\sqrt{y+1/4}$ when $y+1/4\ge0$.
If $y\ge0$, all orders $0\le v\le1/2$ are permitted by the comparison,
and its second branch extends the permitted interval up to
$\sqrt{R^2y+1/4}$.  Thus the set of orders for which
$\lambda^{\mathrm D}_{v,n}\le t^2$ is contained in an interval
starting at zero whose length is at most $P_n\paren{\delta,x}/\delta$,
where $x=\delta t$ and}
{\[
 P_n\paren{\delta,x}=
 \begin{cases}
  \sqrt{\pos{x^2-\pi^2n^2+\delta^2/4}},&x<\pi n,\\
  \sqrt{R^2\paren{x^2-\pi^2n^2}+\delta^2/4},&x\ge\pi n.
 \end{cases}
\]}
{\color{revision}At $x=\pi n$ the two formulas agree and give
$P_n=\delta/2$.  Moreover, $P_n=0$ whenever
$\pi^2n^2\ge x^2+\delta^2/4$, so only finitely many rows contribute.
Integrating each row in the order variable, or equivalently applying
Tonelli's theorem to the nonnegative count, gives
\[
 M_{\mathrm D}\paren t
 =2\sum_{n\ge1}\bigl|\set{v\ge0:\lambda^{\mathrm D}_{v,n}\le t^2}\bigr|
 \le\frac2\delta\sum_{n\ge1}P_n\paren{\delta,x}.
\]
Since $Dt^2/4=Ax^2/\delta$, it follows that}
{\begin{equation}\label{eq:Dq-rows}
 B_{\mathrm D}\paren{t}\ge
 Q_\delta\paren{x}:={\frac{1}{\delta}}
 \paren{Ax^2-2\sum_{n\ge1}P_n\paren{\delta,x}}.
\end{equation}}

{\color{revision}Figure~\ref{fig:Dq-row} illustrates the two branches of this bound.}

{\color{revision}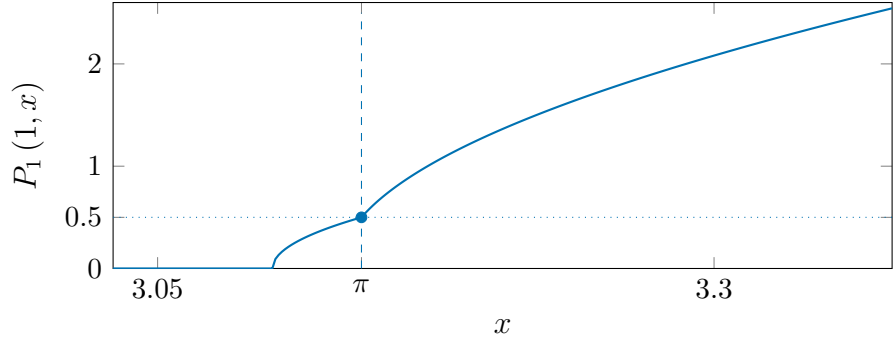
\begin{figure}[htbp]
\centering
\begin{tikzpicture}
\begin{axis}[width=.72\textwidth,height=5.1cm,xlabel={$x$},ylabel={$P_1\paren{1,x}$},xmin=3.03,xmax=3.38,ymin=0,ymax=2.6,xtick={3.05,3.141592654,3.3},xticklabels={$3.05$,$\pi$,$3.3$},ytick={0,.5,1,2},tick label style={font=\small}]
\addplot[annBlue,thick,] coordinates {(3.03,0) (3.03145833,0) (3.03291667,0) (3.034375,0) (3.03583333,0) (3.03729167,0) (3.03875,0) (3.04020833,0) (3.04166667,0) (3.043125,0) (3.04458333,0) (3.04604167,0) (3.0475,0) (3.04895833,0) (3.05041667,0) (3.051875,0) (3.05333333,0) (3.05479167,0) (3.05625,0) (3.05770833,0) (3.05916667,0) (3.060625,0) (3.06208333,0) (3.06354167,0) (3.065,0) (3.06645833,0) (3.06791667,0) (3.069375,0) (3.07083333,0) (3.07229167,0) (3.07375,0) (3.07520833,0) (3.07666667,0) (3.078125,0) (3.07958333,0) (3.08104167,0) (3.0825,0) (3.08395833,0) (3.08541667,0) (3.086875,0) (3.08833333,0) (3.08979167,0) (3.09125,0) (3.09270833,0) (3.09416667,0) (3.095625,0) (3.09708333,0) (3.09854167,0) (3.1,0) (3.10145833,0) (3.10154871,0) (3.10291667,0.0921272988) (3.104375,0.132437682) (3.10583333,0.1630837) (3.10729167,0.188830878) (3.10875,0.211476149) (3.11020833,0.231929894) (3.11166667,0.250729024) (3.113125,0.268221671) (3.11458333,0.2846488) (3.11604167,0.300185388) (3.1175,0.31496325) (3.11895833,0.329084615) (3.12041667,0.342630665) (3.121875,0.35566714) (3.12333333,0.368248164) (3.12479167,0.380418926) (3.12625,0.39221762) (3.12770833,0.403676872) (3.12916667,0.414824814) (3.130625,0.425685905) (3.13208333,0.436281567) (3.13354167,0.446630693) (3.135,0.45675004) (3.13645833,0.466654557) (3.13791667,0.476357645) (3.139375,0.485871371) (3.14083333,0.49520665) (3.14159265,0.5) (3.14229167,0.517271756) (3.14375,0.551578322) (3.14520833,0.583887177) (3.14666667,0.614513499) (3.148125,0.64369749) (3.14958333,0.671627196) (3.15104167,0.698453102) (3.1525,0.724297864) (3.15395833,0.749263017) (3.15541667,0.773433744) (3.156875,0.796882336) (3.15833333,0.819670771) (3.15979167,0.841852661) (3.16125,0.863474751) (3.16270833,0.884578093) (3.16416667,0.905198969) (3.165625,0.925369633) (3.16708333,0.945118911) (3.16854167,0.96447269) (3.17,0.983454318) (3.17145833,1.00208494) (3.17291667,1.0203838) (3.174375,1.03836841) (3.17583333,1.05605485) (3.17729167,1.07345784) (3.17875,1.09059096) (3.18020833,1.10746674) (3.18166667,1.12409675) (3.183125,1.14049176) (3.18458333,1.15666176) (3.18604167,1.17261605) (3.1875,1.18836333) (3.18895833,1.20391171) (3.19041667,1.21926881) (3.191875,1.23444176) (3.19333333,1.24943728) (3.19479167,1.26426167) (3.19625,1.27892089) (3.19770833,1.29342056) (3.19916667,1.30776597) (3.200625,1.32196216) (3.20208333,1.33601388) (3.20354167,1.34992564) (3.205,1.36370173) (3.20645833,1.37734621) (3.20791667,1.39086295) (3.209375,1.40425566) (3.21083333,1.41752784) (3.21229167,1.43068284) (3.21375,1.44372388) (3.21520833,1.45665401) (3.21666667,1.46947616) (3.218125,1.48219312) (3.21958333,1.4948076) (3.22104167,1.50732215) (3.2225,1.51973925) (3.22395833,1.53206127) (3.22541667,1.54429048) (3.226875,1.55642907) (3.22833333,1.56847915) (3.22979167,1.58044273) (3.23125,1.59232178) (3.23270833,1.60411816) (3.23416667,1.61583369) (3.235625,1.62747011) (3.23708333,1.63902911) (3.23854167,1.65051232) (3.24,1.6619213) (3.24145833,1.67325757) (3.24291667,1.68452261) (3.244375,1.69571783) (3.24583333,1.70684461) (3.24729167,1.71790427) (3.24875,1.7288981) (3.25020833,1.73982736) (3.25166667,1.75069325) (3.253125,1.76149694) (3.25458333,1.77223957) (3.25604167,1.78292225) (3.2575,1.79354604) (3.25895833,1.80411199) (3.26041667,1.81462111) (3.261875,1.82507437) (3.26333333,1.83547274) (3.26479167,1.84581714) (3.26625,1.85610847) (3.26770833,1.86634761) (3.26916667,1.87653542) (3.270625,1.88667272) (3.27208333,1.89676033) (3.27354167,1.90679904) (3.275,1.91678961) (3.27645833,1.92673279) (3.27791667,1.93662931) (3.279375,1.94647989) (3.28083333,1.95628521) (3.28229167,1.96604596) (3.28375,1.9757628) (3.28520833,1.98543637) (3.28666667,1.99506729) (3.288125,2.00465619) (3.28958333,2.01420367) (3.29104167,2.02371031) (3.2925,2.03317668) (3.29395833,2.04260334) (3.29541667,2.05199084) (3.296875,2.06133972) (3.29833333,2.0706505) (3.29979167,2.07992369) (3.30125,2.08915979) (3.30270833,2.0983593) (3.30416667,2.10752268) (3.305625,2.11665041) (3.30708333,2.12574295) (3.30854167,2.13480075) (3.31,2.14382425) (3.31145833,2.15281387) (3.31291667,2.16177005) (3.314375,2.1706932) (3.31583333,2.17958372) (3.31729167,2.188442) (3.31875,2.19726845) (3.32020833,2.20606344) (3.32166667,2.21482735) (3.323125,2.22356054) (3.32458333,2.23226337) (3.32604167,2.2409362) (3.3275,2.24957938) (3.32895833,2.25819325) (3.33041667,2.26677813) (3.331875,2.27533436) (3.33333333,2.28386226) (3.33479167,2.29236215) (3.33625,2.30083434) (3.33770833,2.30927912) (3.33916667,2.31769681) (3.340625,2.32608769) (3.34208333,2.33445206) (3.34354167,2.3427902) (3.345,2.35110238) (3.34645833,2.35938888) (3.34791667,2.36764998) (3.349375,2.37588593) (3.35083333,2.384097) (3.35229167,2.39228344) (3.35375,2.40044551) (3.35520833,2.40858345) (3.35666667,2.41669751) (3.358125,2.42478792) (3.35958333,2.43285493) (3.36104167,2.44089876) (3.3625,2.44891964) (3.36395833,2.4569178) (3.36541667,2.46489346) (3.366875,2.47284683) (3.36833333,2.48077814) (3.36979167,2.48868759) (3.37125,2.49657538) (3.37270833,2.50444173) (3.37416667,2.51228684) (3.375625,2.5201109) (3.37708333,2.52791411) (3.37854167,2.53569666) (3.38,2.54345875)};
\addplot[annBlue,dashed] coordinates {(3.141592654,0) (3.141592654,2.5)};\addplot[annBlue,dotted] coordinates {(3.03,.5) (3.38,.5)};\addplot[annBlue,only marks,mark=*] coordinates {(3.141592654,.5)};
\end{axis}
\end{tikzpicture}
\caption{\textcolor{revision}{A permitted-order row for $\delta=1$ and $n=1$. It starts at $x=\sqrt{\pi^2-1/4}$, uses its first branch below $\pi$, and joins its second branch at $P_1=1/2$. This is an upper bound on the actual order interval, not its exact spectral length.}}
\label{fig:Dq-row}
\end{figure}}
{Let $\theta_v$ be the ordinary Bessel phase normalized by
$\theta_v\paren{z}\to-\pi/2$ as $z\downarrow0$.  The real-order
phase count and its global action enclosure are}
{\[
n_v^{\mathrm D}\paren{t^2}=\floor{\gamma_t\paren{v}},\qquad
 \gamma_t\paren{v}={\frac{\theta_v\paren{Rt}-\theta_v\paren{t}}{\pi}}
 < f_t\paren{v}+\frac14,
 \qquad0\le v\le Rt,
\]}
{and the count vanishes for $v\ge Rt$.  These are the real-order
forms of \cite[Lemma~4.1 and Lemma~5.2, equations~(5.6) and~(5.8)]{DirichletAnnuli}.
In particular, the floor has no added constant: the first Dirichlet
crossing is at a phase difference of $\pi$.}

{The floor identity}
\[
 \floor{f_t\paren{v}+\frac14}
 =f_t\paren{v}-\frac14-\psi\paren{f_t\paren{v}+\frac14}
\]
{retains its negative mean.  On the entire inner order interval,}
{\[
 f_t'\paren{0}=0,\qquad
 -f_t''\paren{v}\ge\frac\delta{\pi Rt},\qquad0\le v<t.
\]}
{The lower directional rounding estimate therefore gives}
{\[
 2\int_0^t\psi\paren{f_t\paren{v}+\frac14}\dd v
 \ge-{\frac{\sqrt2}{3}}\sqrt{\frac{\pi Rt}\delta}.
\]}
{On the outer interval set}
\[
 q_{\mathrm{out}}\paren{y}=-f_t\paren{Rt-y}-\frac14,
 \qquad0\le y\le\delta t.
\]
{Then $q_{\mathrm{out}}'\paren{0}=0$, its initial fractional part is
$3/4$, and}
{\[
 -q_{\mathrm{out}}''\paren{y}
 ={\frac{1}{\pi\sqrt{R^2t^2-\paren{Rt-y}^2}}}
 \ge{\frac{1}{\pi t\sqrt D}}.
\]}
{Reflection changes the sign of the sawtooth almost everywhere.  The
upper rounding coefficient at initial fractional part $3/4$ is
$\sqrt2/12$, and hence}
{\[
 2\int_t^{Rt}\psi\paren{f_t\paren{v}+\frac14}\dd v
 \ge-{\frac{\sqrt2}{6}}\sqrt{\pi t\sqrt D}.
\]}
{The endpoint singularities are integrable.  The rounding estimates
thus apply to absolutely continuous derivatives.  Integrating the
action gives}
{\begin{equation}\label{eq:Dq-whole}
 B_{\mathrm D}\paren{t}\ge H_\delta\paren{x}
 :={\frac{R}{2\delta}}\paren{x-\sqrt{S_\delta x}},
\end{equation}}
{where}
{\begin{equation}\label{eq:Dq-S}
 S_\delta={\frac{2\pi}{9}}\paren{2\sqrt h+b\paren{h}}^2,
 \qquad h={\frac{1}{R}},\qquad
 b\paren{h}=\paren{\paren{1-h}^3\paren{1+h}}^{1/4}.
\end{equation}}
{The two relevant derivatives are
\[
 H_\delta'\paren{x}
 =\frac{R}{2\delta}\paren{1-\frac12\sqrt{S_\delta/x}},
 \qquad
 \frac{\mathrm d}{\mathrm dx}
 \paren{\frac{H_\delta\paren{x}}{x/\delta}}
 =\frac{R\sqrt{S_\delta}}{4x^{3/2}}>0.
\]
Thus $H_\delta$ is positive for $x>S_\delta$ and increasing for
$x>S_\delta/4$. Its deficit per unit frequency is increasing
throughout $x>0$. The distinction matters: the all-index argument
compares $B_{\mathrm D}(t)/t$, while the tail inverse below uses
monotonicity of $H_\delta$ itself.}

{The exact geometric threshold in \eqref{eq:Dq-S} has the optimal
universal upper bound furnished by this envelope.  Namely, let
$h_*\in\paren{0,1}$ be the unique root of}
{\begin{equation}\label{eq:Dq-hstar}
 h^2\paren{1+2h}^4=16\paren{1-h}\paren{1+h}^3.
\end{equation}}
Then
{\[
 S_\delta\le S_*:={\frac{2\pi}{9}}
 \paren{2\sqrt{h_*}+b\paren{h_*}}^2,
 \qquad h_*=0.7469694374\ldots,
 \qquad S_*=3.1932979003\ldots.
\]}
{Indeed, logarithmic differentiation gives
\[
 \frac{b'\paren h}{b\paren h}
 =-\frac{1+2h}{2\paren{1-h^2}},\qquad
 \frac{\mathrm d}{\mathrm dh}\paren{2\sqrt h+b\paren h}
 =\frac1{\sqrt h}
  -\frac{b\paren h\paren{1+2h}}{2\paren{1-h^2}}.
\]
Both sides of the critical-point equation are positive. Raising
that equation to the fourth power and cancelling positive factors
gives exactly \eqref{eq:Dq-hstar}. The ratio}
{\[ \frac{h^2\paren{1+2h}^4}{\paren{1-h}\paren{1+h}^3} \]}
{increases strictly from zero to infinity, since its logarithmic derivative is}
{\[
 {\frac{2}{h}}+{\frac{8}{1+2h}}+{\frac{1}{1-h}}-{\frac{3}{1+h}}>0.
\]}
{For completeness, the possibly competing terms satisfy
\[
 \frac8{1+2h}-\frac3{1+h}
 =\frac{5+2h}{\paren{1+2h}\paren{1+h}}>0.
\]
Thus the critical point is unique without recourse to a numerical
plot; numerical evaluation is used only for its displayed decimal
value.}
{\color{revision}Here it is the derivative of $2\sqrt h+b\paren h$ that is
positive near $0$ and negative near $1$; the logarithmic derivative
in the preceding display stays positive.  This distinction proves
the asserted unique maximum.  Since $h=1/\paren{1+\delta}$ reverses
orientation, $S_\delta$ increases up to
$\delta_*=h_*^{-1}-1=0.3387428587\ldots$ and decreases thereafter.
The endpoint limits are
\[
 \lim_{\delta\downarrow0}S_\delta=\frac{8\pi}{9},\qquad
 \lim_{\delta\to+\infty}S_\delta=\frac{2\pi}{9}.
\]
Figure~\ref{fig:Dq-threshold} displays both the maximum and the
full parameter range.  The geometric value $S_\delta$ is preferable
to $S_*$ whenever the radii are known.}

{\color{revision}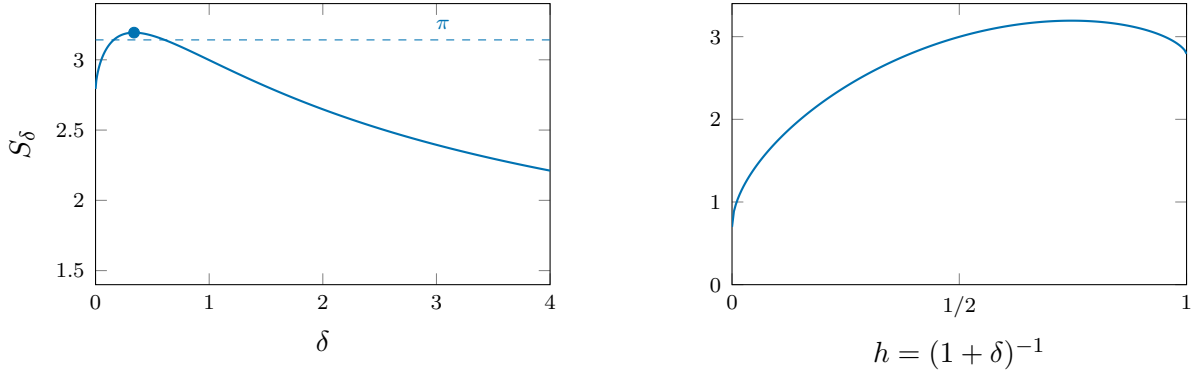
\begin{figure}[htbp]
\centering
\begin{tikzpicture}
\begin{axis}[width=.46\textwidth,height=5.3cm,xlabel={$\delta$},ylabel={$S_\delta$},xmin=0,xmax=4,ymin=1.4,ymax=3.4,xtick={0,1,2,3,4},legend style={draw=none,font=\scriptsize},tick label style={font=\scriptsize},label style={font=\small}]
\addplot[annBlue,thick,] coordinates {(0,2.7925268) (0.00666666667,2.85133935) (0.0133333333,2.88522168) (0.02,2.91234176) (0.0266666667,2.93546089) (0.0333333333,2.95577461) (0.04,2.97394627) (0.0466666667,2.99039554) (0.0533333333,3.00541096) (0.06,3.0192026) (0.0666666667,3.03193011) (0.0733333333,3.04371878) (0.08,3.05466954) (0.0866666667,3.0648654) (0.0933333333,3.07437584) (0.1,3.08325986) (0.106666667,3.09156819) (0.113333333,3.0993449) (0.12,3.10662866) (0.126666667,3.11345366) (0.133333333,3.11985035) (0.14,3.12584601) (0.146666667,3.13146526) (0.153333333,3.13673037) (0.16,3.14166163) (0.166666667,3.14627758) (0.173333333,3.15059524) (0.18,3.15463024) (0.186666667,3.15839707) (0.193333333,3.16190911) (0.2,3.16517879) (0.206666667,3.1682177) (0.213333333,3.17103664) (0.22,3.17364571) (0.226666667,3.17605437) (0.233333333,3.1782715) (0.24,3.18030547) (0.246666667,3.18216411) (0.253333333,3.18385487) (0.26,3.18538473) (0.266666667,3.18676032) (0.273333333,3.18798791) (0.28,3.18907344) (0.286666667,3.19002256) (0.293333333,3.19084061) (0.3,3.19153271) (0.306666667,3.19210368) (0.313333333,3.19255816) (0.32,3.19290054) (0.326666667,3.19313502) (0.333333333,3.19326563) (0.338742859,3.1932979) (0.34,3.19329618) (0.346666667,3.19323035) (0.353333333,3.19307164) (0.36,3.19282341) (0.366666667,3.19248887) (0.373333333,3.19207112) (0.38,3.1915731) (0.386666667,3.19099765) (0.393333333,3.1903475) (0.4,3.18962527) (0.406666667,3.18883346) (0.413333333,3.1879745) (0.42,3.1870507) (0.426666667,3.18606432) (0.433333333,3.18501749) (0.44,3.1839123) (0.446666667,3.18275073) (0.453333333,3.18153472) (0.46,3.18026612) (0.466666667,3.17894671) (0.473333333,3.17757823) (0.48,3.17616234) (0.486666667,3.17470064) (0.493333333,3.17319469) (0.5,3.17164598) (0.506666667,3.17005596) (0.513333333,3.16842603) (0.52,3.16675755) (0.526666667,3.1650518) (0.533333333,3.16331007) (0.54,3.16153356) (0.546666667,3.15972347) (0.553333333,3.15788093) (0.56,3.15600706) (0.566666667,3.15410291) (0.573333333,3.15216954) (0.58,3.15020793) (0.586666667,3.14821907) (0.593333333,3.1462039) (0.6,3.14416332) (0.606666667,3.14209823) (0.613333333,3.14000947) (0.62,3.13789788) (0.626666667,3.13576427) (0.633333333,3.13360942) (0.64,3.13143407) (0.646666667,3.12923897) (0.653333333,3.12702483) (0.66,3.12479234) (0.666666667,3.12254216) (0.673333333,3.12027496) (0.68,3.11799136) (0.686666667,3.11569197) (0.693333333,3.11337738) (0.7,3.11104818) (0.706666667,3.10870493) (0.713333333,3.10634817) (0.72,3.10397843) (0.726666667,3.10159621) (0.733333333,3.09920203) (0.74,3.09679636) (0.746666667,3.09437968) (0.753333333,3.09195244) (0.76,3.08951509) (0.766666667,3.08706805) (0.773333333,3.08461175) (0.78,3.0821466) (0.786666667,3.07967298) (0.793333333,3.0771913) (0.8,3.07470191) (0.824806202,3.06537743) (0.849612403,3.05596878) (0.874418605,3.0464918) (0.899224806,3.03696074) (0.924031008,3.02738846) (0.948837209,3.01778657) (0.973643411,3.00816555) (0.998449612,2.99853484) (1.02325581,2.98890298) (1.04806202,2.97927767) (1.07286822,2.96966585) (1.09767442,2.96007378) (1.12248062,2.95050712) (1.14728682,2.94097094) (1.17209302,2.93146982) (1.19689922,2.92200787) (1.22170543,2.91258877) (1.24651163,2.90321583) (1.27131783,2.89389201) (1.29612403,2.88461994) (1.32093023,2.87540196) (1.34573643,2.86624015) (1.37054264,2.85713636) (1.39534884,2.8480922) (1.42015504,2.83910908) (1.44496124,2.83018824) (1.46976744,2.82133074) (1.49457364,2.81253748) (1.51937984,2.80380923) (1.54418605,2.79514664) (1.56899225,2.78655023) (1.59379845,2.7780204) (1.61860465,2.76955749) (1.64341085,2.76116171) (1.66821705,2.75283321) (1.69302326,2.74457206) (1.71782946,2.73637828) (1.74263566,2.72825179) (1.76744186,2.72019249) (1.79224806,2.71220021) (1.81705426,2.70427473) (1.84186047,2.69641581) (1.86666667,2.68862315) (1.89147287,2.68089643) (1.91627907,2.67323528) (1.94108527,2.66563932) (1.96589147,2.65810814) (1.99069767,2.6506413) (2.01550388,2.64323835) (2.04031008,2.63589881) (2.06511628,2.62862219) (2.08992248,2.62140799) (2.11472868,2.61425569) (2.13953488,2.60716477) (2.16434109,2.60013469) (2.18914729,2.59316489) (2.21395349,2.58625484) (2.23875969,2.57940397) (2.26356589,2.57261171) (2.28837209,2.56587751) (2.31317829,2.55920078) (2.3379845,2.55258097) (2.3627907,2.5460175) (2.3875969,2.53950978) (2.4124031,2.53305726) (2.4372093,2.52665934) (2.4620155,2.52031547) (2.48682171,2.51402508) (2.51162791,2.50778758) (2.53643411,2.50160242) (2.56124031,2.49546902) (2.58604651,2.48938684) (2.61085271,2.48335531) (2.63565891,2.47737388) (2.66046512,2.47144199) (2.68527132,2.46555911) (2.71007752,2.45972468) (2.73488372,2.45393818) (2.75968992,2.44819906) (2.78449612,2.44250681) (2.80930233,2.43686089) (2.83410853,2.43126078) (2.85891473,2.42570599) (2.88372093,2.42019598) (2.90852713,2.41473027) (2.93333333,2.40930836) (2.95813953,2.40392975) (2.98294574,2.39859395) (3.00775194,2.39330048) (3.03255814,2.38804888) (3.05736434,2.38283865) (3.08217054,2.37766935) (3.10697674,2.37254051) (3.13178295,2.36745167) (3.15658915,2.36240239) (3.18139535,2.35739222) (3.20620155,2.35242073) (3.23100775,2.34748747) (3.25581395,2.34259202) (3.28062016,2.33773397) (3.30542636,2.33291289) (3.33023256,2.32812836) (3.35503876,2.32337999) (3.37984496,2.31866737) (3.40465116,2.31399011) (3.42945736,2.30934781) (3.45426357,2.30474008) (3.47906977,2.30016655) (3.50387597,2.29562684) (3.52868217,2.29112057) (3.55348837,2.28664738) (3.57829457,2.28220691) (3.60310078,2.2777988) (3.62790698,2.2734227) (3.65271318,2.26907825) (3.67751938,2.26476512) (3.70232558,2.26048297) (3.72713178,2.25623145) (3.75193798,2.25201025) (3.77674419,2.24781903) (3.80155039,2.24365747) (3.82635659,2.23952526) (3.85116279,2.23542208) (3.87596899,2.23134763) (3.90077519,2.22730159) (3.9255814,2.22328367) (3.9503876,2.21929357) (3.9751938,2.215331) (4,2.21139566)};
\addplot[annBlue,dashed] coordinates {(0,3.141592653589793) (4,3.141592653589793)};\node[annBlue,anchor=south west,font=\scriptsize] at (axis cs:2.9,3.141592653589793) {$\pi$};\addplot[annBlue,only marks,mark=*] coordinates {(0.3387428586735777,3.1932979002651947)};
\end{axis}
\begin{axis}[at={(.51\textwidth,0)},anchor=south west,width=.46\textwidth,height=5.3cm,xlabel={$h=(1+\delta)^{-1}$},xmin=0,xmax=1,ymin=0,ymax=3.4,xtick={0,.5,1},xticklabels={$0$,$1/2$,$1$},tick label style={font=\scriptsize},label style={font=\small}]
\addplot[annBlue,thick,] coordinates {(0,0.698131701) (0.00416666667,0.886732434) (0.00833333333,0.969413805) (0.0125,1.03450217) (0.0166666667,1.09041486) (0.0208333333,1.1404191) (0.025,1.1861938) (0.0291666667,1.22873818) (0.0333333333,1.26870445) (0.0375,1.30654643) (0.0416666667,1.34259513) (0.0458333333,1.37710082) (0.05,1.41025827) (0.0541666667,1.44222261) (0.0583333333,1.47311981) (0.0625,1.50305385) (0.0666666667,1.53211175) (0.0708333333,1.56036723) (0.075,1.5878834) (0.0791666667,1.61471478) (0.0833333333,1.64090887) (0.0875,1.66650735) (0.0916666667,1.69154704) (0.0958333333,1.71606067) (0.1,1.74007747) (0.104166667,1.76362369) (0.108333333,1.78672301) (0.1125,1.80939688) (0.116666667,1.8316648) (0.120833333,1.85354455) (0.125,1.87505243) (0.129166667,1.89620339) (0.133333333,1.91701121) (0.1375,1.93748861) (0.141666667,1.95764735) (0.145833333,1.97749836) (0.15,1.9970518) (0.154166667,2.01631712) (0.158333333,2.03530318) (0.1625,2.05401823) (0.166666667,2.07247001) (0.170833333,2.09066579) (0.175,2.10861242) (0.179166667,2.12631631) (0.183333333,2.14378353) (0.1875,2.16101981) (0.191666667,2.17803056) (0.195833333,2.19482089) (0.2,2.21139566) (0.204166667,2.22775948) (0.208333333,2.24391671) (0.2125,2.2598715) (0.216666667,2.27562781) (0.220833333,2.2911894) (0.225,2.30655986) (0.229166667,2.32174261) (0.233333333,2.33674092) (0.2375,2.35155791) (0.241666667,2.36619657) (0.245833333,2.38065974) (0.25,2.39495017) (0.254166667,2.40907048) (0.258333333,2.42302318) (0.2625,2.43681068) (0.266666667,2.4504353) (0.270833333,2.46389925) (0.275,2.47720468) (0.279166667,2.49035363) (0.283333333,2.50334808) (0.2875,2.51618994) (0.291666667,2.52888103) (0.295833333,2.54142311) (0.3,2.55381788) (0.304166667,2.56606698) (0.308333333,2.57817199) (0.3125,2.59013443) (0.316666667,2.60195576) (0.320833333,2.61363742) (0.325,2.62518076) (0.329166667,2.6365871) (0.333333333,2.64785773) (0.3375,2.65899388) (0.341666667,2.66999674) (0.345833333,2.68086746) (0.35,2.69160716) (0.354166667,2.70221692) (0.358333333,2.71269776) (0.3625,2.72305071) (0.366666667,2.73327674) (0.370833333,2.74337679) (0.375,2.75335177) (0.379166667,2.76320256) (0.383333333,2.77293001) (0.3875,2.78253496) (0.391666667,2.79201819) (0.395833333,2.80138049) (0.4,2.8106226) (0.404166667,2.81974523) (0.408333333,2.82874909) (0.4125,2.83763486) (0.416666667,2.84640318) (0.420833333,2.85505468) (0.425,2.86358998) (0.429166667,2.87200967) (0.433333333,2.8803143) (0.4375,2.88850444) (0.441666667,2.8965806) (0.445833333,2.90454329) (0.45,2.91239302) (0.454166667,2.92013024) (0.458333333,2.92775542) (0.4625,2.935269) (0.466666667,2.94267138) (0.470833333,2.94996299) (0.475,2.9571442) (0.479166667,2.96421538) (0.483333333,2.9711769) (0.4875,2.97802908) (0.491666667,2.98477226) (0.495833333,2.99140673) (0.5,2.9979328) (0.504166667,3.00435074) (0.508333333,3.01066082) (0.5125,3.01686328) (0.516666667,3.02295835) (0.520833333,3.02894626) (0.525,3.03482721) (0.529166667,3.04060139) (0.533333333,3.04626898) (0.5375,3.05183014) (0.541666667,3.05728501) (0.545833333,3.06263373) (0.55,3.06787642) (0.554166667,3.07301318) (0.558333333,3.07804411) (0.5625,3.08296927) (0.566666667,3.08778875) (0.570833333,3.09250257) (0.575,3.09711078) (0.579166667,3.10161339) (0.583333333,3.10601042) (0.5875,3.11030184) (0.591666667,3.11448763) (0.595833333,3.11856776) (0.6,3.12254216) (0.604166667,3.12641077) (0.608333333,3.1301735) (0.6125,3.13383025) (0.616666667,3.1373809) (0.620833333,3.1408253) (0.625,3.14416332) (0.629166667,3.14739478) (0.633333333,3.1505195) (0.6375,3.15353726) (0.641666667,3.15644786) (0.645833333,3.15925104) (0.65,3.16194655) (0.654166667,3.16453411) (0.658333333,3.16701342) (0.6625,3.16938416) (0.666666667,3.17164598) (0.670833333,3.17379853) (0.675,3.17584142) (0.679166667,3.17777424) (0.683333333,3.17959656) (0.6875,3.18130792) (0.691666667,3.18290784) (0.695833333,3.18439581) (0.7,3.1857713) (0.704166667,3.18703373) (0.708333333,3.18818252) (0.7125,3.18921704) (0.716666667,3.19013664) (0.720833333,3.19094062) (0.725,3.19162825) (0.729166667,3.19219879) (0.733333333,3.19265143) (0.7375,3.19298534) (0.741666667,3.19319962) (0.745833333,3.19329338) (0.746969437,3.1932979) (0.75,3.19326563) (0.754166667,3.19311536) (0.758333333,3.1928415) (0.7625,3.19244295) (0.766666667,3.19191852) (0.770833333,3.19126699) (0.775,3.19048707) (0.779166667,3.18957738) (0.783333333,3.18853652) (0.7875,3.18736297) (0.791666667,3.18605516) (0.795833333,3.18461143) (0.8,3.18303003) (0.804166667,3.18130912) (0.808333333,3.17944675) (0.8125,3.17744088) (0.816666667,3.17528935) (0.820833333,3.17298986) (0.825,3.17054) (0.829166667,3.16793722) (0.833333333,3.16517879) (0.8375,3.16226185) (0.841666667,3.15918333) (0.845833333,3.15593999) (0.85,3.15252836) (0.854166667,3.14894476) (0.858333333,3.14518525) (0.8625,3.1412456) (0.866666667,3.1371213) (0.870833333,3.13280751) (0.875,3.12829901) (0.879166667,3.12359016) (0.883333333,3.11867491) (0.8875,3.11354666) (0.891666667,3.10819827) (0.895833333,3.10262197) (0.9,3.09680924) (0.904166667,3.09075077) (0.908333333,3.08443631) (0.9125,3.07785454) (0.916666667,3.07099288) (0.920833333,3.06383735) (0.925,3.05637223) (0.929166667,3.04857985) (0.933333333,3.04044012) (0.9375,3.03193011) (0.941666667,3.02302336) (0.945833333,3.01368911) (0.95,3.00389117) (0.954166667,2.99358645) (0.958333333,2.98272285) (0.9625,2.97123635) (0.966666667,2.95904654) (0.970833333,2.94604979) (0.975,2.93210822) (0.979166667,2.91703035) (0.983333333,2.90053465) (0.9875,2.88217145) (0.991666667,2.86112205) (0.995833333,2.83547699) (1,2.7925268)};

\end{axis}
\end{tikzpicture}
\caption{\textcolor{revision}{The threshold $S_\delta$ in \eqref{eq:Dq-S}. Left: its unique maximum $S_*\simeq3.19330$ occurs at $\delta_*\simeq0.33874$; the dashed line is $\pi$. Right: the variable $h=(1+\delta)^{-1}$ displays the full width range, with $h=0$ and $h=1$ denoting limiting values. The curves illustrate the analytic argument above; they are not used as a proof.}}
\label{fig:Dq-threshold}
\end{figure}}

\subsection{\revblue{An explicit positive gap at every index}}
{A positive tail estimate alone leaves a gap near the first radial
frequency. We close that gap by dividing the frequency axis at
$x=z$: below $z$ we use the first row, and above $z$ we use
\eqref{eq:Dq-whole}. The quantities below specify that division
using only the width $\delta$. In particular, the optimization is
an elementary matching of two explicit lower bounds, not a search
over eigenvalues.}
{For $0<\delta\le5/2$ define}
{\begin{align}
 q_\delta&=A\pi^2-{\frac{R^2}{A}}-{\frac{A\delta^2}{4R^2}},
 &X_\delta&=\sqrt{4\pi^2-\delta^2/4},\label{eq:Dq-qX}\\
 z_\delta^*&={\frac{1}{4}}
 \paren{\sqrt{S_\delta}+\sqrt{S_\delta+8q_\delta/R}}^2,
 &z_\delta&=\min\set{z_\delta^*,X_\delta}.
\notag\end{align}}
{Set}
{\begin{equation}\label{eq:Dq-c}
 c_\delta=
 \begin{cases}
  \displaystyle\max\set{{\frac{R}{2}}\paren{1-\sqrt{S_\delta/\pi}},
  \min\set{{\frac{q_\delta}{z_\delta}},
  {\frac{R}{2}}\paren{1-\sqrt{S_\delta/z_\delta}}}},
  &0<\delta\le5/2,\\[5pt]
  \displaystyle{\frac{R}{2}}\paren{1-\sqrt{S_\delta/\pi}},
  &\delta>5/2.
 \end{cases}
\end{equation}}

{For physical radii $r<R_0$, use $R=R_0/r$ in these constants.
Put}
{\[
 W_k={\frac{8\pi k}{\beta\paren{R_0^2-r^2}}},\qquad
 a_{r,R_0}={\frac{4c_\delta}{rD}}.
\]}
{\color{editorial}We now state the explicit all-index bounds of
Theorem~\ref{thm:dirichlet}, proved immediately after
\eqref{eq:Dq-all-count}. For every $k\ge1$,}

{\begin{equation}\label{eq:Dq-all-eigen}
 \lambda_k\paren{S_{r,R_0,\beta}}
 \ge\paren{\sqrt{W_k+{\frac{a_{r,R_0}^2}{4}}}
                   +{\frac{a_{r,R_0}}{2}}}^2
 =W_k+a_{r,R_0}\sqrt{W_k+{\frac{a_{r,R_0}^2}{4}}}
                   +{\frac{a_{r,R_0}^2}{2}}.
\end{equation}}
{\color{editorial}For every $E>0$, the counting bound is}
{\begin{equation}\label{eq:Dq-all-count}
 N^{\mathrm D}_{S_{r,R_0,\beta}}\paren{E}
 \le{\frac{\abs{S_{r,R_0,\beta}}}{4\pi}}
       \pos{E-a_{r,R_0}\sqrt E}.
\end{equation}}

\begin{proof}[\revblue{Proof of Theorem~\ref{thm:dirichlet}}]
{First, $q_\delta>0$ for $0<\delta\le5/2$.  Dividing by $A$ gives}
{\[
 {\frac{q_\delta}{A}}=\pi^2-
 \paren{{\frac{16\paren{1+\delta}^2}{\paren{2+\delta}^2}}
       +{\frac{\delta^2}{4\paren{1+\delta}^2}}}.
\]}
{Their derivatives are, respectively,
\[
 \frac{32\paren{1+\delta}}{\paren{2+\delta}^3}>0,
 \qquad
 \frac{\delta}{2\paren{1+\delta}^3}>0.
\]
Both terms in parentheses therefore increase with $\delta$, and their sum at
$5/2$ is
$784/81+25/196<\paren{157/50}^2<\pi^2$.
The simpler Cauchy--Schwarz estimate gives}
{\[
 S_\delta\le{\frac{2\pi}{9}}
 \paren{2\sqrt h+\sqrt{1-h}}^2\le{\frac{10\pi}{9}}<{\frac{7}{2}}
 <X_\delta.
\]}
{Moreover, $z_\delta^*>S_\delta$ and $X_\delta>S_\delta$,
so $z_\delta>S_\delta$. Both entries in the inner minimum in
the first branch of \eqref{eq:Dq-c} are positive. The other entry
in its outer maximum need not be positive; this causes no problem.
If $\delta\ge5/2$, then $h\le2/7$
and $b\paren{h}\le1$.  Since $\sqrt{2/7}<11/20$,}
{\[
 S_\delta<{\frac{2\pi}{9}}\paren{1+11/10}^2
 ={\frac{49\pi}{50}}<\pi,
\]}
{which proves positivity of the second branch.}

{At every sector Dirichlet eigenvalue, the order is $jp\ge1/2$.
Equation~\eqref{eq:Dq-signed} consequently gives
$\lambda_k\ge\pi^2/\delta^2$.  Thus at $t=\sqrt{\lambda_k}$
one has $x\ge\pi$.}

{Suppose first that $\delta\le5/2$.  The $n=2$ row can start only
when $x^2+\delta^2/4>4\pi^2$. Consequently, when
$\pi\le x\le X_\delta$ only the first radial row contributes to
\eqref{eq:Dq-rows}, apart from a zero-length second row at the last
endpoint. To make its minimum explicit, put
\[
 Y=\sqrt{R^2\paren{x^2-\pi^2}+\delta^2/4}.
\]
Then the precise square completion is
\[
\begin{aligned}
 Ax^2-2Y
 &=A\pi^2-\frac{A\delta^2}{4R^2}
          +\frac A{R^2}Y^2-2Y\\
 &=q_\delta+\frac A{R^2}
             \paren{Y-\frac{R^2}A}^2
 \ge q_\delta.
\end{aligned}
\]
It gives the following frequency-independent first-row minimum:}
{\[
 Ax^2-2\sqrt{R^2\paren{x^2-\pi^2}+\delta^2/4}
 \ge q_\delta.
\]}
{For $\pi\le x\le z_\delta$ it follows that
$B_{\mathrm D}\paren{t}/t\ge q_\delta/x\ge q_\delta/z_\delta$.
If $z_\delta<\pi$, this first interval is empty and the tail
bound already applies to the entire relevant range.
For $x\ge z_\delta$, \eqref{eq:Dq-whole} gives}
{\[
 {\frac{B_{\mathrm D}\paren{t}}{t}}
 \ge{\frac{R}{2}}\paren{1-\sqrt{S_\delta/x}}
 \ge{\frac{R}{2}}\paren{1-\sqrt{S_\delta/z_\delta}}.
\]}
{This proves the second entry in the maximum defining $c_\delta$.
The first entry follows directly from \eqref{eq:Dq-whole}, since
$x\ge\pi$.  Hence $B_{\mathrm D}\paren{t}\ge c_\delta t$ for
every $x\ge\pi$.  The switching point in
\eqref{eq:Dq-qX} maximizes the minimum of these two bounds
over $S_\delta<z\le X_\delta$: the first decreases with $z$,
the second increases, and they agree at $z=z_\delta^*$.
To see the maximization explicitly, write
\[
 g_1\paren z=\frac{q_\delta}z,
 \qquad g_2\paren z=\frac R2\paren{1-\sqrt{S_\delta/z}}.
\]
Their derivatives have opposite signs:
\[
 g_1'\paren z=-\frac{q_\delta}{z^2}<0,
 \qquad
 g_2'\paren z=\frac{R\sqrt{S_\delta}}{4z^{3/2}}>0.
\]
At their unique crossing, the substitution $y=\sqrt z$ gives
$y^2-\sqrt{S_\delta}\,y-2q_\delta/R=0$. Its positive
root gives $z=z_\delta^*$ in \eqref{eq:Dq-qX}.
The minimum of $g_1,g_2$ increases up to that crossing and
decreases afterwards. If the crossing lies beyond $X_\delta$,
the best admissible division is $z=X_\delta$; otherwise it is
the crossing itself. This proves the stated clipped choice
$z_\delta=\min\{z_\delta^*,X_\delta\}$.
This assertion concerns these two bounds with the completed
square first-row constant $q_\delta$, rather than the exact
spectral deficit. Figure~\ref{fig:Dq-switching} illustrates both
possibilities.}

{For $\delta>5/2$, the same conclusion follows directly from
\eqref{eq:Dq-whole}, since $x\ge\pi>S_\delta$.}

{The Dirichlet part of Proposition~\ref{prop:angular-transfer},
illustrated in Figure~\ref{fig:Dq-angular-transfer}(b), gives}
{\[
 \begin{aligned}
 pN^{\mathrm D}_S\paren{t^2}
 &=p\sum_{j\ge1}n_{jp}^{\mathrm D}\paren{t^2}\\
 &\le\int_0^{+\infty}n_v^{\mathrm D}\paren{t^2}\dd v
   ={\frac{1}{2}}M_{\mathrm D}\paren{t}
   =\frac D8t^2-\frac12B_{\mathrm D}\paren t.
 \end{aligned}
\]}
{This is the Dirichlet counterpart of the Neumann comparison in
Figure~\ref{fig:sampling}(a): right-endpoint sampling gives a lower
rectangle sum, whereas left-endpoint sampling gives an upper
rectangle sum. The positive deficit $B_{\mathrm D}$ estimated
above, rather than monotonicity alone, supplies the quantitative
improvement over the area term.}
{Taking $t^2=\lambda_k$ and using $N^{\mathrm D}_S\paren{\lambda_k}\ge k$
therefore yields, for normalized radii,}
{\[
 \lambda_k-W_k\ge{\frac{4}{D}}B_{\mathrm D}\paren{\sqrt{\lambda_k}}
 \ge{\frac{4c_\delta}{D}}\sqrt{\lambda_k}.
\]}
{For $x\ge\pi$ the same sampling argument yields
$N^{\mathrm D}_S\paren{t^2}\le D\paren{t^2-4c_\delta t/D}/\paren{8p}$.
For $x<\pi$ the sector count is zero, by
\eqref{eq:Dq-signed} and $jp\ge1/2$.  Taking the positive part
therefore gives \eqref{eq:Dq-all-count} at every energy after scaling.
Solving the quadratic inequality for $\sqrt{\lambda_k}$ and
scaling lengths by $r$ proves \eqref{eq:Dq-all-eigen}.}
\end{proof}
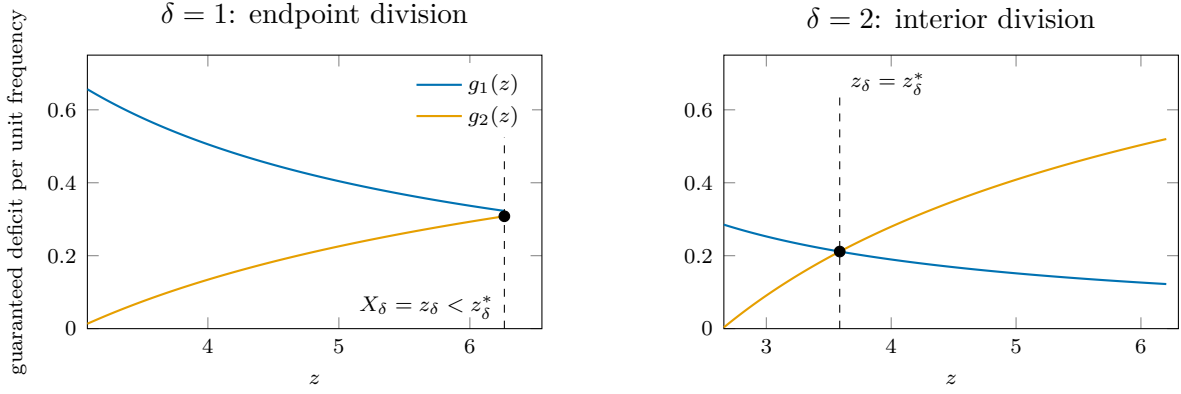
\begin{figure}[htbp]
\centering
\begin{tikzpicture}
\begin{axis}[width=.46\textwidth,height=5.2cm,
 xmin=3.08,xmax=6.55,ymin=0,ymax=.75,
 xlabel={$z$},ylabel={guaranteed deficit per unit frequency},
 title={$\delta=1$: endpoint division},
 tick label style={font=\scriptsize},
 label style={font=\scriptsize},title style={font=\small},
 legend style={draw=none,font=\scriptsize,at={(.98,.97)},anchor=north east}]
\addplot[annBlue,thick,domain=3.08:6.26325934,samples=90]
 {2.02199496748/x};
\addlegendentry{$g_1(z)$}
\addplot[annOrange,thick,domain=3.08:6.26325934,samples=90]
 {1-sqrt(2.99793280175/x)};
\addlegendentry{$g_2(z)$}
\addplot[black,dashed] coordinates {(6.26325934,0) (6.26325934,.72)};
\node[font=\scriptsize,anchor=north east] at
 (axis cs:6.26325934,.12) {$X_\delta=z_\delta<z_\delta^*$};
\addplot[black,only marks,mark=*] coordinates
 {(6.26325934,.308151905)};
\end{axis}
\begin{axis}[at={(.51\textwidth,0)},anchor=south west,
 width=.46\textwidth,height=5.2cm,
 xmin=2.66,xmax=6.3,ymin=0,ymax=.75,
 xlabel={$z$},title={$\delta=2$: interior division},
 tick label style={font=\scriptsize},
 label style={font=\scriptsize},title style={font=\small}]
\addplot[annBlue,thick,domain=2.66:6.20309742,samples=90]
 {.758493289978/x};
\addplot[annOrange,thick,domain=2.66:6.20309742,samples=90]
 {1.5*(1-sqrt(2.64785773370/x))};
\addplot[black,dashed] coordinates {(3.58791674,0) (3.58791674,.65)};
\node[font=\scriptsize,anchor=south west] at
 (axis cs:3.58791674,.61) {$z_\delta=z_\delta^*$};
\addplot[black,only marks,mark=*] coordinates
 {(3.58791674,.211402)};
\end{axis}
\end{tikzpicture}
\caption{Matching the first-row and tail bounds in Subsection~3.2.
The bound valid on both frequency ranges is the lower of the two
curves. For $\delta=1$ their crossing lies beyond the admissible
row range, so its right endpoint is used; for $\delta=2$ the
crossing is admissible. These plots illustrate the derivative and
quadratic calculations in the proof.}
\label{fig:Dq-switching}
\end{figure}

{The whole-inner envelope also gives an additional explicit tail correction}
{\begin{equation}\label{eq:Dq-tail}
 \lambda_k-W_k\ge
 {\frac{2R}{rD}}\sqrt{W_k}
 \paren{1-\sqrt{\frac{S_\delta}{\delta r\sqrt{W_k}}}},
 \qquad \delta r\sqrt{W_k}>S_\delta.
\end{equation}}
{Indeed, for physical radii right-endpoint sampling gives
\[
 \lambda_k-W_k\ge\frac4{r^2D}
 B_{\mathrm D}\paren{r\sqrt{\lambda_k}}.
\]
The function $H_\delta$ is increasing above its positive threshold and
$r\sqrt{\lambda_k}\ge r\sqrt{W_k}$.  One may take the larger
of \eqref{eq:Dq-all-eigen} and \eqref{eq:Dq-tail}.  The stated
cutoff is a formula in the radii and the aperture; it contains no
spectral data.}

\subsection{\revblue{The full circular-arc coefficient}}
{The main theorem is now proved. Its uniform coefficient is
positive, but the whole-inner phase estimate has only $Rt/2$ as
its linear deficit, rather than the circular-arc target
$(R+1)t/2$. The loss comes from using the phase shift $1/4$ on
the entire order range. On a core away from the turning order
$v=t$, a sharper phase estimate replaces that fixed shift by an
integrable error of order $t^{-1}$. The floor then retains its
full mean $-1/2$ on the core. Extending the core towards $t$
recovers the missing $t/2$, at the price of explicit error terms.

On a core $0\le v\le ut$, $0<u<1$, the sharper
ordinary Bessel phase enclosure
\cite[Theorem~1.4]{Phase} is}
{\[
 \max\set{A_v\paren{z}-\frac\pi4-\mathcal E_v\paren{z},-\frac\pi2}
 <\theta_v\paren{z}<A_v\paren{z}-\frac\pi4.
\]}
{\[
 \qquad
 \mathcal E_v\paren{z}={\frac{3z^2+2v^2}{24\paren{z^2-v^2}^{3/2}}},
 \qquad z>v.
\]}
{Consequently $\gamma_t\paren{v}<f_t\paren{v}+\eta_t\paren{v}$ on
the core, where}
{\[
 \eta_t\paren{v}={\frac{3+2z^2}{24\pi t\paren{1-z^2}^{3/2}}},
 \qquad z=v/t.
\]}
{Here the upper bound at $Rt$ and the lower bound at $t$ are
subtracted. More explicitly,
\[
 \theta_v\paren{Rt}-\theta_v\paren t
 <A_v\paren{Rt}-A_v\paren t+\mathcal E_v\paren t.
\]
The $-\pi/4$ terms cancel. This explains both the positive sign
of $\eta_t$ and why no fixed $1/4$ is needed on the core.}
{Its derivatives satisfy}
{\[
 \eta_t'\paren{v}={\frac{z\paren{13+2z^2}}{24\pi t^2\paren{1-z^2}^{5/2}}},\qquad
 \eta_t''\paren{v}={\frac{13+58z^2+4z^4}{24\pi t^3\paren{1-z^2}^{7/2}}}.
\]}
{In particular, the sign of the added phase error makes the curvature
condition necessary.  Put}
{\[
 b_R\paren{z}=1-\sqrt{\frac{1-z^2}{R^2-z^2}},\qquad
 \chi_R\paren{t,u}={\frac{13+58u^2+4u^4}{24t^2\paren{1-u^2}^3b_R\paren{u}}}.
\]}
{For $\chi_R\paren{t,u}<1$, the phase
$q_t=f_t+\eta_t$ obeys}
{\[
q_t'\paren{0}=0,\qquad
 -q_t''\paren{v}\ge\kappa_{t,u}
 :=\paren{1-\chi_R\paren{t,u}}{\frac{\delta}{\pi Rt}},
 \qquad0\le v\le ut.
\]}
{To verify this, the exact curvature of the action on $v=zt<t$
can be written as
\[
 -f_t''\paren{zt}
 =\frac{b_R\paren z}{\pi t\sqrt{1-z^2}}.
\]
Thus the exact ratio $\eta_t''/\paren{-f_t''}$ is
$\chi_R\paren{t,z}$. It increases with $z$.
Indeed, writing
$w=\sqrt{\paren{1-z^2}/\paren{R^2-z^2}}$ gives
$b_R'/b_R=z w\paren{1+w}/\paren{1-z^2}\le2z/\paren{1-z^2}$,
whereas the logarithmic derivative of the remaining denominator
factor contributes $6z/\paren{1-z^2}$ and the numerator increases.
The logarithmic derivative of the ratio, for $z>0$, is
\[
 \frac{116z+16z^3}{13+58z^2+4z^4}
 +\frac{6z}{1-z^2}-\frac{b_R'\paren z}{b_R\paren z}
 \ge \frac{116z+16z^3}{13+58z^2+4z^4}
       +\frac{4z}{1-z^2}>0.
\]
Consequently it suffices to test the endpoint $z=u$.
Finally
$-f_t''\ge\delta/\paren{\pi Rt}$.}

{The integral of the phase error is explicit:}
{\[
 J_{\mathrm D}\paren{u}:=\int_0^{ut}\eta_t\paren{v}\dd v
 ={\frac{5u/\sqrt{1-u^2}-2\arcsin u}{24\pi}}.
\]}
{This follows from the elementary decomposition
\[
 \frac{3+2z^2}{\paren{1-z^2}^{3/2}}
 =\frac5{\paren{1-z^2}^{3/2}}
   -\frac2{\sqrt{1-z^2}},
\]
using $(z/\sqrt{1-z^2})'=\paren{1-z^2}^{-3/2}$.
In particular,
\[
 J_{\mathrm D}'\paren u
 =\frac{3+2u^2}{24\pi\paren{1-u^2}^{3/2}}>0,
 \qquad
 J_{\mathrm D}\paren u\longrightarrow+\infty
 \quad\text{as }u\uparrow1.
\]
The core therefore cannot simply be set equal to the whole inner
interval at a fixed frequency.}
{On the core, $\floor{q_t}=q_t-1/2-\psi\paren{q_t}$.  Thus its
doubled contribution to the deficit is at least}
{\[
 ut-2J_{\mathrm D}\paren{u}
 -{\frac{2C_-\paren{a_{\mathrm D}\paren{t}}}{\sqrt{\kappa_{t,u}}}},
 \qquad
 a_{\mathrm D}\paren{t}
 =\frac{\delta t}\pi+\frac1{8\pi t}
       -\floor{\frac{\delta t}\pi+\frac1{8\pi t}}.
\]}
{The linear gain $ut$ in this display is twice the floor mean
$1/2$ integrated over an interval of length $ut$. The subtracted
terms are, respectively, the doubled phase error and the
one-sided rounding loss. The notation $a_{\mathrm D}(t)$ in this
subsection denotes a fractional phase; the coefficient
$a_{r,R_0}$ in \eqref{eq:Dq-all-eigen} is a different quantity.}

{The turning interval can retain its floor mean because its initial
slope is nonzero.  More generally, if $q$ is decreasing and concave
on $\bracket{a,b}$, with $-q'\paren{a}\ge\sigma>0$, then}
{\begin{equation}\label{eq:Dq-nonstationary}
 \int_a^b\psi\paren{q\paren{v}}\dd v
 \ge-{\frac{\set{q\paren{a}}\paren{1-\set{q\paren{a}}}}{2\sigma}}
 \ge-{\frac{1}{8\sigma}}.
\end{equation}}
{Here $\set{y}=y-\floor{y}$ denotes the fractional part.  For a proof,
the periodic primitive
$F\paren{y}=\set{y}^2/2-\set{y}/2$ has range $\bracket{-1/8,0}$.
After the increasing change of variable $y=q\paren{a}-q\paren{v}$,
the inverse derivative is nonincreasing and bounded by $1/\sigma$.
The primitive of the new integrand is
$F\paren{q\paren{a}}-F\paren{q\paren{a}-y}$, whose lower bound is
$F\paren{q\paren{a}}$.  Integration by parts against the
nonincreasing weight proves \eqref{eq:Dq-nonstationary}.}

{Apply this estimate to $f_t+1/4$ on $\bracket{ut,t}$.  Put}
{\[
 \sigma_R\paren{u}={\frac{\arccos\paren{u/R}-\arccos u}{\pi}}>0,
 \qquad
 b_{t,u}=f_t\paren{ut}+\frac14
           -\floor{f_t\paren{ut}+\frac14}.
\]}
{Indeed,
\[
 -\frac{\mathrm d}{\mathrm dv}
       \paren{f_t\paren v+\tfrac14}\Big|_{v=ut}
 =\sigma_R\paren u,
\]
and concavity makes this negative slope increase in magnitude
as $v$ runs from $ut$ to $t$. Its doubled deficit is at least
$\paren{1-u}t/2-b_{t,u}\paren{1-b_{t,u}}/\sigma_R\paren{u}$.
The outer reflected estimate is unchanged.  Therefore, whenever
$\chi_R\paren{t,u}<1$,}
{\begin{equation}\label{eq:Dq-core-envelope}
 \begin{split}
 B_{\mathrm D}\paren{t}\ge\mathcal C_\delta\paren{t,u}
 :={}&{\frac{R+u}{2}}t-2J_{\mathrm D}\paren{u}\\
 &-{\frac{2C_-\paren{a_{\mathrm D}\paren{t}}}{\sqrt{\kappa_{t,u}}}}
 -{\frac{b_{t,u}\paren{1-b_{t,u}}}{\sigma_R\paren{u}}}
 -{\frac{\sqrt2}{6}}\sqrt{\pi t\sqrt D}.
 \end{split}
\end{equation}}
{To check the linear term, add the three interval contributions:
\[
 \underbrace{ut}_{\text{core}}
 +\underbrace{\frac{1-u}{2}t}_{\text{turning interval}}
 +\underbrace{\frac{R-1}{2}t}_{\text{outer interval}}
 =\frac{R+u}{2}t.
\]
This bound retains both the initial phase of the core and that of
the turning interval. At a specified frequency, $u$ may be
optimized subject to $0<u<1$ and $\chi_R(t,u)<1$. The last
condition has an increasing left side, so its admissible range is
an interval starting at zero whenever it is nonempty. Neither
the phase-dependent objective nor its maximizer is assumed to
be monotone in $u$.}

{For a closed expression with the full circular-arc coefficient,
define}
{\begin{align*}
 \alpha_\delta&=\paren{\frac{25R}{4\delta}}^{1/3},
 &T_\delta&=\paren{\frac{4\alpha_\delta}3}^{3/2},\\
 L_\delta&=\frac{R+1}2,
 &C_\delta&=\frac23\sqrt{\frac{\pi R}\delta}
             +\frac{\sqrt2}6\sqrt{\pi\sqrt D},\\
 J_\delta&=\frac{\alpha_\delta}3
             +\frac5{12\pi\sqrt{\alpha_\delta}},
 &K_\delta&=\frac{\pi R}{2\delta}.
\end{align*}}
{The choice below makes the error terms explicit. The curvature
ratio has the upper bound
\[
 \chi_R\paren{t,u}
 \le\frac{25R}{8\delta\,t^2\paren{1-u^2}^3}.
\]
Thus choosing $1-u^2$ proportional to $t^{-2/3}$ keeps this
bound fixed as $t$ grows and lets $u\uparrow1$.
For $t\ge T_\delta$ use
$u=\sqrt{1-\alpha_\delta t^{-2/3}}\ge1/2$.
Since $b_R\paren{z}\ge\delta/R$ and
$13+58z^2+4z^4\le75$, one has
$\chi_R\paren{t,u}\le1/2$, since
\[
 t^2\paren{1-u^2}^3=\alpha_\delta^3
 =\frac{25R}{4\delta}.
\]
Also $t\ge T_\delta$ gives $1-u^2\le3/4$, hence $u\ge1/2$.
The rounding maxima and
\eqref{eq:Dq-core-envelope} consequently give}
{\begin{equation}\label{eq:Dq-arc-envelope}
 B_{\mathrm D}\paren{t}\ge
 F_\delta^{\mathrm D}\paren{t}:=L_\delta t-C_\delta\sqrt t
                         -J_\delta t^{1/3}-K_\delta.
\end{equation}}
{Indeed,}
{\[
 {\frac{1-u}{2}}t
 ={\frac{\alpha_\delta t^{1/3}}{2\paren{1+u}}}
 \le{\frac{\alpha_\delta}{3}}t^{1/3},\qquad
 2J_{\mathrm D}\paren{u}\le
 {\frac{5}{12\pi\sqrt{\alpha_\delta}}}t^{1/3}.
\]}
{For the stationary core rounding, use
$\max C_-=\sqrt2/6$ and
$\kappa_{t,u}\ge\delta/(2\pi Rt)$ to obtain
\[
 \frac{2C_-\paren{a_{\mathrm D}\paren t}}{\sqrt{\kappa_{t,u}}}
 \le\frac23\sqrt{\frac{\pi R}\delta}\sqrt t.
\]
Together with the unchanged outer rounding term, this is
$C_\delta\sqrt t$.
Furthermore,
$\sigma_R\paren{u}\ge\delta u/\paren{\pi R}\ge\delta/\paren{2\pi R}$,
so the middle rounding loss is at most $K_\delta$.}
{More precisely, the slope bound follows from
\[
 \sigma_R\paren u
 =\frac1\pi\int_{u/R}^u\frac{\dd s}{\sqrt{1-s^2}}
 \ge\frac{\delta u}{\pi R},
\]
and $b_{t,u}(1-b_{t,u})\le1/4$. The resulting three powers
$t^{1/2},t^{1/3},t^0$ in \eqref{eq:Dq-arc-envelope} all have
lower order than its linear term.}

{An explicit positive threshold is}
{\[
\mathcal T_\delta=\max\set{T_\delta,
  \paren{\frac{C_\delta+J_\delta+
  \sqrt{\paren{C_\delta+J_\delta}^2+4L_\delta K_\delta}}
                 {2L_\delta}}^2}.
\]}
{For $t>\mathcal T_\delta$, $F_\delta^{\mathrm D}\paren{t}>0$ and
$\paren{F_\delta^{\mathrm D}}'\paren{t}>0$.  Positivity follows from $t\ge T_\delta>1$
and $t^{1/3}\le\sqrt t$: with $y=\sqrt t$, it is enough to
have $L_\delta y^2-(C_\delta+J_\delta)y-K_\delta>0$.
The second entry in $\mathcal T_\delta$ is the square of its
positive root. If $F_\delta^{\mathrm D}\paren{t}>0$, then
$L_\delta>C_\delta/\sqrt t+J_\delta/t^{2/3}+K_\delta/t$,
and hence
\[
 \paren{F_\delta^{\mathrm D}}'\paren t
 =L_\delta-\frac{C_\delta}{2\sqrt t}
            -\frac{J_\delta}{3t^{2/3}}>0.
\]
The threshold is sufficient and explicit; it need not be the
smallest positive frequency allowed by the sharper core bound.
For interpreting the coefficient, it is also useful that
\[
 \frac{F_\delta^{\mathrm D}\paren t}{t}
 =L_\delta-C_\delta t^{-1/2}-J_\delta t^{-2/3}-K_\delta t^{-1}
 \nearrow L_\delta.
\]
Every error term on the right decreases to zero, so this last
monotonicity holds for all $t>0$ as an algebraic statement.
Figure~\ref{fig:Dq-core-errors} displays the curvature constraint
and this approach to the circular-arc coefficient.}

\begin{figure}[htbp]
\centering
\begin{tikzpicture}
\begin{axis}[width=.46\textwidth,height=5.2cm,
 xmin=0,xmax=.96,ymin=0,ymax=2,
 xlabel={$u$},ylabel={$\chi_2(20,u)$},
 title={Curvature constraint: $R=2$, $t=20$},
 tick label style={font=\scriptsize},
 label style={font=\scriptsize},title style={font=\small}]
\addplot[annBlue,thick,domain=0:.96,samples=160,restrict y to domain=0:2]
 {(13+58*x^2+4*x^4)/(9600*(1-x^2)^3*
 (1-sqrt((1-x^2)/(4-x^2))))};
\addplot[black,dashed] coordinates {(0,1) (.96,1)};
\node[font=\scriptsize,anchor=south west] at
 (axis cs:.05,1) {admissible below $1$};
\end{axis}
\begin{axis}[at={(.51\textwidth,0)},anchor=south west,
 width=.46\textwidth,height=5.2cm,
 xmin=5.44331054,xmax=100,ymin=-.1,ymax=1.62,
 xlabel={$t$},ylabel={$F_1^{\mathrm D}(t)/t$},
 title={Recovery of the coefficient: $\delta=1$},
 tick label style={font=\scriptsize},
 label style={font=\scriptsize},title style={font=\small}]
\addplot[annOrange,thick,domain=5.44331054:100,samples=120]
 {1.5-2.22090357182/sqrt(x)-.860658517265/x^(2/3)-pi/x};
\addplot[black,dashed] coordinates {(5.44331054,1.5) (100,1.5)};
\node[font=\scriptsize,anchor=north east] at
 (axis cs:98,1.5) {$L_1=3/2$};
\addplot[gray,dotted] coordinates {(7.85049210,-.1) (7.85049210,1.62)};
\node[font=\scriptsize,anchor=south west] at
 (axis cs:9,.05) {$\mathcal T_1\simeq7.85$};
\end{axis}
\end{tikzpicture}
\caption{The core constraint becomes more restrictive as $u$
approaches the turning order. After the explicit frequency-dependent
choice of $u$, the normalized deficit envelope increases towards
the circular-arc coefficient $L_\delta$. The plots illustrate the
analytic monotonicity calculations; the dotted line marks the
sufficient positivity threshold, not an exact zero.}
\label{fig:Dq-core-errors}
\end{figure}
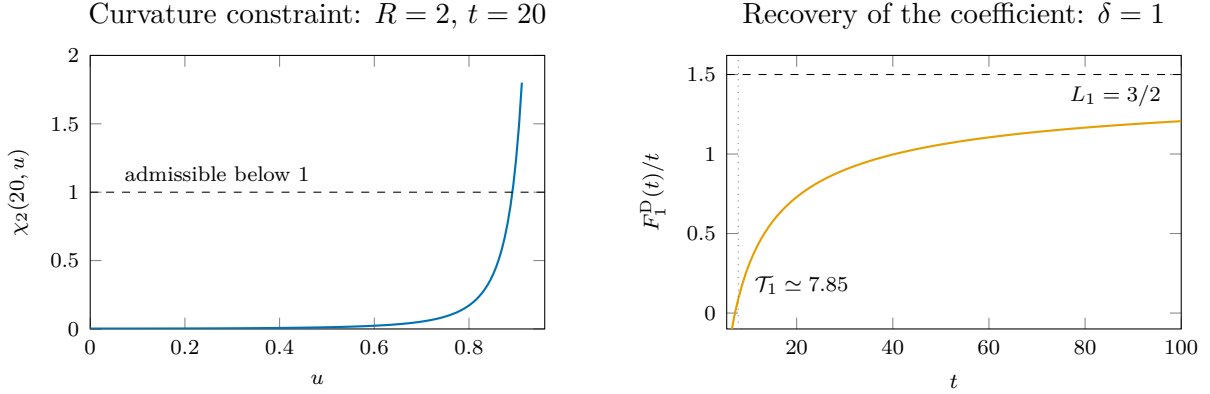

{\begin{corollary}[An explicit gap with the full circular-arc coefficient]
\label{cor:Dq-arc-eigen}
Under the hypotheses of Theorem~\ref{thm:dirichlet}, whenever
$r\sqrt{W_k}>\mathcal T_\delta$,
\[
\lambda_k-W_k\ge{\frac{4}{r^2D}}
    F_\delta^{\mathrm D}\paren{r\sqrt{W_k}}>0.
\]
In particular the coefficient of $\sqrt{W_k}$ is
$2\paren{R_0+r}/\paren{R_0^2-r^2}=2/(R_0-r)$;
the displayed errors are
explicit multiples of $W_k^{1/4}$, $W_k^{1/6}$, and one.
\end{corollary}}
\begin{proof}
{At $E=\lambda_k$, right-endpoint sampling gives
$\lambda_k-W_k\ge4B_{\mathrm D}\paren{r\sqrt{\lambda_k}}/\paren{r^2D}$.
Use \eqref{eq:Dq-arc-envelope}, the positive threshold, and
$r\sqrt{\lambda_k}\ge r\sqrt{W_k}$.}
\end{proof}

{The available explicit bounds may be combined without a loss.
For the whole inner interval, retention of its initial phase gives}
{\[
 \mathcal H_\delta\paren{t}
 ={\frac{Rt}{2}}-{2C_-\paren{a_{\mathrm{wh}}\paren{t}}}
                    \sqrt{\frac{\pi Rt}\delta}
             -{\frac{\sqrt2}{6}}\sqrt{\pi t\sqrt D},
 \qquad
 a_{\mathrm{wh}}\paren{t}=\frac{\delta t}\pi+\frac14
                  -\floor{\frac{\delta t}\pi+\frac14}.
\]}
{It dominates $H_\delta\paren{\delta t}$.  Thus, for $\delta t\ge\pi$,
one may use}
{\begin{equation}\label{eq:Dq-combined}
 \mathcal B_\delta^{\mathrm D}\paren{t}=
 \max\set{c_\delta t,Q_\delta\paren{\delta t},
      \mathcal H_\delta\paren{t},
      \sup_{\substack{0<u<1\\\chi_R\paren{t,u}<1}}
                       \mathcal C_\delta\paren{t,u}}
 \le B_{\mathrm D}\paren{t}.
\end{equation}}
{The supremum of an empty admissible set is omitted.  For
$t\ge T_\delta$ the explicit value $F_\delta^{\mathrm D}\paren{t}$ may
also be included.  This is the strongest lower bound furnished by
these displayed row, whole-inner, and core estimates.  The next subsection inverts their combined counting bound.}

\subsection{\revblue{Discrete radial-row bounds and eigenvalue inversion}}
{The closed gap in Theorem~\ref{thm:dirichlet} discards both the
locations of individual angular samples and the frequency
dependence of the sharper phase bounds. We retain those data
here and invert the best displayed counting bound. In each
row the continuous transfer replaces the integer number of
permitted angular indices by its real upper bound, losing a
fractional part between zero and one. Keeping the floor removes
that loss in the row comparison. Separately, using the actual
phase bounds rather than $c_\delta t$ removes the uniform
lower-envelope loss. Taking the smaller of the resulting upper
counts preserves both advantages. The leading area coefficient
is not enlarged. The resulting lower eigenvalue bound contains
the main theorem's bound and can improve it at low as well as
high indices.}
{For normalized inner radius one, retain
$R=1+\delta$, $D=R^2-1$, and $p=\pi/\beta\ge1/2$.
The signed-potential row lengths in \eqref{eq:Dq-rows} also give a
discrete upper bound.  Define}
{\begin{equation}\label{eq:Dq-discrete-rows}
 \mathcal R_{\delta,p}\paren{t}
 =\sum_{n\ge1}
       \floor{\frac{P_n\paren{\delta,\delta t}}{\delta p}}.
\end{equation}}
{The sum is finite: $P_n\paren{\delta,\delta t}=0$ whenever
$\pi^2n^2>\delta^2t^2+\delta^2/4$.
For a fixed row $n$, the comparison permits the angular indices
satisfying $jp\le P_n/\delta$, so their number is
$\lfloor P_n/(\delta p)\rfloor$. Equality is retained because
the Dirichlet count uses an inclusive energy cutoff.
Every actual mode is permitted by the comparison, although the
converse need not hold. Thus separation of variables
and \eqref{eq:Dq-signed} give}
{\[
N_S^{\mathrm D}\paren{t^2}\le\mathcal R_{\delta,p}\paren{t}.
\]}
{For $\delta t<\pi$, every summand in
\eqref{eq:Dq-discrete-rows} is zero.  Indeed, an active row then has
$P_n\paren{\delta,\delta t}<\delta/2\le\delta p$.}

{In \eqref{eq:Dq-combined}, include the additional entry
$F_\delta^{\mathrm D}\paren{t}$ whenever $t\ge T_\delta$.
Define the fully specified pointwise upper bound}
{\begin{equation}\label{eq:Dq-U}
 U_{\delta,p}\paren{t}=
 \begin{cases}
  0,&\delta t<\pi,\\[2pt]
  \displaystyle\min\set{\mathcal R_{\delta,p}\paren{t},
       {\frac{Dt^2/4-\mathcal B_\delta^{\mathrm D}\paren{t}}{2p}}},
       &\delta t\ge\pi.
 \end{cases}
\end{equation}}
{\color{editorial}The ingredients depend only on $\delta,p,t$, with the core
optimization specified in \eqref{eq:Dq-combined}. Both entries are
nonnegative upper bounds for the sector count: the row entry
retains the angular lattice, whereas the other entry comes from
continuous order integration and right-endpoint sampling.
Taking their minimum preserves an upper bound, so}

{\begin{equation}\label{eq:Dq-U-upper}
 0\le N_S^{\mathrm D}\paren{t^2}\le U_{\delta,p}\paren{t},
 \qquad t>0.
\end{equation}}

{\begin{theorem}[Inverting the refined Dirichlet counting bound]
\label{thm:Dq-bound-inverse}
For $k\ge1$, put
\begin{equation}\label{eq:Dq-sigma}
 \sigma_k\paren{\delta,p}
 =\sup\set{t>0:U_{\delta,p}\paren{t}<k}.
\end{equation}
This supremum is positive and finite.  For an annular sector with
physical radii $r<R_0$, set
$\delta=R_0/r-1$ and $p=\pi/\beta$.  Then
\begin{equation}\label{eq:Dq-sigma-eigen}
 \lambda_k\paren{S_{r,R_0,\beta}}
 \ge{\frac{\sigma_k\paren{\delta,p}^2}{r^2}}.
\end{equation}
In normalized variables, with $w_k=\sqrt{8pk/D}$, one also has
\begin{equation}\label{eq:Dq-sigma-explicit}
 \sigma_k\paren{\delta,p}\ge
 \max\set{
   \sqrt{w_k^2+\frac{4c_\delta^2}{D^2}}+\frac{2c_\delta}D,
   \sqrt{\frac{\pi^2}{\delta^2}+\frac{p^2-1/4}{R^2}}}.
\end{equation}
In particular, \eqref{eq:Dq-sigma-eigen} retains the positive
all-index correction \eqref{eq:Dq-all-eigen}.
An explicit upper bound for this supremum is
\begin{equation}\label{eq:Dq-sigma-search-endpoint}
 \sigma_k\paren{\delta,p}
 \le\sqrt{\frac{\pi^2}{\delta^2}+k^2p^2-\frac14}.
\end{equation}
For fixed $\delta,p$, its asymptotic correction is
\begin{equation}\label{eq:Dq-sigma-asymptotic}
 \sigma_k\paren{\delta,p}^2-w_k^2
 ={\frac{4L_\delta}{D}}w_k+o_{\delta,p}\paren{w_k},
 \qquad k\longrightarrow+\infty,
 \qquad L_\delta=\frac{R+1}2.
\end{equation}
\end{theorem}}

\begin{proof}
{Every $t$ in the set defining \eqref{eq:Dq-sigma} satisfies
$N_S^{\mathrm D}\paren{t^2}<k$, by
\eqref{eq:Dq-U-upper}.  Since the count includes eigenvalues at
equality, this implies $t<\sqrt{\lambda_k\paren{S}}$.
Taking the supremum gives the normalized inequality, and scaling
gives \eqref{eq:Dq-sigma-eigen}.}
{This argument uses the entire set $\{t:U_{\delta,p}(t)<k\}$.
In particular, $U_{\delta,p}$ need not be continuous or monotone:
row floors have jumps, and the retained fractional phases can
make individual phase bounds oscillate. It would be
unjustified to replace the supremum by an arbitrary solution of
$U_{\delta,p}(t)=k$.}

{The all-index bound gives}
{\[
 U_{\delta,p}\paren{t}
 \le{\frac{D}{8p}}\pos{t^2-\frac{4c_\delta}D t},\qquad t>0.
\]}
{For $\delta t<\pi$ this follows from $U=0$; otherwise it follows
from $\mathcal B_\delta^{\mathrm D}\paren{t}\ge c_\delta t$.
To invert this simpler bound directly, put $a=4c_\delta/D$.
Above its zero interval, the equation
$t^2-at-w_k^2=0$ has roots
$a/2\pm\sqrt{w_k^2+a^2/4}$. The negative root is irrelevant.
The right side of the preceding counting bound is strictly below
$k$ for every positive $t$ below}
{\[
 s_k=\sqrt{w_k^2+\frac{4c_\delta^2}{D^2}}+\frac{2c_\delta}D.
\]}
{Hence $\sigma_k\ge s_k$.
Moreover, every sector mode has $n\ge1$ and $jp\ge p\ge1/2$.
The signed comparison gives}
{\[
 \lambda^{\mathrm D}_{jp,n}
 \ge\frac{\pi^2}{\delta^2}+\frac{p^2-1/4}{R^2}.
\]}
{Below the square root of this value the discrete row bound
itself vanishes.  Consequently $U=0$ there and
$\sigma_k$ is at least that square root.  This proves
\eqref{eq:Dq-sigma-explicit}, including its weaker value
$\pi/\delta$.}

{For the upper endpoint, the transformed radial trial
$z(s)=\sqrt{2/\delta}\sin(\pi(s-1)/\delta)$ has unit norm and
kinetic energy $\pi^2/\delta^2$. Since $jp\ge1/2$, its potential
energy is at most $j^2p^2-1/4$. The variational principle gives
$\lambda^{\mathrm D}_{jp,1}\le\pi^2/\delta^2+j^2p^2-1/4$,
because $\paren{j^2p^2-1/4}s^{-2}\le j^2p^2-1/4$.
At the energy in \eqref{eq:Dq-sigma-search-endpoint}, the first
$k$ angular ground modes have therefore been counted.  The actual
count, and hence $U_{\delta,p}$, is at least $k$ at every greater
frequency.  This proves the upper endpoint and also proves
finiteness directly.}

{It remains to identify the high-index correction. This part
uses the continuous asymptotic from
Section~\ref{sec:boundary-optimality}; the pointwise inverse,
its explicit two-sided bounds, and its improvement of the main theorem
have already been proved without that asymptotic.

First observe
that $p\ge1/2$ eliminates every inner-branch row from the discrete
sum.  Thus \eqref{eq:Dq-discrete-rows} is the lattice count}
{\[
 \mathcal R_{\delta,p}\paren{t}
 =\#\set{n,j\ge1:
       \paren{\frac{\pi n}\delta}^2
       +\paren{\frac{pj}R}^2
       \le t^2+\frac1{4R^2}}.
\]}
{Writing $\tau^2=t^2+1/(4R^2)$ makes the area calculation
transparent. The row length before flooring is
$R\sqrt{\tau^2-(\pi n/\delta)^2}/p$, and its integral is
\[
 \frac Rp\int_0^{\delta\tau/\pi}
       \sqrt{\tau^2-\paren{\pi s/\delta}^2}\dd s
 =\frac{\delta R\tau^2}{\pi p}
       \int_0^1\sqrt{1-y^2}\dd y
 =\frac{\delta R}{4p}\tau^2.
\]
The quarter-circle integral is $\pi/4$. This calculation also
gives a finite implementation of the row sum: only
$n\le\lfloor\delta\tau/\pi\rfloor$ need be examined.}
{For fixed $\delta,p$, comparison of each decreasing row with its
integral gives the quadrant-ellipse estimate}
{\begin{equation}\label{eq:Dq-row-asymptotic}
 \mathcal R_{\delta,p}\paren{t}
 ={\frac{\delta R}{4p}}t^2+O_{\delta,p}\paren{t}.
\end{equation}}
{Indeed, replacing each floor by its argument costs at most one per
row, hence $O_{\delta,p}\paren{t}$.  Replacing the sum of the
decreasing row lengths by its integral costs another
$O_{\delta,p}\paren{t}$.  That integral is the quadrant area, and
the added $1/\paren{4R^2}$ changes it by a bounded quantity.}

{The combined continuous bound obeys}
{\[
 \mathcal B_\delta^{\mathrm D}\paren{t}=L_\delta t+o_\delta\paren{t}.
\]}
{Its upper bound is the actual deficit, whose asymptotic is
\eqref{eq:Dq-continuous-asymptotic}{\color{revision}, proved in
Section~\ref{sec:boundary-optimality}}.  Its lower bound is the
included explicit value $F_\delta^{\mathrm D}\paren{t}
=L_\delta t-O_\delta\paren{\sqrt t}$ for $t\ge T_\delta$.
It follows that the second entry of \eqref{eq:Dq-U} is}
{\[
{\frac{D}{8p}}t^2-{\frac{L_\delta}{2p}}t+o_{\delta,p}\paren{t}.
\]}
{The leading coefficient in \eqref{eq:Dq-row-asymptotic} is
strictly larger:}
{\[
 {\frac{\delta R}{4p}}-{\frac{D}{8p}}={\frac{\delta^2}{8p}}>0.
\]}
{Therefore the second entry determines the minimum for all
sufficiently large $t$.  In particular $U_{\delta,p}\paren{t}\to+\infty$,
which proves finiteness of the supremum.  Its positivity follows
already from the zero interval below $\pi/\delta$.}

{Finally, put $a=D/\paren{8p}$ and $b=L_\delta/\paren{2p}$.
For every $\eta>0$, and all sufficiently large $t$,}
{\[
 at^2-\paren{b+\eta}t
 \le U_{\delta,p}\paren{t}
 \le at^2-\paren{b-\eta}t.
\]}
{The roots of the two quadratics at height $k$ show, without a
continuity assumption, that}
{\[
 \sigma_k=w_k+\frac b{2a}+o_{\delta,p}\paren{1}.
\]}
{For the upper bound, the first quadratic is increasing above its
large positive root, so every greater $t$ has $U\ge k$.
For the lower bound, values just below the second quadratic's
large positive root have $U<k$.
Their positive roots are
\[
 \rho_\pm\paren k
 =\frac{b\pm\eta+
        \sqrt{\paren{b\pm\eta}^2+4ak}}{2a}
 =w_k+\frac{b\pm\eta}{2a}+O_{\delta,p,\eta}\paren{w_k^{-1}}.
\]
Thus $\rho_-(k)\le\sigma_k\le\rho_+(k)$ for large $k$,
with the lower inequality understood as a supremum of values
strictly below $\rho_-(k)$. Letting $\eta\downarrow0$ gives
the displayed expansion for $\sigma_k$.
Squaring and using $b/a=4L_\delta/D$ gives
\eqref{eq:Dq-sigma-asymptotic}.}
\end{proof}

{\color{editorial}The number $\sigma_k$ is defined analytically
by the supremum in \eqref{eq:Dq-sigma}. The upper estimate
\eqref{eq:Dq-sigma-search-endpoint} bounds the defining set
from above, while \eqref{eq:Dq-sigma-explicit} gives a lower
bound for its supremum. Its validity as an eigenvalue
lower bound follows from the proof above, independently of
evaluating that supremum. The simpler closed formulas already
give the all-index conclusion of Theorem~\ref{thm:dirichlet}.

For direct use without optimizing all parameters, choose any
admissible $u$ in \eqref{eq:Dq-core-envelope}. The resulting
deficit estimate remains valid; if the resulting counting upper
bound at frequency $t$ is below $k$, then $t^2$ is a lower bound
for $\lambda_k$ in the normalized sector; for physical inner
radius $r$, the corresponding bound is $t^2/r^2$.
Numerical evaluation of the fully optimized bound
is optional. Since $U$ need not be monotone, a single crossing
of $U=k$ does not determine its defining supremum; a finite
sample or an unrestricted bisection is not a substitute for
the analytic definition.

The main theorem supplies a positive correction at every index;
the present inverse retains that guarantee while its leading
high-index correction has the circular-arc coefficient
$4L_\delta/D=2/\delta$ in normalized variables.
Section~\ref{sec:boundary-optimality} explains why this coefficient
is the limit available to the continuous order method, and how
it differs from the larger full-perimeter coefficient of the
sector.}

\section{The Neumann continuous count}\label{sec:neumann}
{\color{editorial}We prove Proposition~\ref{prop:continuous} by combining radial
trials at low frequencies with phase counts at higher frequencies.
The inner initial phase and reflected outer phase give the signs
required by Lemmas~\ref{lem:saw} and~\ref{lem:all-phase-rounding};
the trial bounds ensure overlap for every width.}

\subsection{Radial comparison and real-order phase counting}
{The weighted radial quotient introduced in
Section~\ref{sec:reduction} is}
{\[
 \mathcal R_v\paren{u}
 =\frac{\displaystyle\int_1^R
       \paren{s\abs{u'}^2+v^2s^{-1}\abs{u}^2}\dd s}
       {\displaystyle\int_1^R s\abs{u}^2\dd s},
 \qquad \lambda_{v,0}=\inf_{u\ne0}\mathcal R_v\paren{u}.
\]}
{The form domain is $H^1\paren{1,R}$; a trial function need not
satisfy the Neumann endpoint conditions. For clarity, the usual
conjugation $w\paren{s}=s^{1/2}u\paren{s}$ preserves the norm and gives}
{\[
 \begin{aligned}
 \int_1^R s\abs{u}^2\dd s&=\int_1^R\abs{w}^2\dd s,\\
 \int_1^R\paren{s\abs{u'}^2+v^2s^{-1}\abs{u}^2}\dd s
 &=\int_1^R\paren{\abs{w'}^2+
              \frac{v^2-1/4}{s^2}\abs{w}^2}\dd s
       -\bracket{\frac{\abs{w}^2}{2s}}_1^R.
 \end{aligned}
\]}
{Here the cross term is integrated as
$-\int_1^R\operatorname{Re}\paren{w'\overline w}s^{-1}\dd s
=-\bracket{\abs{w}^2/\paren{2s}}_1^R
-\frac12\int_1^R\abs{w}^2s^{-2}\dd s$.
Thus the differential expression becomes
$-w''+\paren{v^2-1/4}s^{-2}w$, and the Neumann conditions become
$w'\paren{1}=w\paren{1}/2$ and
$w'\paren{R}=w\paren{R}/\paren{2R}$.
Keeping the original weighted quotient avoids carrying these boundary
terms through the trial-function estimates. With $u\equiv1$ it gives}
{\[
 \lambda_{v,0}\le
 \frac{v^2\int_1^R s^{-1}\dd s}{\int_1^R s\dd s}
 =\frac{2v^2\log R}{D}=\frac{v^2}{a_\delta^2},\qquad
 a_\delta=\sqrt{\frac{\delta\paren{2+\delta}}{2\log\paren{1+\delta}}}.
\]}
{This is the real-order form of the ground-mode comparison in
\cite[Section~2.1 and Lemma~2.1]{DaiSunAnnuli}. Its proof uses only
the weighted quotient and therefore applies to every $v\ge0$.
Let $J_v,Y_v$ denote the Bessel functions of real order $v\ge0$,
and define the continuous derivative phase by}
\[
 J_v'\paren{x}+\mathrm iY_v'\paren{x}
 =\mathcal M_v\paren{x} e^{\mathrm i\varphi_v\paren{x}},\qquad
 \mathcal M_v\paren{x}>0,\qquad
 \lim_{x\downarrow0}\varphi_v\paren{x}=\frac\pi2.
\]
Primes denote argument derivatives. The phase decreases on
$\paren{0,v}$ and increases on $\paren{v,+\infty}$, with
$0<\varphi_v\paren{v}<\pi/2$ for $v>0$; at order zero it is
increasing. For $x>v$, put
\[
 E_v\paren{x}=\frac{9x^2-2v^2}{24\paren{x^2-v^2}^{3/2}},
\]
The enclosure of
\cite[Theorem~1.10, equation~\textup{(1.16)}]{Phase} implies
\begin{equation}\label{eq:phase}
 A_v\paren{x}+\frac\pi4<\varphi_v\paren{x}
 <A_v\paren{x}+\frac\pi4+E_v\paren{x}.
\end{equation}
We use the clipped upper envelope from that theorem below; it is
bounded above by the raw approximant in \eqref{eq:phase}.
{\color{revision}The theorem and equation numbers cited here refer to \cite{Phase}; the phase normalization agrees with \cite{DaiSunAnnuli}.}

\begin{lemma}[Real-order phase count]\label{lem:exact}
For every $t>0$ and real $v\ge0$,
\[
 n_v\paren{t^2}
 =\floor{\frac{\varphi_v\paren{Rt}-\varphi_v\paren{t}}{\pi}}+1.
\]
\end{lemma}
\begin{proof}
{This is \cite[Lemma~2.6]{DaiSunAnnuli} with the integer restriction
removed. The Neumann cross-product vanishes at the integer levels
of $d_v\paren{t}=\varphi_v\paren{Rt}-\varphi_v\paren{t}$.
For real $v>0$, the same phase minimum gives $d_v> -\pi/2$
and $d_v<0$ at small positive $t$; for $v=0$, $d_0$ starts
at zero from above. Normalize the radial solution by
$U\paren{1,\lambda}=1$, $\partial_sU\paren{1,\lambda}=0$,
where $\lambda=t^2$. An explicit solution with this normalization is}
{\[
 U\paren{s,t^2}=\frac{\pi t}{2}
   \paren{Y_v'\paren{t}J_v\paren{st}
              -J_v'\paren{t}Y_v\paren{st}}.
\]}
{The Bessel Wronskian gives $U\paren{1,t^2}=1$, and differentiation
in $s$ gives}
{\[
 \partial_sU\paren{R,t^2}
 =-\frac{\pi t^2}{2}\mathcal M_v\paren{t}
        \mathcal M_v\paren{Rt}\sin\paren{d_v\paren{t}}.
\]}
{At a crossing $d_v\paren{t}=k\pi$, the same Wronskian yields
$U\paren{R,t^2}=\paren{-1}^k\mathcal M_v\paren{t}/
\paren{R\mathcal M_v\paren{Rt}}$.
To check the direction of the crossing, put
$V=\partial_\lambda U$. Differentiating
$\paren{sU'}'+\paren{\lambda s-v^2/s}U=0$ gives}
{\[
 \bracket{s\paren{UV'-U'V}}'=-sU^2.
\]}
{The bracket vanishes at $s=1$. At a Neumann eigenvalue
$U'\paren{R}=0$, integration and differentiation of the sine
formula above therefore give}
\[
 -R U\paren{R,\lambda}\,
   \partial_\lambda\partial_sU\paren{R,\lambda}
 =\int_1^R sU\paren{s,\lambda}^2\dd s
 =\frac{\pi t\mathcal M_v\paren{t}^2}{4}\,d_v'\paren{t}>0.
\]
{Thus every integer crossing is upward and simple. For $v>0$ the
initial value of $d_v$ lies between $-\pi/2$ and zero, so its first
crossing is at zero; for $v=0$ the zero eigenvalue is already counted
before the first positive crossing at $\pi$. The discrete unbounded
radial spectrum then gives precisely the displayed floor formula.
At an eigenvalue the integer level is included, as required by the
definition of $n_v$. The argument in
\cite[Lemma~2.6]{DaiSunAnnuli} uses these same initial values,
Wronskian, and Lagrange identity; none requires an integer order.}
\end{proof}

To obtain a smaller turning-point cap, define
\[
 c\paren{v}=\frac1\pi\inf_{X>v}\paren{A_v\paren{X}+E_v\paren{X}}.
\]
The clipped enclosure and its order-monotone minimum
\cite[Lemma~1.8 and Theorem~1.10]{Phase} give
\[
0\le\frac{\varphi_v\paren{x}-A_v\paren{x}-\pi/4}{\pi}
 <c\paren{v},\qquad x\ge v\ge0,
\]
For $x=v>0$, continuity gives the lower inequality, while
the strict upper inequality follows directly from
$\varphi_v(v)<\varphi_v(X_v^*)<\pi/4+\pi c(v)$.
At $x=v=0$, the chosen phase limit gives the normalized error
$1/4<c(0)$. The order range $v\ge0$ is part of
\cite[Lemma~1.8 and Theorem~1.10]{Phase};
$c$ is nonincreasing. More explicitly, if $X_v^*>v$ is the
unique minimizer in \cite[Lemma~1.8]{Phase}, then the clipped upper
bound is $\pi/4+\pi c\paren{v}$ on $\paren{v,X_v^*}$.
Subtracting $A_v\ge0$ proves the error bound there. On
$\bracket{X_v^*,+\infty}$, the raw error $E_v$ decreases and
$E_v\paren{X_v^*}\le A_v\paren{X_v^*}+E_v\paren{X_v^*}
=\pi c\paren{v}$. The order monotonicity of the minimum gives
that of $c$. The same clipped-envelope argument is used
in~\cite[Lemmas~2.7 and~4.6]{DaiSunAnnuli}; here the full real-order
minimum $c(v)$ is retained.

\subsection{The constant trial function at low frequencies}\label{sec:low}
Set $x=\delta t$. Integrating the first radial mode over its
admissible angular interval gives $M\paren{t}\ge2a_\delta t$.
Consequently,
\[
\delta\paren{M\paren{t}-\frac D4t^2}\ge F_\delta\paren{x},
 \qquad F_\delta\paren{x}=2a_\delta x-\frac{2+\delta}{4}x^2.
\]
For fixed $x>0$, the normalized bound $F_\delta\paren{x}/a_\delta$
decreases with $\delta$. Indeed,
\[
 \frac{a_\delta^2}{\paren{2+\delta}^2}
 =\frac{\tanh\paren{L/2}}{2L},\qquad L=\log\paren{1+\delta},
\]
{and $\tanh z/z$ decreases for $z>0$; hence
$\paren{2+\delta}/a_\delta$ increases.
To verify the sign explicitly,}
{\[
 \frac{\dd}{\dd z}\frac{\tanh z}{z}
 =\frac{z\operatorname{sech}^2z-\tanh z}{z^2}<0.
\]}
{Multiplication of its numerator by $\cosh^2z$ leaves
$z-\sinh z\cosh z$. The function
$\sinh z\cosh z-z$ vanishes at zero and has derivative
$\cosh\paren{2z}-1>0$ for $z>0$. Since $L$ increases with
$\delta$, this proves the claimed width monotonicity of}
{\[
 \frac{F_\delta\paren{x}}{a_\delta}
 =2x-\frac{2+\delta}{4a_\delta}x^2.
\]}

Use the parameters
\[
d=\frac52,\qquad x_0=\frac{15}{4},\qquad
 X_0=\frac{27}{5},\qquad u=\frac9{10},\qquad C_0=\frac14.
\]
The elementary estimate $\log\paren{7/2}<63/50$ gives
$a_d^2>125/28>\paren{211/100}^2$. Thus
{\[
 \begin{aligned}
 F_d\paren{x_0}
 &=\frac{15}{4}\paren{2a_d-\frac{135}{32}}\\
 &>\frac{15}{4}\paren{\frac{211}{50}-\frac{135}{32}}
 =\frac3{640}>0.
 \end{aligned}
\]}
{Since $F_d$ is a concave quadratic vanishing at zero, this proves
positivity for $0<x\le x_0$. Width monotonicity extends the same
conclusion to every $0<\delta\le d$.
{\color{revision}The elementary bounds used here and below are proved in
Lemma~\ref{lem:endpoint-estimates} in Appendix~\ref{sec:elementary-estimates}.}}

{These rational parameters are selected to produce overlapping
estimates. The positive root of $F_d$ is $8a_d/\paren{2+d}$,
and $x_0=15/4$ lies just below it with the explicit gap above.
The cutoff $u=9/10$ leaves only one tenth of the inner order
interval for the turning cap, while
$ux_0/d=27/20>1$ permits the smaller cap $C_0=1/4$ throughout
that interval. The choice $d=5/2$ is large enough to overlap
the power-trial range used later, and the generic phase bound
at $X_0=27/5$ will have the positive margin $1/250$.
These positive margins and the common endpoints provide the
overlap needed for a complete cover.}

\subsection{Directional phase estimates and the thin-width range}\label{sec:high}
For $x\ge x_0$, put $T=ut$, $h=\paren{1-u}t$, and define
on $\bracket{0,T}$
\[
 \eta_t\paren{v}=E_v\paren{t}/\pi,\qquad
 q\paren{v}=f_t\paren{v}-\eta_t\paren{v}.
\]
{The inner action curvature and $\eta_t''$ both increase with $v$.
For the action this follows by differentiation:}
{\[
 \frac{\dd}{\dd v}\paren{-f_t''\paren{v}}
 =\frac v\pi\paren{
       \frac1{\paren{t^2-v^2}^{3/2}}
       -\frac1{\paren{R^2t^2-v^2}^{3/2}}}>0,
 \qquad 0<v<t.
\]}
{For the phase error, put $z=v/t$. Then}
{\[
 \begin{aligned}
 \eta_t\paren{v}
 &=\frac{9-2z^2}{24\pi t\paren{1-z^2}^{3/2}},\\
 \eta_t'\paren{v}
 &=\frac{z\paren{23-2z^2}}
       {24\pi t^2\paren{1-z^2}^{5/2}},\\
 \eta_t''\paren{v}
 &=\frac{23+86z^2-4z^4}
       {24\pi t^3\paren{1-z^2}^{7/2}}.
 \end{aligned}
\]}
{In particular, $\eta_t'\paren{0}=0$ and
$\eta_t''\paren{0}=23/\paren{24\pi t^3}$. The derivative of
the dimensionless factor in the last line is}
{\[
 \frac{\dd}{\dd z}
 \frac{23+86z^2-4z^4}{\paren{1-z^2}^{7/2}}
 =\frac{3z\paren{111+138z^2-4z^4}}
       {\paren{1-z^2}^{9/2}}>0,
 \qquad 0<z<1.
\]}
{Both curvature terms therefore attain their lower bounds at zero.
Consequently,}
\[
 q'\paren{0}=0,\qquad
 -q''\paren{v}\ge\kappa_{\mathrm{in}}
 :=\frac{\delta}{\pi Rt}+\frac{23}{24\pi t^3}.
\]
The second term is the exact contribution
$\eta_t''\paren{0}$~\cite[Proposition~4.7]{DaiSunAnnuli}. The core's initial phase is
\[
a_{\mathrm N}\paren{x,\delta}
 =\set{\frac x\pi-\frac{3\delta}{8\pi x}}.
\]
{Indeed, $f_t\paren{0}=\delta t/\pi=x/\pi$ and
$\eta_t\paren{0}=3/\paren{8\pi t}=3\delta/\paren{8\pi x}$.
We retain this phase in Lemma~\ref{lem:all-phase-rounding}, rather
than replacing $C_+\paren{a_{\mathrm N}}$ by its global maximum.}

Let $v_0=ux_0/d>1$. Evaluating the raw phase upper approximant
at the rational point $X_e=2$ gives
\[
c\paren{v_0}
 \le\frac{A_{v_0}\paren{2}+E_{v_0}\paren{2}}\pi<C_0.
\]
{Its order monotonicity therefore supplies the cap $C_0$ throughout
the turning band $\bracket{T,t}$: here
$v\ge T=ux/\delta\ge ux_0/d=v_0$.}

{We next derive the pointwise estimates, including their rounding
signs. On the core $0\le v\le T<t$, the lower bound at $Rt$
and the upper bound at $t$ in~\eqref{eq:phase} give}
{\[
 \frac{\varphi_v\paren{Rt}-\varphi_v\paren{t}}\pi
 >\frac{A_v\paren{Rt}-A_v\paren{t}}\pi
       -\frac{E_v\paren{t}}\pi=q\paren{v}.
\]}
{The floor function is nondecreasing, and
$\floor{q}+1=q+1/2-\psi\paren{q}$. On $T<v\le t$ the clipped
cap gives the same phase-difference lower bound with
$f_t\paren{v}-C_0$ in place of $q\paren{v}$; using
$\floor{z}+1>z$ then discards the rounding at a loss of only
$C_0$. On the outer interval $t<v<Rt$, the phase at $t$ is
below its small-argument limit $\pi/2$, whereas the lower
bound at $Rt$ is $A_v\paren{Rt}+\pi/4$. Hence}
{\[
 \frac{\varphi_v\paren{Rt}-\varphi_v\paren{t}}\pi
 >f_t\paren{v}-\frac14,
 \qquad
 n_v\paren{t^2}\ge\floor{f_t\paren{v}-\frac14}+1.
\]}
{Writing this last floor as
$f_t+1/4-\psi\paren{f_t-1/4}$ yields}
\begin{equation}\label{eq:bands}
 n_v\paren{t^2}\ge
 \begin{cases}
 f_t\paren{v}-\eta_t\paren{v}+\frac12-\psi\paren{q\paren{v}},
     &0\le v\le T,\\
 f_t\paren{v}-C_0,&T<v\le t,\\
 f_t\paren{v}+\frac14-\psi\paren{f_t\paren{v}-1/4},
     &t<v<Rt.
 \end{cases}
\end{equation}
For the outer band, the last line follows from
$n_v\ge\floor{f_t\paren{v}-1/4}+1$. It retains the favorable
mean $1/4$ in the integer rounding.

Figure~\ref{fig:bands} shows the three order intervals.
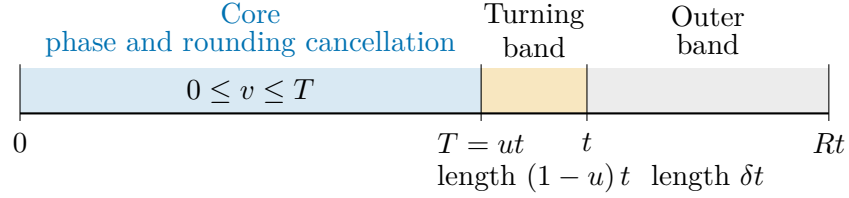
\begin{figure}[htbp]
\centering
\begin{tikzpicture}[x=1cm,y=1cm,font=\small]
 \fill[annBlue,opacity=.15] (0,0) rectangle (6.1,.6);
 \fill[annOrange,opacity=.25] (6.1,0) rectangle (7.5,.6);
 \fill[gray,opacity=.15] (7.5,0) rectangle (10.7,.6);
 \draw[thick] (0,0)--(10.7,0);
 \foreach \xx/\lab in {0/0,6.1/{T=ut},7.5/t,10.7/{Rt}}{
  \draw (\xx,-.1)--(\xx,.65);
  \node[below] at (\xx,-.12) {$\lab$};
 }
 \node[annBlue] at (3.05,1.1)
  {\shortstack{Core\\phase and rounding cancellation}};
 \node at (6.8,1.1) {\shortstack{Turning\\band}};
 \node at (9.1,1.1) {\shortstack{Outer\\band}};
 \node at (3.05,.3) {$0\le v\le T$};
 \node at (6.8,-.85) {length $\paren{1-u}t$};
 \node at (9.1,-.85) {length $\delta t$};
\end{tikzpicture}
\caption{The Neumann angular-order decomposition in
\eqref{eq:bands}. Spacings are schematic. The core retains the
positive mean rounding contribution. The middle band uses a clipped
phase cap, and reflection preserves the favorable mean in the outer
band.}
\label{fig:bands}
\end{figure}

Reflect the outer band by setting
\[
 q_{\mathrm{out}}\paren{y}=-f_t\paren{Rt-y}+\frac14,
 \qquad 0\le y\le\delta t.
\]
Here $q_{\mathrm{out}}'\paren{0}=0$, its initial phase is $1/4$, and
\[
 -q_{\mathrm{out}}''\paren{y}
 =\frac1{\pi\sqrt{R^2t^2-\paren{Rt-y}^2}}
 \ge\kappa_{\mathrm{out}}:=\frac1{\pi t\sqrt D}.
\]
{Indeed, on this interval}
{\[
 \begin{aligned}
 f_t\paren{v}&=\frac{A_v\paren{Rt}}\pi,\\
 f_t'\paren{v}&=-\frac1\pi\arccos\paren{\frac{v}{Rt}},
 \qquad f_t''\paren{v}=\frac1{\pi\sqrt{R^2t^2-v^2}}.
 \end{aligned}
\]}
{Since $Rt-y\ge t$, the square root in the denominator is
at most $t\sqrt D$, proving the uniform curvature bound.
The endpoint singularity is integrable, so the rounding lemma
applies directly to the absolutely continuous derivative. Since
$\psi\paren{f_t\paren{Rt-y}-1/4}=-\psi\paren{q_{\mathrm{out}}\paren{y}}$
almost everywhere, its lower bound controls the required outer
rounding error. The core and outer losses are at most}
\[
 \frac{2C_+\paren{a_{\mathrm N}}}{\sqrt{\kappa_{\mathrm{in}}}},
 \qquad
 \frac{\sqrt2}{3\sqrt{\kappa_{\mathrm{out}}}},
\]
respectively.

{The signs here are essential. The core enters the integrated
count as $-2\int_0^T\psi\paren{q}\dd v$, so it requires the
upper rounding bound $C_+$. Reflection gives}
{\[
 \begin{aligned}
 -2\int_t^{Rt}\psi\paren{f_t\paren{v}-1/4}\dd v
 &=2\int_0^{\delta t}\psi\paren{q_{\mathrm{out}}\paren{y}}\dd y\\
 &\ge-\frac{\sqrt2}{3\sqrt{\kappa_{\mathrm{out}}}},
 \end{aligned}
\]}
{because the lower rounding coefficient at initial phase $1/4$
is $C_-\paren{1/4}=\sqrt2/6$. The values of the sawtooth
at its isolated integer crossings do not affect either integral.}

The integrated phase error is
{\[
 \begin{aligned}
 I\paren{u}&=\int_0^T\eta_t\paren{v}\dd v
 =\frac1{24\pi}\int_0^u
       \frac{9-2z^2}{\paren{1-z^2}^{3/2}}\dd z\\
 &=\frac1{24\pi}\bracket{
       \frac{7z}{\sqrt{1-z^2}}+2\arcsin z}_0^u
 =\frac{7u/\sqrt{1-u^2}+2\arcsin u}{24\pi}.
 \end{aligned}
\]}
{The antiderivative follows from the identity}
{\[
 \frac{\mathrm d}{\mathrm d z}
 \paren{\frac{7z}{\sqrt{1-z^2}}+2\arcsin z}
 =\frac7{\paren{1-z^2}^{3/2}}+\frac2{\sqrt{1-z^2}}
 =\frac{9-2z^2}{\paren{1-z^2}^{3/2}}.
\]}
{In particular, the substitution $v=tz$ removes all dependence
on $t$ from this integrated error.}

{We can now track every contribution to the count. The core has
mean rounding contribution $T/2$, the turning interval loses
$C_0h$, and the outer interval has mean contribution $\delta t/4$.
Multiplication by two, followed by~\eqref{eq:action}, gives}
{\[
 \begin{aligned}
 M\paren{t}-\frac D4t^2
 &\ge T+\frac{\delta t}{2}-2C_0h-2I\paren{u}
       -\frac{2C_+\paren{a_{\mathrm N}}}
                    {\sqrt{\kappa_{\mathrm{in}}}}
       -\frac{\sqrt2}{3\sqrt{\kappa_{\mathrm{out}}}}\\
 &=t\paren{1+\frac\delta2}
       -\paren{1+2C_0}\paren{1-u}t-2I\paren{u}
       -\frac{2C_+\paren{a_{\mathrm N}}}
                    {\sqrt{\kappa_{\mathrm{in}}}}
       -\frac{\sqrt2}{3\sqrt{\kappa_{\mathrm{out}}}}.
 \end{aligned}
\]}
{Here the identity $T=t-h$ explains the additional $h$ in the
turning-interval loss. To normalize by $t$, observe that}
{\[
 \begin{aligned}
 t^2\kappa_{\mathrm{in}}
 &=\frac{\delta t}{\pi R}+\frac{23}{24\pi t}
 =\frac{x}{\pi\paren{1+\delta}}
       +\frac{23\delta}{24\pi x},\\
 \frac1{t\sqrt{\kappa_{\mathrm{out}}}}
 &=\sqrt{\frac{\pi\sqrt D}{t}}
 =\sqrt{\frac{\pi\delta\sqrt{\delta\paren{2+\delta}}}{x}},
 \qquad
 \frac{2I\paren{u}}{t}=\frac{2\delta I\paren{u}}{x}.
 \end{aligned}
\]}
Define
\[
 \begin{aligned}
 K_\delta\paren{x}&=\frac{x}{\pi\paren{1+\delta}}
                    +\frac{23\delta}{24\pi x},\\
 G_\delta\paren{x;c}
 &=1+\frac\delta2-\frac{2c}{\sqrt{K_\delta\paren{x}}}
   -\frac{\sqrt2}{3}\sqrt{\frac{\pi\delta\sqrt{\delta\paren{2+\delta}}}{x}}\\
 &\hspace{1em}-\paren{1+2C_0}\paren{1-u}
                 -\frac{2\delta I\paren{u}}{x}.
 \end{aligned}
\]
{Substitution of these identities gives the complete
phase-dependent lower envelope}
\begin{equation}\label{eq:high}
 \frac{M\paren{t}-Dt^2/4}{t}
 \ge G_\delta\paren{x;C_+\paren{a_{\mathrm N}\paren{x,\delta}}}.
\end{equation}
For $x_0\le x\le X_0$ and $0<\delta\le d$, the unwrapped
initial phase belongs to a single period, with
\[
 \frac19<\frac x\pi-\frac{3\delta}{8\pi x}-1<\frac{18}{25}.
\]
{For its lower bound use $x=x_0$, $\delta=d$, and
$\pi<22/7$; the upper bound follows from $x\le X_0$ and
$\pi>157/50$. Explicitly, the unwrapped phase is increasing
in $x$ and decreasing in $\delta$, and its endpoints satisfy}
{\[
 \frac{7}{2\pi}-1>\frac5{44}>\frac19,
 \qquad
 \frac{27}{5\pi}-1<\frac{113}{157}<\frac{18}{25}.
\]}
{In particular, the same integer part, namely one, is retained
throughout this window. The half-period monotonicity of $C_+$ and the
bounds $C_+\paren{1/9},C_+\paren{18/25}<1/10$
\textcolor{editorial}{from Lemma~\ref{lem:endpoint-estimates}}
therefore control the whole phase window.}

{For constant $c$ and $0<\delta\le d$,
$G_\delta\paren{x;c}$ increases with $x\ge x_0$, because $K_\delta$ increases there and the other errors decrease.
The required inequality is $x_0^2>23d\paren{1+d}/24$.
At $x=x_0$, $K_\delta$ decreases with $\delta\le d$, since
$x_0^2>23\paren{1+d}^2/24$. These assertions follow directly from}
{\[
 \partial_xK_\delta\paren{x}
 =\frac1{\pi\paren{1+\delta}}
       -\frac{23\delta}{24\pi x^2},
 \qquad
 \partial_\delta K_\delta\paren{x}
 =-\frac{x}{\pi\paren{1+\delta}^2}
       +\frac{23}{24\pi x}.
\]}
{The two comparisons reduce to
$225/16>805/96$ and $225/16>1127/96$, respectively.
Thus its core error is at most
$\paren{1/5}/\sqrt{K_d\paren{x_0}}<7/25$.
Using $\sqrt D<\delta+1$ and $I\paren{9/10}<111/500$,
the outer and integrated phase errors satisfy}
\[
 s<\frac{54}{125}\paren{\delta+\frac12},\qquad
 \frac{2\delta I\paren{9/10}}{x_0}<\frac{119}{1000}\delta.
\]
{For the first comparison, write}
{\[
 s^2=\frac{2\pi}{9x_0}\delta\sqrt D
 <\frac{2\pi}{9x_0}\delta\paren{\delta+1}
 <\frac{2\pi}{9x_0}\paren{\delta+\frac12}^2,
\]}
{and use $2\pi/\paren{9x_0}<176/945<\paren{54/125}^2$.
The second follows from
$2\paren{111/500}/x_0=148/1250<119/1000$.
Substitution yields}
\[
 G_\delta\paren{x_0;1/10}
 >\frac{177}{500}-\frac{51}{1000}\delta
 \ge\frac{453}{2000}>0.
\]
Frequency monotonicity proves~\eqref{eq:high} throughout the window.

For $x\ge X_0$, use the generic coefficient $c=1/3$.
Frequency monotonicity remains valid. To control the width, put
$A=X_0/\pi$, $b=23/\paren{24\pi X_0}$, $R=1+\delta$, and
$K=A/R+b\delta$. Denote the core and outer errors by $r$ and $s$.
Their width derivatives are
\[
 r'=\frac{A/R^2-b}{3K^{3/2}},\qquad
 s'=\frac{\sqrt2}{3}\sqrt{\frac\pi{X_0}}
       \frac{\delta+3/2}{\delta^{1/4}\paren{\delta+2}^{3/4}}.
\]
{Both decrease on $0<\delta\le d$. For the core error, the
numerator governing $r''$, multiplied by $R^4/A^2$, is
$-1-10z+4z/R+3z^2<0$, where $z=bR^2/A<1$.
For the outer error, $g=\delta^{3/4}\paren{\delta+2}^{1/4}$
satisfies $g''/g=-3/\paren{4\delta^2\paren{\delta+2}^2}<0$.
To see these signs without omitting terms, use}
{\[
 \begin{aligned}
 K'&=b-\frac A{R^2},\qquad K''=\frac{2A}{R^3},\\
 r''&=\frac{-4AK/R^3+3\paren{A/R^2-b}^2}{6K^{5/2}},\\
 \frac{R^4}{A^2}\paren{-\frac{4AK}{R^3}
                   +3\paren{\frac A{R^2}-b}^2}
 &=-1-10z+\frac{4z}{R}+3z^2
 \le-1-6z+3z^2<0.
 \end{aligned}
\]}
{Here $0<z\le23\paren{7/2}^2/\paren{24X_0^2}<1$.
For $g$, logarithmic differentiation gives}
{\[
 \frac{g'}g=\frac3{4\delta}+\frac1{4\paren{\delta+2}},
 \qquad
 \frac{g''}g
 =\paren{\frac3{4\delta}+\frac1{4\paren{\delta+2}}}^2
       -\frac3{4\delta^2}-\frac1{4\paren{\delta+2}^2}
 =-\frac3{4\delta^2\paren{\delta+2}^2}.
\]}
{\color{editorial}At $\delta=d$, Lemma~\ref{lem:endpoint-estimates} gives}
\[
 r'>\frac{11}{200},\qquad s'>\frac{73}{200},\qquad
 \frac{2I\paren{9/10}}{X_0}>\frac{11}{135}.
\]
Their sum exceeds $1/2$, so $G_\delta\paren{X_0;1/3}$ decreases
with width. At its endpoint the three errors are bounded by
$21/25$, $21/20$, and $103/500$, respectively. Hence
\[
 G_\delta\paren{x;1/3}\ge G_d\paren{X_0;1/3}
 >\frac94-\frac{21}{25}-\frac{21}{20}
             -\frac{103}{500}-\frac3{20}=\frac1{250}>0.
\]
{The width conclusion uses
$\partial_\delta G=1/2-r'-s'-2I\paren{9/10}/X_0$.
Since $r'$ and $s'$ decrease, their values at $d$ give lower
bounds at every smaller width. In particular,
$11/200+73/200+11/135=677/1350>1/2$.
}

{\color{editorial}The endpoint computations are given in
Lemma~\ref{lem:endpoint-estimates}. The margin $1/250$ fixes the
tail threshold $X_0$; together with the low-frequency estimate
and phase window, this proves Proposition~\ref{prop:continuous}
for $0<\delta\le5/2$.}

{The role of retaining the initial phase is illustrated in
Figure~\ref{fig:thin-envelopes}. At the endpoint width $d=5/2$, the
generic coefficient $1/3$ does not suffice at $x_0$, whereas the
coefficient $1/10$ gives a positive bound throughout the intervening
window. The constant trial and generic tail meet this window at its
two ends.}
{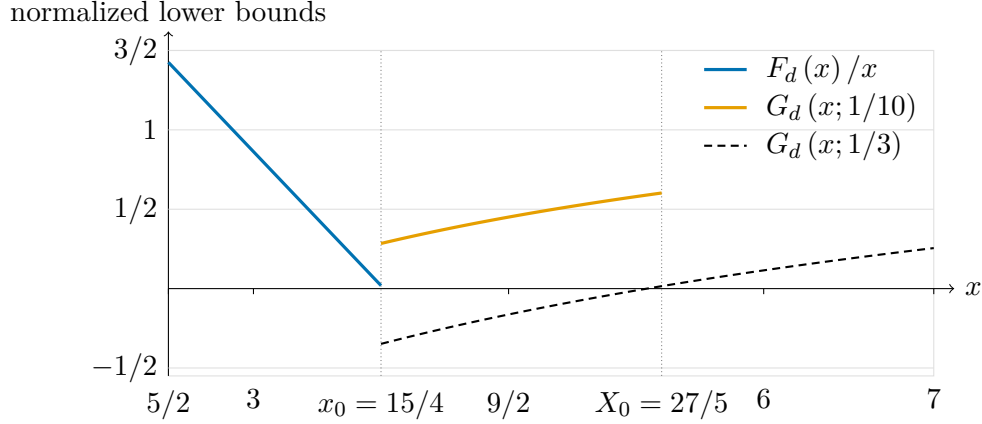
\begin{figure}[htbp]
\centering
\begin{tikzpicture}[x=2.25cm,y=2.1cm,font=\small]
 \draw[gray!25] (2.5,-.55) rectangle (7,1.5);
 \draw[->] (2.5,0)--(7.12,0) node[right] {$x$};
 \draw[->] (2.5,-.55)--(2.5,1.62) node[above] {normalized lower bounds};
 \foreach \y/\lab in {-.5/{-1/2},.5/{1/2},1/{1},1.5/{3/2}}{
   \draw[gray!25] (2.5,\y)--(7,\y);
   \node[left] at (2.5,\y) {$\lab$};
 }
 \foreach \x/\lab in {2.5/{5/2},3/{3},4.5/{9/2},6/{6},7/{7}}{
   \draw (\x,0)--(\x,-.035);
   \node[below] at (\x,-.57) {$\lab$};
 }
 \draw[gray,densely dotted] (3.75,-.55)--(3.75,1.5);
 \draw[gray,densely dotted] (5.4,-.55)--(5.4,1.5);
 \node[below] at (3.75,-.57) {$x_0=15/4$};
 \node[below] at (5.4,-.57) {$X_0=27/5$};
\draw[annBlue,very thick] plot coordinates {(2.5000000,1.4254595) (2.5250000,1.3973345) (2.5500000,1.3692095) (2.5750000,1.3410845) (2.6000000,1.3129595) (2.6250000,1.2848345) (2.6500000,1.2567095) (2.6750000,1.2285845) (2.7000000,1.2004595) (2.7250000,1.1723345) (2.7500000,1.1442095) (2.7750000,1.1160845) (2.8000000,1.0879595) (2.8250000,1.0598345) (2.8500000,1.0317095) (2.8750000,1.0035845) (2.9000000,0.9754595) (2.9250000,0.9473345) (2.9500000,0.9192095) (2.9750000,0.8910845) (3.0000000,0.8629595) (3.0250000,0.8348345) (3.0500000,0.8067095) (3.0750000,0.7785845) (3.1000000,0.7504595) (3.1250000,0.7223345) (3.1500000,0.6942095) (3.1750000,0.6660845) (3.2000000,0.6379595) (3.2250000,0.6098345) (3.2500000,0.5817095) (3.2750000,0.5535845) (3.3000000,0.5254595) (3.3250000,0.4973345) (3.3500000,0.4692095) (3.3750000,0.4410845) (3.4000000,0.4129595) (3.4250000,0.3848345) (3.4500000,0.3567095) (3.4750000,0.3285845) (3.5000000,0.3004595) (3.5250000,0.2723345) (3.5500000,0.2442095) (3.5750000,0.2160845) (3.6000000,0.1879595) (3.6250000,0.1598345) (3.6500000,0.1317095) (3.6750000,0.1035845) (3.7000000,0.0754595) (3.7250000,0.0473345) (3.7500000,0.0192095)};
\draw[annOrange,very thick] plot coordinates {(3.7500000,0.2843192) (3.7706250,0.2895455) (3.7912500,0.2947289) (3.8118750,0.2998700) (3.8325000,0.3049693) (3.8531250,0.3100275) (3.8737500,0.3150450) (3.8943750,0.3200225) (3.9150000,0.3249605) (3.9356250,0.3298595) (3.9562500,0.3347201) (3.9768750,0.3395426) (3.9975000,0.3443278) (4.0181250,0.3490759) (4.0387500,0.3537876) (4.0593750,0.3584633) (4.0800000,0.3631034) (4.1006250,0.3677084) (4.1212500,0.3722788) (4.1418750,0.3768150) (4.1625000,0.3813174) (4.1831250,0.3857864) (4.2037500,0.3902225) (4.2243750,0.3946261) (4.2450000,0.3989975) (4.2656250,0.4033373) (4.2862500,0.4076456) (4.3068750,0.4119230) (4.3275000,0.4161698) (4.3481250,0.4203864) (4.3687500,0.4245730) (4.3893750,0.4287302) (4.4100000,0.4328582) (4.4306250,0.4369574) (4.4512500,0.4410281) (4.4718750,0.4450706) (4.4925000,0.4490853) (4.5131250,0.4530725) (4.5337500,0.4570324) (4.5543750,0.4609655) (4.5750000,0.4648720) (4.5956250,0.4687522) (4.6162500,0.4726065) (4.6368750,0.4764350) (4.6575000,0.4802381) (4.6781250,0.4840161) (4.6987500,0.4877692) (4.7193750,0.4914977) (4.7400000,0.4952019) (4.7606250,0.4988820) (4.7812500,0.5025384) (4.8018750,0.5061712) (4.8225000,0.5097807) (4.8431250,0.5133672) (4.8637500,0.5169309) (4.8843750,0.5204720) (4.9050000,0.5239908) (4.9256250,0.5274875) (4.9462500,0.5309624) (4.9668750,0.5344156) (4.9875000,0.5378474) (5.0081250,0.5412581) (5.0287500,0.5446477) (5.0493750,0.5480166) (5.0700000,0.5513650) (5.0906250,0.5546930) (5.1112500,0.5580008) (5.1318750,0.5612888) (5.1525000,0.5645570) (5.1731250,0.5678056) (5.1937500,0.5710349) (5.2143750,0.5742451) (5.2350000,0.5774362) (5.2556250,0.5806086) (5.2762500,0.5837624) (5.2968750,0.5868977) (5.3175000,0.5900149) (5.3381250,0.5931139) (5.3587500,0.5961950) (5.3793750,0.5992584) (5.4000000,0.6023043)};
\draw[black,thick,densely dashed] plot coordinates {(3.7500000,-0.3481558) (3.7770833,-0.3407189) (3.8041667,-0.3333453) (3.8312500,-0.3260340) (3.8583333,-0.3187840) (3.8854167,-0.3115945) (3.9125000,-0.3044644) (3.9395833,-0.2973930) (3.9666667,-0.2903793) (3.9937500,-0.2834224) (4.0208333,-0.2765216) (4.0479167,-0.2696760) (4.0750000,-0.2628849) (4.1020833,-0.2561475) (4.1291667,-0.2494630) (4.1562500,-0.2428307) (4.1833333,-0.2362498) (4.2104167,-0.2297197) (4.2375000,-0.2232397) (4.2645833,-0.2168091) (4.2916667,-0.2104273) (4.3187500,-0.2040935) (4.3458333,-0.1978073) (4.3729167,-0.1915679) (4.4000000,-0.1853748) (4.4270833,-0.1792273) (4.4541667,-0.1731250) (4.4812500,-0.1670671) (4.5083333,-0.1610533) (4.5354167,-0.1550829) (4.5625000,-0.1491554) (4.5895833,-0.1432702) (4.6166667,-0.1374270) (4.6437500,-0.1316251) (4.6708333,-0.1258640) (4.6979167,-0.1201434) (4.7250000,-0.1144627) (4.7520833,-0.1088214) (4.7791667,-0.1032191) (4.8062500,-0.0976553) (4.8333333,-0.0921297) (4.8604167,-0.0866417) (4.8875000,-0.0811909) (4.9145833,-0.0757770) (4.9416667,-0.0703995) (4.9687500,-0.0650580) (4.9958333,-0.0597521) (5.0229167,-0.0544814) (5.0500000,-0.0492455) (5.0770833,-0.0440442) (5.1041667,-0.0388769) (5.1312500,-0.0337433) (5.1583333,-0.0286431) (5.1854167,-0.0235759) (5.2125000,-0.0185413) (5.2395833,-0.0135391) (5.2666667,-0.0085688) (5.2937500,-0.0036303) (5.3208333,0.0012770) (5.3479167,0.0061533) (5.3750000,0.0109989) (5.4020833,0.0158141) (5.4291667,0.0205992) (5.4562500,0.0253546) (5.4833333,0.0300805) (5.5104167,0.0347773) (5.5375000,0.0394451) (5.5645833,0.0440844) (5.5916667,0.0486954) (5.6187500,0.0532783) (5.6458333,0.0578334) (5.6729167,0.0623611) (5.7000000,0.0668615) (5.7270833,0.0713350) (5.7541667,0.0757818) (5.7812500,0.0802021) (5.8083333,0.0845962) (5.8354167,0.0889643) (5.8625000,0.0933068) (5.8895833,0.0976237) (5.9166667,0.1019155) (5.9437500,0.1061822) (5.9708333,0.1104242) (5.9979167,0.1146417) (6.0250000,0.1188348) (6.0520833,0.1230039) (6.0791667,0.1271490) (6.1062500,0.1312706) (6.1333333,0.1353687) (6.1604167,0.1394436) (6.1875000,0.1434954) (6.2145833,0.1475245) (6.2416667,0.1515309) (6.2687500,0.1555150) (6.2958333,0.1594768) (6.3229167,0.1634167) (6.3500000,0.1673347) (6.3770833,0.1712311) (6.4041667,0.1751060) (6.4312500,0.1789598) (6.4583333,0.1827924) (6.4854167,0.1866042) (6.5125000,0.1903952) (6.5395833,0.1941658) (6.5666667,0.1979159) (6.5937500,0.2016460) (6.6208333,0.2053560) (6.6479167,0.2090461) (6.6750000,0.2127167) (6.7020833,0.2163677) (6.7291667,0.2199993) (6.7562500,0.2236118) (6.7833333,0.2272053) (6.8104167,0.2307800) (6.8375000,0.2343359) (6.8645833,0.2378733) (6.8916667,0.2413922) (6.9187500,0.2448930) (6.9458333,0.2483756) (6.9729167,0.2518403) (7.0000000,0.2552872)};

 \draw[annBlue,very thick] (5.65,1.38)--(5.90,1.38);
 \node[anchor=west] at (5.95,1.38) {$F_d\left(x\right)/x$};
 \draw[annOrange,very thick] (5.65,1.14)--(5.90,1.14);
 \node[anchor=west] at (5.95,1.14) {$G_d\left(x;1/10\right)$};
 \draw[black,thick,densely dashed] (5.65,.9)--(5.90,.9);
 \node[anchor=west] at (5.95,.9) {$G_d\left(x;1/3\right)$};
\end{tikzpicture}
\caption{Complementary lower bounds at $d=5/2$. The constant
trial controls $x\le x_0$, the initial-phase estimate controls
$x_0\le x\le X_0$, and the generic estimate controls $x\ge X_0$.
The generic curve is also shown below $X_0$: its negative part
explains why replacing the initial phase by its worst-case value
would leave a gap. Both $F_d\left(x\right)/x$ and $G_d$ bound
$\paren{M\left(t\right)-Dt^2/4}/t$, with $x=dt$.}
\label{fig:thin-envelopes}
\end{figure}}

\subsection{A power trial function and unrestricted radial width}
The preceding estimates settle $0<\delta\le d_0:=5/2$.
For larger widths, a trial function adapted to the angular potential
gives a different low-frequency interval. This removes the obstruction
created by radially constant trial functions.

For $v>0$, take $u_v\paren{s}=s^v$. Direct integration gives
{\[
 \begin{aligned}
 \int_1^R\paren{s\abs{u_v'}^2
                      +v^2s^{-1}\abs{u_v}^2}\dd s
 &=2v^2\int_1^R s^{2v-1}\dd s
 =v\paren{R^{2v}-1},\\
 \int_1^R s\abs{u_v}^2\dd s
 &=\frac{R^{2v+2}-1}{2v+2}.
 \end{aligned}
\]}
{Substituting these two integrals into the radial quotient yields}
\begin{equation}\label{eq:power-trial}
 \lambda_{v,0}
 \le2v\paren{v+1}\frac{R^{2v}-1}{R^{2v+2}-1}
 <\frac{2v\paren{v+1}}{R^2}.
\end{equation}
The strict comparison uses
$R^2\paren{R^{2v}-1}<R^{2v+2}-1$. At $v=0$ the constant
function has eigenvalue zero. Therefore the first radial mode alone
gives, for every $t>0$,
\[
M\paren{t}>\sqrt{1+2R^2t^2}-1.
\]
To justify strictness, let $v_*\paren{t}>0$ solve
$2v_*\paren{v_*+1}=R^2t^2$. The trial quotient in
\eqref{eq:power-trial} is strictly below $t^2$ at $v_*$ and is
continuous in $v$, so the admissible order interval extends beyond
$v_*$. Its integrated count is strictly greater than $2v_*$.

Squaring the two nonnegative sides shows
\[
 \sqrt{1+2R^2t^2}-1\ge\frac D4t^2
 \quad\Longleftrightarrow\quad
 t^2\le\frac{8\paren{3R^2+1}}{D^2}.
\]
{More explicitly, adding one before squaring gives}
{\[
 1+2R^2t^2\ge\paren{1+\frac D4t^2}^2
 \quad\Longleftrightarrow\quad
 2R^2-\frac D2\ge\frac{D^2}{16}t^2.
\]}
{Since $D=R^2-1$, the left side is
$\paren{3R^2+1}/2$, which gives the stated threshold.
In the variable $x=\delta t$, the power trial thus settles}
\[
0<x\le p\paren\delta,
 \qquad
 p\paren\delta=\frac{\sqrt{8\paren{3\paren{1+\delta}^2+1}}}{2+\delta}.
\]
{The function $p$ increases with $\delta$ and tends to $\sqrt{24}$:
its logarithmic derivative, with $R=1+\delta$, is
$\paren{3R-1}/\paren{\paren{3R^2+1}\paren{R+1}}>0$.
Indeed,}
{\[
 \frac{p'}p=\frac{3R}{3R^2+1}-\frac1{R+1}
 =\frac{3R-1}{\paren{3R^2+1}\paren{R+1}},
 \qquad
 p^2=24-\frac{48}{\delta+2}+\frac{32}{\paren{\delta+2}^2}.
\]}
{Already $p\paren{5/2}^2=1208/81>\paren{15/4}^2$.
Thus the low-frequency intervals for the constant and power
trials overlap at the width where the proof changes estimates.}

We also need the turning cap for small real orders. The clipped
phase argument used above does not require $v\ge1$. Its raw-envelope
minimum decreases with real order and is at most the order-zero
minimum $\pi/4+\sqrt{3/2}$ \cite[Lemma~1.8]{Phase}. Below that
minimum the clipped envelope is constant; above it $E_v$ decreases.
Consequently,
\[
0\le\frac{\varphi_v\paren{x}-A_v\paren{x}-\pi/4}{\pi}
 <\frac{\sqrt{3/2}}\pi<\frac25,
 \qquad x\ge v\ge0.
\]
At a turning argument use continuity; at $v=x=0$ use the phase
limit, whose error is $\pi/4<\sqrt{3/2}$. This extension permits
cutoffs with $ut<1$ and introduces no new Bessel estimate.

For $0<u<1$, repeat the three-band calculation with this global cap.
Define
\[
 \begin{aligned}
 \mathcal H_{\delta,u}\paren{x;c}
 &=1+\frac\delta2-\frac{2c}{\sqrt{K_\delta\paren{x}}}
   -\frac{\sqrt2}{3}\sqrt{\frac{\pi\delta\sqrt{\delta\paren{2+\delta}}}{x}}\\
 &\hspace{1em}-\frac95\paren{1-u}-\frac{2\delta I\paren{u}}{x}.
 \end{aligned}
\]
Here $K_\delta$ and $I$ are the explicit functions already defined.
The same exact phase count, inner curvature, and reflected outer
rounding give
\[
\frac{M\paren{t}-Dt^2/4}{t}
 \ge\mathcal H_{\delta,u}
 \paren{x;C_+\paren{a_{\mathrm N}\paren{x,\delta}}}.
\]

For every $\delta\ge d_0$ and $0\le c\le1/3$, this envelope
increases with $x>0$. This fact does not require $K_\delta$ to
increase. Denote its core and outer rounding errors by $r$ and $s$.
If $K_\delta'\ge0$, both errors decrease. Otherwise
$r'\le r/\paren{2x}$, since
$-K_\delta'/K_\delta\le1/x$, while $s'=-s/\paren{2x}$.
Furthermore,
\[
 s^2K_\delta\paren{x}
 \ge\frac{2\delta\sqrt{\delta\paren{2+\delta}}}{9\paren{1+\delta}}
 >\frac49\ge4c^2.
\]
{At $d_0=5/2$, the ratio $\delta\sqrt D/\paren{1+\delta}$
 equals $15\sqrt5/14>2$. The strict inequality persists because
$\delta\sqrt{\delta\paren{2+\delta}}/\paren{1+\delta}$ increases.
Thus $s>r$, and the outer decrease dominates any possible increase
of the core error. The phase error also decreases with $x$.
Here primes refer to $x$ derivatives. To make the comparison
explicit, write $K=\alpha x+\beta/x$ with
$\alpha=1/\paren{\pi R}$ and $\beta=23\delta/\paren{24\pi}$.
Then}
{\[
 -\frac{K'}K=\frac{\beta-\alpha x^2}
                    {x\paren{\alpha x^2+\beta}}\le\frac1x,
 \qquad r'=-\frac{rK'}{2K}.
\]}
{The estimate $s>r$ follows by retaining the first summand in $K$:}
{\[
 s^2K=\frac{2\pi\delta\sqrt D}{9x}K
 \ge\frac{2\delta\sqrt D}{9R}>\frac49
 \ge4c^2=r^2K.
\]}
{If $Q\paren\delta=\delta\sqrt D/R$, then}
{\[
 Q'=\frac{\sqrt D}{R}+\frac\delta{R^2\sqrt D}>0,
 \qquad Q\paren{5/2}=\frac{15\sqrt5}{14}>2.
\]}
{Consequently, even in the case $K'<0$,}
{\[
 \partial_x\mathcal H_{\delta,u}\paren{x;c}
 =-r'+\frac{s}{2x}+\frac{2\delta I\paren{u}}{x^2}
 \ge\frac{s-r}{2x}+\frac{2\delta I\paren{u}}{x^2}>0.
\]}

{\color{revision}For fixed $\delta\ge5/2$, $0<u<1$, and $0\le c\le1/3$,
the preceding argument proves strict increase on all $x>0$.
The endpoint limits also follow directly from the defining formula:
\[
 \lim_{x\downarrow0}\mathcal H_{\delta,u}\paren{x;c}=-\infty,
 \qquad
 \lim_{x\to+\infty}\mathcal H_{\delta,u}\paren{x;c}
 =1+\frac\delta2-\frac95\paren{1-u}>0.
\]
Thus each fixed-parameter envelope has exactly one zero.
Figure~\ref{fig:Neumann-H} illustrates this shape and the improvement
from retaining the initial phase.  This monotonicity concerns a
fixed $c$; it makes no assertion about an envelope in which the
fractional initial phase varies with $x$.}

{\color{revision}\begin{figure}[htbp]
\centering
\begin{tikzpicture}
\begin{axis}[width=.46\textwidth,height=5.6cm,xlabel={$x$},ylabel={$\mathcal H_{\delta,9/10}\paren{x;1/3}/(1+\delta/2)$},xmin=1,xmax=15,ymin=-1,ymax=.85,legend pos=south east,legend style={draw=none,font=\scriptsize},tick label style={font=\scriptsize},label style={font=\scriptsize}]
\addplot[annBlue,thick,solid] coordinates {(1,-0.968030384) (1.03833333,-0.934567944) (1.07666667,-0.90328395) (1.115,-0.873958394) (1.15333333,-0.846399568) (1.19166667,-0.820439653) (1.23,-0.795931127) (1.26833333,-0.772743792) (1.30666667,-0.750762323) (1.345,-0.729884215) (1.38333333,-0.71001807) (1.42166667,-0.691082145) (1.46,-0.673003126) (1.49833333,-0.655715087) (1.53666667,-0.639158598) (1.575,-0.623279953) (1.61333333,-0.608030519) (1.65166667,-0.593366158) (1.69,-0.579246733) (1.72833333,-0.565635677) (1.76666667,-0.552499614) (1.805,-0.53980803) (1.84333333,-0.527532979) (1.88166667,-0.515648828) (1.92,-0.50413203) (1.95833333,-0.492960923) (1.99666667,-0.482115551) (2.035,-0.471577504) (2.07333333,-0.461329778) (2.11166667,-0.451356646) (2.15,-0.441643545) (2.18833333,-0.432176975) (2.22666667,-0.422944403) (2.265,-0.413934187) (2.30333333,-0.405135494) (2.34166667,-0.396538239) (2.38,-0.388133017) (2.41833333,-0.379911054) (2.45666667,-0.371864155) (2.495,-0.363984653) (2.53333333,-0.356265375) (2.57166667,-0.348699601) (2.61,-0.341281027) (2.64833333,-0.334003736) (2.68666667,-0.32686217) (2.725,-0.319851101) (2.76333333,-0.312965608) (2.80166667,-0.306201055) (2.84,-0.29955307) (2.87833333,-0.293017528) (2.91666667,-0.28659053) (2.955,-0.280268391) (2.99333333,-0.274047624) (3.03166667,-0.267924926) (3.07,-0.261897166) (3.10833333,-0.255961375) (3.14666667,-0.250114731) (3.185,-0.244354556) (3.22333333,-0.238678301) (3.26166667,-0.23308354) (3.3,-0.227567961) (3.33833333,-0.222129363) (3.37666667,-0.216765643) (3.415,-0.211474794) (3.45333333,-0.2062549) (3.49166667,-0.201104126) (3.53,-0.196020717) (3.56833333,-0.191002994) (3.60666667,-0.186049346) (3.645,-0.181158229) (3.68333333,-0.17632816) (3.72166667,-0.171557715) (3.76,-0.166845528) (3.79833333,-0.162190283) (3.83666667,-0.157590713) (3.875,-0.1530456) (3.91333333,-0.14855377) (3.95166667,-0.144114091) (3.99,-0.13972547) (4.02833333,-0.135386854) (4.06666667,-0.131097224) (4.105,-0.126855597) (4.14333333,-0.122661021) (4.18166667,-0.118512576) (4.22,-0.114409372) (4.25833333,-0.110350547) (4.29666667,-0.106335264) (4.335,-0.102362713) (4.37333333,-0.0984321097) (4.41166667,-0.0945426912) (4.45,-0.0906937176) (4.48833333,-0.0868844706) (4.52666667,-0.0831142519) (4.565,-0.079382383) (4.60333333,-0.0756882043) (4.64166667,-0.0720310739) (4.68,-0.0684103672) (4.71833333,-0.0648254763) (4.75666667,-0.0612758091) (4.795,-0.0577607887) (4.83333333,-0.0542798532) (4.87166667,-0.0508324543) (4.91,-0.0474180578) (4.94833333,-0.0440361424) (4.98666667,-0.0406861994) (5.025,-0.0373677322) (5.06333333,-0.0340802561) (5.10166667,-0.0308232977) (5.14,-0.0275963943) (5.17833333,-0.0243990941) (5.21666667,-0.0212309552) (5.255,-0.0180915457) (5.29333333,-0.0149804433) (5.33166667,-0.0118972348) (5.37,-0.00884151589) (5.40833333,-0.0058128911) (5.44666667,-0.00281097322) (5.485,0.000164616828) (5.52333333,0.00311425024) (5.56166667,0.00603829058) (5.6,0.00893709399) (5.6591195,0.0133590984) (5.71823899,0.0177231519) (5.77735849,0.0220304644) (5.83647799,0.0262822099) (5.89559748,0.0304795273) (5.95471698,0.0346235227) (6.01383648,0.0387152702) (6.07295597,0.0427558133) (6.13207547,0.0467461662) (6.19119497,0.0506873148) (6.25031447,0.0545802181) (6.30943396,0.0584258087) (6.36855346,0.0622249943) (6.42767296,0.0659786583) (6.48679245,0.0696876607) (6.54591195,0.0733528391) (6.60503145,0.0769750094) (6.66415094,0.0805549665) (6.72327044,0.084093485) (6.78238994,0.0875913202) (6.84150943,0.0910492083) (6.90062893,0.0944678673) (6.95974843,0.0978479976) (7.01886792,0.101190282) (7.07798742,0.104495388) (7.13710692,0.107763966) (7.19622642,0.110996651) (7.25534591,0.114194063) (7.31446541,0.117356808) (7.37358491,0.120485478) (7.4327044,0.123580649) (7.4918239,0.126642887) (7.5509434,0.129672744) (7.61006289,0.132670757) (7.66918239,0.135637455) (7.72830189,0.138573352) (7.78742138,0.141478952) (7.84654088,0.144354748) (7.90566038,0.147201221) (7.96477987,0.150018842) (8.02389937,0.152808073) (8.08301887,0.155569365) (8.14213836,0.158303159) (8.20125786,0.161009887) (8.26037736,0.163689972) (8.31949686,0.166343827) (8.37861635,0.168971859) (8.43773585,0.171574463) (8.49685535,0.174152029) (8.55597484,0.176704937) (8.61509434,0.17923356) (8.67421384,0.181738263) (8.73333333,0.184219404) (8.79245283,0.186677334) (8.85157233,0.189112396) (8.91069182,0.191524927) (8.96981132,0.193915257) (9.02893082,0.19628371) (9.08805031,0.198630602) (9.14716981,0.200956246) (9.20628931,0.203260945) (9.26540881,0.205544998) (9.3245283,0.207808701) (9.3836478,0.210052338) (9.4427673,0.212276195) (9.50188679,0.214480546) (9.56100629,0.216665663) (9.62012579,0.218831814) (9.67924528,0.22097926) (9.73836478,0.223108257) (9.79748428,0.225219058) (9.85660377,0.227311909) (9.91572327,0.229387054) (9.97484277,0.231444732) (10.0339623,0.233485175) (10.0930818,0.235508614) (10.1522013,0.237515274) (10.2113208,0.239505377) (10.2704403,0.241479142) (10.3295597,0.243436781) (10.3886792,0.245378504) (10.4477987,0.247304519) (10.5069182,0.249215028) (10.5660377,0.251110229) (10.6251572,0.252990319) (10.6842767,0.25485549) (10.7433962,0.256705932) (10.8025157,0.258541829) (10.8616352,0.260363365) (10.9207547,0.262170719) (10.9798742,0.263964067) (11.0389937,0.265743584) (11.0981132,0.26750944) (11.1572327,0.269261802) (11.2163522,0.271000836) (11.2754717,0.272726704) (11.3345912,0.274439565) (11.3937107,0.276139577) (11.4528302,0.277826893) (11.5119497,0.279501666) (11.5710692,0.281164046) (11.6301887,0.282814178) (11.6893082,0.284452209) (11.7484277,0.286078279) (11.8075472,0.28769253) (11.8666667,0.289295099) (11.9257862,0.290886122) (11.9849057,0.292465732) (12.0440252,0.29403406) (12.1031447,0.295591237) (12.1622642,0.29713739) (12.2213836,0.298672643) (12.2805031,0.300197121) (12.3396226,0.301710944) (12.3987421,0.303214233) (12.4578616,0.304707106) (12.5169811,0.306189678) (12.5761006,0.307662064) (12.6352201,0.309124377) (12.6943396,0.310576727) (12.7534591,0.312019223) (12.8125786,0.313451974) (12.8716981,0.314875085) (12.9308176,0.316288661) (12.9899371,0.317692805) (13.0490566,0.319087618) (13.1081761,0.320473199) (13.1672956,0.321849648) (13.2264151,0.323217062) (13.2855346,0.324575535) (13.3446541,0.325925163) (13.4037736,0.327266039) (13.4628931,0.328598253) (13.5220126,0.329921896) (13.5811321,0.331237057) (13.6402516,0.332543824) (13.6993711,0.333842284) (13.7584906,0.335132521) (13.8176101,0.336414619) (13.8767296,0.337688662) (13.9358491,0.338954732) (13.9949686,0.340212909) (14.0540881,0.341463272) (14.1132075,0.3427059) (14.172327,0.34394087) (14.2314465,0.345168259) (14.290566,0.346388141) (14.3496855,0.347600592) (14.408805,0.348805684) (14.4679245,0.350003489) (14.527044,0.351194079) (14.5861635,0.352377525) (14.645283,0.353553895) (14.7044025,0.354723258) (14.763522,0.355885682) (14.8226415,0.357041232) (14.881761,0.358189976) (14.9408805,0.359331978) (15,0.360467302)};
\addlegendentry{$\delta=2.5$}
\addplot[annBlue,thick,dashed] coordinates {(1,-1.1339829) (1.03833333,-1.08911401) (1.07666667,-1.04714483) (1.115,-1.0077935) (1.15333333,-0.970814011) (1.19166667,-0.935990603) (1.23,-0.903133223) (1.26833333,-0.872073763) (1.30666667,-0.842662951) (1.345,-0.814767758) (1.38333333,-0.788269226) (1.42166667,-0.763060639) (1.46,-0.73904597) (1.49833333,-0.716138564) (1.53666667,-0.694260011) (1.575,-0.673339178) (1.61333333,-0.65331138) (1.65166667,-0.634117654) (1.69,-0.615704138) (1.72833333,-0.598021527) (1.76666667,-0.581024594) (1.805,-0.564671778) (1.84333333,-0.548924813) (1.88166667,-0.533748409) (1.92,-0.519109965) (1.95833333,-0.504979316) (1.99666667,-0.491328509) (2.035,-0.478131603) (2.07333333,-0.465364489) (2.11166667,-0.453004733) (2.15,-0.441031431) (2.18833333,-0.42942508) (2.22666667,-0.418167463) (2.265,-0.407241546) (2.30333333,-0.39663138) (2.34166667,-0.386322021) (2.38,-0.376299446) (2.41833333,-0.36655049) (2.45666667,-0.357062776) (2.495,-0.347824659) (2.53333333,-0.338825174) (2.57166667,-0.330053988) (2.61,-0.32150135) (2.64833333,-0.313158059) (2.68666667,-0.305015418) (2.725,-0.297065208) (2.76333333,-0.289299649) (2.80166667,-0.281711374) (2.84,-0.274293406) (2.87833333,-0.267039125) (2.91666667,-0.259942254) (2.955,-0.25299683) (2.99333333,-0.246197192) (3.03166667,-0.239537956) (3.07,-0.233014002) (3.10833333,-0.226620457) (3.14666667,-0.220352684) (3.185,-0.214206261) (3.22333333,-0.208176976) (3.26166667,-0.20226081) (3.3,-0.196453931) (3.33833333,-0.190752677) (3.37666667,-0.185153555) (3.415,-0.179653223) (3.45333333,-0.17424849) (3.49166667,-0.1689363) (3.53,-0.163713732) (3.56833333,-0.158577989) (3.60666667,-0.153526392) (3.645,-0.148556377) (3.68333333,-0.143665482) (3.72166667,-0.138851351) (3.76,-0.13411172) (3.79833333,-0.129444422) (3.83666667,-0.124847371) (3.875,-0.120318567) (3.91333333,-0.11585609) (3.95166667,-0.111458092) (3.99,-0.107122798) (4.02833333,-0.102848502) (4.06666667,-0.0986335618) (4.105,-0.0944763968) (4.14333333,-0.0903754864) (4.18166667,-0.0863293662) (4.22,-0.0823366256) (4.25833333,-0.0783959054) (4.29666667,-0.0745058957) (4.335,-0.0706653336) (4.37333333,-0.066873001) (4.41166667,-0.0631277228) (4.45,-0.0594283652) (4.48833333,-0.0557738336) (4.52666667,-0.0521630713) (4.565,-0.0485950573) (4.60333333,-0.0450688056) (4.64166667,-0.0415833632) (4.68,-0.0381378087) (4.71833333,-0.0347312513) (4.75666667,-0.0313628295) (4.795,-0.0280317096) (4.83333333,-0.0247370851) (4.87166667,-0.021478175) (4.91,-0.0182542234) (4.94833333,-0.0150644979) (4.98666667,-0.0119082892) (5.025,-0.00878491001) (5.06333333,-0.00569369399) (5.10166667,-0.00263399525) (5.14,0.000394812591) (5.17833333,0.00339333711) (5.21666667,0.00636216781) (5.255,0.00930187677) (5.29333333,0.0122130193) (5.33166667,0.0150961347) (5.37,0.0179517465) (5.40833333,0.0207803634) (5.44666667,0.0235824797) (5.485,0.0263585757) (5.52333333,0.0291091185) (5.56166667,0.0318345619) (5.6,0.0345353476) (5.6591195,0.0386532802) (5.71823899,0.0427151042) (5.77735849,0.0467222797) (5.83647799,0.0506762075) (5.89559748,0.0545782329) (5.95471698,0.0584296482) (6.01383648,0.0622316954) (6.07295597,0.065985569) (6.13207547,0.069692418) (6.19119497,0.0733533487) (6.25031447,0.0769694262) (6.30943396,0.080541677) (6.36855346,0.0840710904) (6.42767296,0.0875586206) (6.48679245,0.0910051881) (6.54591195,0.0944116817) (6.60503145,0.0977789595) (6.66415094,0.101107851) (6.72327044,0.104399156) (6.78238994,0.107653652) (6.84150943,0.110872086) (6.90062893,0.114055186) (6.95974843,0.117203652) (7.01886792,0.120318165) (7.07798742,0.123399385) (7.13710692,0.12644795) (7.19622642,0.12946448) (7.25534591,0.132449575) (7.31446541,0.135403819) (7.37358491,0.138327778) (7.4327044,0.141222002) (7.4918239,0.144087025) (7.5509434,0.146923365) (7.61006289,0.149731527) (7.66918239,0.152512002) (7.72830189,0.155265266) (7.78742138,0.157991785) (7.84654088,0.160692009) (7.90566038,0.16336638) (7.96477987,0.166015325) (8.02389937,0.168639262) (8.08301887,0.171238598) (8.14213836,0.173813729) (8.20125786,0.176365043) (8.26037736,0.178892916) (8.31949686,0.181397716) (8.37861635,0.183879802) (8.43773585,0.186339525) (8.49685535,0.188777227) (8.55597484,0.191193243) (8.61509434,0.193587897) (8.67421384,0.195961511) (8.73333333,0.198314394) (8.79245283,0.200646852) (8.85157233,0.202959183) (8.91069182,0.205251677) (8.96981132,0.207524621) (9.02893082,0.209778292) (9.08805031,0.212012964) (9.14716981,0.214228903) (9.20628931,0.216426371) (9.26540881,0.218605624) (9.3245283,0.220766912) (9.3836478,0.222910481) (9.4427673,0.225036571) (9.50188679,0.227145417) (9.56100629,0.229237252) (9.62012579,0.231312301) (9.67924528,0.233370786) (9.73836478,0.235412925) (9.79748428,0.237438931) (9.85660377,0.239449013) (9.91572327,0.241443377) (9.97484277,0.243422224) (10.0339623,0.245385752) (10.0930818,0.247334154) (10.1522013,0.249267622) (10.2113208,0.251186342) (10.2704403,0.253090497) (10.3295597,0.254980268) (10.3886792,0.256855832) (10.4477987,0.258717363) (10.5069182,0.260565031) (10.5660377,0.262399004) (10.6251572,0.264219447) (10.6842767,0.266026522) (10.7433962,0.267820389) (10.8025157,0.269601203) (10.8616352,0.271369119) (10.9207547,0.273124288) (10.9798742,0.274866858) (11.0389937,0.276596976) (11.0981132,0.278314786) (11.1572327,0.280020429) (11.2163522,0.281714045) (11.2754717,0.283395769) (11.3345912,0.285065738) (11.3937107,0.286724083) (11.4528302,0.288370934) (11.5119497,0.290006421) (11.5710692,0.29163067) (11.6301887,0.293243804) (11.6893082,0.294845947) (11.7484277,0.296437219) (11.8075472,0.298017738) (11.8666667,0.299587621) (11.9257862,0.301146984) (11.9849057,0.302695939) (12.0440252,0.304234599) (12.1031447,0.305763073) (12.1622642,0.307281469) (12.2213836,0.308789895) (12.2805031,0.310288455) (12.3396226,0.311777253) (12.3987421,0.313256391) (12.4578616,0.314725969) (12.5169811,0.316186087) (12.5761006,0.317636842) (12.6352201,0.31907833) (12.6943396,0.320510647) (12.7534591,0.321933885) (12.8125786,0.323348138) (12.8716981,0.324753496) (12.9308176,0.326150049) (12.9899371,0.327537884) (13.0490566,0.32891709) (13.1081761,0.330287752) (13.1672956,0.331649955) (13.2264151,0.333003783) (13.2855346,0.334349318) (13.3446541,0.335686641) (13.4037736,0.337015833) (13.4628931,0.338336972) (13.5220126,0.339650138) (13.5811321,0.340955406) (13.6402516,0.342252853) (13.6993711,0.343542555) (13.7584906,0.344824584) (13.8176101,0.346099015) (13.8767296,0.347365919) (13.9358491,0.348625367) (13.9949686,0.349877429) (14.0540881,0.351122176) (14.1132075,0.352359675) (14.172327,0.353589993) (14.2314465,0.354813199) (14.290566,0.356029356) (14.3496855,0.357238532) (14.408805,0.358440788) (14.4679245,0.35963619) (14.527044,0.360824799) (14.5861635,0.362006677) (14.645283,0.363181886) (14.7044025,0.364350485) (14.763522,0.365512535) (14.8226415,0.366668093) (14.881761,0.367817218) (14.9408805,0.368959967) (15,0.370096398)};
\addlegendentry{$\delta=5$}
\addplot[annBlue,thick,dotted] coordinates {(1,-1.28881225) (1.03833333,-1.23558852) (1.07666667,-1.18574467) (1.115,-1.13895484) (1.15333333,-1.09493457) (1.19166667,-1.05343435) (1.23,-1.01423436) (1.26833333,-0.977140118) (1.30666667,-0.941978905) (1.345,-0.90859677) (1.38333333,-0.87685601) (1.42166667,-0.846633065) (1.46,-0.81781672) (1.49833333,-0.790306588) (1.53666667,-0.764011805) (1.575,-0.738849912) (1.61333333,-0.714745899) (1.65166667,-0.691631365) (1.69,-0.669443801) (1.72833333,-0.648125963) (1.76666667,-0.627625317) (1.805,-0.607893561) (1.84333333,-0.588886203) (1.88166667,-0.570562187) (1.92,-0.55288356) (1.95833333,-0.535815187) (1.99666667,-0.519324486) (2.035,-0.503381197) (2.07333333,-0.487957181) (2.11166667,-0.473026229) (2.15,-0.458563901) (2.18833333,-0.444547376) (2.22666667,-0.430955318) (2.265,-0.417767755) (2.30333333,-0.404965975) (2.34166667,-0.392532419) (2.38,-0.380450599) (2.41833333,-0.368705014) (2.45666667,-0.357281074) (2.495,-0.346165038) (2.53333333,-0.335343949) (2.57166667,-0.324805579) (2.61,-0.314538377) (2.64833333,-0.304531424) (2.68666667,-0.294774388) (2.725,-0.285257485) (2.76333333,-0.275971444) (2.80166667,-0.266907471) (2.84,-0.25805722) (2.87833333,-0.249412765) (2.91666667,-0.240966571) (2.955,-0.232711473) (2.99333333,-0.224640651) (3.03166667,-0.216747611) (3.07,-0.209026163) (3.10833333,-0.201470407) (3.14666667,-0.194074711) (3.185,-0.1868337) (3.22333333,-0.179742241) (3.26166667,-0.172795427) (3.3,-0.165988565) (3.33833333,-0.159317167) (3.37666667,-0.152776935) (3.415,-0.146363755) (3.45333333,-0.140073683) (3.49166667,-0.13390294) (3.53,-0.1278479) (3.56833333,-0.121905085) (3.60666667,-0.116071156) (3.645,-0.110342907) (3.68333333,-0.104717256) (3.72166667,-0.0991912428) (3.76,-0.0937620198) (3.79833333,-0.088426848) (3.83666667,-0.0831830919) (3.875,-0.078028214) (3.91333333,-0.0729597709) (3.95166667,-0.0679754084) (3.99,-0.0630728579) (4.02833333,-0.0582499322) (4.06666667,-0.0535045217) (4.105,-0.0488345917) (4.14333333,-0.0442381782) (4.18166667,-0.0397133854) (4.22,-0.0352583826) (4.25833333,-0.0308714012) (4.29666667,-0.0265507325) (4.335,-0.0222947248) (4.37333333,-0.0181017812) (4.41166667,-0.0139703574) (4.45,-0.00989895938) (4.48833333,-0.00588614131) (4.52666667,-0.0019305039) (4.565,0.00196930769) (4.60333333,0.00581460553) (4.64166667,0.00960666062) (4.68,0.0133467045) (4.71833333,0.0170359307) (4.75666667,0.0206754963) (4.795,0.0242665232) (4.83333333,0.0278100997) (4.87166667,0.0313072813) (4.91,0.0347590923) (4.94833333,0.0381665269) (4.98666667,0.0415305502) (5.025,0.0448520992) (5.06333333,0.048132084) (5.10166667,0.0513713885) (5.14,0.0545708718) (5.17833333,0.0577313684) (5.21666667,0.0608536897) (5.255,0.0639386244) (5.29333333,0.0669869393) (5.33166667,0.0699993804) (5.37,0.0729766731) (5.40833333,0.0759195234) (5.44666667,0.078828618) (5.485,0.0817046257) (5.52333333,0.084548197) (5.56166667,0.0873599656) (5.6,0.0901405484) (5.6591195,0.094369146) (5.71823899,0.0985271248) (5.77735849,0.102616534) (5.83647799,0.106639341) (5.89559748,0.11059744) (5.95471698,0.114492652) (6.01383648,0.118326729) (6.07295597,0.122101355) (6.13207547,0.125818156) (6.19119497,0.129478695) (6.25031447,0.133084477) (6.30943396,0.136636956) (6.36855346,0.140137531) (6.42767296,0.143587552) (6.48679245,0.146988321) (6.54591195,0.150341095) (6.60503145,0.153647088) (6.66415094,0.15690747) (6.72327044,0.160123374) (6.78238994,0.163295891) (6.84150943,0.166426079) (6.90062893,0.169514959) (6.95974843,0.172563517) (7.01886792,0.175572708) (7.07798742,0.178543457) (7.13710692,0.181476658) (7.19622642,0.184373174) (7.25534591,0.187233845) (7.31446541,0.190059481) (7.37358491,0.192850868) (7.4327044,0.195608766) (7.4918239,0.198333915) (7.5509434,0.201027029) (7.61006289,0.203688801) (7.66918239,0.206319904) (7.72830189,0.20892099) (7.78742138,0.211492691) (7.84654088,0.214035622) (7.90566038,0.216550378) (7.96477987,0.219037537) (8.02389937,0.221497663) (8.08301887,0.2239313) (8.14213836,0.226338978) (8.20125786,0.228721213) (8.26037736,0.231078504) (8.31949686,0.233411339) (8.37861635,0.23572019) (8.43773585,0.238005517) (8.49685535,0.240267767) (8.55597484,0.242507375) (8.61509434,0.244724765) (8.67421384,0.246920349) (8.73333333,0.249094527) (8.79245283,0.25124769) (8.85157233,0.253380218) (8.91069182,0.255492481) (8.96981132,0.257584839) (9.02893082,0.259657644) (9.08805031,0.261711238) (9.14716981,0.263745954) (9.20628931,0.265762117) (9.26540881,0.267760044) (9.3245283,0.269740043) (9.3836478,0.271702417) (9.4427673,0.27364746) (9.50188679,0.275575456) (9.56100629,0.277486687) (9.62012579,0.279381425) (9.67924528,0.281259937) (9.73836478,0.283122482) (9.79748428,0.284969315) (9.85660377,0.286800683) (9.91572327,0.288616828) (9.97484277,0.290417987) (10.0339623,0.29220439) (10.0930818,0.293976264) (10.1522013,0.295733828) (10.2113208,0.297477299) (10.2704403,0.299206887) (10.3295597,0.300922797) (10.3886792,0.302625232) (10.4477987,0.304314388) (10.5069182,0.305990458) (10.5660377,0.307653631) (10.6251572,0.30930409) (10.6842767,0.310942015) (10.7433962,0.312567585) (10.8025157,0.31418097) (10.8616352,0.315782341) (10.9207547,0.317371863) (10.9798742,0.318949697) (11.0389937,0.320516004) (11.0981132,0.322070937) (11.1572327,0.32361465) (11.2163522,0.325147292) (11.2754717,0.326669009) (11.3345912,0.328179943) (11.3937107,0.329680237) (11.4528302,0.331170026) (11.5119497,0.332649446) (11.5710692,0.33411863) (11.6301887,0.335577707) (11.6893082,0.337026804) (11.7484277,0.338466045) (11.8075472,0.339895554) (11.8666667,0.34131545) (11.9257862,0.342725852) (11.9849057,0.344126874) (12.0440252,0.345518629) (12.1031447,0.346901231) (12.1622642,0.348274786) (12.2213836,0.349639404) (12.2805031,0.350995189) (12.3396226,0.352342244) (12.3987421,0.353680672) (12.4578616,0.355010572) (12.5169811,0.356332041) (12.5761006,0.357645176) (12.6352201,0.358950073) (12.6943396,0.360246823) (12.7534591,0.361535517) (12.8125786,0.362816247) (12.8716981,0.364089099) (12.9308176,0.365354161) (12.9899371,0.366611518) (13.0490566,0.367861253) (13.1081761,0.369103449) (13.1672956,0.370338188) (13.2264151,0.371565548) (13.2855346,0.372785609) (13.3446541,0.373998447) (13.4037736,0.375204138) (13.4628931,0.376402757) (13.5220126,0.377594377) (13.5811321,0.378779071) (13.6402516,0.37995691) (13.6993711,0.381127963) (13.7584906,0.382292301) (13.8176101,0.383449989) (13.8767296,0.384601097) (13.9358491,0.385745688) (13.9949686,0.386883828) (14.0540881,0.388015581) (14.1132075,0.389141009) (14.172327,0.390260175) (14.2314465,0.39137314) (14.290566,0.392479962) (14.3496855,0.393580703) (14.408805,0.394675419) (14.4679245,0.395764168) (14.527044,0.396847008) (14.5861635,0.397923993) (14.645283,0.398995179) (14.7044025,0.400060619) (14.763522,0.401120368) (14.8226415,0.402174478) (14.881761,0.403223001) (14.9408805,0.404265989) (15,0.40530349)};
\addlegendentry{$\delta=10$}
\addplot[annBlue,thin] coordinates {(1,0) (15,0)};\addplot[annBlue,thin,densely dotted] coordinates {(5.6,-1) (5.6,.85)};
\end{axis}
\begin{axis}[at={(.51\textwidth,0)},anchor=south west,width=.46\textwidth,height=5.6cm,xlabel={$x$},ylabel={$\mathcal H_{5/2,9/10}\paren{x;c}$},xmin=3.8618,xmax=5.6,ymin=-.55,ymax=.65,xtick={4,4.8,5.6},legend pos=north west,legend style={draw=none,font=\scriptsize},tick label style={font=\scriptsize},label style={font=\scriptsize}]
\addplot[annBlue,thick,solid] coordinates {(3.86181049,0.255145239) (3.87422613,0.258172469) (3.88664177,0.261185272) (3.89905741,0.264183765) (3.91147305,0.267168062) (3.92388869,0.270138279) (3.93630432,0.273094528) (3.94871996,0.276036921) (3.9611356,0.278965569) (3.97355124,0.28188058) (3.98596688,0.284782062) (3.99838252,0.287670122) (4.01079816,0.290544865) (4.0232138,0.293406396) (4.03562944,0.296254817) (4.04804508,0.29909023) (4.06046072,0.301912736) (4.07287636,0.304722435) (4.085292,0.307519424) (4.09770764,0.310303801) (4.11012328,0.313075663) (4.12253891,0.315835104) (4.13495455,0.31858222) (4.14737019,0.321317102) (4.15978583,0.324039842) (4.17220147,0.326750533) (4.18461711,0.329449264) (4.19703275,0.332136125) (4.20944839,0.334811202) (4.22186403,0.337474584) (4.23427967,0.340126357) (4.24669531,0.342766606) (4.25911095,0.345395416) (4.27152659,0.34801287) (4.28394223,0.350619051) (4.29635787,0.353214041) (4.30877351,0.35579792) (4.32118914,0.358370769) (4.33360478,0.360932666) (4.34602042,0.363483691) (4.35843606,0.366023921) (4.3708517,0.368553433) (4.38326734,0.371072302) (4.39568298,0.373580605) (4.40809862,0.376078415) (4.42051426,0.378565806) (4.4329299,0.381042852) (4.44534554,0.383509624) (4.45776118,0.385966194) (4.47017682,0.388412633) (4.48259246,0.39084901) (4.4950081,0.393275396) (4.50742374,0.395691859) (4.51983937,0.398098466) (4.53225501,0.400495286) (4.54467065,0.402882384) (4.55708629,0.405259828) (4.56950193,0.407627682) (4.58191757,0.409986011) (4.59433321,0.412334879) (4.60674885,0.414674349) (4.61916449,0.417004485) (4.63158013,0.419325349) (4.64399577,0.421637002) (4.65641141,0.423939506) (4.66882705,0.42623292) (4.68124269,0.428517306) (4.69365833,0.430792722) (4.70607397,0.433059226) (4.7184896,0.435316878) (4.73090524,0.437565735) (4.74332088,0.439805854) (4.75573652,0.442037291) (4.76815216,0.444260104) (4.7805678,0.446474346) (4.79298344,0.448680074) (4.80539908,0.450877342) (4.81781472,0.453066204) (4.83023036,0.455246713) (4.842646,0.457418923) (4.85506164,0.459582886) (4.86747728,0.461738654) (4.87989292,0.463886279) (4.89230856,0.466025812) (4.9047242,0.468157304) (4.91713983,0.470280804) (4.92955547,0.472396363) (4.94197111,0.47450403) (4.95438675,0.476603855) (4.96680239,0.478695884) (4.97921803,0.480780167) (4.99163367,0.482856751) (5.00404931,0.484925684) (5.01646495,0.486987011) (5.02888059,0.48904078) (5.04129623,0.491087036) (5.05371187,0.493125825) (5.06612751,0.495157192) (5.07854315,0.497181182) (5.09095879,0.49919784) (5.10337443,0.501207209) (5.11579006,0.503209333) (5.1282057,0.505204255) (5.14062134,0.507192018) (5.15303698,0.509172665) (5.16545262,0.511146237) (5.17786826,0.513112778) (5.1902839,0.515072327) (5.20269954,0.517024927) (5.21511518,0.518970617) (5.22753082,0.520909439) (5.23994646,0.522841433) (5.2523621,0.524766638) (5.26477774,0.526685093) (5.27719338,0.528596839) (5.28960902,0.530501912) (5.30202466,0.532400353) (5.31444029,0.534292199) (5.32685593,0.536177488) (5.33927157,0.538056258) (5.35168721,0.539928545) (5.36410285,0.541794387) (5.37651849,0.54365382) (5.38893413,0.54550688) (5.40134977,0.547353604) (5.41376541,0.549194028) (5.42618105,0.551028186) (5.43859669,0.552856114) (5.45101233,0.554677847) (5.46342797,0.55649342) (5.47584361,0.558302866) (5.48825925,0.56010622) (5.50067489,0.561903516) (5.51309052,0.563694787) (5.52550616,0.565480066) (5.5379218,0.567259386) (5.55033744,0.569032781) (5.56275308,0.570800283) (5.57516872,0.572561923) (5.58758436,0.574317734) (5.6,0.576067748)};
\addlegendentry{$c=11/100$}
\addplot[annBlue,thick,dashed] coordinates {(3.86181049,-0.347857617) (3.87422613,-0.344557859) (3.88664177,-0.341270682) (3.89905741,-0.337995998) (3.91147305,-0.334733721) (3.92388869,-0.331483765) (3.93630432,-0.328246046) (3.94871996,-0.325020479) (3.9611356,-0.321806983) (3.97355124,-0.318605475) (3.98596688,-0.315415875) (3.99838252,-0.312238102) (4.01079816,-0.309072078) (4.0232138,-0.305917725) (4.03562944,-0.302774965) (4.04804508,-0.299643722) (4.06046072,-0.29652392) (4.07287636,-0.293415484) (4.085292,-0.290318342) (4.09770764,-0.287232418) (4.11012328,-0.284157642) (4.12253891,-0.281093941) (4.13495455,-0.278041246) (4.14737019,-0.274999485) (4.15978583,-0.271968589) (4.17220147,-0.26894849) (4.18461711,-0.26593912) (4.19703275,-0.262940412) (4.20944839,-0.259952299) (4.22186403,-0.256974716) (4.23427967,-0.254007597) (4.24669531,-0.251050878) (4.25911095,-0.248104495) (4.27152659,-0.245168385) (4.28394223,-0.242242485) (4.29635787,-0.239326734) (4.30877351,-0.236421071) (4.32118914,-0.233525434) (4.33360478,-0.230639763) (4.34602042,-0.227764) (4.35843606,-0.224898084) (4.3708517,-0.222041959) (4.38326734,-0.219195565) (4.39568298,-0.216358846) (4.40809862,-0.213531745) (4.42051426,-0.210714206) (4.4329299,-0.207906173) (4.44534554,-0.205107592) (4.45776118,-0.202318407) (4.47017682,-0.199538565) (4.48259246,-0.196768011) (4.4950081,-0.194006694) (4.50742374,-0.191254559) (4.51983937,-0.188511557) (4.53225501,-0.185777633) (4.54467065,-0.183052738) (4.55708629,-0.180336821) (4.56950193,-0.177629831) (4.58191757,-0.174931719) (4.59433321,-0.172242435) (4.60674885,-0.169561931) (4.61916449,-0.166890158) (4.63158013,-0.164227067) (4.64399577,-0.161572612) (4.65641141,-0.158926744) (4.66882705,-0.156289418) (4.68124269,-0.153660587) (4.69365833,-0.151040204) (4.70607397,-0.148428224) (4.7184896,-0.145824603) (4.73090524,-0.143229294) (4.74332088,-0.140642254) (4.75573652,-0.138063438) (4.76815216,-0.135492803) (4.7805678,-0.132930305) (4.79298344,-0.130375902) (4.80539908,-0.12782955) (4.81781472,-0.125291208) (4.83023036,-0.122760833) (4.842646,-0.120238384) (4.85506164,-0.117723819) (4.86747728,-0.115217098) (4.87989292,-0.11271818) (4.89230856,-0.110227025) (4.9047242,-0.107743592) (4.91713983,-0.105267843) (4.92955547,-0.102799737) (4.94197111,-0.100339235) (4.95438675,-0.0978863001) (4.96680239,-0.0954408923) (4.97921803,-0.093002974) (4.99163367,-0.0905725072) (5.00404931,-0.0881494545) (5.01646495,-0.0857337787) (5.02888059,-0.0833254427) (5.04129623,-0.0809244099) (5.05371187,-0.078530644) (5.06612751,-0.0761441087) (5.07854315,-0.0737647682) (5.09095879,-0.0713925871) (5.10337443,-0.0690275299) (5.11579006,-0.0666695616) (5.1282057,-0.0643186474) (5.14062134,-0.0619747529) (5.15303698,-0.0596378437) (5.16545262,-0.0573078858) (5.17786826,-0.0549848455) (5.1902839,-0.0526686892) (5.20269954,-0.0503593837) (5.21511518,-0.0480568959) (5.22753082,-0.045761193) (5.23994646,-0.0434722425) (5.2523621,-0.041190012) (5.26477774,-0.0389144695) (5.27719338,-0.036645583) (5.28960902,-0.034383321) (5.30202466,-0.032127652) (5.31444029,-0.0298785448) (5.32685593,-0.0276359685) (5.33927157,-0.0253998922) (5.35168721,-0.0231702856) (5.36410285,-0.0209471181) (5.37651849,-0.0187303598) (5.38893413,-0.0165199807) (5.40134977,-0.014315951) (5.41376541,-0.0121182414) (5.42618105,-0.00992682245) (5.43859669,-0.00774166511) (5.45101233,-0.00556274047) (5.46342797,-0.00339001983) (5.47584361,-0.00122347469) (5.48825925,0.000936923286) (5.50067489,0.00309120221) (5.51309052,0.00523939002) (5.52550616,0.00738151446) (5.5379218,0.00951760311) (5.55033744,0.0116476833) (5.56275308,0.0137717824) (5.57516872,0.0158899272) (5.58758436,0.0180021447) (5.6,0.0201084615)};
\addlegendentry{$c=1/3$}
\addplot[annBlue,thin] coordinates {(3.8618,0) (5.6,0)};
\end{axis}
\end{tikzpicture}
\caption{\textcolor{revision}{The explicit Neumann envelope. Left: three widths at fixed $u=9/10$ and $c=1/3$, divided by $1+\delta/2$ to compare their shapes; the vertical line is $x=28/5$. Right: at $\delta=5/2$, retaining the phase coefficient $c=11/100$ raises the envelope above its generic $c=1/3$ version on $[p(5/2),28/5]$. Validity of the smaller coefficient on the bridge and all positive margins are proved in the following subsections. These sampled curves supply no numerical step in the proof.}}
\label{fig:Neumann-H}
\end{figure}
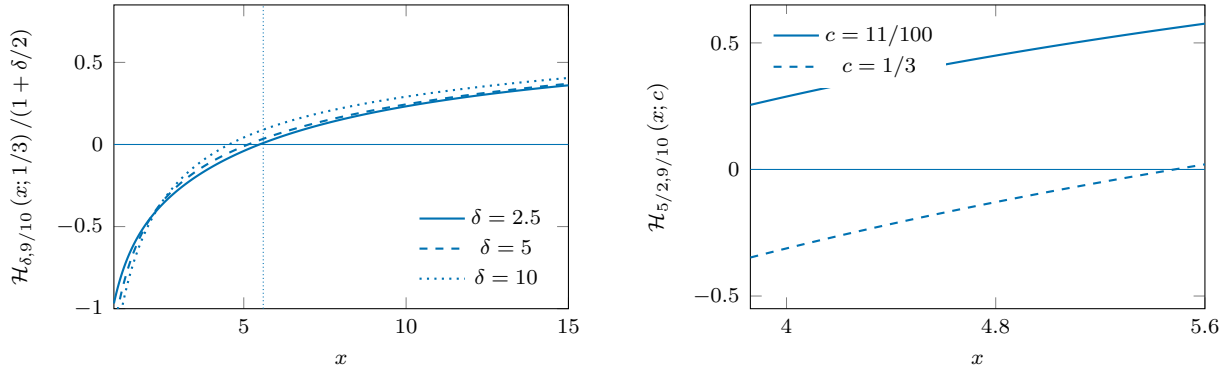}
\subsection{\revblue{The high-frequency tail for larger widths}}
Put $d_0=5/2$, $X=28/5$, $u=9/10$, and use the generic
rounding coefficient $c=1/3$. We prove directly that
\[
 \mathcal H_{\delta,9/10}\paren{X;1/3}>\frac7{400},
 \qquad \delta\ge\frac52.
\]
{Together with frequency monotonicity, this settles every $x\ge X$.
The larger threshold $X=28/5$ allows the turning cap to increase
from $1/4$ to $2/5$ while preserving a uniform positive margin
across all widths $\delta\ge5/2$.
Only the elementary bounds}
\[
 \frac{333}{106}<\pi<\frac{22}{7},\qquad
 I\paren{9/10}<\frac{111}{500}
\]
are needed.

Write $A=X/\pi$, $b=23/\paren{24\pi X}$,
$K=A/\paren{1+\delta}+b\delta$, and set
\[
 r\paren\delta=\frac2{3\sqrt K},\qquad
 s\paren\delta=\frac{\sqrt2}{3}
 \sqrt{\frac{\pi\delta\sqrt{\delta\paren{2+\delta}}}{X}},
 \qquad
 E\paren\delta=r\paren\delta+s\paren\delta
                +\frac{2\delta I\paren{9/10}}{X}.
\]
Then $\mathcal H_{\delta,9/10}\paren{X;1/3}
=1+\delta/2-9/50-E\paren\delta$.
At $\delta=5/2$,
\[
 K=\frac{13627}{6720\pi},\qquad
 s^2=\frac{25\pi\sqrt5}{168}.
\]
The rational bounds $\sqrt5<161/72$ and $\pi<22/7$ give
\[
 r\paren{5/2}<\frac{83}{100},\qquad
 s\paren{5/2}<\frac{10227}{10000}.
\]
These follow by squaring positive quantities; explicitly,
\[
 \frac{4\paren{22/7}\,6720}{9\cdot13627}
 <\paren{\frac{83}{100}}^2,\qquad
 \frac{25\paren{22/7}\paren{161/72}}{168}
 <\paren{\frac{10227}{10000}}^2.
\]
Consequently,
\[
 \begin{aligned}
 \mathcal H_{5/2,9/10}\paren{X;1/3}
 &>\frac94-\frac{83}{100}-\frac{10227}{10000}
       -\frac9{50}-\frac{111}{560}\\
 &=\frac{19}{1000}+\frac3{35000}>\frac{19}{1000}.
 \end{aligned}
\]

Differentiation yields
\[
 r'=\frac{A/\paren{1+\delta}^2-b}{3K^{3/2}},\qquad
 s'=\frac{\sqrt2}{3}\sqrt{\frac\pi X}
 \frac{\delta+3/2}{\delta^{1/4}\paren{\delta+2}^{3/4}}.
\]
{The derivative $r'$ decreases while it is positive and is
nonpositive after its single zero. Indeed, putting
$R=1+\delta$ and $z=bR^2/A$, the numerator governing $r''$,
after multiplication by the positive factor $R^4/A^2$, is
$-1-10z+4z/R+3z^2<0$ for $0<z<1$.
The derivative $s'$ decreases for all $\delta>0$, since
$g=\delta^{3/4}\paren{\delta+2}^{1/4}$ satisfies
$g''/g=-3/\paren{4\delta^2\paren{\delta+2}^2}<0$.
The calculation of $r''$ is the one displayed above, with $X$
in place of $X_0$. The derivative $r'$ changes sign at
$R=\sqrt{A/b}=X\sqrt{24/23}$. Before this point $z<1$,
so the displayed numerator of $r''$ is negative; after it the
numerator $A/R^2-b$ of $r'$ is nonpositive. Therefore, for
any $\delta_1>0$ and every $\delta\ge\delta_1$,}
{\[
 r'\paren\delta\le\max\set{r'\paren{\delta_1},0},
 \qquad s'\paren\delta\le s'\paren{\delta_1}.
\]}
{These are exactly the bounds needed to control the derivative
of the total error $E$ beyond the two chosen widths.}

For completeness, the required derivative bounds can all be
checked by the same rational inequalities, without evaluating a
transcendental function numerically. At the two widths,
\[
 \begin{array}{ccc}
 \delta&5/2&3\\ \hline
 K&13627/\paren{6720\pi}&2143/\paren{1120\pi}\\
 \sqrt K&>803/1000&>39/50\\
 r'&<59/1000&<1/25\\
 \paren{s'}^2&160\pi/\paren{1701\sqrt5}
                    &9\pi/\paren{56\sqrt{15}}\\
 s'&<364/1000&<362/1000
 \end{array}
\]
{To verify the third row, bound the numerator of $r'$ with
$\pi>333/106$ and its denominator with $\pi<22/7$ and the
preceding square-root bounds. For the last row, use
$\sqrt5>559/250$ and $\sqrt{15}>387/100$, and square.
In terms of $K=k/\pi$ and
$n=X/R^2-23/\paren{24X}$, the two core comparisons have
$r'=\sqrt\pi\,n/\paren{3k^{3/2}}$, where}
{\[
 \begin{array}{c|cc}
 \delta&5/2&3\\ \hline
 n&112/245-115/672&7/20-115/672.
 \end{array}
\]}
{Equivalently, each stated upper bound $r'<\rho$ follows from
$\paren{22/7}n^2<9\rho^2k^3$; all quantities on its two
sides are rational. The same squared comparison applies to
the exact formulas for $s'$ in the table.
Thus}
\[
 E'\paren{5/2}<\frac{59+364}{1000}+\frac{111}{1400}
              <\frac{503}{1000},\qquad
 E'\paren{3}<\frac1{25}+\frac{362}{1000}+\frac{111}{1400}
             <\frac12.
\]
On $\bracket{5/2,3}$, the high envelope can therefore decrease
by less than $\paren{3-5/2}\,3/1000=3/2000$.
For $\delta\ge3$ it increases, because the bounds $r'\paren\delta\le\max\set{r'\paren{3},0}$ and
$s'\paren\delta\le s'\paren{3}$ give $E'\paren\delta<1/2$.
Hence its value is everywhere larger than
$19/1000-3/2000=7/400$, as claimed.

\subsection{The intermediate-width bridge}
\begin{lemma}[The intermediate-width bridge]\label{lem:analytic-bridge}
For $5/2\le\delta\le10$ and
$p\paren\delta\le x\le28/5$,
\[
 C_+\paren{a_{\mathrm N}\paren{x,\delta}}<\frac{11}{100},
 \qquad
 \mathcal H_{\delta,9/10}\paren{x;11/100}>\frac{163}{2000}.
\]
\end{lemma}
\begin{proof}
We use the elementary rational bounds
$333/106<\pi<22/7$ and
\[
 p\paren{5/2}>\frac{77}{20},\qquad
 p\paren{5}>\frac{21}{5},\qquad
 p\paren{10}>\frac{22}{5}.
\]
The three latter inequalities follow by squaring, using
$p\paren{5/2}^2=1208/81$, $p\paren{5}^2=872/49$, and
$p\paren{10}^2=182/9$.

First we control the initial phase throughout the whole bridge.
Set $y=\delta+2$ and
$g\paren\delta=p\paren\delta-3\delta/\paren{8p\paren\delta}$.
Since
\[
 p^2=24-\frac{48}{y}+\frac{32}{y^2},\qquad
 p'=\frac{8\paren{3y-4}}{y^3p},\qquad
 p''=-\frac{48\paren{y-2}}{y^4p}-\frac{\paren{p'}^2}{p},
\]
we have
\[
 g''
 =p''\paren{1+\frac{3\delta}{8p^2}}
  +\frac{3p'}{4p^2}-\frac{3\delta\paren{p'}^2}{4p^3}
 <p''+\frac{3p'}{4p^2}<0.
\]
{For clarity, differentiating once first gives}
{\[
 g'=p'\paren{1+\frac{3\delta}{8p^2}}-\frac3{8p},
\]}
{and differentiating again gives the displayed expression for
$g''$. The extra terms containing $\delta$ are negative because
$p''<0$, $p'>0$, and $\delta>0$.
Indeed, on $9/2\le y\le12$,}
\[
 \frac{y\paren{3y-4}}{8\paren{y-2}}
 =\frac{3y+2+4/\paren{y-2}}8
 \le\frac{99}{20}<9<p^2,
\]
{which makes the last upper bound strictly negative.
Indeed,}
{\[
 p''+\frac{3p'}{4p^2}
 <-\frac{48\paren{y-2}}{y^4p}
       +\frac{6\paren{3y-4}}{y^3p^3}<0
\]}
{is precisely the comparison
$p^2>y\paren{3y-4}/\paren{8\paren{y-2}}$ proved above.
The concavity of $g$ therefore reduces its lower bound to its two
endpoints. The function $z-3\delta/\paren{8z}$ increases with $z>0$;
hence}
\[
 \frac{g\paren{5/2}}\pi
 >\frac{19439}{16940}>\frac{111}{100},\qquad
 \frac{g\paren{10}}\pi
 >\frac{10927}{9680}>\frac{111}{100}.
\]
The unwrapped initial phase increases with $x$. Its upper bound is
attained at $x=28/5$ and $\delta=5/2$, so
\[
 \frac{x-3\delta/\paren{8x}}\pi
 \le\frac{12169}{2240\pi}
 <\frac{12169\cdot106}{2240\cdot333}<\frac{173}{100}.
\]
Consequently its fractional part stays in $\paren{11/100,73/100}$.
The half-period monotonicity of $C_+$ reduces its bound to these
two endpoints. Its explicit formulas give
\[
 C_+\paren{11/100}
 <\frac12\paren{\frac{39}{100}}^2
       \frac{21/10}{\paren{121/100}^2}
 =\frac{31941}{292820}<\frac{11}{100},
\]
where we used
$\sqrt2<3/2$, $3/4<\sqrt{61/100}<4/5$, and
$23/50<\sqrt{22/100}<1/2$. Also,
\[
 C_+\paren{73/100}
 =\frac{2\sqrt2}{3}\paren{\frac{23}{100}}^{3/2}
 <\frac{20}{21}\frac{23}{100}\frac{12}{25}
 =\frac{92}{875}<\frac{11}{100}.
\]

It remains to show positivity of the envelope. Denote its core and
outer rounding errors by $r=\paren{11/50}/\sqrt{K_\delta\paren{p}}$
and $s$. The arithmetic--geometric mean inequality implies
\[
 K_\delta\paren{p}
 \ge2\sqrt{\frac{23\delta}{24\pi^2\paren{1+\delta}}}
 >2\sqrt{\frac{5635}{81312}}>\frac{13}{25},
 \qquad r<\frac{31}{100}.
\]
{Here $\delta/\paren{1+\delta}\ge5/7$ and $\pi<22/7$
give the middle lower bound. The last two comparisons reduce to}
{\[
 \frac{4\cdot5635}{81312}>\paren{\frac{13}{25}}^2,
 \qquad
 \frac{\paren{11/50}^2}{13/25}
 =\frac{121}{1300}<\paren{\frac{31}{100}}^2.
\]}
Since $\sqrt{\delta\paren{\delta+2}}<\delta+1$,
\[
 s<\sqrt{\frac{2\pi}{9p}}\paren{\delta+\frac12}.
\]
{\color{editorial}Using $I\paren{9/10}<111/500$ from
Lemma~\ref{lem:endpoint-estimates}, on $5/2\le\delta\le5$ we obtain}
\[
 s<\frac{427}{1000}\paren{\delta+\frac12},\qquad
 \frac{2\delta I\paren{9/10}}p<\frac{116}{1000}\delta,
\]
and hence
\[
 \mathcal H_{\delta,9/10}\paren{p;11/100}
 >\frac{593}{2000}-\frac{43}{1000}\delta
 \ge\frac{163}{2000}.
\]
On $5\le\delta\le10$, the sharper lower bound $p>21/5$ gives
\[
 s<\frac{408}{1000}\paren{\delta+\frac12},\qquad
 \frac{2\delta I\paren{9/10}}p<\frac{106}{1000}\delta,
\]
and therefore
\[
 \mathcal H_{\delta,9/10}\paren{p;11/100}
 >\frac{153}{500}-\frac7{500}\delta
 \ge\frac{83}{500}>\frac{163}{2000}.
\]
Each constant comparison just used is a comparison of positive
rational numbers after squaring; for example the two outer
coefficients follow from
$880/4851<\paren{427/1000}^2$ and
$220/1323<\paren{408/1000}^2$.
Frequency monotonicity of the envelope extends these lower bounds
from $x=p\paren\delta$ to the full bridge. Its proof applies already
at $\delta=5/2$, because
$\delta\sqrt{\delta\paren{\delta+2}}/\paren{1+\delta}>2$
there and this expression increases with $\delta$.
\end{proof}

\subsection{Unbounded widths and completion of the proof}
For widths $\delta\ge10$, take $u=3/4$ and $x_B=22/5$.
The power threshold satisfies $p\paren{10}^2=182/9>x_B^2$.
\textcolor{editorial}{By Lemma~\ref{lem:endpoint-estimates}, at
$\delta=10$ the core, outer, and integrated phase errors in}
$\mathcal H_{10,3/4}\paren{x_B;1/3}$ are bounded by
$737/1000$, $417/100$, and $587/1000$, respectively. Thus
\[
 \mathcal H_{10,3/4}\paren{x_B;1/3}
 >6-\frac{737}{1000}-\frac{417}{100}
           -\frac{587}{1000}-\frac9{20}
 =\frac7{125}>\frac1{20}.
\]
{\color{editorial}At this fixed frequency,
\[
 \partial_\delta K_\delta\paren{x_B}
 =\frac1\pi\paren{\frac{23}{24x_B}
                         -\frac{x_B}{\paren{1+\delta}^2}}>0
 \qquad(\delta\ge10),
\]
because $x_B^2<23\cdot11^2/24$, so the core error decreases with
width. Lemma~\ref{lem:endpoint-estimates} gives
$\partial_\delta s_{10}\paren{x_B}<2/5$
and $2I\paren{3/4}/x_B<3/50$. Since the outer-error derivative
decreases with width, these bounds imply
\[
 \partial_\delta\mathcal H_{\delta,3/4}\paren{x_B;1/3}
 >\frac12-\frac25-\frac3{50}=\frac1{25}.
\]
Frequency monotonicity then extends the positive bound to
$x\ge x_B$. The power trial covers $x\le p\paren\delta$,
and $x_B<p\paren{10}\le p\paren\delta$ ensures overlap.}

{Figure~\ref{fig:frequency-cover} collects the overlaps in the
frequency variable $x=\delta t$. Every junction is a closed
endpoint of one of the two adjacent estimates, and the
large-width range has a nonempty overlap because
$x_B<p\paren{10}\le p\paren\delta$.}
{\begin{figure}[htbp]
\centering
\begin{tikzpicture}[x=1.25cm,y=1cm,font=\small]
 \node[left] at (-.15,2.18) {$0<\delta\le5/2$};
 \node[left] at (-.15,.93) {$5/2\le\delta\le10$};
 \node[left] at (-.15,-.32) {$\delta\ge10$};
 \fill[annBlue,opacity=.2] (0,2) rectangle (3.75,2.36);
 \fill[annOrange,opacity=.25] (3.75,2) rectangle (5.4,2.36);
 \fill[gray,opacity=.2] (5.4,2) rectangle (6.35,2.36);
 \node at (1.85,2.18) {Constant trial};
 \node at (4.57,2.18) {Phase};
 \node at (5.86,2.18) {Tail};
 \draw (3.75,1.92)--(3.75,2.42)
       node[above] {$x_0=15/4$};
 \draw (5.4,1.92)--(5.4,2.42)
       node[above] {$X_0=27/5$};
 \fill[annBlue,opacity=.2] (0,.75) rectangle (4.15,1.11);
 \fill[annOrange,opacity=.25] (4.15,.75) rectangle (5.6,1.11);
 \fill[gray,opacity=.2] (5.6,.75) rectangle (6.35,1.11);
 \node at (1.85,.93) {Power trial};
 \node at (4.88,.93) {Bridge};
 \node at (5.97,.93) {Tail};
 \draw (4.15,.67)--(4.15,1.17)
       node[above] {$p\paren\delta$};
 \draw (5.6,.67)--(5.6,1.17)
       node[above] {$X=28/5$};
 \fill[annBlue,opacity=.2] (0,-.5) rectangle (4.7,-.14);
 \fill[gray,opacity=.2] (4.4,-.5) rectangle (6.35,-.14);
 \node at (1.85,-.32) {Power trial};
 \node at (5.46,-.32) {Phase};
 \draw (4.4,-.58)--(4.4,-.08)
       node[above] {$x_B=22/5$};
 \draw (4.7,-.58)--(4.7,-.08)
       node[below,yshift=-.45cm] {$p\paren\delta>x_B$};
 \foreach \yy in {2,.75,-.5}{
   \draw[->,thick] (0,\yy)--(6.65,\yy);
   \node[right] at (6.65,\yy) {$x$};
 }
\end{tikzpicture}
\caption{Frequency coverage for the three width ranges.
The positions of the variable thresholds $p\paren\delta$ are
schematic. The estimates overlap at each junction, and the last
row has an interval covered by both the power trial and the
generic phase bound.}
\label{fig:frequency-cover}
\end{figure}
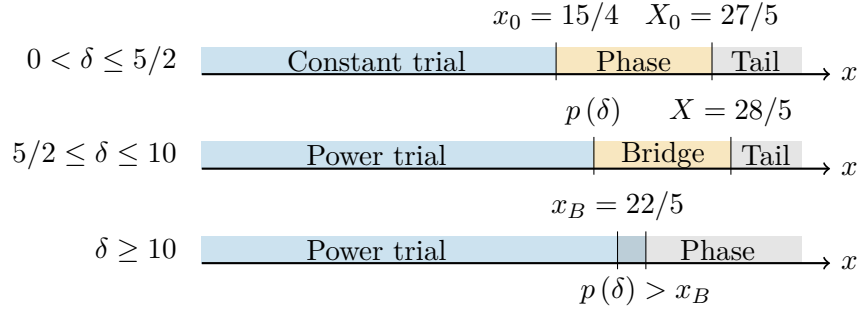}

{For $5/2\le\delta\le10$, the power trial covers the low frequencies,
Lemma~\ref{lem:analytic-bridge} covers the interval up to $28/5$,
and the tail estimate covers the rest. Together with the already
proved range $0<\delta\le5/2$, these overlapping intervals cover
every $\delta>0$ and $t>0$. This completes the
proof of Proposition~\ref{prop:continuous}.}

\section{\revblue{Positive Neumann gaps and optimal cutoff selection}}
{\label{sec:neumann-quantitative}}
{The qualitative argument in Section~\ref{sec:neumann} proves
$M(t)>Dt^2/4$, which already implies the strict Neumann P\'olya
inequality through Proposition~\ref{prop:angular-transfer}. We now
retain the positive difference from the area term. There are two
levels of conclusion. A single explicit positive lower bound for
this difference divided by $t$, uniform in $t>0$, proves
Theorem~\ref{thm:main}; retaining
its frequency dependence gives a stronger bound at each individual
index. Thus the optimizations below refine the same proof and are
not additional hypotheses needed for the qualitative inequality.}

{The improvement comes from three specific pieces of information.
The trial family contains both earlier radial trials; the turning
cap is integrated in the order variable instead of frozen at its
largest value on a band; and the cutoff is chosen to minimize the
remaining loss. Each change can only improve a lower bound for the
count. The exact action identity \eqref{eq:action}, and hence the
area coefficient, remains the same. The limit of the resulting
positive envelope describes its high-frequency strength, whereas
the overlapping estimates of Section~\ref{sec:neumann} are what
guarantee positivity at every frequency.}

{Continue to write}
{\[
 R=1+\delta>1,\qquad D=R^2-1,\qquad x=\delta t,
 \qquad p=\pi/\beta,\qquad b=\frac{R+1}{2}.
\]}
{Here $Dt^2/4$ is the area term, and $bt$ is the circular-arc
boundary term in the continuous count $M$. Indeed, multiplication
by the angular-transfer factor $1/(2p)$ gives
$bt/(2p)=\beta(R+1)t/(4\pi)$, corresponding to the two arcs of
total length $\beta(R+1)$. The radial sides enter the discrete
two-term law in Section~\ref{sec:boundary-optimality}. Put}
{\[
 K_\delta\left(x\right)=\frac{x}{\pi R}+\frac{23\delta}{24\pi x},\qquad
 s_\delta\left(t\right)=\frac{\sqrt2}{3}\sqrt{\frac{\pi\sqrt D}{t}},\qquad
 a\left(t\right)=\left\{\frac{\delta t}{\pi}-\frac{3}{8\pi t}\right\}.
\]}
{The function $C_+$ is defined in Lemma~\ref{lem:all-phase-rounding}. Its endpoint values agree, so $C_+\left(a\left(t\right)\right)$ is continuous across integer phase crossings; $0\le C_+\le1/3$.}

\subsection{\revblue{Optimized power trial functions}}
{For $\alpha\ge0$, define}
{\[
 \rho_R\left(\alpha\right)=
 \begin{cases}
 \displaystyle\frac{\alpha+1}{\alpha}
       \frac{R^{2\alpha}-1}{R^{2\alpha+2}-1},&\alpha>0,\\[5pt]
 \displaystyle\frac{2\log R}{D},&\alpha=0.
 \end{cases}
\]}
{This is the continuous positive function
$\int_1^R s^{2\alpha-1}\dd s/\int_1^R s^{2\alpha+1}\dd s$;
in particular $R^{-2}<\rho_R\left(\alpha\right)<1$. Its value at
zero follows by taking the limit in the integrals. Its monotonicity
also has a direct explanation: with probability density
$s^{2\alpha+1}/\int_1^R s^{2\alpha+1}\dd s$, differentiation gives
\[
 \rho_R'(\alpha)=2\operatorname{Cov}_\alpha(s^{-2},\log s)<0.
\]
The covariance is negative because its double-integral formula
has the factor $(s^{-2}-z^{-2})(\log s-\log z)<0$ for $s\ne z$.
Increasing $\alpha$ therefore puts more mass near the outer circle,
where the angular cost $s^{-2}$ is smaller. The radial derivative
has the competing cost $\alpha^2$. More explicitly, the trial
function $s^\alpha$ at angular order $v$ gives}
{\[
 \lambda_{v,0}\le\left(\alpha^2+v^2\right)\rho_R\left(\alpha\right).
\]}
{Consequently the strongest ground-mode order interval provided by
these trial functions is}
{\begin{equation}\label{eq:Nq-trial-opt}
 V_R\left(t\right)=\max_{0\le\alpha\le Rt}
       \sqrt{\left(t^2/\rho_R\left(\alpha\right)-\alpha^2\right)_+},
 \qquad M\left(t\right)>2V_R\left(t\right).
\end{equation}}
{The maximum exists. Values $\alpha>Rt$ have nonpositive radicand
and cannot improve it, while $\alpha=0$ already gives $V_R\left(t\right)>0$.
For numerical evaluation, it is useful to maximize the smooth
radicand $H_t(\alpha)=t^2/\rho_R(\alpha)-\alpha^2$ before taking
its positive square root. Since
\[
 H_t'(\alpha)=-\frac{t^2\rho_R'(\alpha)}{\rho_R(\alpha)^2}
                    -2\alpha,
 \qquad H_t'(0)>0,\qquad H_t(Rt)<0<H_t(0),
\]
every global maximizer is an interior point of $(0,Rt)$ with
positive radicand and satisfies
\[
 2\alpha=-\frac{t^2\rho_R'(\alpha)}{\rho_R(\alpha)^2}.
\]
This equation is a stationarity condition; a uniqueness claim for
this trial optimization is not required. All its maximizing
points give the same certified interval $[0,V_R(t)]$.
At a maximizing $\alpha$, the trial quotient at $v=V_R\left(t\right)$ is
$t^2$.  Its value strictly exceeds the first radial eigenvalue:
for $\alpha>0$ the trial function fails the Neumann endpoint
conditions; for $\alpha=0$ and $v>0$ a constant function fails
the radial eigenvalue equation.  Continuity in $v$ therefore
extends the admissible interval strictly beyond $V_R\left(t\right)$ and
proves the strict inequality in \eqref{eq:Nq-trial-opt}.
The choices $\alpha=0$ and $\alpha=v$ include the constant and
power trials used in the preceding proof.}

{We record a closed improvement of the latter trial, which will
ensure a positive quantitative remainder at its endpoint.  Define}
{\[
 v_*\left(t\right)=\frac{\sqrt{1+2R^2t^2}-1}{2},\qquad
 \vartheta\left(t\right)=\frac{1-R^{-2v_*\left(t\right)}}{1-R^{-2v_*\left(t\right)-2}},\qquad
 \widehat V_R\left(t\right)=\sqrt{v_*\left(t\right)^2+t^2\left(1-\vartheta\left(t\right)\right)}.
\]}
{The definition of $v_*$ is equivalent to
$2v_*(v_*+1)=R^2t^2$. Substituting $\alpha=v=v_*$ in the trial
quotient gives
\[
 2v_*^2\rho_R(v_*)
 =\frac{2v_*(v_*+1)}{R^2}
       \frac{1-R^{-2v_*}}{1-R^{-2v_*-2}}
 =t^2\vartheta<t^2.
\]
Thus the older power threshold leaves an explicitly measurable
surplus. Also
\[
 1-\vartheta
 =\frac{(1-R^{-2})R^{-2v_*}}{1-R^{-2v_*-2}}>0.
\]
The order comparison
$\lambda_{v,0}-\lambda_{v_*,0}\le v^2-v_*^2$ shows that
$M\left(t\right)>2\widehat V_R\left(t\right)$.  The optimized trial interval is at
least as large: the quotient for the fixed trial $s^{v_*}$ at
$v=\widehat V_R\left(t\right)$ is}
{\[
 t^2\vartheta+
 \left(\widehat V_R\left(t\right)^2-v_*^2\right)\rho_R\left(v_*\right)\le t^2,
\]}
{so $V_R\left(t\right)\ge\widehat V_R\left(t\right)>v_*\left(t\right)$.}

\subsection{\revblue{An elementary cap and the unique optimal cutoff}}
{Keep the order-dependent clipped phase cap}
{\[
 c\left(v\right)=\frac1\pi\inf_{X>v}\left(A_v\left(X\right)+E_v\left(X\right)\right),\qquad v\ge0.
\]}
{The real-order clipped enclosure of
\cite[Lemma~1.8 and Theorem~1.10]{Phase}, already used in
Section~\ref{sec:neumann}, holds for all nonnegative orders and gives}
{\[
 0\le\frac{\varphi_v\left(X\right)-A_v\left(X\right)-\pi/4}{\pi}<c\left(v\right),
 \qquad X\ge v\ge0.
\]}
{The function $c$ is continuous and nonincreasing.  It has the
following elementary evaluation.  If $w_v$ is the unique positive
root of}
{\begin{equation}\label{eq:Nq-cap-cubic}
 8w^3-3w^2-10v^2w-7v^4=0,
\end{equation}}
{then}
{\[
c\left(v\right)=\frac1\pi\left[
  \sqrt{w_v}-v\arctan\frac{\sqrt{w_v}}v
  +\frac{7v^2}{24w_v^{3/2}}+\frac{3}{8\sqrt{w_v}}\right],
\]}
{where the product involving $v$ is zero at $v=0$. At the
turning endpoint $X=v>0$, the lower phase bound follows by
continuity and the strict upper bound follows from
$\varphi_v(v)<\varphi_v(X_v^*)<\pi/4+\pi c(v)$.
At $X=v=0$ the chosen phase limit gives the normalized error
$1/4<c(0)$. In particular
$w_0=3/8$ and $c\left(0\right)=\sqrt{3/2}/\pi<2/5$.
To verify the evaluation, put $w=X^2-v^2$. In this variable the
quantity being minimized is
\[
 F_v(w)=\sqrt w-v\arctan(\sqrt w/v)
               +\frac{7v^2}{24w^{3/2}}+\frac3{8\sqrt w},
 \qquad
 F_v'(w)=\frac{\sqrt w}{2(w+v^2)}
                   -\frac{3w+7v^2}{16w^{5/2}}.
\]
Its critical-point equation is
$8w^3=\left(w+v^2\right)\left(3w+7v^2\right)$, which is
\eqref{eq:Nq-cap-cubic}. Dividing by $w^3$ gives the particularly
simple uniqueness test
\[
 8=\frac3w+\frac{10v^2}{w^2}+\frac{7v^4}{w^3},\qquad
 \frac{\dd}{\dd w}\left(\frac3w+\frac{10v^2}{w^2}
                         +\frac{7v^4}{w^3}\right)
 =-\frac3{w^2}-\frac{20v^2}{w^3}-\frac{21v^4}{w^4}<0.
\]
The right-hand side decreases continuously from infinity to zero
on $(0,+\infty)$, so it crosses $8$ exactly once. This includes
$v=0$, when the root is $3/8$. Since $F_v$ diverges at both
endpoints, its unique critical point is its unique global minimum.
The nonzero derivative in the divided equation also implies
continuous dependence of $w_v$, and hence of $c(v)$, on $v\ge0$.
There is also a direct check of the order monotonicity.  At the
minimizer put $r_v=\sqrt{w_v}/v$ for $v>0$.  Differentiating
the minimum with respect to $v$ and using the critical-point
equation gives}
{\[
 \pi c'\left(v\right)=-\arctan r_v+
 \frac{r_v\left(23r_v^2+21\right)}{3\left(1+r_v^2\right)\left(3r_v^2+7\right)}<0.
\]}
{Indeed, the difference between the arctangent and the rational
term vanishes at zero and has derivative
$48r^4\left(6r^2+7\right)/\left(9r^4+30r^2+21\right)^2>0$ for $r>0$.}

{One further derivative identity will explain the motion of the
optimal cutoff. For $v>0$, put $e(r)=7/(24r^3)+3/(8r)$.
The substitution $w=v^2r^2$ gives
\[
 \pi v c(v)=\min_{r>0}\{v^2(r-\arctan r)+e(r)\}.
\]
Differentiating at its unique minimizer yields
\[
 c(v)+vc'(v)=\frac{2v}{\pi}(r_v-\arctan r_v)>0,
 \qquad v>0.
\]
Thus $c$ decreases, but $vc(v)$ increases. These are different
monotonicity statements and both will be used below.}

{Write}
{\[
 I\left(u\right)=\frac{7u/\sqrt{1-u^2}+2\arcsin u}{24\pi},
 \quad 0\le u<1,\qquad
 q_0=1+\frac{2\sqrt{3/2}}\pi,\qquad
 t_c=\frac3{4\pi q_0}.
\]}
{The order-dependent cap may be integrated on the turning interval
instead of replaced by its value at the left endpoint.  The
resulting cutoff loss is}
{\begin{equation}\label{eq:Nq-cutoff-loss}
 J\left(t,u\right)=t\left(1-u\right)+2\int_{ut}^{t}c\left(v\right)\dd v+2I\left(u\right),
 \qquad j\left(t\right)=\min_{0\le u<1}J\left(t,u\right).
\end{equation}}
{The three summands have separate roles. The length $t(1-u)$
is the loss of the favorable core mean on the turning band;
$2\int_{ut}^tc(v)\dd v$ is its phase-cap loss, including the
factor two in $M$; and $2I(u)$ is the integrated raw phase error
on the core. Since $c$ decreases,
\[
 2\int_{ut}^tc(v)\dd v\le2t(1-u)c(ut).
\]
Consequently this loss improves the corresponding frozen-cap
loss at each fixed cutoff, before any optimization is performed.}

{\begin{lemma}\label{lem:Nq-cutoff-opt}
For every $t>0$ the minimum in \eqref{eq:Nq-cutoff-loss} is
attained at a unique cutoff $u_*\left(t\right)$.  The minimizing cutoff is $u_*\left(t\right)=0$ for
$0<t\le t_c$.  For $t>t_c$, $u_*\left(t\right)$ is the unique root in $\left(0,1\right)$ of
\[
\frac{9-2u^2}{12\pi\left(1-u^2\right)^{3/2}}
       =t\left(1+2c\left(ut\right)\right).
\]
Both $u_*$ and $j$ are continuous. On $(t_c,+\infty)$ the cutoff
$u_*(t)$ is strictly increasing, and $u_*(t)\to1$ as $t\to+\infty$.
\end{lemma}}
\begin{proof}
{The loss tends to infinity as $u\uparrow1$, and}
{\[
 \begin{aligned}
\partial_uJ\left(t,u\right)&=2I'\left(u\right)-t\left(1+2c\left(ut\right)\right),\\
 I'\left(u\right)&=\frac{9-2u^2}{24\pi\left(1-u^2\right)^{3/2}},\\
 I''\left(u\right)&=\frac{u\left(23-2u^2\right)}{24\pi\left(1-u^2\right)^{5/2}}.
\end{aligned}
\]}
{For $u>0$ the second derivative is
\[
 \partial_u^2J(t,u)=2I''(u)-2t^2c'(ut)>0.
\]
Equivalently, $I'$ is strictly increasing on $[0,1)$ and
$c(ut)$ is nonincreasing, so $\partial_uJ$ is strictly increasing
on the whole interval, including its endpoint. Its initial value is
\[
 \partial_uJ(t,0)=\frac3{4\pi}-q_0t.
\]
If $t\le t_c$, this value is nonnegative and $J$ increases to
the right of zero, giving the boundary minimum $u_*=0$.
If $t>t_c$, this value is negative, whereas
$\partial_uJ(t,u)\to+\infty$ as $u\uparrow1$. There is then
exactly one zero, and the derivative changes from negative to
positive there. This proves existence and uniqueness in both
branches, without a numerical minimization argument.

For continuity near a fixed positive $t$, the blow-up of $I(u)$
keeps all minimizing cutoffs in a common compact subinterval of
$[0,1)$. The strict root characterization gives continuity away
from $t_c$. As $t\downarrow t_c$, any positive limit of the root
would contradict $\partial_uJ(t_c,u)>0$ for $u>0$, so the root
tends to zero. Substitution in $J$ proves continuity of $j$ as well.

Finally put $q(v)=1+2c(v)$ and
$F(t,u)=2I'(u)-tq(ut)$. At an interior root,
\[
 F_u=2I''(u)-2t^2c'(ut)>0,\qquad
 F_t=-\{1+2c(ut)+2ut c'(ut)\}<0.
\]
The last sign follows from the identity for $c+vc'$ above.
Implicit differentiation therefore gives $u_*'(t)=-F_t/F_u>0$.
If $u_*$ stayed bounded away from $1$ along an unbounded sequence
of $t$, the left side $2I'(u_*)$ of its root equation would stay
bounded, while the right side would be at least $t$. This proves
$u_*(t)\to1$.}
\end{proof}

\begin{figure}[htbp]
\centering
\begin{tikzpicture}
\begin{axis}[width=.46\textwidth,height=5.0cm,xmin=0,xmax=20,
 ymin=.18,ymax=.41,xlabel={$v$},ylabel={$c(v)$},
 tick label style={font=\scriptsize},label style={font=\small}]
\addplot[annBlue,very thick] coordinates {(0.00000000,0.38984840) (0.02500000,0.37792428) (0.05000000,0.36713238) (0.07500000,0.35741781) (0.10000000,0.34869980) (0.15000000,0.33386953) (0.20000000,0.32187516) (0.30000000,0.30389239) (0.40000000,0.29113000) (0.50000000,0.28159583) (0.70000000,0.26822191) (1.00000000,0.25568064) (1.50000000,0.24361782) (2.00000000,0.23642393) (3.00000000,0.22801591) (4.00000000,0.22312838) (5.00000000,0.21987141) (7.00000000,0.21572346) (10.00000000,0.21217486) (15.00000000,0.20900032) (20.00000000,0.20719763)};
\addplot[black,dashed,domain=0:20] {sqrt(14)/(6*pi)};
\node[anchor=north east,font=\scriptsize] at (axis cs:19.6,.388)
 {$c(0)=\sqrt{3/2}/\pi$};
\node[anchor=south east,font=\scriptsize] at (axis cs:19.6,.201)
 {$c(+\infty)=\sqrt{14}/(6\pi)$};
\end{axis}
\begin{semilogxaxis}[at={(.51\textwidth,0)},anchor=south west,
 width=.46\textwidth,height=5.0cm,xmin=.03,xmax=20,ymin=-.03,ymax=1.03,
 xlabel={$t$},ylabel={$u_*(t)$},
 xtick={.03,0.134142183,1,16},xticklabels={$0.03$,$t_c$,$1$,$16$},
 tick label style={font=\scriptsize},label style={font=\small}]
\addplot[annOrange,very thick] coordinates {(0.030000000,0.00000000) (0.032920226,0.00000000) (0.036124709,0.00000000) (0.039641120,0.00000000) (0.043499821,0.00000000) (0.047734131,0.00000000) (0.052380613,0.00000000) (0.057479387,0.00000000) (0.063074481,0.00000000) (0.069214205,0.00000000) (0.075951576,0.00000000) (0.083344768,0.00000000) (0.091457620,0.00000000) (0.100360184,0.00000000) (0.110129331,0.00000000) (0.120849416,0.00000000) (0.132613003,0.00000000) (0.134142183,0.00000000) (0.145521668,0.22111712) (0.159686873,0.32877770) (0.175230932,0.40516293) (0.192288063,0.46545432) (0.211005549,0.51540528) (0.231545012,0.55797674) (0.254083805,0.59494609) (0.278816542,0.62748572) (0.305956786,0.65641962) (0.335738885,0.68235394) (0.368420000,0.70575004) (0.404282322,0.72696858) (0.443635513,0.74629758) (0.486819379,0.76397117) (0.534206799,0.78018253) (0.586206952,0.79509318) (0.643268845,0.80883969) (0.705885193,0.82153874) (0.774596669,0.83329091) (0.849996581,0.84418364) (0.932735985,0.85429348) (1.023529314,0.86368792) (1.123160545,0.87242679) (1.232489966,0.88056348) (1.352461608,0.88814578) (1.484111394,0.89521673) (1.628576084,0.90181520) (1.787103092,0.90797645) (1.961061256,0.91373254) (2.151952659,0.91911273) (2.361425597,0.92414378) (2.591288812,0.92885023) (2.843527112,0.93325460) (3.120318506,0.93737766) (3.424053015,0.94123852) (3.757353304,0.94485487) (4.123097332,0.94824303) (4.524443202,0.95141816) (4.964856426,0.95439426) (5.448139855,0.95718437) (5.978466512,0.95980057) (6.560415626,0.96225413) (7.199012171,0.96455549) (7.899770257,0.96671442) (8.668740744,0.96873999) (9.512563484,0.97064067) (10.438524661,0.97242436) (11.454619702,0.97409845) (12.569622315,0.97566981) (13.793160250,0.97714489) (15.135798428,0.97852970) (16.609130170,0.97982988) (18.225877301,0.98105067) (20.000000000,0.98219701)};
\addplot[black,densely dotted] coordinates {(0.134142183,0) (0.134142183,1)};
\end{semilogxaxis}
\end{tikzpicture}
\caption{Two monotonicities used in cutoff selection. The clipped
cap decreases from its order-zero value toward its positive limit
(left). The minimizing cutoff stays at zero until $t_c\approx
0.13414$, then increases toward $1$ (right, logarithmic frequency
axis). Values are sampled from the unique cubic root and the
unique cutoff root. The derivative calculations in the text prove
the signs and uniqueness; the plots illustrate them.}
\label{fig:Nq-cap-and-cutoff}
\end{figure}
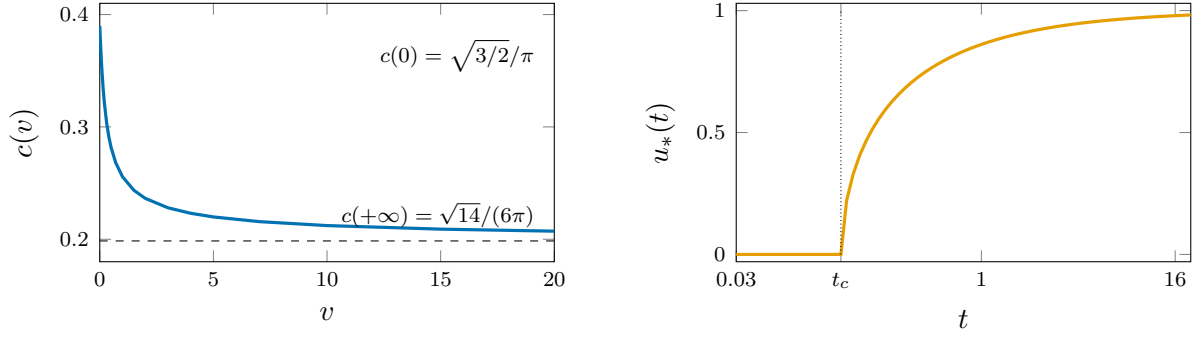

{For comparison, using only the rational global cap $2/5$ gives
the simpler loss}
{\[
 J_{9/5}\left(t,u\right)=\frac95t\left(1-u\right)+2I\left(u\right).
\]}
{Its unique optimizer is zero for $t\le5/\left(12\pi\right)$; otherwise it
solves $I'\left(u\right)=9t/10$.  Replacing $9/5$ by $q_0$ improves this
constant-cap optimizer.  The loss $j\left(t\right)$ in
\eqref{eq:Nq-cutoff-loss} is no larger than either optimized
constant-cap loss.}

{\begin{figure}[htbp]
\centering
\begin{tikzpicture}[x=9cm,y=2.5cm,font=\small,line join=round]
 \foreach \yy in {0.5,1,1.5,2}{
  \draw[gray!18] (0,\yy)--(1,\yy);
  \draw (0,\yy)--(-.008,\yy) node[left] {$\yy$};
 }
 \draw[->] (0,0)--(1.06,0) node[right] {$u$};
 \draw[->] (0,0)--(0,2.15) node[above] {$J\left(t,u\right)/t$};
 \foreach \xx in {0,0.25,0.5,0.75,1}{
  \draw (\xx,0)--(\xx,-.035) node[below] {$\xx$};
 }
 \draw[annBlue,very thick] plot coordinates {
  (0.000000,1.586196) (0.015000,1.563192) (0.030000,1.540403) (0.045000,1.517818) (0.060000,1.495428) (0.075000,1.473223)
  (0.090000,1.451194) (0.105000,1.429333) (0.120000,1.407632) (0.135000,1.386084) (0.150000,1.364683) (0.165000,1.343424)
  (0.180000,1.322301) (0.195000,1.301310) (0.210000,1.280448) (0.225000,1.259712) (0.240000,1.239099) (0.255000,1.218608)
  (0.270000,1.198236) (0.285000,1.177984) (0.300000,1.157850) (0.315000,1.137836) (0.330000,1.117941) (0.345000,1.098167)
  (0.360000,1.078515) (0.375000,1.058987) (0.390000,1.039587) (0.405000,1.020317) (0.420000,1.001182) (0.435000,0.982185)
  (0.450000,0.963332) (0.465000,0.944628) (0.480000,0.926081) (0.495000,0.907697) (0.510000,0.889486) (0.525000,0.871457)
  (0.540000,0.853621) (0.555000,0.835990) (0.570000,0.818578) (0.585000,0.801401) (0.600000,0.784477) (0.615000,0.767827)
  (0.630000,0.751473) (0.645000,0.735444) (0.660000,0.719769) (0.675000,0.704485) (0.690000,0.689634) (0.705000,0.675266)
  (0.720000,0.661438) (0.735000,0.648221) (0.750000,0.635700) (0.765000,0.623975) (0.780000,0.613175) (0.795000,0.603455)
  (0.810000,0.595016) (0.825000,0.588115) (0.840000,0.583091) (0.855000,0.580402) (0.870000,0.580683) (0.885000,0.584848)
  (0.900000,0.594268) (0.901583,0.595647) (0.903167,0.597112) (0.904750,0.598664) (0.906333,0.600309) (0.907917,0.602049)
  (0.909500,0.603890) (0.911083,0.605834) (0.912667,0.607887) (0.914250,0.610054) (0.915833,0.612340) (0.917417,0.614751)
  (0.919000,0.617292) (0.920583,0.619970) (0.922167,0.622792) (0.923750,0.625765) (0.925333,0.628897) (0.926917,0.632196)
  (0.928500,0.635672) (0.930083,0.639334) (0.931667,0.643193) (0.933250,0.647260) (0.934833,0.651549) (0.936417,0.656071)
  (0.938000,0.660843) (0.939583,0.665879) (0.941167,0.671198) (0.942750,0.676819) (0.944333,0.682763) (0.945917,0.689053)
  (0.947500,0.695714) (0.949083,0.702776) (0.950667,0.710269) (0.952250,0.718230) (0.953833,0.726696) (0.955417,0.735713)
  (0.957000,0.745330) (0.958583,0.755604) (0.960167,0.766598) (0.961750,0.778387) (0.963333,0.791054) (0.964917,0.804696)
  (0.966500,0.819427) (0.968083,0.835380) (0.969667,0.852710) (0.971250,0.871605) (0.972833,0.892287) (0.974417,0.915028)
  (0.976000,0.940159) (0.977583,0.968096) (0.979167,0.999357) (0.980750,1.034612) (0.982333,1.074733) (0.983917,1.120886)
  (0.985500,1.174671) (0.987083,1.238355) (0.988667,1.315276) (0.990250,1.410602) (0.991833,1.532859) (0.993417,1.697381)
  (0.995000,1.935421)
 };
 \draw[annBlue,densely dotted] (0.861400,0)--(0.861400,0.580114);
 \fill[annBlue] (0.861400,0.580114) circle (2pt);
 \draw[annOrange,very thick] plot coordinates {
  (0.000000,1.497334) (0.015000,1.471951) (0.030000,1.447277) (0.045000,1.423153) (0.060000,1.399455) (0.075000,1.376095)
  (0.090000,1.353007) (0.105000,1.330145) (0.120000,1.307472) (0.135000,1.284962) (0.150000,1.262594) (0.165000,1.240350)
  (0.180000,1.218217) (0.195000,1.196184) (0.210000,1.174242) (0.225000,1.152384) (0.240000,1.130602) (0.255000,1.108893)
  (0.270000,1.087251) (0.285000,1.065672) (0.300000,1.044154) (0.315000,1.022695) (0.330000,1.001291) (0.345000,0.979941)
  (0.360000,0.958644) (0.375000,0.937398) (0.390000,0.916204) (0.405000,0.895060) (0.420000,0.873966) (0.435000,0.852922)
  (0.450000,0.831929) (0.465000,0.810987) (0.480000,0.790098) (0.495000,0.769261) (0.510000,0.748479) (0.525000,0.727754)
  (0.540000,0.707088) (0.555000,0.686483) (0.570000,0.665942) (0.585000,0.645469) (0.600000,0.625068) (0.615000,0.604744)
  (0.630000,0.584502) (0.645000,0.564348) (0.660000,0.544290) (0.675000,0.524337) (0.690000,0.504499) (0.705000,0.484788)
  (0.720000,0.465219) (0.735000,0.445808) (0.750000,0.426577) (0.765000,0.407550) (0.780000,0.388760) (0.795000,0.370246)
  (0.810000,0.352056) (0.825000,0.334256) (0.840000,0.316930) (0.855000,0.300192) (0.870000,0.284201) (0.885000,0.269185)
  (0.900000,0.255488) (0.901583,0.254138) (0.903167,0.252810) (0.904750,0.251504) (0.906333,0.250221) (0.907917,0.248962)
  (0.909500,0.247728) (0.911083,0.246521) (0.912667,0.245340) (0.914250,0.244188) (0.915833,0.243066) (0.917417,0.241974)
  (0.919000,0.240916) (0.920583,0.239892) (0.922167,0.238904) (0.923750,0.237953) (0.925333,0.237043) (0.926917,0.236174)
  (0.928500,0.235349) (0.930083,0.234571) (0.931667,0.233843) (0.933250,0.233166) (0.934833,0.232545) (0.936417,0.231982)
  (0.938000,0.231482) (0.939583,0.231048) (0.941167,0.230685) (0.942750,0.230397) (0.944333,0.230190) (0.945917,0.230069)
  (0.947500,0.230042) (0.949083,0.230114) (0.950667,0.230295) (0.952250,0.230592) (0.953833,0.231016) (0.955417,0.231577)
  (0.957000,0.232289) (0.958583,0.233164) (0.960167,0.234220) (0.961750,0.235475) (0.963333,0.236949) (0.964917,0.238667)
  (0.966500,0.240657) (0.968083,0.242953) (0.969667,0.245593) (0.971250,0.248624) (0.972833,0.252103) (0.974417,0.256095)
  (0.976000,0.260686) (0.977583,0.265978) (0.979167,0.272101) (0.980750,0.279223) (0.982333,0.287561) (0.983917,0.297407)
  (0.985500,0.309161) (0.987083,0.323390) (0.988667,0.340929) (0.990250,0.363068) (0.991833,0.391941) (0.993417,0.431379)
  (0.995000,0.489198)
 };
 \draw[annOrange,densely dotted] (0.947162,0)--(0.947162,0.230039);
 \fill[annOrange] (0.947162,0.230039) circle (2pt);
 \draw[black,thick,dashed] plot coordinates {
  (0.000000,1.444184) (0.015000,1.419041) (0.030000,1.395365) (0.045000,1.372353) (0.060000,1.349725) (0.075000,1.327351)
  (0.090000,1.305160) (0.105000,1.283108) (0.120000,1.261166) (0.135000,1.239315) (0.150000,1.217540) (0.165000,1.195829)
  (0.180000,1.174173) (0.195000,1.152567) (0.210000,1.131004) (0.225000,1.109480) (0.240000,1.087992) (0.255000,1.066536)
  (0.270000,1.045110) (0.285000,1.023711) (0.300000,1.002339) (0.315000,0.980990) (0.330000,0.959664) (0.345000,0.938361)
  (0.360000,0.917077) (0.375000,0.895814) (0.390000,0.874570) (0.405000,0.853345) (0.420000,0.832139) (0.435000,0.810950)
  (0.450000,0.789779) (0.465000,0.768625) (0.480000,0.747489) (0.495000,0.726371) (0.510000,0.705270) (0.525000,0.684188)
  (0.540000,0.663123) (0.555000,0.642078) (0.570000,0.621052) (0.585000,0.600046) (0.600000,0.579061) (0.615000,0.558098)
  (0.630000,0.537158) (0.645000,0.516244) (0.660000,0.495355) (0.675000,0.474496) (0.690000,0.453667) (0.705000,0.432873)
  (0.720000,0.412116) (0.735000,0.391401) (0.750000,0.370733) (0.765000,0.350118) (0.780000,0.329564) (0.795000,0.309081)
  (0.810000,0.288681) (0.825000,0.268380) (0.840000,0.248199) (0.855000,0.228167) (0.870000,0.208323) (0.885000,0.188725)
  (0.900000,0.169457) (0.901583,0.167448) (0.903167,0.165443) (0.904750,0.163445) (0.906333,0.161452) (0.907917,0.159465)
  (0.909500,0.157484) (0.911083,0.155510) (0.912667,0.153543) (0.914250,0.151583) (0.915833,0.149630) (0.917417,0.147685)
  (0.919000,0.145749) (0.920583,0.143820) (0.922167,0.141901) (0.923750,0.139992) (0.925333,0.138092) (0.926917,0.136203)
  (0.928500,0.134325) (0.930083,0.132458) (0.931667,0.130604) (0.933250,0.128763) (0.934833,0.126936) (0.936417,0.125123)
  (0.938000,0.123327) (0.939583,0.121546) (0.941167,0.119783) (0.942750,0.118040) (0.944333,0.116316) (0.945917,0.114614)
  (0.947500,0.112935) (0.949083,0.111282) (0.950667,0.109655) (0.952250,0.108058) (0.953833,0.106492) (0.955417,0.104961)
  (0.957000,0.103467) (0.958583,0.102014) (0.960167,0.100606) (0.961750,0.099248) (0.963333,0.097945) (0.964917,0.096703)
  (0.966500,0.095529) (0.968083,0.094431) (0.969667,0.093420) (0.971250,0.092506) (0.972833,0.091704) (0.974417,0.091031)
  (0.976000,0.090507) (0.977583,0.090158) (0.979167,0.090018) (0.980750,0.090126) (0.982333,0.090539) (0.983917,0.091330)
  (0.985500,0.092597) (0.987083,0.094482) (0.988667,0.097196) (0.990250,0.101059) (0.991833,0.106606) (0.993417,0.114794)
  (0.995000,0.127577)
 };
 \draw[black,densely dotted] (0.979317,0)--(0.979317,0.090016);
 \fill[black] (0.979317,0.090016) circle (2pt);
 \draw[annBlue,very thick] (.08,2.02)--(.14,2.02)
    node[right,black] {$t=1$};
 \draw[annOrange,very thick] (.36,2.02)--(.42,2.02)
    node[right,black] {$t=4$};
 \draw[black,thick,dashed] (.65,2.02)--(.71,2.02)
    node[right] {$t=16$};
\end{tikzpicture}
\caption{The normalized cutoff loss at three frequencies.
The unique minimizing cutoffs are approximately $0.8614$, $0.9472$,
and $0.9793$ for $t=1,4,16$, respectively.  Increasing frequency
moves the optimum toward the turning point.  The plotted values
use the elementary cubic cap and one-dimensional quadrature;
convexity and uniqueness follow from Lemma~\ref{lem:Nq-cutoff-opt}.}
\label{fig:Nq-cutoff-opt}
\end{figure}
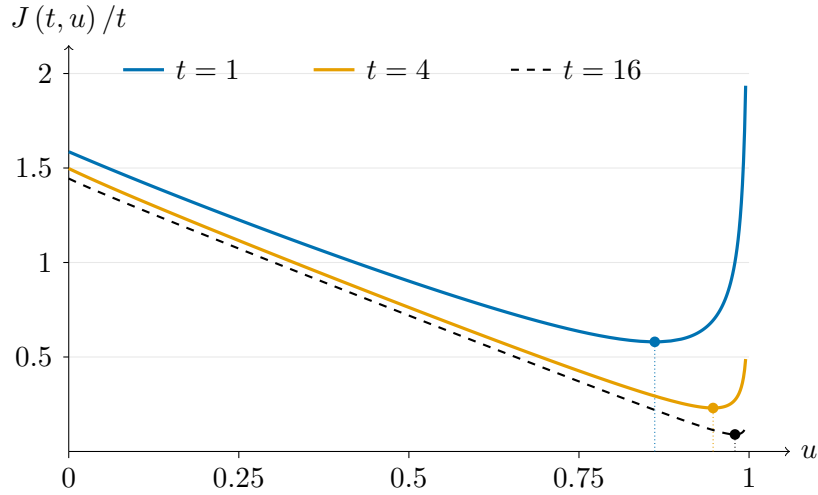}

{The cutoff keeps a longer core as the frequency increases. Put $q_{+\infty}=1+\sqrt{14}/\left(3\pi\right)$.  If
$h=1-u$, its turning loss has leading size $q_{+\infty} th$, whereas
the integrated phase error has leading size
$7h^{-1/2}/\left(12\pi\sqrt2\right)$.  Equating their marginal costs gives}
{\[
 h\sim\left(\frac7{24\pi\sqrt2\,q_{+\infty} t}\right)^{2/3}.
\]}
{In particular, there are constants $0<c_1\le c_2$ and $t_0>0$,
independent of $t$, such that
$c_1t^{1/3}\le j(t)\le c_2t^{1/3}$ for $t\ge t_0$.
This is a two-sided comparison in size; the sharper expansion
below specifies the leading constant. The positive spectral
correction is obtained by intersecting the
optimized lower-count envelope with the horizontal level $2pk$;
the first intersection lies strictly before the Weyl frequency.}

\subsection{\revblue{A positive envelope for all widths and frequencies}}
Define
{\begin{align}
 B_{\rm tr}\left(t\right)&=\frac{2V_R\left(t\right)}t-\frac D4t,\notag\\
 B_{\rm ph}\left(t\right)&=b-\frac{2C_+\left(a\left(t\right)\right)}{\sqrt{K_\delta\left(\delta t\right)}}
                    -s_\delta\left(t\right)-\frac{j\left(t\right)}t,\label{eq:Nq-Bphase}\\
 B_0\left(t\right)&=b-s_\delta\left(t\right)-\frac{J\left(t,0\right)}t,\notag\\
 B_\delta\left(t\right)&=\max\left\{B_{\rm tr}\left(t\right),B_{\rm ph}\left(t\right),B_0\left(t\right)\right\}.
\notag\end{align}}
{The term $B_0$ removes the unnecessary core rounding loss when
the core has length zero.  The positive-core estimate retains
the sharp initial-phase coefficient; its loss is independent of
$u$ except for $J$, which is why Lemma~\ref{lem:Nq-cutoff-opt}
optimizes that whole estimate. Thus $B_{\rm tr}$ measures the
ground-mode surplus over the area term, while $B_{\rm ph}$ measures
the arc coefficient left after paying the three phase and rounding
losses. Taking their maximum allows each estimate to be used where
it is most effective. In particular, the sign of any one branch
outside its useful frequency range is immaterial.}

{\begin{proposition}\label{prop:Nq-count}
For every $\delta>0$ and $t>0$,
\begin{equation}\label{eq:Nq-count}
 M\left(t\right)>\frac D4t^2+tB_\delta\left(t\right),\qquad B_\delta\left(t\right)>0.
\end{equation}
The envelope is continuous and satisfies
$B_\delta\left(t\right)\to\left(R+1\right)/2$ as $t\to+\infty$.
\end{proposition}}
{In the width-scaled frequency one may equivalently write
$\mathcal B_\delta\left(x\right)=B_\delta\left(x/\delta\right)>0$; then
$M\left(t\right)>Dt^2/4+t\mathcal B_\delta\left(\delta t\right)$.}
\begin{proof}
{On the core $\left[0,ut\right]$ use exactly the phase and directional
rounding estimates proved above.  On the turning interval use
$n_v\left(t^2\right)>f_t\left(v\right)-c\left(v\right)$ pointwise.  Its positive length gives
strictness after integration.  The reflected outer rounding
estimate is unchanged.  The result is}
{\[
 M\left(t\right)>\frac D4t^2+bt
  -\frac{2C_+\left(a\left(t\right)\right)}{\sqrt{\kappa_{\rm in}}}
  -\frac{\sqrt2}{3\sqrt{\kappa_{\rm out}}}-J\left(t,u\right),
\]}
{where $t^2\kappa_{\rm in}=K_\delta\left(\delta t\right)$ and
$t^{-1}/\sqrt{\kappa_{\rm out}}=
\sqrt{\pi\sqrt D/t}$. At $u=0$, apply the turning and outer
estimates directly: the core is empty, while the turning interval
still has positive length and gives strictness. This gives the
sharper bound $B_0$, and hence also the displayed estimate with
its extra nonnegative core loss. Evaluate the positive-core bound at the attained
optimizer, and combine it with \eqref{eq:Nq-trial-opt}. To see the
linear coefficient explicitly, integrating the favorable means
in the core and outer bands gives
\[
 2\left(\frac{ut}{2}+\frac{(R-1)t}{4}\right)
 =\left(u+\frac{R-1}{2}\right)t
 =bt-(1-u)t.
\]
Also $\int_0^{ut}E_v(t)\dd v/\pi=I(u)$ by $v=tz$.
The discarded turning mean and both integrated phase errors are
therefore exactly the terms collected in $J(t,u)$. This accounts
for every summand in \eqref{eq:Nq-Bphase}.}

{We spell out why the envelope is quantitatively positive at
every junction of the earlier cover.  For $0<\delta\le5/2$,
the constant trial gives a positive remainder on
$0<x\le15/4$, the phase window has margin $453/2000$ on
$15/4\le x\le27/5$, and the thin tail has margin $1/250$
above $27/5$.  Our integrated cap is at least as good as the
cap $1/4$ there, because $v\ge\left(9/10\right)t\ge27/20$ on the
turning interval.  For $5/2\le\delta\le10$, the power trial
has positive remainder up to and including $x=p\left(\delta\right)$
because $\widehat V_R\left(t\right)>v_*\left(t\right)$.  The intermediate bridge
has margin $163/2000$, and the larger-width tail has margin
$7/400$.  For $\delta\ge10$, the power interval overlaps the
generic $u=3/4$ range with margin $1/20$.  Each old $G$ or
$\mathcal H$ estimate is bounded above by
\eqref{eq:Nq-Bphase}, since the integrated cap is smaller and
the cutoff has been optimized.  These intervals cover every
positive $\delta,t$, proving positivity.}

{For reference, the cover used here is summarized below. The
listed margins are lower bounds for the remainder divided by
$t$, namely for $(M(t)-Dt^2/4)/t$; the trial entries are supplied
by $B_{\rm tr}$. The function $p(\delta)$ is the radial power-trial
threshold from Section~\ref{sec:neumann}, whereas $p=\pi/\beta$
without an argument is the angular spacing.
\begin{center}
\small
\begin{tabular}{cclc}
\hline
Width & Frequency $x=\delta t$ & Estimate retained & Margin\\
\hline
$0<\delta\le5/2$ & $0<x\le15/4$ & Constant trial &
$2a_\delta-15(\delta+2)/16$\\
 & $15/4\le x\le27/5$ & Thin phase window & $453/2000$\\
 & $x\ge27/5$ & Thin phase tail & $1/250$\\
$5/2\le\delta\le10$ & $0<x\le p(\delta)$ & Improved power trial & $e_\delta$\\
 & $p(\delta)\le x\le28/5$ & Intermediate bridge & $163/2000$\\
 & $x\ge28/5$ & Larger-width tail & $7/400$\\
$\delta\ge10$ & $0<x\le p(\delta)$ & Improved power trial & $e_\delta$\\
 & $x\ge22/5$ & Wide-width phase tail & $1/20$\\
\hline
\end{tabular}
\end{center}
Here $e_\delta>0$ is evaluated immediately below. The last two
intervals overlap because $p(\delta)\ge p(10)>22/5$.
The table records previously proved analytic margins; it does
not infer positivity from plotted samples.}

{Continuity of the trial maximum follows from its continuous
formula on a compact interval, or by reparametrizing
$\alpha=zRt$, $0\le z\le1$.  Continuity of the other terms
has already been established.  Finally $V_R\left(t\right)\le Rt$, so
$B_{\rm tr}\left(t\right)\to-\infty$.  The core and outer rounding
losses in $B_{\rm ph}$ tend to zero.  Choosing
$1-u=t^{-2/3}$ for $t>1$ shows $j\left(t\right)=O\left(t^{1/3}\right)$.
Hence $B_{\rm ph}\left(t\right)\to b$; all its losses are nonnegative,
and $B_0\left(t\right)\le b$, proving the asserted limit.}
\end{proof}

\subsection{\revblue{A closed positive subtraction at every eigenvalue index}}
{Here is a simpler formula which requires no optimization.  Put}
{\[
 a_\delta=\sqrt{\frac{D}{2\log R}},\qquad
 v_p=\frac{3R^2+1}{D},\qquad
 e_\delta=\frac{2\left(1-R^{-2}\right)R^{-2v_p}}{\sqrt2R+1},
\]}
{and define the explicit positive constant}
{\begin{equation}\label{eq:Nq-uniform-b}
 b_\delta^{\rm c}=
 \begin{cases}
 \displaystyle\min\left\{2a_\delta-
            \frac{15\left(\delta+2\right)}{16},\frac1{250}\right\},
                   &0<\delta\le5/2,\\[5pt]
 \displaystyle\min\left\{e_\delta,\frac7{400}\right\},
                   &\delta>5/2.
 \end{cases}
\end{equation}}
{The first entry is positive by the constant-trial endpoint
estimate and its width monotonicity: on $x\le15/4$,
\[
 B_{\rm tr}(t)\ge2a_\delta-\frac{\delta+2}{4}x
       \ge2a_\delta-\frac{15(\delta+2)}{16}>0.
\]
At $\delta=5/2$ the final expression equals
$F_{5/2}(15/4)/(15/4)>0$; for smaller widths its quotient by
$a_\delta$ is no smaller, as proved in Section~\ref{sec:low}.
For the second entry, write
\[
 t_p=\frac{p(\delta)}\delta
       =\frac{\sqrt{8(3R^2+1)}}D.
\]
The function $v_*(t)$ increases in $t$, and
\[
 1+2R^2t_p^2=\frac{(7R^2+1)^2}{D^2},\qquad
 v_*(t_p)=\frac{3R^2+1}{D}=v_p.
\]
Thus $t\le t_p$ implies $v_*(t)\le v_p$. Using the explicit
formula for $1-\vartheta$ above and its denominator being at most
one gives}
{\[
 1-\vartheta\left(t\right)\ge\left(1-R^{-2}\right)R^{-2v_p},\qquad
 \widehat V_R\left(t\right)+v_*\left(t\right)\le\left(\sqrt2R+1\right)t.
\]}
{For the second inequality, use
$v_*\le Rt/\sqrt2$ and
$\widehat V_R\le v_*+t$. Therefore}
{\[
 \frac{2\left(\widehat V_R\left(t\right)-v_*\left(t\right)\right)}t
 =\frac{2t\left(1-\vartheta\left(t\right)\right)}{\widehat V_R\left(t\right)+v_*\left(t\right)}
 \ge e_\delta.
\]}
{The coarse power comparison is nonnegative throughout that
interval, because
\[
 \frac{2v_*(t)}t-\frac D4t\ge0
 \quad\Longleftrightarrow\quad t\le t_p.
\]
Consequently $B_{\rm tr}(t)\ge e_\delta$ all the way up to the
endpoint, where the coarse comparison alone would vanish.
Together with the bridge and tail margins just
listed, these estimates prove the uniform bound}
{\begin{equation}\label{eq:Nq-closed-count}
 B_\delta\left(t\right)\ge b_\delta^{\rm c}>0,
 \qquad M\left(t\right)>\frac D4t^2+b_\delta^{\rm c}t,
 \qquad t>0.
\end{equation}}

{For physical radii $r>0$ and $rR$, set}
{\[
 w_k=\sqrt{\frac{8pk}{D}},\qquad
 A_\delta=\frac{2b_\delta^{\rm c}}D,\qquad
 \tau_k^{\rm c}=\sqrt{w_k^2+A_\delta^2}-A_\delta.
\]}
{Here $w_k^2/r^2=W_k\left(S\right)=4\pi k/\left|S\right|$.}

{This is the precise link with the main theorem. The uniform
choice $b_\delta^{\rm c}$ replaces the whole positive function
$B_\delta(t)$ by one explicit number. After scaling and angular
transfer it gives
\[
 \frac{|S|}{4\pi}\,a_{\rm N}
 =\frac{rb_\delta^{\rm c}}{2p},\qquad
 a_{\rm N}=\frac{4b_\delta^{\rm c}}{rD}
           =\frac{4r b_\delta^{\rm c}}{R_{\rm phys}^2-r^2},
 \quad R_{\rm phys}=rR.
\]
The strict P\'olya inequality needs only a positive remainder;
Theorem~\ref{thm:main} records this uniform quantitative choice.
The next inverse retains the larger, frequency-dependent remainder.}

{\color{revision}We now state the eigenvalue form of Theorem~\ref{thm:main},
proved immediately below.  For every annular sector, including the
retained slit at $\beta=2\pi$, and every integer $k\ge1$,}
{\begin{equation}\label{eq:Nq-all-gap}
 \mu_k\left(S\right)<W_k\left(S\right)-\Gamma_k^{\rm c}\left(S\right),\qquad
 \Gamma_k^{\rm c}\left(S\right)=
 \frac{2A_\delta}{r^2}
       \left(\sqrt{w_k^2+A_\delta^2}-A_\delta\right)>0.
\end{equation}}

\begin{proof}[\revblue{Proof of Theorem~\ref{thm:main}}]
{The proof uses exactly the continuous angular integral in
Proposition~\ref{prop:angular-transfer}, illustrated by
Figure~\ref{fig:sampling}(a). Apply that proposition at the
normalized frequency $t=r\sqrt E$ and then use
\eqref{eq:Nq-closed-count} to obtain}
{\[
 \begin{aligned}
 N_S^<\left(E\right)
 &\ge\frac1p\int_0^{+\infty}n_v\left(r^2E\right)\dd v
   =\frac{M\left(r\sqrt E\right)}{2p}\\
 &>\underbrace{\frac{\left|S\right|}{4\pi}E}_{\text{area term}}
   +\underbrace{\frac{rb_\delta^{\rm c}}{2p}\sqrt E}
               _{\text{positive correction}}.
 \end{aligned}
\]}
{The first comparison comes from the left-endpoint rectangles
and continuity in energy, as proved in
Proposition~\ref{prop:angular-transfer}. The strict inequality
comes from the positive continuous surplus in
\eqref{eq:Nq-closed-count}. Thus the integral is not bypassed:
Section~\ref{sec:neumann-quantitative} strengthens its lower
bound, and the same angular transfer converts that improvement
into the quantitative sector inequality.}
{At $E=\left(\tau_k^{\rm c}/r\right)^2$ the last expression is $k$,
because $\left(\tau_k^{\rm c}\right)^2+2A_\delta\tau_k^{\rm c}=w_k^2$.
Indeed, multiplying that expression by $8p/D$ gives the quadratic
$t^2+(4b_\delta^{\rm c}/D)t=w_k^2$; its positive root is precisely
$\sqrt{w_k^2+A_\delta^2}-A_\delta$.
The integer count is therefore at least $k+1$, including the
zero mode.  Thus $\mu_k\left(S\right)<E$, and subtraction from $W_k$
gives
\[
 W_k-E=\frac{w_k^2-(\tau_k^{\rm c})^2}{r^2}
       =\frac{2A_\delta\tau_k^{\rm c}}{r^2}>0,
\]
which is \eqref{eq:Nq-all-gap}. Since $a_{\rm N}=2A_\delta/r$,
the displayed eigenvalue formula in Theorem~\ref{thm:main} follows
by completing the same square.}
\end{proof}

{The fully optimized envelope gives a stronger, still explicit
one-variable version.  Define}
{\[
 \Phi_\delta\left(0\right)=0,\qquad
 \Phi_\delta\left(t\right)=\frac D4t^2+tB_\delta\left(t\right)\quad\left(t>0\right),
 \qquad
 \tau_k=\min\left\{t\in\left(0,w_k\right):\Phi_\delta\left(t\right)=2pk\right\}.
\]}
{The extension at zero is continuous: the trial contribution to
$\Phi_\delta$ is $2V_R(t)\le2Rt$, while the other two contributions
tend to zero, using $j(t)\le J(t,0)\le q_0t$ and
$ts_\delta(t)=O(\sqrt t)$. Since
$\Phi_\delta(0)=0<2pk$ and $\Phi_\delta(w_k)>2pk$, the intermediate
value theorem gives a crossing strictly inside $(0,w_k)$.
Continuity at both endpoints keeps its full set of crossings a
positive distance from $0$ and $w_k$, so this closed set has a
smallest element. No monotonicity of the fractional initial phase
is asserted or needed. Then}
{\begin{equation}\label{eq:Nq-envelope-inverse}
 \mu_k\left(S\right)<W_k\left(S\right)-\Gamma_k\left(S\right),\qquad
 \Gamma_k\left(S\right)=\frac{w_k^2-\tau_k^2}{r^2}
             \ge\Gamma_k^{\rm c}\left(S\right)>0.
\end{equation}}
{Indeed, the count is strictly larger than $k$ at $\tau_k$.
The inequality between the two gaps follows from
\eqref{eq:Nq-closed-count}: at $\tau_k^{\rm c}$,
\[
 \Phi_\delta(\tau_k^{\rm c})
 \ge\frac D4(\tau_k^{\rm c})^2+b_\delta^{\rm c}\tau_k^{\rm c}
 =2pk.
\]
There is therefore a crossing at or before $\tau_k^{\rm c}$,
so $\tau_k\le\tau_k^{\rm c}$ and $\Gamma_k\ge\Gamma_k^{\rm c}$.
Equivalently, one may invert the continuous nondecreasing
function $\max_{0\le s\le t}\Phi_\delta\left(s\right)$.}

\subsection{\revblue{Full arc coefficient and explicit high-frequency errors}}
{The optimized formula recovers the complete circular-arc
coefficient, with an explicit error rather than an unspecified
frequency threshold.  Set}
{\[
 L_\delta^{\mathrm N}=\frac23\sqrt{\frac{\pi R}{\delta}}
                   +\frac{\sqrt2}{3}\sqrt{\pi\sqrt D},\qquad
 C_0^{\rm h}=3q_0^{1/3}
       \left(\frac7{24\pi\sqrt{3/2}}\right)^{2/3},\qquad
 T_0^{\rm h}=\frac7{6\pi\sqrt3\,q_0}.
\]}
{For $t\ge T_0^{\rm h}$, choose}
{\[
 h=1-u=\left(\frac7{24\pi q_0\sqrt{3/2}\,t}\right)^{2/3}
 \le\frac12.
\]}
{This choice minimizes the elementary upper model
$q_0th+Ah^{-1/2}$, where
$A=7/(12\pi\sqrt{3/2})$: its derivative is
$q_0t-(A/2)h^{-3/2}$ and its second derivative is
$(3A/4)h^{-5/2}>0$. The condition that this minimizer lies in
$(0,1/2]$ is exactly $t\ge T_0^{\rm h}$.}
{Since $1-u^2=h\left(2-h\right)\ge3h/2$, $u\le1$, and
$2\arcsin u\le\pi$, this gives}
{\[
 j\left(t\right)\le q_0th+\frac7{12\pi\sqrt{3/2}}h^{-1/2}
                     +\frac1{12}
       =C_0^{\rm h}t^{1/3}+\frac1{12}.
\]}
{Here $t(1-u)+2\int_{ut}^tc(v)\dd v\le q_0th$ uses
$c(v)\le c(0)$, and the equality after substitution follows
from the critical-point identity
$Ah^{-1/2}=2q_0th$. Thus the explicit constants are obtained by
minimizing a proved upper bound for $j$, rather than by assuming
that its exact optimizer has already reached an asymptotic regime.}
{The core curvature gives
$2C_+\left(a\left(t\right)\right)/\sqrt{\kappa_{\rm in}}
 \le\left(2/3\right)\sqrt{\pi Rt/\delta}$.  Hence}
{\[
M\left(t\right)>\frac D4t^2+bt-L_\delta^{\mathrm N}\sqrt t
                        -C_0^{\rm h}t^{1/3}-\frac1{12},
 \qquad t\ge T_0^{\rm h}.
\]}
{In particular the completely explicit threshold}
{\[
 T_\delta^{\rm h}=\max\left\{T_0^{\rm h},
       \left(\frac{6L_\delta^{\mathrm N}}b\right)^2,
       \left(\frac{6C_0^{\rm h}}b\right)^{3/2},
       \frac1{2b}\right\}
\]}
{ensures $M\left(t\right)>Dt^2/4+\left(b/2\right)t$ for $t\ge T_\delta^{\rm h}$.
Indeed, the three errors are then bounded respectively by
$(b/6)t$, $(b/6)t$, and $(b/6)t$; their sum is at most
$(b/2)t$. The resulting quadratic lower-count expression is
strictly increasing for $t>0$.
Thus one can also obtain a closed tail subtraction by replacing
$b_\delta^{\rm c}$ by $b/2$ in \eqref{eq:Nq-all-gap}, whenever
$2pk\ge D\left(T_\delta^{\rm h}\right)^2/4+\left(b/2\right)T_\delta^{\rm h}$.}

{There is a sharper asymptotic evaluation of the optimal cutoff.
Equation~\eqref{eq:Nq-cap-cubic} with $w=v^{4/3}z$ gives
\[
 8z^3=7+10v^{-2/3}z+3v^{-4/3}z^2.
\]
Its unique positive root stays bounded above and away from zero,
so $z=(7/8)^{1/3}+O(v^{-2/3})$. Consequently
$r_v=\sqrt{w_v}/v=O(v^{-1/3})$. Substitution in
$r-\arctan r=r^3/3+O(r^5)$ gives
\[
 A_v(X_v^*)=\frac{\sqrt{7/8}}3+O(v^{-2/3}),\qquad
 E_v(X_v^*)=\frac7{24\sqrt{7/8}}+O(v^{-2/3}).
\]
The two displayed constants are equal, and their sum is
$\sqrt{14}/6$. Therefore}
{\[
 c\left(v\right)=\frac{\sqrt{14}}{6\pi}+O\left(v^{-2/3}\right),\qquad
 q_{+\infty}=1+\frac{\sqrt{14}}{3\pi}.
\]}
{For $h=1-u\downarrow0$,}
{\[
 2I\left(1-h\right)=\frac7{12\pi\sqrt2}h^{-1/2}
                       +\frac1{12}+O\left(h^{1/2}\right).
\]}
{Since $c\left(v\right)$ decreases to its limit, the optimal loss satisfies}
{\[
\begin{aligned}
 1-u_*\left(t\right)&\sim
      \left(\frac7{24\pi\sqrt2\,q_{+\infty} t}\right)^{2/3},\\
 j\left(t\right)&=3q_{+\infty}^{1/3}
       \left(\frac7{24\pi\sqrt2}\right)^{2/3}t^{1/3}
                       +\frac1{12}+O\left(t^{-1/3}\right).
 \end{aligned}
\]}
{For the lower comparison use
$J\left(t,u\right)\ge q_{+\infty} t\left(1-u\right)+2I\left(u\right)$.  For the matching upper
comparison choose the displayed scale of $1-u$; on its turning
interval $c\left(v\right)-c\left(+\infty\right)=O\left(t^{-2/3}\right)$, whose integrated
contribution is $O\left(t^{-1/3}\right)$. To make the balance explicit,
put $a=7/(24\pi\sqrt2)$. The singular part of the model loss is
$q_{+\infty}th+2ah^{-1/2}$; it is minimized at
$h=(a/(q_{+\infty}t))^{2/3}$ and has minimum
$3q_{+\infty}^{1/3}a^{2/3}t^{1/3}$.
For the exact minimizer, $u_*(t)\to1$ and its root equation becomes
\[
 a h^{-3/2}\{1+O(h)\}
   =t\{q_{+\infty}+O(t^{-2/3})\},
\]
which gives the first asymptotic line and also
$h=O(t^{-2/3})$. The remainder $O(h^{1/2})$ in the expansion of
$2I(1-h)$ is thus $O(t^{-1/3})$, as is the integrated cap error.
This proves the stated expansion of $j$. The optimization is
optimal within
the present phase-error and clipped-cap family.  The independent
rounding error remains of order $\sqrt t$.}

{For fixed geometry the stronger gap in
\eqref{eq:Nq-envelope-inverse} consequently has the leading size}
{\begin{equation}\label{eq:Nq-gap-asymptotic}
 \Gamma_k\left(S\right)\sim\frac{2}{r\left(R-1\right)}\sqrt{W_k\left(S\right)}
               =\frac{2}{R_{\rm phys}-r}\sqrt{W_k\left(S\right)}.
\end{equation}}
{\color{revision}Indeed, continuity makes $\Phi_\delta$ bounded on
every finite interval, whereas its crossing level $2pk$ tends to
infinity. Hence $\tau_k\to+\infty$. Since $B_\delta(t)\to b$, the
crossing equation implies
\[
 \frac{w_k^2}{\tau_k^2}
    =1+\frac{4B_\delta(\tau_k)}{D\tau_k}\longrightarrow1,
\]
so $\tau_k/w_k\to1$.
The equation $\Phi_\delta\paren{\tau_k}=Dw_k^2/4$ gives
\[
 w_k^2-\tau_k^2=\frac4D\tau_kB_\delta\paren{\tau_k}.
\]
Dividing this identity by $r^2\sqrt{W_k}=rw_k$ gives the limit
$4b/(rD)=2/(r(R-1))$ in \eqref{eq:Nq-gap-asymptotic}.
After angular sampling, the associated count coefficient is
\[
 \frac{\beta\paren{R_{\rm phys}+r}}{4\pi},
\]
the circular-arc part of the boundary coefficient.
Section~\ref{sec:boundary-optimality} identifies the additional
radial-side contribution and explains the scope of these bounds.}

{The three conclusions therefore have distinct roles: the sign
$B_\delta(t)>0$ proves strict P\'olya, the explicit lower bound
$B_\delta(t)\ge b_\delta^{\rm c}$ proves the uniform main theorem,
and the full envelope retains the information needed for the sharp
arc-scale correction. The last statement concerns the leading
high-frequency coefficient and does not replace the all-frequency
overlap argument.}

\section{\revblue{Boundary coefficients and the scope of optimality}}
{\label{sec:boundary-optimality}}
{\color{editorial}The main theorems give explicit positive corrections at
every index, uniformly in the aperture.  The optimized estimates in
Sections~\ref{sec:dirichlet-quantitative} and
\ref{sec:neumann-quantitative} then raise a different question: which
part of the classical Weyl boundary term do these corrections capture?
We first determine the actual second term of the continuous order
integrals; this justifies calling their coefficient the full
\emph{circular-arc coefficient}.  We next recover the contribution of
the two radial sides when the angular orders are sampled.  Finally, a
low-index example explains why the larger full-perimeter coefficient
cannot replace the all-index constants in the main theorems.  Thus the
asymptotic results identify the scale and scope of the improvements;
they do not replace the finite-frequency estimates used to prove the
main theorems.}

{Throughout the normalized argument, the radii are $1$ and $R>1$,
$\delta=R-1$, $D=R^2-1$, and $p=\pi/\beta>0$ is fixed.  Write
$n_v^{\mathrm N}=n_v$ and $M_{\mathrm N}=M$ for the Neumann counts
already defined, and retain $n_v^{\mathrm D}$ and $M_{\mathrm D}$
for their Dirichlet counterparts.  All radial counts in the proof
are inclusive.  The strict sector counts will be recovered below.}

\subsection{\revblue{Phase shifts and the actual continuous counts}}
{Put $g_R\left(z\right)=f_t\left(tz\right)/t$.  Homogeneity of the action gives
the explicit expression}
\[
 g_R(z)=\frac1\pi\begin{cases}
 \sqrt{R^2-z^2}-z\arccos(z/R)-\sqrt{1-z^2}+z\arccos z,
       &0\le z\le1,\\[2pt]
 \sqrt{R^2-z^2}-z\arccos(z/R),&1\le z\le R.
 \end{cases}
\]
{In particular $g_R$ is independent of $t$ and continuous at the
transition $z=1$.  Differentiation is particularly simple because}
\[
 \frac{\dd}{\dd z}\left(\sqrt{a^2-z^2}-z\arccos(z/a)\right)
       =-\arccos(z/a),\qquad 0<z<a.
\]
{Its derivatives on the two open order regions are therefore}
{\begin{equation}\label{eq:Bq-action-curvature}
 \begin{aligned}
 g_R'\left(z\right)&=\frac{\arccos z-\arccos\left(z/R\right)}\pi,
 &g_R''\left(z\right)&=\frac1\pi\left(
       \frac1{\sqrt{R^2-z^2}}-\frac1{\sqrt{1-z^2}}\right),
 &&0<z<1,\\
 g_R'\left(z\right)&=-\frac1\pi\arccos\left(z/R\right),
 &g_R''\left(z\right)&=\frac1{\pi\sqrt{R^2-z^2}},
 &&1<z<R.
 \end{aligned}
\end{equation}}
{Since $z/R<z$ on $0<z<1$ and $\arccos$ decreases, both formulas
for $g_R'$ are negative.  Thus $g_R$ decreases strictly from
$g_R(0)=\delta/\pi$ to $g_R(R)=0$.  Moreover, $g_R''<0$ before
$z=1$ and $g_R''>0$ after $z=1$: the slope decreases from zero to
its unique minimum $-\arccos(1/R)/\pi$ and then increases to zero.
Consequently $g_R'$ is nonzero on each compact subinterval of the
two regions, and $|g_R''|$ is bounded above and below by positive
constants there.  These are precisely the properties needed for
the continuous change of variables and the discrete cancellation
argument below.}

\begin{figure}[htbp]
\centering
\begin{tikzpicture}
\begin{axis}[width=.74\linewidth,height=5cm,
 xmin=0,xmax=2,ymin=-.38,ymax=.035,
 xlabel={$z=v/t$},ylabel={$g_2'(z)$},
 xtick={0,1,2},ytick={-.333333,0},
 yticklabels={$-1/3$,$0$},grid=major,
 tick label style={font=\small},label style={font=\small}]
 \addplot[annBlue,very thick,domain=0:1,samples=81]
       {(acos(x)-acos(x/2))/180};
 \addplot[annOrange,very thick,domain=1:2,samples=81]
       {-acos(x/2)/180};
 \addplot[black,densely dashed] coordinates {(1,0) (1,-.333333)};
 \addplot[only marks,mark=*,black] coordinates {(1,-.333333)};
 \node[font=\small,anchor=south] at (axis cs:.46,-.19)
       {$g_2''<0$};
 \node[font=\small,anchor=south] at (axis cs:1.55,-.19)
       {$g_2''>0$};
\end{axis}
\end{tikzpicture}
\caption{The action slope for $R=2$.  It is negative throughout the
open order interval and has a unique minimum at the transition
$z=1$.  For every $R>1$ the same pattern gives both the strict
decrease of $g_R$ and the bounded variation of its slope.}
\label{fig:Bq-action-slope}
\end{figure}
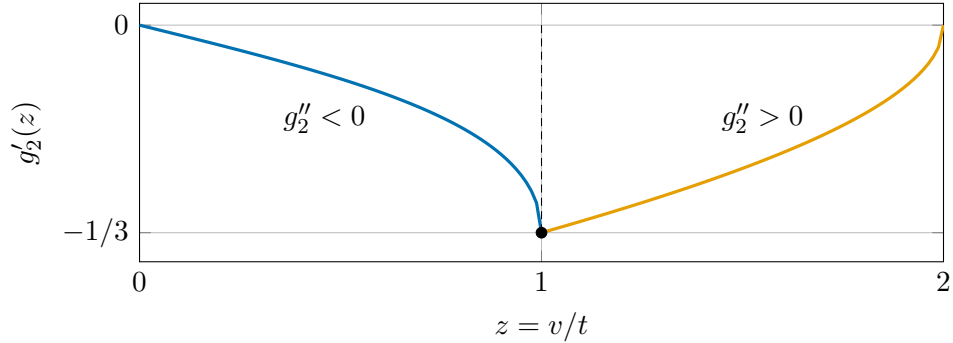

{\color{revision}The Dirichlet phase-count formula of
\cite[Lemma~4.1]{DirichletAnnuli} and the Neumann formula of
\cite[Lemma~2.6]{DaiSunAnnuli}, extended from integer to real order
in Lemma~\ref{lem:exact} above, take the form}
{\begin{equation}\label{eq:Bq-phase-counts}
 \begin{aligned}
 n_v^{\mathrm D}\left(t^2\right)&=\left\lfloor\gamma_t^{\mathrm D}\left(v\right)\right\rfloor,
 &\gamma_t^{\mathrm D}\left(v\right)&=
       \frac{\theta_v\left(Rt\right)-\theta_v\left(t\right)}\pi,\\
 n_v^{\mathrm N}\left(t^2\right)&=\left\lfloor\gamma_t^{\mathrm N}\left(v\right)\right\rfloor+1,
 &\gamma_t^{\mathrm N}\left(v\right)&=
       \frac{\varphi_v\left(Rt\right)-\varphi_v\left(t\right)}\pi.
 \end{aligned}
\end{equation}}
{Their phase normalizations are those used in the preceding sections.
For $\sigma\in\left\{\mathrm D,\mathrm N\right\}$, write}
{\[
 \gamma_t^\sigma\left(tz\right)=t g_R\left(z\right)+r_t^\sigma\left(z\right).
\]}
{The ordinary and derivative phase enclosures of
\cite[Theorems~1.4 and~1.10]{Phase} give locally uniform convergence
to zero for both shifts on $\left(0,1\right)$.  On $\left(1,R\right)$ the argument $Rt$
remains oscillatory, while $t$ is in the forbidden region.  The
uniform expansions for positive real order in
\cite[\S10.20(i), equations~10.20.4--10.20.8]{DLMF} give locally
uniformly}
{\[
 \frac{J_{tz}\left(t\right)}{Y_{tz}\left(t\right)}\longrightarrow0,
 \qquad
 \frac{J'_{tz}\left(t\right)}{Y'_{tz}\left(t\right)}\longrightarrow0.
\]}
{For $v=tz>t$, the chosen branches satisfy
$-\pi/2<\theta_v(t)\le-\pi/4$ and $0<\varphi_v(t)<\pi/2$.
These ranges select the phase values from the ratios above;
the corresponding normalized phases tend to $-\pi/2$ and
$\pi/2$, respectively. Combining these limits with the regular
phase enclosures at $Rt$ yields}
{\begin{equation}\label{eq:Bq-phase-shifts}
 \begin{array}{c|cc}
       &0<z<1&1<z<R\\ \hline
 r_t^{\mathrm D}\left(z\right)&0&1/4\\
 r_t^{\mathrm N}\left(z\right)&0&-1/4
 \end{array}
 \qquad\text{in the limit }t\longrightarrow+\infty,
\end{equation}}
{with convergence uniform on compact subintervals.}

{The signs in this table are worth making explicit.  In the
oscillatory region the ordinary phase equals
$A_v(y)-\pi/4+o(1)$, whereas the derivative phase equals
$A_v(y)+\pi/4+o(1)$.  For $z<1$ both arguments are oscillatory,
so their constant shifts cancel.  For $z>1$ only the outer
argument is oscillatory; the inner limits just stated give
$(-1/4)-(-1/2)=1/4$ in the Dirichlet case and
$(1/4)-(1/2)=-1/4$ in the Neumann case.  The change at $z=1$
will account for the two different mean corrections below.}

{We also need bounds that do not deteriorate at the endpoints or at
$z=1$.  The Neumann clipped cap from the preceding section gives}
{\begin{equation}\label{eq:Bq-uniform-N}
 \left|n_v^{\mathrm N}\left(t^2\right)-f_t\left(v\right)\right|\le\frac54+c\left(0\right),
 \qquad 0\le v\le Rt.
\end{equation}}
{Indeed, on $v\le t$ the difference of the two regular phase errors
is at most $\pi c\left(0\right)$ in absolute value.  On $t<v<Rt$ one has
$0<\varphi_v\left(t\right)<\pi/2$, so the normalized shift lies between
$-1/4$ and $1/4+c\left(0\right)$.  Formula~\eqref{eq:Bq-phase-counts} then
gives the bound.}

{For Dirichlet conditions define, for $y>v$,}
{\[
 e_v\left(y\right)=A_v\left(y\right)-\frac\pi4-\theta_v\left(y\right),\qquad
 \mathcal E_v\left(y\right)=\frac{3y^2+2v^2}{24\left(y^2-v^2\right)^{3/2}}.
\]}
{The clipped ordinary phase enclosure gives
$0<e_v\left(y\right)<\min\left\{\mathcal E_v\left(y\right),A_v\left(y\right)+\pi/4\right\}$.
Moreover $A_v\left(y\right)\mathcal E_v\left(y\right)\le1/8$.  For $v>0$, setting
$w=\sqrt{y^2-v^2}/v$ reduces the latter assertion to
$\arctan w\ge5w/\left(3w^2+5\right)$.  Indeed,}
\[
 \frac{\dd}{\dd w}\left(\arctan w-\frac{5w}{3w^2+5}\right)
 =\frac{4w^2(6w^2+5)}{(1+w^2)(3w^2+5)^2}\ge0,
\]
{and the difference vanishes at zero.  At $v=0$ the product equals
$1/8$ directly.  Writing $a=A_v(y)>0$ now gives}
\[
 e_v(y)\le\min\{a+\pi/4,\,1/(8a)\}.
\]
{The first expression increases with $a$ and the second decreases.
Their intersection therefore gives a uniform upper bound, namely}
{\[
 e_v\left(y\right)\le c_*:=
 \frac{\pi/4+\sqrt{\pi^2/16+1/2}}2.
\]}
{For $v\le t$, the error in the phase difference is a difference
of two such bounded errors.  For $t<v<Rt$, the phase monotonicity
and its turning value give $-\pi/2<\theta_v\left(t\right)\le-\pi/4$.
Thus}
{\begin{equation}\label{eq:Bq-uniform-D}
 \left|n_v^{\mathrm D}\left(t^2\right)-f_t\left(v\right)\right|\le1+\frac{c_*}\pi,
 \qquad0\le v\le Rt.
\end{equation}}
{Turning arguments follow by continuity.  Both radial counts vanish
for $v\ge Rt$: their Rayleigh quotients are at least $v^2/R^2$,
and equality at a positive order cannot be attained.}

\begin{samepage}
{\begin{proposition}[The sharp continuous boundary coefficient]
\label{prop:Nq-actual-coefficient}
\label{prop:Dq-continuous-asymptotic}
For each fixed $R>1$,
\[
M_{\mathrm N}\left(t\right)=\frac D4t^2+\frac{R+1}2t+o_R\left(t\right),
 \qquad t\longrightarrow+\infty,
\]
and
\begin{equation}\label{eq:Dq-continuous-asymptotic}
 M_{\mathrm D}\left(t\right)=\frac D4t^2-\frac{R+1}2t+o_R\left(t\right).
\end{equation}
Hence $\left(R+1\right)/2$ is the largest possible leading coefficient in a
lower bound, with an $o_R\left(t\right)$ remainder, for either the actual
Neumann surplus $M_{\mathrm N}\left(t\right)-Dt^2/4$ or the actual Dirichlet
deficit $Dt^2/4-M_{\mathrm D}\left(t\right)$.
\end{proposition}}
\end{samepage}
\begin{proof}
{On a compact interval $I$ contained in $\left(0,1\right)$ or $\left(1,R\right)$, the
change of variable $y=g_R\left(z\right)$ gives the smooth weight
$w(y)=1/|g_R'(g_R^{-1}(y))|$ on the interval $J=g_R(I)$;
both $w$ and $w'$ are bounded there.  If $H$ is a bounded
Riemann integrable function of period one and mean zero, then
$K(y)=\int_0^yH(s)\dd s$ is a bounded periodic primitive.
Integration by parts gives}
\[
 \begin{aligned}
 \int_J H(ty+c)w(y)\dd y
 &=\frac1t[K(ty+c)w(y)]_{\partial J}
       -\frac1t\int_J K(ty+c)w'(y)\dd y\\
 &=O_I(t^{-1}),
 \end{aligned}
\]
{and hence}
{\[
 \int_I H\left(tg_R\left(z\right)+c\right)\,\dd z=O_I\left(t^{-1}\right).
\]}
{Apply this to the centered sawtooth
$H(y)=1/2-\{y\}$, or to an interval indicator modulo one
after subtracting its mean.
In particular,}
{\[
 \int_I\left(\left\lfloor tg_R\left(z\right)+c\right\rfloor-tg_R\left(z\right)\right)\dd z
       \longrightarrow\left(c-1/2\right)\left|I\right|.
\]}
{The uniform convergence of $r_t^\sigma$ to the constant $c$ in
\eqref{eq:Bq-phase-shifts} preserves this limit.  For a phase
perturbation of size at most $\eta<1/2$, the floor can change
only within distance $\eta$ of an integer.  The change of variable
just used bounds the measure of this set by $O_I\left(\eta+t^{-1}\right)$.
Let $t\to+\infty$ and then $\eta\downarrow0$.}

{The mean corrections $n_{tz}^\sigma\left(t^2\right)-f_t\left(tz\right)$ are therefore
$1/2$ and $1/4$ for Neumann conditions on the inner and outer order
regions, and $-1/2$ and $-1/4$ for Dirichlet conditions.  More
explicitly, if $\{y\}=y-\lfloor y\rfloor$ then the periodic
function $-\{y\}$ has mean $-1/2$, and}
\[
 \begin{array}{c|cc}
 &0<z<1&1<z<R\\ \hline
 \text{Dirichlet mean: }c-1/2&-1/2&1/4-1/2=-1/4\\
 \text{Neumann mean: }c+1/2&1/2&-1/4+1/2=1/4.
 \end{array}
\]
{The extra $1$ in the Neumann phase-count formula is essential
for this sign.  The
uniform bounds \eqref{eq:Bq-uniform-N}--\eqref{eq:Bq-uniform-D}
allow intervals about $z=0,1,R$ to be removed and then shrunk.
It follows that}
{\[
 \int_0^R\left(n_{tz}^{\mathrm N}\left(t^2\right)-f_t\left(tz\right)\right)\dd z
       \longrightarrow\frac12+\frac\delta4=\frac{R+1}4,
\]}
{and the Dirichlet limit is its negative.  Multiplication by $2t$
and the exact action identity \eqref{eq:action} give, for example,}
\[
 \begin{aligned}
 M_{\mathrm N}(t)
 &=2\int_0^{Rt}f_t(v)\dd v
      +2t\int_0^R\bigl(n_{tz}^{\mathrm N}(t^2)-f_t(tz)\bigr)\dd z\\
 &=\frac D4t^2+2t\left(\frac12+\frac\delta4+o_R(1)\right)
  =\frac D4t^2+\frac{R+1}{2}t+o_R(t).
 \end{aligned}
\]
{The Dirichlet calculation changes the sign of the second integral.
The final assertion follows by dividing the corresponding surplus
or deficit by $t$.}
\end{proof}

{For physical radii $r<R_{\rm phys}$, the continuous contribution to
the sector count is $M_\sigma(r\sqrt E)/(2p)$, with normalized
$R=R_{\rm phys}/r$.  The proposition gives}
\[
 \begin{gathered}
 \frac{M_\sigma(r\sqrt E)}{2p}
   =\frac A{4\pi}E
       \mathbin{\pm}\frac{\beta(R_{\rm phys}+r)}{4\pi}\sqrt E
       +o_{r,R_{\rm phys},\beta}(\sqrt E),\\
 A=\frac\beta2(R_{\rm phys}^2-r^2).
 \end{gathered}
\]
{where the plus sign is Neumann.  Thus the second coefficient is
exactly the total circular-arc length $\beta(R_{\rm phys}+r)$
divided by $4\pi$.  After counting inversion its coefficient is
$\beta(R_{\rm phys}+r)/A=2/(R_{\rm phys}-r)$, as in
\eqref{eq:intro-arc-coefficient}.  This is the precise optimality
statement proved here: the continuous surplus or deficit cannot
have a larger leading coefficient.  It does not assert optimality
of each all-index constant, nor does it yet include the radial sides.}

\subsection{\revblue{Discrete phase averaging at fixed angular spacing}}
{The continuous averaging argument does not estimate the angular
rounding sum.  We prove the needed discrete cancellation directly.}

{\begin{lemma}[Discrete averaging of the action]
\label{lem:Bq-discrete-averaging}
Fix $R>1$, $p>0$, and a compact interval $I$ contained in $\left(0,1\right)$
or $\left(1,R\right)$.  For every nonzero integer $h$,
\begin{equation}\label{eq:Bq-exponential-sum}
 \sum_{pj/t\in I}\exp\left(2\pi\mathrm i h f_t\left(pj\right)\right)
       =O_{R,p,I,h}\left(\sqrt t\right).
\end{equation}
Consequently, for every bounded Riemann integrable function $H$
of period one and every constant $c$,
\begin{equation}\label{eq:Bq-discrete-mean}
 \frac1t\sum_{pj/t\in I}H\left(f_t\left(pj\right)+c\right)
       \longrightarrow\frac{\left|I\right|}p\int_0^1H\left(y\right)\dd y.
\end{equation}
\end{lemma}}
\begin{proof}
{Set $F_t\left(x\right)=h t g_R\left(px/t\right)$ on the real interval $\left(t/p\right)I$.
By \eqref{eq:Bq-action-curvature}, $F_t'$ is monotone with a
range contained in a fixed bounded interval, and}
{\[
 \frac{a}{t}\le\left|F_t''\left(x\right)\right|\le\frac{b}{t}
\]}
{for positive constants $a,b$ depending only on the fixed data.
For $0<\varepsilon<1/4$, remove the set}
{\[
 \left\{x:\operatorname{dist}\left(F_t'\left(x\right),\mathbb Z\right)<\varepsilon\right\}.
\]}
{There are only a fixed number of its component intervals, and
their total length is $O\left(\varepsilon t\right)$, by the lower derivative
bound.  Discard every integer index $j$ whose edge $\left[j,j+1\right]$
meets that set, together with the endpoint indices whose edges
are not wholly contained in $\left(t/p\right)I$.  This discards
$O\left(\varepsilon t+1\right)$ terms.}

{The retained indices form a fixed number of consecutive blocks.
On each block the differences
$d_j=F_t\left(j+1\right)-F_t\left(j\right)$ are monotone and belong to a single interval
$\left[m+\varepsilon,m+1-\varepsilon\right]$ for an integer $m$.
Put}
{\[
 u_j=e^{2\pi\mathrm iF_t\left(j\right)},\qquad
 a_j=\frac1{e^{2\pi\mathrm i d_j}-1}
       =-\frac12-\frac{\mathrm i}2\cot\left(\pi d_j\right).
\]}
{Then $u_j=a_j\left(u_{j+1}-u_j\right)$.  Summation by parts on a block
$A\le j\le B$ gives}
{\[
 \sum_{j=A}^B u_j
 =a_Bu_{B+1}-a_Au_A+
       \sum_{j=A+1}^B\left(a_{j-1}-a_j\right)u_j.
\]}
{The cotangent is monotone between successive integers.  Hence
$\sup\left|a_j\right|+\sum\left|a_{j-1}-a_j\right|=O\left(\varepsilon^{-1}\right)$ on the
block.  Restoring the discarded terms proves}
{\[
 \left|\sum_{pj/t\in I}e^{2\pi\mathrm i h f_t\left(pj\right)}\right|
       =O\left(\varepsilon t+\varepsilon^{-1}+1\right).
\]}
{Choose $\varepsilon=t^{-1/2}$ for sufficiently large $t$ to obtain
\eqref{eq:Bq-exponential-sum}.}

{The number of indices is $t\left|I\right|/p+O\left(1\right)$.  The exponential-sum
estimate first proves \eqref{eq:Bq-discrete-mean} for any fixed
trigonometric polynomial: its constant Fourier mode contributes
the claimed mean, and each of its finitely many nonzero modes is
$O(t^{-1/2})$ after division by $t$.  To pass to a continuous
one-periodic function $H$, define its Fourier coefficients and
the Fej\'er means of its Fourier series by}
\[
 \begin{aligned}
 \widehat H(h)&=\int_0^1H(y)e^{-2\pi\mathrm i h y}\dd y,\\
 \sigma_NH(y)&=\sum_{|h|\le N}
       \left(1-\frac{|h|}{N+1}\right)
          \widehat H(h)e^{2\pi\mathrm i h y}.
 \end{aligned}
\]
{These are the Ces\`aro averages of the Fourier partial sums.
For continuous periodic $H$ they converge uniformly to $H$ and
preserve its mean: $\int_0^1\sigma_NH=\widehat H(0)=\int_0^1H$.
For each fixed $N$ the finite-mode argument applies to
$\sigma_NH$, while}
\[
 \left|\frac1t\sum_{pj/t\in I}
       \bigl(H-\sigma_NH\bigr)(f_t(pj)+c)\right|
 \le\left(\frac{|I|}{p}+O(t^{-1})\right)
           \|H-\sigma_NH\|_\infty.
\]
{Letting $t\to+\infty$ first, and only then $N\to+\infty$,
proves the assertion for continuous $H$.  No rate uniform in $N$
has been used, so this passage gives a limit, not an
$O(\sqrt t)$ remainder for the unnormalized sum for general $H$.}

{A discontinuous sawtooth or interval indicator cannot be
approximated uniformly by continuous functions.  Instead, bound
an interval indicator above and below by continuous periodic
functions whose integrals differ by an arbitrarily small amount;
this proves \eqref{eq:Bq-discrete-mean} for indicators and hence for
step functions.  The same bounds show that a single phase value
has limiting frequency zero, so the choice of endpoint values
does not affect the argument.  Finally, upper and lower step
approximations, whose integrals approach that of a bounded
Riemann integrable $H$, give the stated general case. For complex
$H$, apply the argument to its real and imaginary parts. This also
supplies a direct proof of the required equidistribution.}
\end{proof}

{Apply the lemma to the sawtooth and use the locally uniform shifts
\eqref{eq:Bq-phase-shifts}.  As in the continuous argument, a
uniformly small perturbation changes a floor only at phases near
integers.  Formula~\eqref{eq:Bq-discrete-mean} shows that a fixed
$\eta$-neighborhood contains $2\eta t\left|I\right|/p+o\left(t\right)$ sampled phases.
First letting $t\to+\infty$ and then $\eta\downarrow0$ therefore
shows that the perturbation changes only $o\left(t\right)$ floors.  We obtain}
{\begin{equation}\label{eq:Bq-sampled-corrections}
 \begin{aligned}
 \sum_{j\ge0}\left(n_{jp}^{\mathrm N}\left(t^2\right)-f_t\left(jp\right)\right)
       &=\frac{R+1}{4p}t+o_{R,p}\left(t\right),\\
 \sum_{j\ge1}\left(n_{jp}^{\mathrm D}\left(t^2\right)-f_t\left(jp\right)\right)
       &=-\frac{R+1}{4p}t+o_{R,p}\left(t\right).
 \end{aligned}
\end{equation}}
{\color{revision}The omitted neighborhoods of $0$, $1$, and $R$ contain
$O\paren{\eta t+1}$ indices.  By the uniform radial bounds, their
total contribution is $O_{R,p}\paren{\eta t+1}$.  This justifies
their removal before taking $\eta\downarrow0$.}

\subsection{\revblue{The straight sides and the full two-term sector law}}
{The action sum retains an endpoint contribution that is absent
from the order integral.  Its trapezoidal estimate has a bounded
error.  Indeed, extend $f_t$ by zero beyond $Rt$.  Its derivative
is continuous on $\left[0,+\infty\right)$, decreases from zero to
$-\arccos\left(1/R\right)/\pi$ at $v=t$, and then increases to zero.
Consequently}
{\[
 \operatorname{Var}_{\left[0,+\infty\right)}\left(f_t'\right)=
       \frac2\pi\arccos\left(1/R\right).
\]}
{On a mesh interval $\left[a,a+p\right]$, the difference between the
trapezoidal integral and the actual integral is the integral of
$\left(v-a\right)\left(a+p-v\right)/2$ against $\dd f_t'\left(v\right)$.  This kernel has absolute
value at most $p^2/8$.  In detail, two integrations by parts,
valid also for the finite signed measure $\dd f_t'$, give}
\[
 \frac p2\bigl(f_t(a)+f_t(a+p)\bigr)
       -\int_a^{a+p}f_t(v)\dd v
 =\frac12\int_a^{a+p}(v-a)(a+p-v)\dd f_t'(v).
\]
{The vanishing at $Rt$ removes the far endpoint after the
intervals are summed.  Summing and dividing by $p$
therefore gives}
{\[
 \left|\frac{f_t\left(0\right)}2+\sum_{j\ge1}f_t\left(jp\right)
       -\frac1p{\color{revision}\int_0^{+\infty}} f_t\left(v\right)\dd v\right|
       \le\frac p8\operatorname{Var}\left(f_t'\right)
       =\frac p{4\pi}\arccos\left(1/R\right).
\]}
{Since $f_t\left(0\right)=\delta t/\pi$, the exact action identity yields}
{\begin{equation}\label{eq:Bq-action-sums}
 \begin{aligned}
 \sum_{j\ge0}f_t\left(jp\right)
       &=\frac D{8p}t^2+\frac\delta{2\pi}t+O_{R,p}\left(1\right),\\
 \sum_{j\ge1}f_t\left(jp\right)
       &=\frac D{8p}t^2-\frac\delta{2\pi}t+O_{R,p}\left(1\right).
 \end{aligned}
\end{equation}}
{The different signs reflect the Neumann angular row at $j=0$
and the first Dirichlet angular row at $j=1$.  In particular,
the linear term is exactly $\pm f_t(0)/2$, and is independent of
the angular spacing $p$ before physical scaling.  This endpoint
term is the contribution that was lost when a monotone angular
sum was replaced by its integral in the all-frequency transfer.}

{\color{revision}The formula in the next theorem is the classical
two-term Weyl law \eqref{eq:intro-two-term-weyl}, specialized to this
geometry; it is not a new form of the Weyl formula.  Its general
coefficient and signs are recorded, for example, in
\cite[Section~1.1, formula~(1.1.2)]{IvriiSurvey}.  Ivrii's
general smooth-boundary theorem~\cite{Ivrii} assumes that the
periodic billiard trajectories have measure zero; see also
\cite[formula~(1.1.4)]{IvriiSurvey}.  We do not apply that theorem
directly to the corners or the slit.  Instead the separated phase
counts and the discrete cancellation proved above establish the
law here for every fixed aperture, including the retained slit,
and identify the two sources of its boundary coefficient.  This
is the role of the following proof within the present argument.
Its $o(\sqrt E)$ remainder gives no explicit frequency threshold
and is not asserted to be uniform as the aperture varies.}
{\begin{theorem}[Two-term Weyl law for a fixed annular sector]
\label{thm:Bq-full-perimeter}
Let $S=S_{r,R_{\rm phys},\beta}$, with $0<r<R_{\rm phys}$ and
$0<\beta\le2\pi$, retaining the slit at the endpoint.  Put
\[
 A=\left|S\right|=\frac\beta2\left(R_{\rm phys}^2-r^2\right),\qquad
 P=\beta\left(R_{\rm phys}+r\right)+2\left(R_{\rm phys}-r\right).
\]
At the slit, $P$ counts the two distinct boundary faces, each of
length $R_{\rm phys}-r$, as prescribed by the sector form.
For either the strict or the inclusive counting convention,
\begin{equation}\label{eq:Bq-full-perimeter}
 \begin{aligned}
 N_S^{\mathrm N,<}\left(E\right),\ N_S^{\mathrm N,\le}\left(E\right)
       &=\frac A{4\pi}E+\frac P{4\pi}\sqrt E+o_S\left(\sqrt E\right),\\
 N_S^{\mathrm D,<}\left(E\right),\ N_S^{\mathrm D,\le}\left(E\right)
       &=\frac A{4\pi}E-\frac P{4\pi}\sqrt E+o_S\left(\sqrt E\right),
 \qquad E\longrightarrow+\infty.
 \end{aligned}
\end{equation}
\end{theorem}}
\begin{proof}
{For normalized radii, separation of variables gives the inclusive
Neumann sum over $j\ge0$ and the Dirichlet sum over $j\ge1$.
Combining \eqref{eq:Bq-sampled-corrections} and
\eqref{eq:Bq-action-sums} gives}
{\[
 N_S^{\sigma,\le}\left(t^2\right)=\frac D{8p}t^2
       \mathbin{\pm}\left(\frac{R+1}{4p}
                     +\frac\delta{2\pi}\right)t+o_{R,p}\left(t\right),
\]}
{with the plus sign for Neumann and the minus sign for Dirichlet.}
{The linear coefficient has thus been obtained as a sum:
$(R+1)/(4p)$ comes from the mean radial phase corrections in
\eqref{eq:Bq-sampled-corrections}, and $\delta/(2\pi)$ comes from
the angular endpoint in \eqref{eq:Bq-action-sums}.  Neither term
is absorbed into an unspecified error.}

{For a fixed angular order every radial eigenvalue is simple.
Thus the strict and inclusive radial counts differ by at most one,
and a difference requires $\gamma_t^\sigma\left(jp\right)$ to be an integer.
On each compact order interval, local uniformity of the phase
shift forces $f_t\left(jp\right)+c$ into a shrinking neighborhood of the
integers.  Lemma~\ref{lem:Bq-discrete-averaging} shows that there
are $o\left(t\right)$ such indices.  The endpoint intervals contribute at
most $O\left(\eta t+1\right)$ further coincidences.  Shrinking them proves
that the total difference between the two sector conventions is
$o\left(t\right)$.}

{Finally take $R=R_{\rm phys}/r$ and $t=r\sqrt E$.
The quadratic coefficient becomes $A/\left(4\pi\right)$, the coefficient
$\left(R+1\right)/\left(4p\right)$ becomes the circular-arc term
$\beta\left(R_{\rm phys}+r\right)/\left(4\pi\right)$, and $\delta/\left(2\pi\right)$ becomes
the term $2\left(R_{\rm phys}-r\right)/\left(4\pi\right)$ for the two radial sides.
This proves \eqref{eq:Bq-full-perimeter}, including the retained
slit, whose two faces both belong to the boundary.}
\end{proof}

{To see the inversion explicitly, put $\alpha=A/(4\pi)$ and
$\gamma=P/(4\pi)$.  At a Dirichlet eigenvalue $E_k=\lambda_k$,
$N^{\mathrm D,<}(E_k)\le k-1<k\le N^{\mathrm D,\le}(E_k)$;
at a Neumann eigenvalue $E_k=\mu_k$,
$N^{\mathrm N,<}(E_k)\le k<k+1\le N^{\mathrm N,\le}(E_k)$.
The laws for both conventions therefore give}
\[
 k=\alpha E_k\mathbin{\pm}\gamma\sqrt{E_k}
            +o_S(\sqrt{E_k}),
\]
{with the plus sign for Neumann.  The leading term first implies
$E_k/W_k\to1$, where $W_k=4\pi k/A$.  Subtracting $\alpha E_k$
and then replacing $\sqrt{E_k}$ by
$\sqrt{W_k}+o(\sqrt{W_k})$ gives}
{\[
\begin{aligned}
 W_k\left(S\right)-\mu_k\left(S\right)&=\frac PA\sqrt{W_k\left(S\right)}
                         +o_S\left(\sqrt{W_k\left(S\right)}\right),\\
 \lambda_k\left(S\right)-W_k\left(S\right)&=\frac PA\sqrt{W_k\left(S\right)}
                         +o_S\left(\sqrt{W_k\left(S\right)}\right).
 \end{aligned}
\]}
{In particular,}
{\begin{equation}\label{eq:Bq-arc-versus-perimeter}
 \frac PA=\frac2{R_{\rm phys}-r}
              +\frac4{\beta\left(R_{\rm phys}+r\right)}.
\end{equation}}
{\color{editorial}The first summand is the circular-arc coefficient attained by
\eqref{eq:Dq-sigma-asymptotic} and \eqref{eq:Nq-gap-asymptotic};
the second comes from the angular endpoint contribution after
discrete averaging.  The main counting theorems use a one-sided
sum--integral comparison at every frequency, which does not retain
this additional term.  Monotonicity alone cannot supply it.  For
example, for the decreasing step $q(v)=\mathbf1_{[0,p-\varepsilon]}(v)$,
$0<\varepsilon<p$, the inclusive left sum satisfies}
\[
 p\sum_{j\ge0}q(jp)-\int_0^\infty q(v)\dd v
      =p-(p-\varepsilon)=\varepsilon\longrightarrow0.
\]
{Thus a positive fixed endpoint gain cannot be deduced for every
decreasing count merely from left-endpoint sampling.  For the
actual sector spectrum the high-frequency gain follows instead
from the phase averaging and the smooth action sum.}

\subsection{\revblue{A low-index obstruction to the full coefficient}}
{The full coefficient $P/A$ cannot in general be inserted into a
remainder-free square-root gap at every index.  To test such an
upgrade it is enough to examine the first eigenvalues in a family
with a simple geometric limit.  Consider the sectors}
{\[
 S_L=S_{L,L+1,1/L},\qquad L\longrightarrow+\infty.
\]}
{The map from the unit square $Q=\left(0,1\right)^2$,}
{\[
 \left(x,y\right)\longmapsto
 \left(\left(L+x\right)\cos\left(y/L\right),\left(L+x\right)\sin\left(y/L\right)\right),
\]}
{has pullback metric $\dd x^2+\left(1+x/L\right)^2\dd y^2$.
Writing $w_L=1+x/L$, its energy and norm are}
{\[
 q_L\left[u\right]=\int_Q\left(w_L\left|u_x\right|^2+w_L^{-1}\left|u_y\right|^2\right)\dd x\dd y,
 \qquad \left\|u\right\|_L^2=\int_Qw_L\left|u\right|^2\dd x\dd y.
\]}
{Let $b_L=1+1/L$, so $1\le w_L\le b_L$ and
$b_L^{-1}\le w_L^{-1}\le1$.  If $q_Q[u]=\int_Q|\nabla u|^2$
and $\|u\|_Q$ is the ordinary square norm, then}
\[
 b_L^{-1}q_Q[u]\le q_L[u]\le b_Lq_Q[u],\qquad
 \|u\|_Q^2\le\|u\|_L^2\le b_L\|u\|_Q^2.
\]
{Consequently every nonzero trial function satisfies}
\[
 b_L^{-2}\frac{q_Q[u]}{\|u\|_Q^2}
 \le\frac{q_L[u]}{\|u\|_L^2}
 \le b_L\frac{q_Q[u]}{\|u\|_Q^2}.
\]
{The map is a diffeomorphism of the polar rectangle and preserves
the relevant form domains.  Apply min--max on $H^1(Q)$ for
Neumann conditions and on $H_0^1(Q)$ for Dirichlet conditions.
Since $b_L\to1$, every fixed-index eigenvalue converges to that
of $Q$.  In particular, the first nonzero Neumann square modes
are $\cos(\pi x)$ and $\cos(\pi y)$, and the first Dirichlet mode
is $\sin(\pi x)\sin(\pi y)$, giving}
{\[
 \mu_1\left(S_L\right)\longrightarrow\pi^2,\qquad
 \lambda_1\left(S_L\right)\longrightarrow2\pi^2.
\]}
{At the same time}
{\[
 A_L=1+\frac1{2L}\longrightarrow1,
 \qquad P_L=4+\frac1L\longrightarrow4,
 \qquad W_1\left(S_L\right)\longrightarrow4\pi.
\]}
{If an all-index gap were to have the form
$W_k-\mu_k\ge c\sqrt{W_k}$ or
$\lambda_k-W_k\ge c\sqrt{W_k}$, the first eigenvalue along this
family would require, respectively,}
\[
 c\le\frac{W_1(S_L)-\mu_1(S_L)}{\sqrt{W_1(S_L)}},\qquad
 c\le\frac{\lambda_1(S_L)-W_1(S_L)}{\sqrt{W_1(S_L)}}.
\]
{Their limiting ceilings are}
{\[
c_{\mathrm N}\le\frac{4\pi-\pi^2}{\sqrt{4\pi}}<4,
 \qquad
 c_{\mathrm D}\le\frac{2\pi^2-4\pi}{\sqrt{4\pi}}<4.
\]}
{Both are strictly below $P_L/A_L\to4$.  Hence choosing the
geometry-dependent coefficient $c=P_L/A_L$ fails already at
$k=1$ for all sufficiently large $L$, for each boundary condition.
The same obstruction
holds in the quadratic inversion normalization: the inequalities
$W_k-\mu_k\ge a\sqrt{\mu_k}$ and
$\lambda_k-W_k\ge a\sqrt{\lambda_k}$ instead have limiting
first-index ceilings}
{\[
 a_{\mathrm N}\le4-\pi<4,
 \qquad a_{\mathrm D}\le\sqrt2\left(\pi-2\right)<4.
\]}
{For example, the Neumann ceiling here is obtained by dividing
$4\pi-\pi^2$ by $\pi$, and the Dirichlet ceiling by dividing
$2\pi^2-4\pi$ by $\sqrt2\pi$.  These are the exact
normalizations arising when a counting correction
$\pm(Aa/(4\pi))\sqrt E$ is inverted.  Thus the obstruction
applies to directly replacing $a_{\mathrm D}$ or $a_{\mathrm N}$
in the main theorems by $P/A$, as well as to the
$\sqrt{W_k}$ formulation.}

{\color{editorial}The logical distinction is now explicit.
Theorems~\ref{thm:dirichlet} and~\ref{thm:main} give computable,
strict improvements throughout the spectrum.  The optimized
bounds of Sections~\ref{sec:dirichlet-quantitative} and
\ref{sec:neumann-quantitative} retain that finite-frequency
control and recover the full circular-arc coefficient at high
frequency.  Proposition~\ref{prop:Nq-actual-coefficient} proves
that this coefficient is sharp for the continuous count.
Theorem~\ref{thm:Bq-full-perimeter} identifies the additional
radial-side term in the actual sector asymptotic, while the
present example prevents a remainder-free all-index upgrade to
the full coefficient.  The optimizations claimed earlier are
within the stated trial and phase families.}

\section{Products and further questions}\label{sec:products}
The linear two-dimensional Weyl count makes an annular sector a
useful factor in higher dimensions. The argument below is a strict
Neumann analogue of Laptev's Dirichlet product
argument~\cite[Theorem~2.8]{Laptev}; its averaged input is the
Neumann trace estimate~\cite[Theorem~3.1]{Laptev}, proved again
below for the present Laplacian setting. The same strict-count product argument is used
in~\cite[Corollary~1.2 and its proof]{DaiSunAnnuli}.

{\begin{corollary}\label{cor:product}
Let $S=S_{r,R,\beta}$ satisfy the hypotheses of
Theorem~\ref{thm:main}. Let $\Omega\subset\R^d$, $d\ge1$, be a
bounded connected Lipschitz domain. With Neumann conditions on
the whole product boundary,
\[
 \mu_k\paren{S\times\Omega}
 <4\pi^2\paren{\frac{k}{\omega_{d+2}\abs{S}\abs\Omega}}^{2/\paren{d+2}},
 \qquad k\ge1.
\]
At $\beta=2\pi$ the slit is retained in the product form.
\end{corollary}}

\begin{proof}
{Write the eigenvalues of $\Omega$ as $\nu_\ell$, $\ell\ge0$,
with $\nu_0=0$, and put $N_S^{<}\paren{z}=0$ for $z\le0$.
The tensor products of the two factor eigenbases are a complete
orthonormal eigenbasis for the product form, with eigenvalues
$\mu_j\paren{S}+\nu_\ell$. The compactness of the form embeddings
on the factors gives a discrete product spectrum; for a slit sector
the same statement holds on its polar rectangle. Product separation
and Theorem~\ref{thm:main} therefore give, for $E>0$,}
{\[
\begin{aligned}
 N_{S\times\Omega}^{<}\paren{E}
 &=\sum_{\ell\ge0}N_S^{<}\paren{E-\nu_\ell}\\
 &>\frac{\abs{S}}{4\pi}
             \sum_{\ell\ge0}\pos{E-\nu_\ell}.
\end{aligned}
\]}
{Only finitely many summands are nonzero. The inequality remains
strict because the summand $\ell=0$ has positive argument $E$.}

{For clarity, we give the sharp lower bound on the Riesz mean of
the other factor; it is the plane-wave argument underlying the
Neumann spectral-sum estimates~\cite{Kroger,Laptev}. Let
$\phi_\ell$ be an orthonormal Neumann eigenbasis of $L^2\paren\Omega$,
and define $e_\xi\paren{y}=\paren{2\pi}^{-d/2}e^{i\xi\cdot y}$
on $\Omega$. This function belongs to the Neumann form domain.
Its spectral expansion satisfies}
{\[
\begin{aligned}
 \sum_{\ell\ge0}\abs{\left\langle e_\xi,\phi_\ell\right\rangle}^2
 &=\paren{2\pi}^{-d}\abs\Omega,\\
 \sum_{\ell\ge0}\nu_\ell
          \abs{\left\langle e_\xi,\phi_\ell\right\rangle}^2
 &=\int_\Omega\abs{\nabla e_\xi}^2\dd y
 =\paren{2\pi}^{-d}\abs\Omega\abs\xi^2.
\end{aligned}
\]}
{Replacing $E-\nu_\ell$ by its positive part increases the sum, so}
{\[
 \sum_{\ell\ge0}\pos{E-\nu_\ell}
       \abs{\left\langle e_\xi,\phi_\ell\right\rangle}^2
 \ge\paren{2\pi}^{-d}\abs\Omega\paren{E-\abs\xi^2}.
\]}
{Integrate over $\abs\xi<\sqrt E$. Plancherel's theorem applied
to the zero extension of $\phi_\ell$ gives}
{\[
 \int_{\abs\xi<\sqrt E}
       \abs{\left\langle e_\xi,\phi_\ell\right\rangle}^2\dd\xi
 \le\int_{\R^d}\abs{\left\langle e_\xi,\phi_\ell\right\rangle}^2\dd\xi=1.
\]}
It follows that
{\[
\begin{aligned}
 \sum_{\ell\ge0}\pos{E-\nu_\ell}
 &\ge\frac{\abs\Omega}{\paren{2\pi}^d}
            \int_{\abs\xi<\sqrt E}\paren{E-\abs\xi^2}\dd\xi\\
 &=\frac{\abs\Omega}{\paren{2\pi}^d}
        d\omega_d\int_0^{\sqrt E}\paren{E-\rho^2}\rho^{d-1}\dd\rho\\
 &=\frac{2\omega_d\abs\Omega}{\paren{d+2}\paren{2\pi}^d}
                                   E^{\paren{d+2}/2}.
\end{aligned}
\]}
{Since $\omega_{d+2}=2\pi\omega_d/\paren{d+2}$, we conclude that}
{\[
 N_{S\times\Omega}^{<}\paren{E}
 >\frac{\omega_{d+2}\abs{S}\abs\Omega}{\paren{2\pi}^{d+2}}
                                  E^{\paren{d+2}/2}.
\]}
{Taking the right-hand side equal to $k$ proves the eigenvalue bound
with the same indexing argument as in Theorem~\ref{thm:main}.}
\end{proof}

{\color{editorial}For an interval factor, Corollary~\ref{cor:product} applies to
annular-sector cylinders; no individual P\'olya assumption on the
other factor is required. The Dirichlet product inequality follows
from~\cite{TangWangDirichlet,Laptev}. Extensions to curved strips
or mixed boundary conditions would require analogues of the exact
action identity and the directional rounding estimates.}

\appendix
\section{\texorpdfstring{\textcolor{revision}{Elementary endpoint estimates}}{Elementary endpoint estimates}}\label{sec:elementary-estimates}
{
We collect the constant estimates used in the frequency and width cover.
All square-root comparisons below have positive sides and are verified
by squaring; this convention will be used throughout the proof.}
{\begin{lemma}[\textcolor{revision}{Constants for the frequency and width cover}]\label{lem:endpoint-estimates}
Set
\[
 d=\frac52,\qquad x_0=\frac{15}{4},\qquad
 X_0=\frac{27}{5},\qquad x_B=\frac{22}{5},
\]
and write the three frequency-dependent errors as
\[
 r_\delta\paren{x;c}=\frac{2c}{\sqrt{K_\delta\paren{x}}},\qquad
 s_\delta\paren{x}=\sqrt{\frac{2\pi}{9x}}
              \delta^{3/4}\paren{\delta+2}^{1/4},\qquad
 \ell_{\delta,u}\paren{x}=\frac{2\delta I\paren{u}}{x}.
\]
Then
\[
 \frac{333}{106}<\pi<\frac{355}{113}<\frac{22}{7},\qquad
 \frac{11}{50}<I\paren{9/10}<\frac{111}{500},\qquad
 I\paren{3/4}<\frac{129}{1000}.
\]
The low-frequency endpoint satisfies
\[
 a_d>\frac{211}{100},\qquad
 2a_dx_0-\frac{2+d}{4}x_0^2>\frac3{640}.
\]
On the thin-width phase window its fractional initial phase lies
between $1/9$ and $18/25$, and
\[
 C_+\paren{1/9}<\frac1{10},\qquad
 C_+\paren{18/25}<\frac1{10},\qquad
 c\paren{27/20}<\frac14.
\]
The window constants obey
\[
 r_d\paren{x_0;1/10}<\frac7{25},\qquad
 \sqrt{\frac{2\pi}{9x_0}}<\frac{54}{125}.
\]
At $\paren{\delta,x}=\paren{d,X_0}$, the generic
tail errors and their width derivatives satisfy
\[
 \begin{aligned}
 r_d\paren{X_0;1/3}&<\frac{21}{25},&
 s_d\paren{X_0}&<\frac{21}{20},&
 \ell_{d,9/10}\paren{X_0}&<\frac{103}{500},\\
 \partial_\delta r_d\paren{X_0;1/3}&>\frac{11}{200},&
 \partial_\delta s_d\paren{X_0}&>\frac{73}{200},&
 \partial_\delta\ell_{d,9/10}\paren{X_0}&>\frac{11}{135}.
 \end{aligned}
\]
Here and below $\partial_\delta r_d$ means the derivative evaluated
at $\delta=d$; the frequency and the coefficient are held fixed.
For the wide-width argument,
\[
 \begin{aligned}
 r_{10}\paren{x_B;1/3}&<\frac{737}{1000},&
 s_{10}\paren{x_B}&<\frac{417}{100},&
 \ell_{10,3/4}\paren{x_B}&<\frac{587}{1000},\\
 \partial_\delta s_{10}\paren{x_B}&<\frac25,&
 \partial_\delta\ell_{10,3/4}\paren{x_B}&<\frac3{50},&
 p\paren{10}&>x_B.
 \end{aligned}
\]
\end{lemma}}
\begin{proof}
{For $0<z<1$, the alternating expansion}
{\[
 \arctan z=z-\frac{z^3}{3}+\frac{z^5}{5}-\cdots
\]}
{places $\arctan z$ between consecutive truncations. Machin's identity
$\pi=16\arctan\paren{1/5}-4\arctan\paren{1/239}$ follows from
$\tan\paren{4\arctan\paren{1/5}}=120/119$ and the tangent
subtraction formula, which gives
$\tan\paren{4\arctan\paren{1/5}-\arctan\paren{1/239}}=1$.
The angle is in $\paren{0,\pi/2}$, so it equals $\pi/4$.
The truncation bounds now give}
\[
 \begin{aligned}
 \pi&>
 16\paren{\frac15-\frac1{3\cdot5^3}
       +\frac1{5\cdot5^5}-\frac1{7\cdot5^7}}-\frac4{239}
 >\frac{333}{106},\\
 \pi&<
 16\paren{\frac15-\frac1{3\cdot5^3}
       +\frac1{5\cdot5^5}-\frac1{7\cdot5^7}
       +\frac1{9\cdot5^9}}
 -4\paren{\frac1{239}-\frac1{3\cdot239^3}}
 <\frac{355}{113}.
 \end{aligned}
\]
{The first display uses a four-term lower bound for $\arctan\paren{1/5}$
and the one-term upper bound for $\arctan\paren{1/239}$;
the second uses a five-term upper bound and a two-term lower bound.
Thus the inequality directions account for both signs in Machin's identity.}

For $w=\sqrt{19}$, we have $87/20<w<109/25$ and
\[
 I\paren{9/10}
 =\frac1{24}+\frac{63/w-2\arctan\paren{w/9}}{24\pi}.
\]
The inequalities $z-z^3/3<\arctan z<z$ for $0<z<1$ give
\[
 \begin{aligned}
 \frac{63}{w}-2\arctan\paren{w/9}
 &>\frac{1575}{109}-\frac{218}{225}
   >\frac{107}{25}\frac{22}{7},\\
 \frac{63}{w}-2\arctan\paren{w/9}
 &<\frac{420}{29}-\frac{29}{30}
      +\frac23\paren{\frac{109}{225}}^3
   <\frac{541}{125}\frac{333}{106}.
 \end{aligned}
\]
They prove the asserted two-sided bound for $I\paren{9/10}$.
For $w=\sqrt7>66/25$, put $z=\paren{3-w}/\paren{3+w}<3/47$.
Since $\arctan\paren{w/3}=\pi/4-\arctan z$,
\[
 I\paren{3/4}
 =\frac1{48}+\frac{21/w+2\arctan z}{24\pi}
 <\frac1{48}+\frac{175/22+6/47}{24\paren{333/106}}
 <\frac{129}{1000}.
\]

For the logarithm, the positive power series for
$\log\paren{\paren{1+z}/\paren{1-z}}$ at $z=5/9$ gives
\[
 \log\paren{7/2}
 <2\paren{\frac59+\frac13\paren{\frac59}^3}
   +\frac{2\paren{5/9}^5}{5\paren{1-\paren{5/9}^2}}
 =\frac{76915}{61236}<\frac{63}{50}.
\]
{The tail estimate above follows by replacing the denominators
$5,7,9,\ldots$ by $5$ in the remaining positive terms and summing
the resulting geometric series. Since $D=d\paren{d+2}=45/4$,}
{\[
 a_d^2=\frac{45/4}{2\log\paren{7/2}}
 >\frac{45/4}{2\paren{63/50}}=\frac{125}{28}
 >\paren{\frac{211}{100}}^2.
\]}
{Substitution of $x_0=15/4$ then gives}
{\[
 2a_dx_0-\frac{2+d}{4}x_0^2
 >2\frac{211}{100}\frac{15}{4}
    -\frac98\paren{\frac{15}{4}}^2=\frac3{640}.
\]}

{The endpoint bounds for the shifted initial phase on the thin-width window are}
\[
 \frac7{2\pi}-1>\frac5{44}>\frac19,
 \qquad
 \frac{27}{5\pi}-1<\frac{113}{157}<\frac{18}{25}.
\]
{The unwrapped initial phase $x/\pi-3\delta/\paren{8\pi x}$
increases with $x$ and decreases with $\delta$. Its extrema on
$x_0\le x\le X_0$, $0<\delta\le d$ are therefore bounded
by the two displayed endpoints. For the left phase, use}
{\[
 C_+\paren{a}=\frac{\sqrt2}{3}\paren{\frac12-a}^2
 \frac{2\sqrt{1/2+a}+\sqrt{2a}}
       {\paren{\sqrt{1/2+a}+\sqrt{2a}}^2},
 \qquad 0\le a\le\frac12,
\]}
{together with}
\[
 \frac{39}{50}<\sqrt{\frac{11}{18}}<\frac{79}{100},\qquad
 \frac{47}{100}<\sqrt{\frac29}<\frac{12}{25},\qquad
 \sqrt2<\frac{177}{125}.
\]
It follows that
\[
 C_+\paren{1/9}
 <\frac{59}{125}\frac{49}{324}\frac{103}{50}\frac{16}{25}
 =\frac{595546}{6328125}<\frac1{10}.
\]
For the right phase,
\[
 C_+\paren{18/25}^2
 =\frac89\paren{\frac{11}{50}}^3<\frac1{100}.
\]
The half-period monotonicity of $C_+$ now controls every phase
in the window.

For the turning cap, let $w=\sqrt{871}$, so
$59/2<w<148/5$, and define
\[
 z=\frac{w-27}{w+27}>\frac5{113}.
\]
Then
\[
 \arctan\paren{w/27}
 =\frac\pi4+\arctan z
 >\frac{157}{200}+\frac5{113}
          -\frac13\paren{\frac5{113}}^3
 >\frac{829}{1000}.
\]
At $v=27/20$ and the evaluation argument $2$, direct substitution
in $A_v+E_v$ gives
\[
 \begin{aligned}
 A_{27/20}\paren{2}+E_{27/20}\paren{2}
 &=\frac w{20}-\frac{27}{20}\arctan\paren{w/27}
                 +\frac{10785}{871w}\\
 &<\frac{148}{100}-\frac{27}{20}\frac{829}{1000}
                 +\frac{21570}{51389}
 <\frac{157}{200}<\frac\pi4.
 \end{aligned}
\]
The infimum defining $c$ is therefore less than $1/4$.

Write $K_d\paren{x}=k\paren{x}/\pi$. At $x=x_0$,
$k=431/252>17/10$, so
\[
 r_d\paren{x_0;1/10}^2
 <\frac{22/7}{25\paren{17/10}}<\paren{\frac7{25}}^2,
 \qquad
 \frac{2\pi}{9x_0}<\frac{176}{945}
                  <\paren{\frac{54}{125}}^2.
\]
At $x=X_0$, put
\[
 k=\frac{54}{35}+\frac{575}{1296},\qquad
 n=\frac{108}{245}-\frac{115}{648}.
\]
{These are exactly}
{\[
 k=\frac{90109}{45360},\qquad n=\frac{41809}{158760},
 \qquad\frac{99}{50}<k<\frac{199}{100},\qquad
 n>\frac{263}{1000}.
\]}
{To identify the derivative, write
$K_\delta\paren{x}=k_\delta\paren{x}/\pi$ with
$k_\delta\paren{x}=x/\paren{1+\delta}+23\delta/\paren{24x}$.
Differentiating $r_\delta\paren{x;1/3}
=2\sqrt\pi/\paren{3\sqrt{k_\delta\paren{x}}}$ gives}
{\[
 \partial_\delta r_\delta\paren{x;1/3}
 =\frac{\sqrt\pi}{3k_\delta\paren{x}^{3/2}}
    \paren{\frac{x}{\paren{1+\delta}^2}-\frac{23}{24x}}.
\]}
{At $\paren{\delta,x}=\paren{d,X_0}$ the last bracket is $n$.
Since $\sqrt\pi>177/100$ and $k^{3/2}<281/100$, we obtain}
\[
 \begin{aligned}
 r_d\paren{X_0;1/3}^2
 &<\frac{4\paren{22/7}}{9\paren{99/50}}
  <\paren{\frac{21}{25}}^2,\\
 \partial_\delta r_d\paren{X_0;1/3}
 &=\frac{\sqrt\pi n}{3k^{3/2}}
  >\frac{177}{100}\frac{263}{1000}\frac{100}{843}
  =\frac{15517}{281000}>\frac{11}{200}.
 \end{aligned}
\]
{The required inequality $k^{3/2}<281/100$ follows simply by
cubing $k<199/100$ and comparing
$\paren{199/100}^3<\paren{281/100}^2$.
The derivative of the outer error has the explicit form}
{\[
 \partial_\delta s_\delta\paren{x}
 =\sqrt{\frac{2\pi}{9x}}
   \frac{\delta+3/2}{\delta^{1/4}\paren{\delta+2}^{3/4}}>0.
\]}
{Indeed, this follows from logarithmic differentiation of
$\delta^{3/4}\paren{\delta+2}^{1/4}$. Substituting $d=5/2$
and $X_0=27/5$, and using $\sqrt5<9/4$, yields}
\[
 \begin{aligned}
 s_d\paren{X_0}^2
 &=\frac{25\pi\sqrt5}{162}
  <\frac{275}{252}<\paren{\frac{21}{20}}^2,\\
 \paren{\partial_\delta s_d\paren{X_0}}^2
 &=\frac{640\pi}{6561\sqrt5}
  >\frac{640\paren{157/50}}{6561\paren{9/4}}
  >\paren{\frac{73}{200}}^2.
 \end{aligned}
\]
The bounds for $I\paren{9/10}$ give
\[
 \ell_{d,9/10}\paren{X_0}
 <\frac{111}{540}<\frac{103}{500},\qquad
 \partial_\delta\ell_{d,9/10}\paren{X_0}
 >\frac{11}{135}.
\]
{In particular, their sum exceeds $1/2$ by the explicit amount}
{\[
 \frac{11}{200}+\frac{73}{200}+\frac{11}{135}
 =\frac12+\frac1{675}.
\]}

{At the wide-width endpoint, direct substitution gives}
{\[
 k_{10}\paren{x_B}=\frac25+\frac{575}{264}=\frac{3403}{1320},
 \qquad K_{10}\paren{x_B}=\frac{3403}{1320\pi}.
\]}
{The inequalities $219/20<\sqrt{120}<2191/200$ and
$\pi<22/7$ then give}
\[
 \begin{aligned}
 r_{10}\paren{x_B;1/3}^2
 &<\frac{38720}{71463}<\paren{\frac{737}{1000}}^2,\\
 s_{10}\paren{x_B}^2
 &=\frac{50\pi\sqrt{120}}{99}
  <\frac{313}{18}<\paren{\frac{417}{100}}^2,\\
 \paren{\partial_\delta s_{10}\paren{x_B}}^2
 &=\frac{2645\pi}{4752\sqrt{120}}
  <\frac{13225}{82782}<\frac4{25}.
 \end{aligned}
\]
{For the outer error and derivative, the intermediate rational
bounds in this display come from}
{\[
 \frac{50\paren{22/7}\paren{2191/200}}{99}=\frac{313}{18},
 \qquad
 \frac{2645\paren{22/7}}{4752\paren{219/20}}
 =\frac{13225}{82782}.
\]}
{The strict gaps are}
{\[
 \paren{\frac{417}{100}}^2-\frac{313}{18}=\frac1{90000},
 \qquad\frac4{25}-\frac{13225}{82782}
 =\frac{503}{2069550}.
\]}
{The phase bounds and the trial overlap are}
\[
 \ell_{10,3/4}\paren{x_B}
 <\frac{129}{220}<\frac{587}{1000},\qquad
 \partial_\delta\ell_{10,3/4}\paren{x_B}
 <\frac{129}{2200}<\frac3{50},\qquad
 p\paren{10}^2=\frac{182}{9}>\paren{\frac{22}{5}}^2.
\]
{Thus the sum of the two positive derivative bounds is less than
$2/5+3/50=23/50<1/2$. The high-envelope margin is}
{\[
 6-\frac{737}{1000}-\frac{417}{100}
       -\frac{587}{1000}-\frac9{20}=\frac7{125}>\frac1{20}.
\]}
\end{proof}

\noindent\textbf{Conflict of interest.}
The authors declare no conflict of interest.

\medskip
\noindent\textbf{Data availability.}
No datasets were generated or analysed during this study.

\medskip
\noindent\textbf{Declaration of AI-assisted writing.}
During the preparation of this work, the authors used ChatGPT and DeepSeek
for language polishing, grammar checking, verification of some computations,
and refining the exposition of certain technical arguments. After using
these tools, the authors carefully reviewed and edited the content as
needed and take full responsibility for the final version of this manuscript.

\end{document}